\documentclass[11pt,leqno,a4paper]{amsart}

\usepackage[]{hyperref}
\hypersetup{
    colorlinks=true,       
    linkcolor=red,          
    citecolor=blue,        
    filecolor=magenta,      
    urlcolor=cyan           
}

\usepackage{anysize}
\marginsize{4cm}{4cm}{3cm}{3cm}

\usepackage{amsmath}
\usepackage{amsfonts,amssymb}
\usepackage{enumerate}
\usepackage{mathrsfs}
\usepackage{amsthm}
\usepackage{booktabs}
\usepackage{multirow}
\usepackage{colortbl}

\newtheorem{theo}{Theorem}[section]
\newtheorem{lemm}[theo]{Lemma}
\newtheorem{prop}[theo]{Proposition}
\newtheorem{coro}[theo]{Corollary}

\newtheorem{rema}[theo]{Remark}
\newtheorem{defi}[theo]{Definition}
\newtheorem{nota}[theo]{Notation}

\usepackage{enumerate}

\DeclareMathOperator{\cnx}{div}

\DeclareMathOperator{\RE}{Re}

\DeclareMathOperator{\IM}{Im}

\DeclareMathOperator{\curl}{curl}

\DeclareMathOperator{\supp}{supp}

\DeclareSymbolFont{pletters}{OT1}{cmr}{m}{sl}
\DeclareMathSymbol{s}{\mathalpha}{pletters}{`s}

\def\B{B }

\def\blA{\bigl\Vert}
\def\brA{\bigr\Vert}

\def\defn{\mathrel{:=}}

\def\eps{\varepsilon}

\def\la{\left\lvert}
\def\lA{\left\lVert}
\def\le{\leq}
\def\les{\lesssim}

\def\mez{\frac{1}{2}}

\def\ra{\right\rvert}
\def\rA{\right\rVert}

\def\tdm{\frac{3}{2}}

\def\xC{\mathbf{C}}
\def\xN{\mathbf{N}}
\def\xR{\mathbf{R}}
\def\xS{\mathbf{S}}
\def\xT{\mathbf{T}}
\def\xZ{\mathbf{Z}}

\def\ba{\begin{align}}
\def\bad{\begin{aligned}}
\def\be{\begin{equation}}
\def\ea{\end{align}}
\def\ead{\end{aligned}}
\def\ee{\end{equation}}
\def\e{\eqref}

\numberwithin{equation}{section}

\begin{document}
 
\begin{abstract}
In  this paper we consider the Cauchy problem for the gravity 
water-wave equations, in a domain with flat bottom and in arbitrary space dimension. 
We prove that if the data are of size $\eps$ in a space 
of analytic functions which have a holomorphic 
extension in a strip of size $\sigma$, then the 
solution exists up to a time of size $C/\eps$ 
in a space of analytic functions having at time $t$ a 
holomorphic extension in a strip of size $\sigma - C'\eps t$. 
\end{abstract}

\title[Analytic Cauchy theory for water-waves]{Cauchy theory for the water waves system in an analytic framework}

\author{Thomas Alazard, Nicolas Burq and Claude Zuily}

\date{}


\maketitle

\section{Introduction}

The water-wave problem consists in describing, by means of the Euler equations, 
the dynamics of the free surface of a fluid. 
There are many different equations associated with this problem. 
Indeed, there are many different factors that dictate the dynamics of water-waves: 
The equations may be incompressible or not, irrotationnal or not, the fluid may have a fixed or 
moving bottom, and the restoring forces may be determined by gravitation or surface tension. 
The study of these equations has received a lot of attention during the last decades and 
there are now many cases in which the mathematical analysis is well developed. 
In particular, there are many recent results 
concerning the well-posedness of the water-waves equations 
in Sobolev spaces in large time, including global existence results 
(see \cite{Wu09,Wu10,GMS,AlDe,IonescuPusateri,HIT,XuechengWang-FD2,Zheng2019longterm} for the equations 
without surface tension). 

In addition to the analysis of the Cauchy problem, another line of research is the 
mathematical justification of the derivation of approximate equations describing 
water-waves dynamics in asymptotic regimes. 
The most famous examples are the equations introduced by Boussinesq and 
Korteweg-de Vries (see~\cite{Schneider-Wayne-2002,LannesLivre,Saut-IMPA} and references there-in). 
Kano and Nishida \cite{KanoNishida0,Kano-Nishida1} gave, in the two 
dimensional case, 
the first justification of the Friedrichs expansion for the water-waves equations in terms 
of the shallowness parameter (by definition this is the ratio of the 
mean depth to the wavelength). 
In order to guarantee the existence of the 
solution for the full equations, they used an abstract Cauchy-Kowalevski 
theorem in a scale of Banach spaces, so that analyticity of the initial data is required 
(see also Kano~\cite{Kano1986} and Kano-Nishida~\cite{Kano-Nishida2}). These results have been extended to include initial data belonging to usual Sobolev spaces, 
by Craig~\cite{Craig1985}, Iguchi~\cite{IguchiCPDE,Iguchi-JDE2018}, Bona, Lannes and Saut~\cite{BLS} or Alvarez-Samaniego and Lannes~\cite{ASL} (see \cite{LannesLivre} for more references).

The study of various nonlinear partial differential equations 
in spaces of analytic functions has also received a great   attention. 
We can mention the 
well-posedness results in analytic spaces by 
Kato and Masuda~\cite{Kato-Masuda-1986} which apply to many equations in fluid dynamics, the study of 
the Rayleigh-Taylor instability by 
Sulem-Sulem~\cite{SulemSulem-1985}, 
the study of the Cauchy problem for the semi-linear one dimensional Schr\" odinger equations 
by Bona-Grujic and Kalisch \cite{Bona-al}, 
Selberg-D.O.da Silva \cite{Sel-Da},  
the work on the KdV equation by Hayashi \cite{Haya}, Tesfahun \cite{Tesfa} 
and on the periodic BBM equation by Himonas-Petronilho \cite{Him-Petr}, 
the work by Kucavica-Vicol \cite{Kuca-Vic} on the Euler equation, 
the work 
on quasilinear wave equations and other quasilinear systems by Alinhac 
and M\'etivier \cite{Alinhac-Metivier-1984} and Kuksin-Nadirashvili~\cite{KuksinNadirashvili}, 
the work by Matsuyama and Ruzhansky~\cite{Matsuyama-2019} on the Kirchhoff equation, 
Gancedo-Granero-Belinch{\'o}n-Scrobogna~\cite{GanBelScr-2019} for the Muskat problem 
and the one of 
Pierre~\cite{Pierre-2018} for the MHD equations. 
We should also mention the recent works by Mouhot-Villani~\cite{MouhotVillani}, 
Bedrossian-Masmoudi-Mouhot~\cite{BeMaMo} 
and by Grenier-Nguyen-Rodnianski~\cite{GrNgRo} on the Landau damping for analytic and Gevrey data.

Inspired by the  pioneering works of Kano-Nishida, our goal is to revisit the analysis of the water-problem 
with analytic data, using tools and methods that we developed previously 
to study the Cauchy problem with rough initial data. Our main result in 
this direction states that the solutions remain analytic for large time intervals.

Let us now state our problem more precisely. 
We are mainly interested in the study of the Cauchy problem for 
the gravity wave system, in any space dimension. 
There are many possible formulation for this problem. 
Here we use the classical 
Eulerian formulation and work with the so-called Craig-Sulem-Zakharov formulation, following \cite{CrSu,Zakharov1968}. 
In this formulation, there are two unknowns: (i) the free surface elevation $\eta$ and (ii) the trace $\psi$ 
of the velocity potential on the free surface. These two unknowns depend on 
the time variable $t$ and the horizontal space variable $x$. 
Motivated by possible applications to control theory (\cite{ABHK,Zhu1}), we 
assume below that $x$ belongs to the $d$ dimensional torus $\xT^d=(\xR/2\pi\xZ)^d$, 
which means that the solutions are $2\pi$-periodic in each variable $x_j$, $1\le j\le d$.

We consider initial data in spaces 
of functions having a holomorphic extension to a fixed strip in the complex plane. Furthermore, in view of applications to control theory, we assume 
that the fluid domain has a flat bottom and consider a source term on the bottom which belongs 
merely to a classical Sobolev space.
This problem can be written as follows. 
Given functions $\eta_0, \psi_0 $ on $\xT^d$, and $b$ on $\xR \times \xT^d$, solve the system
\begin{equation}\label{system}
\left\{
\begin{aligned}
&\partial_{t}\eta-G(\eta)(\psi,b)=0,\\
&\partial_{t}\psi+g \eta    + \frac{1}{2}\la \nabla_x \psi\ra^2  -\frac{1}{2}
\frac{\bigl(\nabla_x  \eta\cdot \nabla_x  \psi +G(\eta)(\psi,b) \bigr)^2}{1+|\nabla_x  \eta|^2}
= 0,\\
&\eta\arrowvert_{t=0}= \eta_0, \quad \psi \arrowvert_{t=0}= \psi_0.
 \end{aligned}
\right.
\end{equation}  
Here $G(\eta)$ denotes the Dirichlet-Neuman operator, which is defined as follows. Given 
$h>0$ and some fixed time $t$, let us introduce the fluid domain
$$
\Omega(t) = \{(x,y) \in \xT^d\times \xR: -h < y< \eta(t,x)\}.
$$
Then, define the potential $\phi= \phi(t,x,y)$ as the unique  solution of the problem
\begin{equation}\label{dphi}
\Delta \phi = 0 \quad\text{in } \Omega(t), \quad \phi\arrowvert_{y = \eta(t,x)} = \psi (t,x),\quad \partial_y \phi\arrowvert_{y= -h} = b(t,x). 
\end{equation}
Then the  Dirichlet-Neumann operator is defined by
\begin{equation}\label{def-DN}
\begin{aligned}
G(\eta)(\psi,b) (t,x)&=
\sqrt{1+|\nabla_x  \eta|^2}\,
\partial _n \phi\arrowvert_{y=\eta(t,x)}=(\partial_y \phi-\nabla_x  \eta \cdot \nabla_x \phi) \big\arrowvert _{y=\eta(t,x)}.
\end{aligned}
\end{equation}
Recall that the equations \eqref{system} are derived from the following Euler equations in a set with moving boundary:
\begin{equation}\label{Euler}
\left\{
\begin{aligned}
&\partial_t v + (v \cdot\nabla_{x,y})v = - \nabla_{x,y} P - gy\quad \text{in } \Omega = \{(t,x,y): (x,y)\in \Omega(t)\},\\
&\cnx_{x,y}\, v = 0, \quad \text{in } \Omega,\\
& \curl_{x,y} v = 0, \quad \text{in } \Omega,\\
&P\arrowvert_{y= \eta(t,x)} = 0,
\end{aligned}
\right.
\end{equation}
where we set $v(t,x,y) = \nabla_{x,y} \phi(t,x,y)$ 
and $\psi(t,x) = \phi(t,x, \eta(t,x)).$ We refer to~\cite{Bertinoro} for the proof that from solutions of~\e{system} one may define solutions 
of the   Euler system \eqref{Euler}.
As in~\cite{Kano-Nishida1}, 
we shall work in the spaces defined as follows. 
Given $d \geq 1$, $\sigma\ge 0$ and  $s\ge 0$, we define
\[
\mathcal{H}^{\sigma,s}(\xT^d) 
= \Big\{ u \in L^2(\xT^d): \Vert u \Vert_{ \mathcal{H}^{\sigma,s}}^2
\defn\sum_{\xi\in\xZ^d}e^{2\sigma \vert \xi \vert} 
\langle \xi \rangle^{2s} \vert \widehat{u}(\xi)\vert^2 <+ \infty \Big\}
\]
with
\[
\widehat{u}(\xi)=\int_{\xT^d}e^{-ix\cdot\xi}u(x)\, dx,\quad \langle \xi\rangle=\big(1+\la \xi\ra^2\big)^{1/2}.
\]
Several properties of these spaces are gathered in Appendix~\ref{AppendixA}.

Roughly speaking,  the main result of this paper asserts 
that if the norms of the data $\eta_0, \psi_0$ in such spaces 
and 
the norm of $b$  in some Sobolev space are of size $\eps>0$,  
then our system has a unique solution in these spaces up to the time $c_*/\eps$ for some  $c_*>0$. 
It is classical since the work of Kato and Masuda~\cite{Kato-Masuda-1986} that, 
for solutions with analytic initial data, the width of the strip of analyticity might decrease with time. 
The main novelty here is that we show that for small data of size $\eps$, 
the decrease is at most linear in $\eps$. 
To prove this result, we cannot rely  
on an abstract Cauchy-Kowalevski theorem (as the ones introduced by 
Nirenberg~\cite{Nirenberg-1972}, Ovsjannikov~\cite{Ovsjannikov-1973}, 
Nishida~\cite{Nishida-1977} or Baouendi-Goulaouic~\cite{BaouendiGoulaouic-1977}; used 
by Ovsjannikov~\cite{Ovsjannikov-1974,Ovsjannikov-1976} 
and Castro-C\'ordoba-Fefferman-Gancedo-G\'omez-Serrano~\cite{CCFGS-Annals} to study the Cauchy problem for the 
water-waves equations). A key difference between our work and previous ones is that 
we shall use energy estimates in the spaces defined above, using  the methods introduced in \cite{AM,ABZ1,ABZ3} to study 
the water-waves equations. To achieve these estimates we begin  
by a precise analysis of the Dirichlet-Neumann 
operator in the spaces of analytic functions. This requires a careful 
study of elliptic equations with variable coefficients, which is of independent interest.

\subsection{The spaces of analytic functions and their characterizations.}
It is well known that functions in $\mathcal{H}^{\sigma,s}(\xT^d)$ can be expressed as the traces 
on the real of functions which are holomorphic in a strip of the form
$$S_\sigma= \{z\in \xC^d: \RE z\in \xT^d, \vert \IM z \vert <\sigma\}.$$ More precisely, for $U \in \mathcal{H}ol(S_\sigma)$ 
and  $\vert y \vert <\sigma$ we shall 
denote by $U_y$ the function from $\xT^d$ to $\xC$ defined by $x \mapsto  U(x+iy)$ (here $y\in \xR^d$ and 
$| y|$ denotes its Euclidean norm). 
Then, for any $u \in \mathcal{H}^{\sigma,s}(\xT^d)$, 
there exists $U\in \mathcal{H}ol(S_\sigma)$  such that $U_0 = u$ and
$$
\sup _{\vert y \vert<\sigma} \Vert U_y\Vert_{H_x^s(\xT^d)} \leq C\Vert u \Vert_{ \mathcal{H}^{\sigma,s}}.
$$
In Appendix~\ref{AppendixA} we prove a result which clarifies the converse. 

\begin{theo}\label{Gsigmas}
Let $\sigma>0$ and $s\geq 0$. 
\begin{itemize}
\item[\rm{(1)}] Let $U\in \mathcal{H}ol(S_\sigma)$ be 
such that $$M_0:= \sup_{\vert y \vert<\sigma} \Vert U_y\Vert_{H_x^s(\xT^d)}<+ \infty$$ and set $u = U_0$.
Then 
\begin{itemize}
\item[\rm{(i)}]  If $d=1$, then $u$ 
belongs to $\mathcal{H}^{\sigma,s}(\xT^d)$ and  $\Vert u \Vert_{\mathcal{H}^{\sigma,s}} \leq 2M_0.$

\item[\rm{(ii)}]  If $d\geq 2$, 
then $u$ belongs to $\mathcal{H}^{\delta,s}(\xT^d)$ for any $\delta<\sigma$ and there exists a constant 
$C_\delta>0 $ such 
that $\Vert u \Vert_{\mathcal{H}^{\delta,s}} \leq C_\delta M_0$.
\end{itemize}
\item[\rm{(2)}] Let $U\in \mathcal{H}ol(S_\sigma)$ be such that
$$
M_1:= \sup_{\vert y \vert<\sigma} \Vert U_y\Vert_{H_x^{s'}(\xT^d)}<+ \infty \quad\text{with}\quad 
s'>s + \frac{d-1}{4}.
$$
Then the function $u = U_0$ belongs to $\mathcal{H}^{\sigma,s}(\xT^d)$ 
and there exists a constant $C>0$ such that 
$\Vert u \Vert_{\mathcal{H}^{\sigma,s}} \leq C   M_1$.
\end{itemize}
\end{theo}
\begin{rema}
\begin{itemize}
\item [\rm{(i)}] In the case (ii), in general we do not 
 have  $u \in  \mathcal{H}^{\sigma,s}(\xT^d)$.  For instance, if $s =0$ and $d \geq 2,$  a counterexample is provided by  the function $u$ such that $\widehat{u} (\xi) = e^{-\sigma \vert \xi \vert} \langle \xi \rangle^{-\frac{\mu}{2}} $ where  $\mu= d-\mez + \eps, 0<\eps <\frac{1}{10}.$
\item [\rm{(ii)}]If $U_0$ is radial, it suffices to assume in (2) that  $\sup_{\vert y \vert<\sigma} \Vert U_y\Vert_{H_x^{s + \frac{d-1}{4}}(\xT^d)}$ is finite.
\item [\rm{(iii)}]All the properties of these spaces needed in this paper are gathered in the Appendix.
\item [\rm{(iv)}]The same results hold with $\xT^d$ replaced by $\xR^d$.
\end{itemize}
\end{rema}

\subsection{Local in time well-posedness.} 
Our first result states that, 
for data $\eta_0,\psi_0$ of size $\eps$, the Cauchy-problem \e{system} has a unique solution 
in the space of analytic functions  on a time interval   of size $1$. 

\begin{defi}
Given a real number $s$ and continuous time-dependent index $\sigma=\sigma(t)\ge 0$, we denote by 
$C^0\big([0,T], \mathcal{H}^{\sigma,s}\big)$ the subspace of 
$C^0\big([0,T], H^{s}\big)$ which consists of those functions $f$ such that
$$
F\in C^0\big([0,T], H^{s}\big), 
$$
where
$$F(t,\cdot )=e^{\sigma(t)\la D_x\ra}f(t,\cdot).
$$
\end{defi}
\begin{rema}
When $\sigma(t) = \sigma_0$, this definition coincides with the usual definition. In general, it is easy to show that for  any $\sigma_0 \leq \inf_{t\in [0,T]} \sigma(t)$, we have 
$$ C^0\big([0,T], \mathcal{H}^{\sigma,s}\big) \subset C^0\big([0,T], \mathcal{H}^{\sigma_0,s})
$$
(with continuous embedding).
In particular, since $\sigma(t)\geq 0$,  we have 
$$  C^0\big([0,T], \mathcal{H}^{\sigma,s}\big)\subset C^0\big([0,T],{H}^{s}).$$
\end{rema}
\begin{theo}\label{T=1}
Let $d\ge 1, s> 2+\frac{d}{2}, g>0, h>0$ and $0<\lambda_0<1.$ 
Then there exist   positive constants 
$\eps_0,K,M$  such that for all $ 0<\lambda \leq \lambda_0,$ all $\eps < \eps_0$,  for all $(\eta_0, \psi_0)\in \mathcal{H}^{\lambda h, s} \times \mathcal{H}^{\lambda h, s}$,  and all  $b\in L^\infty(\xR, H^{s-1}(\xT^d))\cap L^2(\xR, H^{s-\mez}(\xT^d))$ 
satisfying
$$
\Vert b \Vert_{L^\infty(\xR, H^{s-1})\cap L^2(\xR, H^{s-\mez})} 
+  \Vert \eta_0 \Vert_{\mathcal{H}^{\lambda h, s}}
+\Vert \psi_0 \Vert_{\mathcal{H}^{\lambda h, s}}\leq \eps,
$$
the Cauchy problem \eqref{system} has a  unique solution  
\be\label{result:T=1}
(\eta, \psi)  \in C^0\big([0,T], \mathcal{H}^{\sigma,s}\times \mathcal{H}^{\sigma,s}\big)\cap L^2\big((0,T), \mathcal{H}^{\sigma,s+\mez}\times \mathcal{H}^{\sigma,s+\mez}\big)
\ee
with 
$$T=  (\lambda h)/K \quad\text{and}\quad \sigma(t)=\lambda h-Kt, $$
such that
\begin{multline*}
\sup_{t\in[0,T]}\Vert \eta(t)\Vert^2_{\mathcal{H}^{\sigma,s}} 
+ \sup_{t\in[0,T]}\Vert \psi(t)\Vert^2_{\mathcal{H}^{\sigma,s}}\\
\int_0^T \big(\Vert \eta(t)\Vert^2_{\mathcal{H}^{\sigma,s+\mez}}+\Vert \psi(t)\Vert^2_{\mathcal{H}^{\sigma,s+\mez}}\big)\, dt\leq M\eps^2.
\end{multline*}
\end{theo}
\begin{rema}
This result complements the analysis by Kano and Nishida \cite{Kano-Nishida1} and Kano~\cite{Kano1986} 
in which 
we allow a non-zero and non-analytic source term $b$. 
\end{rema}

\subsection{Well-posedness on large time intervals.}
Our main result improves Theorem~\ref{T=1} by showing 
that the solution exists and remains analytic on a large time interval whose size 
is proportional to the inverse of the size of the initial data. 
To state this result, we need to introduce two auxiliary functions. Following~\cite{ABZ3}, we set
\begin{equation}\label{def-VB}
B=  \frac{G(\eta)(\psi,b) + \nabla_x \eta \cdot \nabla_x \psi }{  1+ \vert \nabla_x \eta \vert^2},\quad
V =  \nabla_x \psi - B  \nabla_x \eta.
\end{equation}
They are the traces on the free surface of the Eulerian velocity field. Moreover, we shall set
\begin{align*}
 N_s(b) &=   \Vert  b \Vert_{L^\infty(\xR,H^{s+ \mez})} + \Vert  \partial_t  b \Vert_{L^\infty(\xR,H^{s- \mez})} +\Vert b \Vert_{L^1(\xR, H^{s+ \mez})},\\
   a(D_x)f &= G(0)(f,0)= \vert D_x\vert\tanh(h\vert D_x\vert)f.
\end{align*}
\begin{theo}\label{T=2}
Let $d\ge 1, g>0, h>0, 0<\lambda_0<1$ and consider  real-numbers 
$s>3+ \frac{d}{2}$ and $s'\in [s-1,s)$.
Then there exist positive constants $\eps_*, K_*,c_*,$  
such that for all $\eps \leq \eps_*,$ 
 all $0<\lambda \leq \lambda_0$, and all $(\eta_0, \psi_0) \in \mathcal{H}^{\lambda h, s+\mez}\times\mathcal{H}^{\lambda h, s+\mez},$ if
\be\label{assu:T=2}
N_s(b)
+\Vert   \eta_0  \Vert_{\mathcal{H}^{\lambda h,s+ \mez}}
+\Vert   a(D_x)^\mez \psi_0  \Vert_{\mathcal{H}^{\lambda h,s}} 
+\Vert   V_0  \Vert_{\mathcal{H}^{\lambda h,s}}  +  \Vert  B_0 \Vert_{\mathcal{H}^{\lambda h,s}}\leq \eps,
\ee
then the Cauchy problem \eqref{system} has a 
unique solution on the time interval 
$[0,\frac{c_*}{\eps}]$ such that
\be\label{result:T=2}
(\eta, \psi,V,B)\in C^0\Big(\Big[0,\frac{c_*}{\eps}\Big], \mathcal{H}^{\sigma,s'+\mez}
\times \mathcal{H}^{\sigma,s'+\mez}\times\mathcal{H}^{\sigma,s'}\times \mathcal{H}^{\sigma,s'} \Big),
\ee
with $\sigma(t) = \lambda h - K_*\eps t$.
\end{theo}
\begin{rema}
\begin{itemize}
\item [\rm{(i)}] One can assume without loss of generality that $\lambda h>K_*c_*$, so that $\sigma(t)> 0$ 
for all time $t$ in $[0,c_*/\eps]$.
\item [\rm{(ii)}] A loss in the radius of analyticity of size $\eps$ is optimal. 
\item [\rm{(iii)}] With a little extra work one could prove that the above result holds with $s'=s$.
\item [\rm{(iv)}]  As explained in the introduction some motivation for this work are future applications to control theory. Namely, we would like to understand the states reachable from initial rest ($\eta_0, \Psi_0, V_0, B_0) = (0,0,0,0)$) by means of actions on the bottom (the function $b$). In some sense, this results shows that in this frame-work, the use of analytic regularity spaces is unavoidable. Indeed, subject to the control $b$ in Sobolev spaces, the solution will remain analytic for large times.
\end{itemize}
\end{rema}

\subsection{Organization of the paper.}
In the next three Sections (see Sections \ref{S:Elliptic}, \ref{S:DN} and \ref{S:DNsuite}), 
we prove auxiliary elliptic regularity results and apply them 
to study several different properties of the Dirichlet-Neumann operator $G(\eta)$ when $\eta$ belongs to 
some analytic space. Theorem~\ref{T=1} 
 is proved in Section \ref{S:size1}, 
 by performing a fixed point in a suitable family of analytic spaces.  
 Theorem~\ref{T=2} is proved in Section \ref{S:sizeeps}. 
 Here the analysis is more involved as we need to keep track 
 of the linear terms, and we were not able to perform a fixed 
 point. As a consequence, we have to use a compactness method 
 and prove {\em a priori} estimates on a family of  regularized 
 systems in order to pass to the limit.    
In Appendix~\ref{AppendixA} we gathered several results concerning analytic spaces, 
including the proof of Theorem~\ref{Gsigmas}.

\section{Elliptic regularity}\label{S:Elliptic}
All functions considered here will be real valued. We fix two 
real numbers $s_0,h$ and a function $\eta=\eta(x)$ such that
$$
s_0> \frac{d}{2},\quad h>0,\quad \eta \in \mathcal{H}^{h, s_0+1}(\xT^d),\quad \inf_{x\in\xR^d}\eta(x)>-h.
$$
Set
\begin{align*}
\Omega &= \{ (x, y) \,:\, x \in \xT^d, -h < y< \eta(x)\},\\
\Sigma & = \{(x, y) \,:\, x \in \xT^d,   y=\eta(x)\},\\
\Gamma & = \{(x, y)\,:\, x \in \xT^d, y= -h\}.
\end{align*}
We denote by  $n$  the unit  normal to  $\Sigma$ 
and by  $\partial_n$ the normal derivative:
$$
n=\frac{1}{\sqrt{1+ \vert \nabla_x \eta \vert^2}}\begin{pmatrix}-\nabla \eta \\ 1\end{pmatrix},\quad 
\partial_n = \frac{1}{\sqrt{1+ \vert \nabla_x \eta \vert^2}} \left(\partial_y
- \nabla_x \eta  \cdot \nabla_x\right).
$$
Given two functions $\psi=\psi(x)$ and $b=b(x)$, 
we consider the following elliptic problem:
\begin{equation}\label{Diri1-1}
\Delta_{x,y} u = 0 \text{ in  } \Omega, \quad u\arrowvert_{y=\eta}= \psi, \quad \partial_y u\arrowvert_{y= -h} = b,
\end{equation}
where $\Delta_{x,y}=\partial_y^2+\Delta_x$. Hereafter, given a function $f=f(x,y)$, we use $f\arrowvert_{y=\eta}$ as a short notation for the function $x\mapsto f(x,\eta(x))$.
The goal of this  section  is to obtain elliptic regularity results for 
the solutions of \eqref{Diri1-1} in the spaces of analytic functions. 

\subsection{Preliminaries.}\label{defiCOV}
For $h>0$ we set $I_h = (-h,0).$

\subsubsection{Straightening the free surface.}

We begin by making a change of variables to reduce the problem to 
a fixed domain of the form
\begin{equation*}
\widetilde{\Omega} = \xT^d \times I_h=\{(x,z): x \in \xT^d, -h < z <0\}.
\end{equation*}
This change of variables will take $\Delta_{x,y}$ to a
strictly elliptic operator and the normal derivative $\partial_{n}$ to a vector field which is transverse to the boundary $\{z=0\}$.

We consider a simple change  of variables of the form $(x,z)\mapsto (x,\rho(x,z))$. 
The simplest change of variables reads
$(x,z)\mapsto \left(x, \frac{z+h}{h}\eta(x) + z\right)$. 
For technical reasons, we will consider another choice and introduce a {\em smoothing} change of variables (following Lannes~\cite{LannesJAMS}). 
This means that the function $\rho$ is given by 
\begin{equation*}
\rho(x,z) := \frac{1}{h}(z+h) (e^{z \vert D_{x} \vert}  \eta)(x) + z, \quad x\in \xT^d, \quad -h \leq z\leq 0,
\end{equation*}
where $e^{z\la D_x\ra}$ is the Fourier multiplier with symbol $e^{z\la \xi\ra}$. Since $z\le 0$, this is a smoothing operator, bounded from $H^\mu(\xT^d)$ to $\mathcal{H}^{z,\mu}(\xT^d)$ for any real number $\mu$. Notice that 
$\rho(x,0)= \eta(x)$ and  $ \rho(x,-h) = -h$.
Since
$$
\partial_z \rho(\cdot,z)-1= \frac{1}{h}e^{z\vert D_x\vert} \eta + \frac{1}{h}(z+h)e^{z\vert D_x\vert}\vert D_x \vert  \eta,
$$
for $s_0> d/2$ the Sobolev embedding implies that for all $z \in I_h$ we have
\begin{equation*}
\Vert \partial_z \rho(\cdot,z)-1 \Vert_{L^\infty(\xT^d)} \les  \Vert \partial_z \rho(\cdot,z) -1\Vert_{H^{s_0}(\xT^d)} \les \Vert \eta \Vert_{H^{s_0+1}(\xT^d)}.
\end{equation*}
We refer the reader to Lemma~\ref{est-rho} in Appendix \ref{AppendixA} for the proof of   more 
general estimates. 
Therefore, if $ \Vert \eta \Vert_{H^{s_0+1}}\leq \eps_0$ with $\eps_0$ small enough, 
the map $(x,z) \mapsto (x,\rho(x,z))$ is a  diffeomorphism from $\widetilde{\Omega}$ 
to $\Omega$. By this  change of variables, the derivatives $\partial_y$  and  
$\nabla_x$ become, respectively,  
$$
\Lambda_1 = \frac{1}{\partial_z \rho}  \partial_z , \quad \Lambda_2 = \nabla_{x} - \frac{\nabla_{x}\rho}{\partial_z \rho}  \partial_z.
$$
More precisely, set
\begin{equation*}
\widetilde{u}(x,z) = u(x,\rho(x,z)).
\end{equation*}
Then $\widetilde{u}$ solves
 \begin{equation}\label{eqwideu}
 (\Lambda_1^2 + \Lambda_2^2)\, \widetilde{u}=0, \text{ in } \widetilde{\Omega}, \quad \widetilde{u}\arrowvert_{z=0} = \psi, \quad   (\partial_ z \widetilde{u}) \arrowvert_{z=-h} = (\partial_z \rho\arrowvert_{z=-h}) b.
 \end{equation}
Using the chain rule, one can expand $\Lambda_1^2 + \Lambda_2^2$ as follows:
\[
\Lambda_1^2 + \Lambda_2^2=\frac{1}{\partial_z\rho}\Big(\alpha\partial_z^2+\beta\Delta_x, 
+\gamma\cdot\nabla_x\partial_z -\delta\partial_z \Big),
\]
where 
\[
\left\{
\begin{aligned}
&\alpha=\frac{1+|\nabla_x\rho|^2}{\partial_z\rho},\quad 
\beta=\partial_z\rho,\quad \gamma=-2\nabla_x\rho,\\
&\delta=\frac{1+|\nabla_x\rho|^2}{\partial_z\rho}\partial_z^2\rho+\partial_z\rho \Delta_x \rho -2\nabla_x\rho
\cdot\nabla_x\partial_z\rho.
\end{aligned}
\right.
\]
It will be useful to observe that $\Lambda_1^2 + \Lambda_2^2$ 
is a perturbation of 
$\Delta_{x,z}=\partial_z^2+\Delta_x$, 
which can be written in divergence form. More precisely, 
by a direct computation, one can verify that
\be\label{formediv}
(\partial_z\rho)(\Lambda_1^2 + \Lambda_2^2)\, 
\widetilde{u}=\partial_z\Big(\frac{1+\la\nabla_x \rho\ra^2}{\partial_z \rho}
\partial_z \widetilde{u}-\nabla_x \rho\cdot\nabla_x\widetilde{u}\Big)
+\cnx_x\big(\partial_z \rho \nabla_x \widetilde{u}-\partial_z \widetilde{u}\nabla_x\rho\big).
\ee
Consequently, it follows from~\e{eqwideu} that
 \begin{equation*}
\Delta_{x,z}\widetilde{u} +R\widetilde{u}=0 \text{ in } \widetilde{\Omega},
\end{equation*}
where
\begin{equation}\label{P=}
R\widetilde{u}= \partial_z\Big(\frac{1+\la\nabla_x \rho\ra^2-\partial_z \rho}{\partial_z \rho}\partial_z \widetilde{u}-\nabla_x \rho\cdot\nabla_x\widetilde{u}\Big)+\cnx_x\Big((\partial_z \rho-1) \nabla_x \widetilde{u}- \nabla_x\rho \, \partial_z\widetilde{u}\Big).
\end{equation}

\subsubsection{The lifting of the trace.}
Another standard approach consists in further transforming the problem by simplifying the Dirichlet boundary 
condition on $z=0$. To do so, given a function $\underline{\psi}=\underline{\psi}(x,z)$ satisfying $\underline{\psi}(x,0)=\psi(x)$, we shall set $v=\widetilde{u}-\underline{\psi}$, solution to 
\begin{align*}
 &(\Delta_{x,z} +R)\,v= - (\Delta_{x,z} +R)\,\underline{\psi} \quad \text{ in  } 
 \widetilde{\Omega}=\xT^d \times I_h,\\
 & v \arrowvert_{z=0} = 0, \quad   (\partial_ z v) \arrowvert_{z=-h} = (\partial_z \rho\arrowvert_{z=-h})b -(\partial_z \underline{\psi} \arrowvert_{z=-h}) b.
\end{align*}
Parallel to the choice of the coordinate $\rho$ in the above paragraph, 
a convenient choice for $\underline{\psi}$ is to 
consider the solution of an elliptic problem, to gain some extra 
regularity inside the domain $\widetilde{\Omega}$. Namely, 
we determine $\underline{\psi}$ by solving the problem
\begin{equation}\label{relev}
(\partial_z^2 + \Delta_x) \underline{\psi} = 0\quad\text{in }\widetilde{\Omega}, \quad \underline{\psi}\arrowvert_{z=0}=\psi, \quad \partial_z \underline{\psi}\arrowvert_{z=-h}=0.
\end{equation}
Notice that this problem can be explicitly solved using the  Fourier transform in  $x$. 
More precisely, we have $$\underline{\psi}(x,z)=(2\pi)^{-d}\sum_{\xi \in\xZ^d}e^{ix\cdot\xi}\widehat{\underline{\psi}}(\xi,z),$$ 
where
\be\label{Fourierpsi}
\widehat{\underline{\psi}}(\xi,z)=\frac{e^{z|\xi|}}{1+e^{-2h|\xi|}}
\widehat{\psi}(\xi)+\frac{e^{-2h|\xi|-z|\xi|}}{1+e^{-2h|\xi|}}\widehat{\psi}(\xi).
\ee
Now, we set
\begin{equation}\label{DNpsi}
 G_0(0)\psi = \partial_z \underline{\psi} \arrowvert_{z=0}.
 \end{equation}
This is the Dirichlet-Neumann  operator  associated to the  problem \eqref{relev}. 
By using \e{Fourierpsi}, we find that 
 \begin{equation*}
  G_0(0)  = \vert D_x \vert \tanh  (h \vert D_x \vert).
 \end{equation*}
 To shorten the  notations, we  shall set in the sequel,
 \begin{equation}\label{a(D)=}
 a(D_x) =  G_0(0).
 \end{equation}
 By using the previous notation, we have the following result (see Lemma~\ref{est-relevappendix} in Appendix \ref{AppendixA}).
\begin{lemm}
 For all $\mu \in \xR$, 
 there exists a constant $C>0$ such that for all $\sigma \geq 0$ and all  $\psi$ such that $a(D_x) ^\mez \psi \in \mathcal{H}^{\sigma,\mu}$,  there holds
\begin{align*}
  &\Vert \nabla_{x,z} \underline{\psi}\Vert_{L^2(I_h, \mathcal{H}^{\sigma,\mu})} \leq C \Vert  a(D_x)^\mez \psi\Vert_{ \mathcal{H}^{\sigma,\mu}},\\
  & \Vert \partial_z^2 \underline{\psi}\Vert_{L^2(I_h, \mathcal{H}^{\sigma,\mu-1})} \leq C \Vert  a(D_x)^\mez \psi\Vert_{ \mathcal{H}^{\sigma,\mu}},\\
  &\Vert \nabla_{x,z} \underline{\psi}\Vert_{L^\infty(I_h, \mathcal{H}^{\sigma,\mu-\mez})} \leq C \Vert  a(D_x)^\mez \psi\Vert_{ \mathcal{H}^{\sigma,\mu}}.
  \end{align*}
 \end{lemm}
\begin{rema}
There exists a constant $C>0$ such that for all  $\sigma\geq 0$ and $\mu \in \xR$ we have
\begin{equation}\label{nabla-G}
\Vert \vert D_x \vert \psi \Vert_{\mathcal{H}^{\sigma,\mu - \mez}} + \Vert \nabla_x  \psi \Vert_{\mathcal{H}^{\sigma,\mu- \mez}} \leq C \Vert  a(D_x)^\mez \psi \Vert_{\mathcal{H}^{\sigma,\mu}}. 
\end{equation}
This follows from the inequality $\vert \xi \vert \leq C \langle \xi \rangle \tanh\, (h \vert \xi \vert)$ for all $\xi \in \xR^d.$
  \end{rema}

\subsection{Elliptic regularity in analytic spaces.}

In this paragraph, we specify the spaces in which we shall work to study the 
elliptic regularity theory. Recall the notation $I_h=(-h,0)$ and 
consider the Dirichlet problem in a half-space:
\begin{equation*}
(\partial_z^2 + \Delta_x) w = 0\quad\text{in }\xT^d\times I_h, 
\quad w\arrowvert_{z=0}=\psi,\quad 
\partial_z w\arrowvert_{z=-h}=\theta.
\end{equation*}
Then, by using a Fourier calculation analogous to~\e{Fourierpsi}, one verifies that if 
$\psi\in \mathcal{H}^{h,\mu}(\xT^d)$ and $\theta\in H^{\mu-1}(\xT^d)$, then 
$$
e^{(h+z)\la D_x\ra}w\in C^0([-h,0], H^{\mu}(\xT^d)),
$$
which is equivalent to
$$
e^{z\la D_x\ra}w\in C^0([-h,0], \mathcal{H}^{h,\mu}(\xT^d)).
$$
Our aim is to obtain a similar result for solutions to 
the general problem with variable coefficients. However 
for the latter problem  
we will loose on the radius of analyticity. Namely, 
we will replace $e^{z\la D_x\ra}$ 
(resp.\ $\mathcal{H}^{h,\mu}(\xT^d)$) by $e^{\lambda z\la D_x\ra}$ (resp.\ $\mathcal{H}^{\lambda h,\mu}(\xT^d)$) for some 
$\lambda\in [0,1)$. This leads us to introduce the following spaces.
 \begin{defi}\label{def-EF}
Let $\lambda\in [0,1]$. For $\mu \in \xR$, 
we introduce the spaces
\begin{equation}\label{E-F}
 \begin{aligned}
 E^{\lambda,\mu} &= \{u: e^{\lambda z\vert D_x \vert} 
 u \in C^0([-h,0], \mathcal{H}^{\lambda h,\mu}(\xT^d))\},\\
 F^{\lambda,\mu} &=  \{u: e^{\lambda  z\vert D_x \vert} u \in 
 L^2(I_h, \mathcal{H}^{\lambda h,\mu}(\xT^d))\},\\
 \mathcal{X}^{\lambda ,\mu} &= E^{\lambda,\mu} \cap F^{\lambda,\mu+ \mez}.
  \end{aligned}
  \end{equation}
\end{defi}
\begin{rema}
  Lemma~\ref{lions} shows that  $\nabla_{x,z} u \in F^{\sigma,\mu+ \mez}$ and $  D^\alpha_{x,z} u \in F^{\sigma,\mu- \mez}$ for $ \vert \alpha \vert = 2$  imply that $\nabla_{x,z} u \in E^{\sigma,\mu}.$
\end{rema}

We are now in position to state our first two results concerning elliptic regularity in 
analytic Sobolev spaces. 

\begin{prop}\label{regul0}
Consider    real numbers $ \lambda_0, s, \mu$ such that
$$
0\le \lambda_0 <1, \quad s>\frac{d}{2}+1,\quad 0\le \mu\le s-1.
$$
Then there exist two constants $\overline{\eps}>0$ and $C>0$ such that for all $0\leq \lambda\leq \lambda_0$,    all $\eta \in \mathcal{H}^{\lambda h, s}(\xT^d)$  satisfying $\Vert \eta \Vert_{\mathcal{H}^{\lambda h,s}} \leq \overline{\eps}$,  all  $ F \in F^{\lambda, \mu-1}$, all $\theta \in H^{\mu-\mez}(\xT^d)$ and all $w \in L^2(I_h, H^{\mu+1}(\xT^d))$ solutions of the problem  
   \begin{equation}\label{eqv1}
 (\Delta_{x,z}+R)w=  F   \text{ in } \xT^d \times I_h,\quad
 w \arrowvert_{z=0} = 0, \quad   (\partial_ z w) \arrowvert_{z=-h} = \theta, 
 \end{equation}
the function $w$ belongs to $F^{\lambda, \mu}$ and satisfies 
\begin{equation}\label{est-wideu}
\Vert  \nabla_{x,z}  w \Vert _{F^{\lambda,\mu}}\leq
C\Big(\Vert   F \Vert_{F^{\lambda,\mu-1}}  + \Vert \theta \Vert_{H^{\mu- \mez}(\xT^d)}\Big).
\end{equation}
\end{prop}
\begin{rema}
For our purposes  the estimate~\e{est-wideu} 
is interesting for $\lambda$ close to $1$. 
\end{rema}

Before proving this result, we pause to show how to deduce 
a variant of Proposition~\ref{regul0} with a non-vanishing trace on $z=0$, assuming that the index $\mu$ is equal to $s-1$. 
\begin{coro}\label{regul12}
Consider two  real numbers $ \lambda_0, s$ such that
$$
0\le \lambda_0 <1, \quad s>\frac{d}{2}+1.
$$
Then there exists two constants $\overline{\eps}>0$ and $C>0$ such that for all $ 0 \leq \lambda \leq \lambda_0$, 
  all $\eta \in \mathcal{H}^{\lambda h, s}(\xT^d)$ 
satisfying $\Vert \eta \Vert_{\mathcal{H}^{\lambda h,s}} \leq \overline{\eps}$, 
all $\psi  \in   \mathcal{H}^{\lambda h, s}(\xT^d)$, 
all $F \in F^{\lambda, s-2}$, all $\theta \in H^{s-\tdm}(\xT^d)$ and all 
$\widetilde{w} \in L^2(I_h, H^{s}(\xT^d))$ solutions of the problem 
$$
(\Delta_{x,z}+R)\widetilde{w}=  F   \text{ in } \xT^d \times I_h,\quad
\widetilde{w}  \arrowvert_{z=0} = \psi, \quad   (\partial_ z \widetilde{w}) \arrowvert_{z=-h} = \theta, 
$$
the function $\nabla_{x,z}\widetilde{w}$ belongs to $F^{\lambda,s-1}$ and satisfies 
\begin{equation}\label{est-wideu12}
\Vert  \nabla_{x,z}  \widetilde{w} \Vert _{F^{\lambda,s-1}}\leq
C\Big(\Vert   F \Vert_{F^{\lambda,s-2}} +\blA a(D_x)^\mez \psi\brA_{\mathcal{H}^{\lambda h,s-1}} 
+ \Vert \theta \Vert_{H^{s- \tdm}(\xT^d)}\Big).
\end{equation}
\end{coro}
\begin{proof}
Since by \eqref{relev} we have, $\Delta_{x,z} \underline{\psi} =0,$ 
the function $w=\widetilde{w}-\underline{\psi}$ satisfies 
$$
(\Delta_{x,z}+R)w=  F -R\underline{\psi}  \text{ in } \xT^d \times I_h,
\quad w  \arrowvert_{z=0} = 0, \quad   (\partial_ z w) \arrowvert_{z=-h} = \theta.
$$
Consequently, \e{est-wideu12} follows from the estimate~\e{est-wideu} given 
by Proposition~\ref{regul0}, applied with $\mu=s-1$, together 
with the estimate~\e{HF2(s-2)} below for the remainder 
$R\underline{\psi}$.
\end{proof}
\begin{proof}[Proof of Proposition~\ref{regul0}]
As we have seen in the previous paragraph, one can use the Fourier 
transform to study the analytic regularity of the solutions to the 
linearized problem
\begin{equation*}
\Delta_{x,z}w=  F   \text{ in } \xT^d \times I_h,\quad
w\arrowvert_{z=0} = 0, \quad   (\partial_ z w) \arrowvert_{z=-h} = \theta.
\end{equation*}
However, since the operator $R$ is a differential operator 
with variable coefficients, to study the regularity of the 
solution to \e{eqv1}, we must proceed differently. We will use the multiplier method. 
More precisely, our strategy consists in conjugating the operator $\Delta_{x,z}+R$ 
by the weight $e^{\lambda(h+z)\la D_x\ra}$. The trick here is that, 
when $\lambda<1$, we obtain another coercive operator and then the 
desired estimate \e{est-wideu} will follow from an energy estimate. 
The proof thus consists in estimating the function 
$e^{\lambda(h+z)\la D_x\ra} w$. 
To rigorously justify the computations, we shall truncate 
the symbol $e^{\lambda(h+z)\la \xi\ra}$, using the following lemma.

\begin{lemm}\label{qeps}
Given $\eps>0$, $\lambda \in [0,1)$ and  $z \in I_h$, define  $q_\eps(z,\cdot)\colon\xR^d\to\xR$ by
$$
q_\eps(z, \xi) = \lambda \Big( \frac{ h \vert \xi \vert}{1+ \eps \vert \xi \vert} + z \vert \xi \vert \Big).
$$ 
Then, for all  $\xi, \zeta \in \xR^d$ we have
\begin{align*}
&  q_\eps(z, \xi) \leq \lambda(h+z) \vert \xi\vert , \quad q_\eps(-h,\xi) = - \frac{\lambda h \eps \vert \xi \vert^2}{1+ \eps \vert \xi \vert}  \leq 0, \\
& q_\eps (z,  \xi) - q_\eps (z,   \zeta) \leq \lambda h \vert \xi- \zeta\vert.
\end{align*}
\end{lemm}
\begin{proof} 
The two first claims are  obvious. Then set  $\langle \eps \xi \rangle =  1+ \eps \vert \xi \vert$. We have
$$
 q_\eps (z,  \xi) - q_\eps (z, \zeta)= 
\lambda(\vert \xi \vert - \vert \zeta \vert) \left(\frac{h}{\langle \eps \xi \rangle\langle \eps \zeta \rangle} + z \right).
$$
%
If $\vert \xi \vert - \vert \zeta \vert \geq 0$, we use the fact that  
$\frac{h}{\langle \eps \xi \rangle\langle \eps \zeta \rangle}\leq h$,   
$z \leq 0$, together with the inequality $\vert \xi \vert - \vert \zeta \vert \leq \vert \xi - 
\zeta \vert$. 
If  $\vert \xi \vert - \vert \zeta \vert \leq 0$ we use the fact that $\frac{h}{\langle \eps \xi \rangle\langle \eps \zeta \rangle} \geq 0$ and $ 0 \leq -z \leq h.$
 \end{proof}

Now we fix $s,\mu$ satisfying $s>d/2+1$ and $0\le \mu\le s$ and set
$$
\Lambda_{\eps}(z)=e^{q_\eps(z,D_x)}\langle D_x\rangle^\mu,\quad \langle D_x\rangle^\mu
=(I-\Delta_x)^{\mu/2},
$$
where  $q_\eps$ is defined  in  Lemma~\ref{qeps}. 
Given a function $f=f(x,z)$ defined for $x\in \xT^d$ and $z\in I_h$, 
we define $\Lambda_\eps f$ as usual by $(\Lambda_\eps f)(\cdot,z)=\Lambda_\eps(z)f(\cdot,z)$. Then we set 
\begin{equation*}
w_\eps  = \Lambda_{\eps}w. 
\end{equation*}
Notice that this definition is meaningful, since  the symbol  $ e^{q_\eps(x,\xi)} $ 
is bounded and  $w$ belongs to $L^2(I_h, H^1(\xT^d))$. 
Our goal is to estimate the $H^1(\widetilde{\Omega})$ norm of $w_\eps$ uniformly in $\eps$, 
and this will imply the desired result by means of   Fatou's lemma.

To form an equation on $w_\eps$, we notice that, for any function $f=f(x,z)$, 
$$
(\partial_ z  - \lambda \vert D_x \vert  )e^{q_\eps(z,D_x)} f
=e^{q_\eps(z,D_x)} \partial_z f,\quad \nabla_x\big( e^{q_\eps(z,D_x)}f\big)=e^{q_\eps(z,D_x)} \nabla_x f.
$$
Therefore, setting
\be\label{defiP}
P_\lambda=(\partial_ z  - \lambda \vert D_x \vert  )^2+\Delta_x,
\ee
we obtain
$$
P_\lambda w_\eps=\Lambda_\eps  (\partial_z^2+\Delta_x) w.
$$
Since $v$ solves \e{eqv1}, we conclude that $w_\eps$ is  solution of the problem
 \begin{equation}\label{EQ3}
 \left\{
 \begin{aligned}
 &P_\lambda w_\eps = 
 \Lambda_{\eps}(-Rw+F),\\
 & w_\eps \arrowvert_{z=0}   = 0, \quad  
 (\partial_ z w_\eps - \lambda \vert D_x \vert w_\eps ) \arrowvert_{z=-h} =
 e^{q_\eps(-h,D_x)}\langle D_x\rangle^\mu\theta:=\theta_\eps.
 \end{aligned}
 \right.
  \end{equation}
Now, the rest of the proof is divided into three steps:
\begin{itemize}
\item First, we will prove that the operator $P_\lambda$ is elliptic.
\item The second step is elementary. We check that the contributions of the Cauchy data $F$ and $\theta$ 
are estimated by the right-hand side of \eqref{est-wideu}.
\item In the third step, we prove a commutator estimate in analytic spaces 
and use it to deduce that the 
contribution of 
$e^{q_\eps(z,D_x)}Rv$ can be absorbed by the elliptic regularity, 
under a smallness assumption on the coefficients in the operator $R$.
\end{itemize}

\textbf{Step 1: The conjugated operator.} We begin by studying 
the operator $P_\lambda$ introduced in~\e{defiP}. We will see that it is an elliptic operator and prove some elementary elliptic estimates. 

Recall that, by notation,  $\widetilde{\Omega} = \xT^d \times I_h$. We denote by $\mathcal{H}^1_0(\widetilde{\Omega})$ 
the subspace of  $H^1(\widetilde{\Omega})$ which consists of those functions whose trace on $z=0$ vanishes, 
equipped with the~$H^1(\widetilde{\Omega})$-norm.  Poincar{\'e}'s inequality applies in this setting  and 
there is a positive constant $C_{\widetilde{\Omega}}$ such that
\be\label{Poincaretilde}
\lA u\rA_{L^2(\widetilde{\Omega})}\le C_{\widetilde{\Omega}}\lA \nabla_{x,z}u\rA_{L^2(\widetilde{\Omega})},\quad 
\forall u\in \mathcal{H}^1_0(\widetilde{\Omega}).
\ee
Now, consider the bilinear form
\begin{align*} 
 a(u,v) &= \big( \partial_z u, \partial_z v\big)_{L^2(\widetilde{\Omega})} 
 +  (1- \lambda^2) \big( \nabla_x u, \nabla_x v\big)_{L^2(\widetilde{\Omega})} \\
 &\quad + \lambda \big( \partial_z u, \vert D_x \vert v\big)_{L^2(\widetilde{\Omega})} 
 - \lambda \big(  \vert D_x \vert u, \partial_z v\big)_{L^2(\widetilde{\Omega})}.
 \end{align*}
This is a continuous bilinear  form on $\mathcal{H}^1_0(\widetilde{\Omega})  
\times \mathcal{H}^1_0(\widetilde{\Omega})$. Moreover,
if  $u \in \mathcal{H}_0^1(\widetilde{\Omega}) \cap H^2(\widetilde{\Omega}) $, 
we can make the following computations:
\begin{align*}
&\big( \partial_z u, \partial_z v\big)_{L^2(\widetilde{\Omega})} 
= - \big( \partial^2_z u,  v\big)_{L^2(\widetilde{\Omega})} 
- \int_{\xT^d} (\partial_z u) v\arrowvert_{z= -h}\, dx,\\
& \big( \nabla_x u, \nabla_x v\big)_{L^2(\widetilde{\Omega})} 
= - \big(\Delta_x u,v)_{L^2(\widetilde{\Omega})},\\
& \big( \partial_z u, \vert D_x \vert v\big)_{L^2(\widetilde{\Omega})} 
= \big( \vert D_x \vert\partial_z u,   v\big)_{L^2(\widetilde{\Omega})}, \\
  -& \big(  \vert D_x \vert u, \partial_z v\big)_{L^2(\widetilde{\Omega})} 
  =  \big( \vert D_x \vert\partial_z u,   v\big)_{L^2(\widetilde{\Omega})} 
  + \int_{\xT^d} (\vert  D_x \vert u) v\arrowvert_{z=-h} \, dx.
 \end{align*}
It follows that 
\begin{equation}\label{a-P}
a(u,v) = \big(-P_\lambda u,v \big)_{L^2(\widetilde{\Omega})} 
- \int_{\xT^d} \  \big[(\partial_z u  - \lambda \vert D_x \vert  u)   v \big]\arrowvert_{z=-h} \, dx.
 \end{equation}
On the other hand, by using the assumption that $\lambda<1$, 
and remembering that we are considering real-valued functions, we have
\begin{equation}\label{coerc}
a(u,u) = \Vert \partial_z u \Vert^2_{L^2(\widetilde{\Omega})} + (1- \lambda^2) \Vert \nabla_x u \Vert^2_{L^2(\widetilde{\Omega})} \geq C(1- \lambda^2) \Vert u \Vert^2_{H^1(\widetilde{\Omega})},
\end{equation}
where we used the classical Poincar\'e 
inequality~\e{Poincaretilde}. Here $C>0$ is independent of $\lambda.$
 
With the  notations in \eqref{EQ3} consider the linear form on  $\mathcal{H}_0^1(\widetilde{\Omega}),$
\begin{equation*}
L(f)= -\big\langle \Lambda_\eps(-Rw+F) ,f\big\rangle - (( \theta_\eps, f\arrowvert_{z=-h})),
\end{equation*}
where 
$\big\langle \cdot, \cdot \big\rangle$ denotes the 
duality between    $L^2_z\big( I_h, H^{-1}(\xT^d)\big)$ 
and $L^2_z\big( I_h, H^{1}(\xT^d)\big)$ and  $(( \cdot, \cdot))$ 
denotes the duality between $H^{-\mez}(\xT^d)$ 
and  $H^{\mez}(\xT^d)$. 
We deduce from   \eqref{a-P} that 
$w_\eps$ is solution of the  problem,
\be\label{LaxMilgram}
a(w_\eps,f) = L(f), \quad 
\forall f \in \mathcal{H}^1_0(\widetilde{\Omega}).
\ee

Recall from \e{P=} that $Rw$ is given by,
\begin{equation*}
Rv= \partial_z F_1 +\cnx_x F_2,
\end{equation*}
where
$$
F_1=\frac{1+\la\nabla_x \rho\ra^2-\partial_z \rho}{\partial_z \rho}\partial_z w
-\nabla_x \rho\cdot\nabla_xw,\quad F_2=(\partial_z \rho-1) \nabla_x w-\partial_z w\nabla_x\rho.
$$
Parallel to the computations above  we immediately verify that
\begin{align*}
\Lambda_{\eps}\partial_z F_1&=(\partial_z-\lambda \la D_x\ra)\big(\Lambda_{\eps} F_1\big),\qquad
\Lambda_{\eps}\cnx_x F_2=\cnx_x \big(\Lambda_{\eps}F_2\big),
\end{align*}
so, 
$$
\Lambda_{\eps}Rw =(\partial_z-\lambda \la D_x\ra)\big(\Lambda_{\eps} F_1\big)   
+\cnx_x \big(\Lambda_{\eps}F_2\big).
$$
Integrating by parts with respect to $z$ or $x$, 
we find that
\begin{align*}
\big\langle \Lambda_{\eps}Rw  ,f\big\rangle 
   &=\big\langle \Lambda_{\eps}F_1,(\partial_z  -\lambda \la D_x\ra)f\big\rangle 
    - ((e^{q_\eps(-h,D_x)}\langle D_x\rangle^\mu 
    F_1\arrowvert_{z=-h},f\arrowvert_{z=-h}
))\\
&\quad  +\big\langle \Lambda_{\eps} F_2,\nabla_x f\big\rangle. 
\end{align*}
The absolute value of the first and last term in the right-hand 
side above are estimated by means of the Cauchy-Schwarz inequality by
$$
\Big(\blA \Lambda_{\eps}F_1\brA_{L^2(\widetilde{\Omega})}
+\blA \Lambda_{\eps}F_2\brA_{L^2(\widetilde{\Omega})}\Big)
\lA \nabla_{x,z}f\rA_{L^2(\widetilde{\Omega})}.
$$
To estimate the second term, we use the fact that 
$q_\eps(-h,\xi)\le 0$ and 
the trace theorem to write
\begin{align*}
&\la (( e^{q_\eps(-h,D_x)}\langle D_x\rangle ^\mu F_1\arrowvert_{z=-h},f\arrowvert_{z=-h}))\ra\\
\le &\ \lA e^{q_\eps(-h,D_x)}\langle D_x\rangle^\mu F_1\arrowvert_{z=-h}\rA_{H^{-\mez}}\lA f\arrowvert_{z=-h}\rA_{H^\mez(\xT^d)}\\
\le &\,  \lA \langle D_x\rangle^\mu F_1\arrowvert_{z=-h}\rA_{H^{-\mez}}
\lA f\rA_{H^1(\widetilde{\Omega})}\\
\les &\, \lA \langle D_x\rangle^\mu F_1\arrowvert_{z=-h}\rA_{H^{-\mez}}
\lA \nabla_{x,z}f\rA_{L^2(\widetilde{\Omega})},
\end{align*}
where we used Poincar\'e's inequality. 
Similarly, we have
$$
\la (( \theta_\eps, f\arrowvert_{z=-h}))\ra \les \lA \theta_\eps\rA_{H^{-\mez}}\lA \nabla_{x,z}f\rA_{L^2(\widetilde{\Omega})}.
$$
Hence we conclude that
\begin{equation}\label{LaxMilgram3}
\la L(f)\ra 
\le \big( A+B+C+D\big)
\lA \nabla_{x,z}f\rA_{L^2(\widetilde{\Omega})}, 
\end{equation}
where
\[
\begin{aligned}
&A=\blA \Lambda_{\eps}F\brA_{L^2( I_h, H^{-1}(\xT^d))}, \\
& B=\lA \theta_\eps\rA_{H^{-\mez}},\\
&C=\blA \Lambda_{\eps}F_1\brA_{L^2(\widetilde{\Omega})}+
\blA \Lambda_{\eps}F_2\brA_{L^2(\widetilde{\Omega})},
\\
&D=\blA \langle D_x\rangle^\mu F_1\arrowvert_{z=-h}\brA_{H^{-\mez}}.
\end{aligned}
\]
Thus, by combining \e{LaxMilgram} and \e{LaxMilgram3} 
(applied with $f=w_\eps$) 
together with the 
coercive inequality~\e{coerc}, we conclude that
\be\label{LaxMilgram4}
\Vert \nabla_{x,z} w_\eps\Vert_{L^2_z (I_h, L^2)}  
\leq C_\lambda(A +B+C+D),
\ee
where $C_\lambda = \frac{C}{1-\lambda}.$ \\

\textbf{Step 2: Estimates of $A$ and $B$}. 
Directly from Lemma~\ref{qeps} and \eqref{EQ3}  we deduce that
\begin{equation*}
\Vert  \theta_\eps \Vert_{H^{-\mez}(\xT^d)} \leq   \Vert \theta \Vert_{H^{\mu- \mez}(\xT^d)}.
\end{equation*}
Similarly, using the property $q_\eps(z, \xi) \leq \lambda(h+z) \vert \xi \vert$, we obtain
$$
\Vert \Lambda_\eps F \Vert_{L^2_z(I_h, H^{-1}(\xT^d))}
\leq  
\Vert e^{\lambda z \vert D_x \vert}\langle D_x\rangle^\mu F
\Vert_{L^2_z(I_h, \mathcal{H}^{\lambda h, -1})}\le \lA F\rA_{F^{\lambda,\mu-1}}.
$$
Then the terms $A$ and $B$ in \e{LaxMilgram4} are estimated by the 
right-hand side of the desired inequality~\e{est-wideu}. \\

\textbf{Step 3: estimate of $C$}. We shall use the following lemma.
\begin{lemm}\label{lemme1}
Consider two real numbers 
$$
\nu >\frac{d}{2},\quad 0\le \mu\le \nu.
$$
Given two functions $r,u\colon \xT^d\times I_h\to \xR$, set
$$
U(\cdot,z) = e^{q_\eps(z,D_x)} \langle D_x\rangle^\mu (r(\cdot,z) u(\cdot,z)).
$$
Then  there exists a constant $C=C(d,\nu,\mu,h)>0$ such that
$$
\Vert U \Vert_{L^2_z(I_h, L^2(\xT^d))}
\leq C \Vert r \Vert_{L^\infty_z(I_h,\mathcal{H}^{\lambda h, \nu})}
\Vert e^{q_\eps(z,D_x)}\langle D_x\rangle^\mu u \Vert_{L^2_z(I_h, L^{2}(\xT^d))}.
$$
\end{lemm}
\begin{proof}
By Lemma~\ref{qeps} we have
$$
e^{q_\eps(z, \xi)}=e^{q_\eps(z, \xi)-q_\eps(z,\zeta)}
e^{q_\eps(z, \zeta)}\le e^{\lambda  h \la \xi-\zeta\ra}e^{q_\eps(z, \zeta)}.
$$
So writing $\widehat{U}(\xi, z)$ under the form
$$
 \frac{1}{(2\pi)^d}\int \frac{\langle \xi \rangle^\mu }
{\langle \xi- \zeta\rangle^{\nu}\langle \zeta \rangle^\mu}
e^{q_\eps(z, \xi)- q_\eps(z, \zeta)}  \langle \xi- \zeta \rangle^{\nu} 
\widehat{r}(\xi- \zeta,z)  e^{ q_\eps(z, \zeta)}  \langle \zeta \rangle^\mu\widehat{u}(\zeta,z) \, d\zeta,
$$
we find that
$$
\big\vert \widehat{U}(\xi, z)\big\vert\le  \frac{1}{(2\pi)^d}\int F(\xi,\zeta)f_1(\xi-\zeta,z)f_2(\zeta,z)\, d\zeta,
$$
where
$$
F(\xi,\zeta)=\langle \xi \rangle^\mu \langle \xi- \zeta\rangle^{-\nu}
\langle \zeta \rangle^{-\mu},
$$
and
$$
f_1(\theta,z)=e^{\lambda h\la \theta\ra}\langle \theta\rangle^\nu\la \widehat{r}(\theta,z)\ra,\quad 
f_2(\zeta,z)=e^{ q_\eps(z, \zeta)}  \langle \zeta \rangle^\mu \la\widehat{u}(\zeta,z)\ra.
$$
Now, we are exactly in the setting introduced 
by H\"ormander  to 
study the continuity of the product in 
Sobolev spaces (see Lemma~\ref{Theo:Hormander} 
in Appendix~\ref{AppendixA} and take $s_3 = \mu$, $s_1= \nu$ 
and $s_2= \mu$). We infer that, for any 
fixed $z\in I_h$,
$$
\lA U(\cdot,z)\rA_{L^2}^2\le C\lA f_1(\cdot,z)\rA_{L^2}^2\lA f_2(\cdot,z)\rA_{L^2}^2.
$$
Using the Plancherel's identity and then integrating in $z$, we obtain the desired result. 
This completes the proof of Lemma \ref{lemme1}. 
\end{proof}

We are now in position to estimate $\blA \Lambda_{\eps}F_1\brA_{L^2(\widetilde{\Omega})}+
\blA \Lambda_{\eps}F_2\brA_{L^2(\widetilde{\Omega})}$. 
Remembering that
$$
F_1=\frac{1+\la\nabla_x \rho\ra^2-\partial_z \rho}{\partial_z \rho}
\partial_z w-\nabla_x \rho\cdot\nabla_x w,\quad 
F_2=(\partial_z \rho-1) \nabla_x w-\partial_z w\nabla_x\rho,
$$
we see that it is sufficient to estimate the $L^2(\widetilde{\Omega})$-norm of a term of the form $\Lambda_\eps(\alpha\beta)$, where the factors of the product are
$$
\alpha\in \left\{\frac{1+\la\nabla_x \rho\ra^2-\partial_z \rho}{\partial_z \rho},\nabla_x\rho, \, \partial_z\rho-1\right\},\qquad 
\beta\in \left\{\nabla_x w,\, \partial_z w\right\}.
$$ 
Since $s>d/2+1$ and $0\le \mu\le s-1$ by assumptions, 
we can apply Lemma~\ref{lemme1} with $\nu=s-1$, 
to obtain
\begin{align*}
\lA \Lambda_{\eps}(\alpha \beta)\rA_{L^2(\widetilde{\Omega})} 
\les 
\lA \alpha\rA_{L^\infty(I_h,\mathcal{H}^{\lambda h,s-1})}\lA \Lambda_\eps \beta\rA_{L^2(\widetilde{\Omega})}.
\end{align*}
Now we claim that there exists $C>0$ such that
\begin{align}
&\lA \Lambda_\eps \beta\rA_{L^2(\widetilde{\Omega})}\le C \lA \nabla_{x,z}w_\eps\rA_{L^2(\widetilde{\Omega})},\label{claimbeta}\\
&\lA \alpha\rA_{L^\infty(I_h,\mathcal{H}^{\lambda h,s-1})}\le C\overline{\eps}\quad\text{provided }\lA 
\eta\rA_{\mathcal{H}^{\lambda h,s}}\le \overline{\eps}\label{claimalpha}.
\end{align}
The proof of~\e{claimbeta} is straightforward; indeed
$$
\Lambda_\eps \partial_{x_j}w=\partial_{x_j}\Lambda_\eps w=\partial_{x_j}w_\eps,
\quad 
\Lambda_\eps \partial_z w=(\partial_z-\lambda \la D_x\ra)\Lambda_\eps w=(\partial_z-\lambda \la D_x\ra)w_\eps.
$$
The second estimate follows from Lemma~\ref{algebre0}. 
Now, by using \e{claimalpha} and \e{claimbeta}, we obtain that, if 
$\lA 
\eta\rA_{\mathcal{H}^{\lambda h,s}}\le \overline{\eps}$, then we have
$$
\blA \Lambda_{\eps}F_1\brA_{L^2(\widetilde{\Omega})}+
\blA \Lambda_{\eps}F_2\brA_{L^2(\widetilde{\Omega})}\les \overline{\eps}\lA \nabla_{x,z}w_\eps\rA_{L^2(\widetilde{\Omega})}.
$$

\vspace{5mm}

\textbf{Step 4: estimate of $D$}. It remains only to estimate 
$$
D=\blA \langle D_x\rangle^\mu F_1\arrowvert_{z=-h}\brA_{H^{-\mez}},
$$
where recall that
$$
F_1=\frac{1+\la\nabla_x \rho\ra^2-\partial_z \rho}{\partial_z \rho}\partial_z w-\nabla_x \rho\cdot\nabla_x w.
$$
By definition of $\rho$ (see~\e{drho})  we have
$$
\partial_z \rho\arrowvert_{z=-h}  = 1 + \frac{1}{h} e^{-h \vert D_{x} \vert} \eta(x) ,\quad
\nabla_{x} \rho\arrowvert_{z=-h}  = 0.
$$
Since $\partial_z v \arrowvert_{z=-h}=\theta$ we obtain
$$
F_1\arrowvert_{z=-h}=-\frac{e^{-h \vert D_{x} \vert} \eta}{h+e^{-h\la D_x\ra}\eta}\theta.
$$
Then $D\leq C \lA \theta\rA_{H^{\mu -\mez}}$, where $C$   depends only on $h$. \\

\textbf{Step 5: End of the proof.}
By plugging the previous estimates in \e{LaxMilgram4}, we conclude that
$$
\Vert \nabla_{x,z} w_\eps\Vert_{L^2_z (I_h, L^2)}  
\leq C_\lambda \big( \Vert \theta \Vert_{H^{\mu- \mez}(\xT^d)}+ \lA F\rA_{F^{\lambda,\mu-1}}\big)
+C_\lambda' \overline{\eps}\blA \nabla_{x,z} w_\eps\brA_{L^2_z (I_h, L^2)}.
$$
Now, taking $\overline{\eps}$ such that $C_\lambda' \overline{\eps}<1$, we obtain the uniform estimate
$$
\sup_{\eps\in (0,\overline{\eps}]}\Vert \nabla_{x,z} w_\eps\Vert_{L^2_z (I_h, L^2)}  
\leq C''_\lambda \big( \Vert \theta \Vert_{H^{\mu- \mez}(\xT^d)}+ \lA F\rA_{F^{\lambda,\mu-1}}\big).
$$
It follows from Fatou's lemma that
$$
\Vert \nabla_{x,z} w\Vert_{L^2_z (I_h, L^2)}  
\leq C''_\lambda \big( \Vert \theta \Vert_{H^{\mu- \mez}(\xT^d)}+ \lA F\rA_{F^{\lambda,\mu-1}}\big).
$$
This completes the proof of Proposition~\ref{regul0}.
\end{proof}

\subsection{Sharp elliptic estimates.}

We consider again the  problem,
\begin{equation}\label{eq-U}
(\Delta_{x,z}+R)\widetilde{w}=  F \,  \text{ in } \xT^d \times I_h,\quad
\widetilde{w}\arrowvert_{z=0} = \psi, \quad    \partial_ z \widetilde{w} \arrowvert_{z=-h} = \theta,
\end{equation}
and our purpose is to refine the result proved in Corollary~\ref{regul12}.

Let us recall some elliptic estimates proved in \cite{ABZ3,AM} for this problem:
$$
\Vert   \nabla_{x,z}  \widetilde{w} \Vert _{L^2(I_h;H^{s})}
\leq  \mathcal{F}\left(\lA \eta\rA_{H^{s+\mez}}\right)
\left(\Vert  F\Vert_{L^2_z(I_h;H^{s- 1})} + \Vert \psi \Vert_{H^{s+\mez}}
+\Vert \theta \Vert_{H^{s-\mez}}\right)
$$
for some non-decreasing function $\mathcal{F}\colon \xR_+\to\xR_+$. 
Here we will prove an analogue of this estimate in the analytic setting. 
As in the previous paragraph, we will make a smallness assumption on $\eta$. 
But we will only assume that the $\mathcal{H}^{\lambda h,s}$-norm of $\eta$ 
is small, and not the $\mathcal{H}^{\lambda h,s+\mez}$-norm. 
Also, for later purposes, we will prove a sharp estimate which is tame 
in the sense that the $\mathcal{H}^{\lambda h,s+\mez}$-norm of $\eta$ 
will multiply only the $\mathcal{H}^{\lambda h,s}$-norm of $\psi$. 

\begin{theo}\label{regul-fin'} 
Consider two real numbers $ \lambda_0, s$ such that
$$
0\le \lambda_0 <1, \quad s>\frac{d}{2}+2.
$$
Then there exist two constants $\overline{\eps}>0$ and $C>0$ 
such that for all $ 0\leq \lambda \leq \lambda_0$,   all   $\eta \in \mathcal{H}^{\lambda  h,s+ \mez }(\xT^d)$  
satisfying $\Vert \eta \Vert_{\mathcal{H}^{\lambda h,s}}\leq \overline{\eps}$, 
all $\psi  \in   \mathcal{H}^{\lambda h, s+ \mez}(\xT^d), $ all $ \theta \in H^{s-\mez}(\xT^d)$ 
and all solutions $\widetilde{w}$ of  problem \eqref{eq-U}, 
if we set
$$
I_1= \Vert   \nabla_{x,z}  \widetilde{w} \Vert _{F^{\lambda, s}}, 
\quad I_2= \Vert   \partial_z^2  \widetilde{w} \Vert_{F^{\lambda, s-1}}, 
\quad I_3=   \Vert   \nabla_{x,z}  \widetilde{w} \Vert _{E^{\lambda, s-\mez}},
$$  
then
\be\label{HF20}
\begin{aligned}
\sum_{k=1}^3 I_k &\leq  C
\Big(\Vert  F\Vert_{F^{\lambda,s- 1}}  + \Vert a(D_x)^\mez \psi \Vert_{\mathcal{H}^{\lambda h,s}}+\Vert \theta \Vert_{H^{s-\mez}}\Big)
\\
&\quad +C \Vert \eta \Vert_{  \mathcal{H}^{\lambda h, s+ \mez}}
\Big(\Vert  F\Vert_{F^{\lambda,s-2}}+ \Vert a(D_x)^\mez \psi \Vert_{\mathcal{H}^{\lambda h,s-\mez}}+\Vert \theta \Vert_{H^{s-\tdm}}\Big),
\end{aligned}
\ee
where recall that $a(D_x) =  G_0(0)$ as introduced in \eqref{a(D)=}.
\end{theo}
\begin{rema}
Compared to the statement of Corollary~\ref{regul12} we notice that, 
under the same smallness assumption on $\eta$, 
we are considering smoother data $\psi$, $F$ and $\theta$ as well as 
smoother solutions. Indeed, now $\theta$ and $F$ belongs to $H^{s-1/2}(\xT^d)$ and $F^{\lambda,s-1}$ respectively, while 
we considered before the case where they belong to $H^{\mu-1/2}(\xT^d)$ and $F^{\lambda,\mu-1}$ 
for some $\mu \le s-1$. We will reduce the analysis to the previous case by considering the equation satisfied by~$\nabla_x \widetilde{w}$. 
 \end{rema}
 
The rest of this section is devoted to the proof of this result.

\begin{proof}[Proof of Theorem~\ref{regul-fin'}]
We divide the proof into three steps. \\

{\bf Step 1.} Notice that by interpolation (see Lemma~\ref{uvF}, (iv) with $\mu = s - \mez) $, we have
\begin{equation}\label{I3-I2-I1}
I_3 \leq C(I_1+ I_2).
\end{equation}
Therefore, it is enough to estimate $I_1$ and $I_2$. \\

{\bf Step 2.} We now prove that  
\begin{equation}\label{I2-I1}
\begin{aligned}
 I_2  \leq C\big( I_1& +   \Vert F\Vert_{F^{\lambda, s-1}} \\
&+ \Vert \eta \Vert_{\mathcal{H}^{\lambda h, s+\mez}}\big(\Vert   F \Vert_{F^{\lambda,s-2}} +\blA a(D_x)^\mez \psi\brA_{\mathcal{H}^{\lambda h,s-\mez}} + \Vert \theta \Vert_{H^{s- \tdm}(\xT^d)}\big)\big).
\end{aligned}
\end{equation}
Since $(\Delta_{x,z}+R)w=F$, we have $\partial_z^2w=F-\Delta_xw-Rw$. 
Now, we use the technical estimate~\eqref{estR:s-2} in Appendix \ref{AppendixA}
to handle the contribution of $Rw$. It follows that
$$
\Vert \partial_z^2 \widetilde{w}\Vert_{F^{\lambda, s-2}} \leq \Vert F\Vert_{F^{\lambda, s-2}}+ \Vert \nabla_x \widetilde{w}\Vert_{F^{\lambda, s-1}}
+  C\overline{\eps}\big(\Vert \nabla_{x,z} \widetilde{w}\Vert_{F^{\lambda, s-1}} + \Vert \partial_z^2\widetilde{w}\Vert_{F^{\lambda, s-2}}\big).
$$
Taking $\overline{\eps}$ such that $C \overline{\eps}\leq \mez$, we deduce that
\begin{equation}\label{HF2-u4}
\Vert \partial_z^2 \widetilde{w}\Vert_{F^{\lambda, s-2}} \leq C\big( \Vert F\Vert_{F^{\lambda, s-2}}+ \Vert \nabla_{x,z} \widetilde{w}\Vert_{F^{\lambda, s-1}}\big).
\end{equation}
 By the same way, using \eqref{HF2-u3} below, we get
 \begin{align*}
  \Vert \partial_z^2 \widetilde{w}\Vert_{F^{\lambda, s-1}} &\leq    \Vert F\Vert_{F^{\lambda, s-1}} +   \Vert \nabla_{x} \widetilde{w}\Vert_{F^{\lambda, s}}+ C\overline{\eps} \big( \lA \nabla_{x}\widetilde{w}\rA_{F^{\lambda h,s}}+\lA \partial_z^2\widetilde{w}\rA_{F^{\lambda,s-1}})\\
  & + C \big(\lA   \partial_z \widetilde{w}\rA_{F^{\lambda h,s-1}}+ \Vert \eta \Vert_{\mathcal{H}^{\lambda h, s+\mez}}(\Vert \partial_z^2 \widetilde{w}\Vert_{F^{\lambda, s-\frac{5}{2}}} + \lA   \partial_z \widetilde{w}\rA_{F^{\lambda,s-\frac{3}{2}}}\big).
  \end{align*}
As before, taking $\overline{\eps}$ so small that $C\overline{\eps}\leq \mez$, 
we deduce that
\begin{align*}
\Vert \partial_z^2 \widetilde{w}\Vert_{F^{\lambda, s-1}}
\leq C\big(   &\Vert F\Vert_{F^{\lambda, s-1}}+ \Vert \nabla_{x,z} \widetilde{w}\Vert_{F^{\lambda, s}}\\
&\quad+\Vert \eta \Vert_{\mathcal{H}^{\lambda h, s+\mez}}(\Vert \partial_z^2 \widetilde{w}\Vert_{F^{\lambda, s-\frac{5}{2}}} + \lA   \partial_z \widetilde{w}\rA_{F^{\lambda,s-\frac{3}{2}}}\big).
\end{align*}
Now, since $s -\frac{5}{2} \leq s-2$ and $s -\frac{3}{2} \leq s-1$, we can use \eqref{HF2-u4}. It follows that
\be\label{HF2-u89}
\begin{aligned}
\Vert \partial_z^2 \widetilde{w}\Vert_{F^{\lambda, s-1}} \leq C\big(   \Vert F\Vert_{F^{\lambda, s-1}}&+ \Vert \nabla_{x,z} \widetilde{w}\Vert_{F^{\lambda, s}}\\
&+\Vert \eta \Vert_{\mathcal{H}^{\lambda h, s+\mez}}(\Vert F \Vert_{F^{\lambda, s-2}}+ \Vert \nabla_{x,z} \widetilde{w}\Vert_{F^{\lambda, s-1}})\big).
\end{aligned}
\ee
 Now by Corollary~\ref{regul12} we have
\be\label{HFI-37}
\Vert  \nabla_{x,z}  \widetilde{w} \Vert _{F^{\lambda,s-1}}\leq
C\Big(\Vert   F \Vert_{F^{\lambda,s-2}} +\blA a(D_x)^\mez \psi\brA_{\mathcal{H}^{\lambda h,s-\mez}} + \Vert \theta \Vert_{H^{s- \tdm}(\xT^d)}\Big).
\ee
Since $\Vert \nabla_{x,z} \widetilde{w}\Vert_{F^{\lambda, s}}=I_1$ by definition, 
we obtain the desired estimate~\eqref{I2-I1}  
by plugging the previous inequality in \e{HF2-u89}.\\

{\bf Step 3.} We are left with the estimate of $I_1.$ We shall prove that
\be\label{I1-I2}
\begin{aligned}
I_1 &\leq  C
\Big(\Vert  F\Vert_{F^{\lambda,s- 1}}  + \Vert a(D_x)^\mez \psi \Vert_{\mathcal{H}^{\lambda h,s}}+\Vert \theta \Vert_{H^{s-\mez}}\Big)
\\
&\quad +C \Vert \eta \Vert_{  \mathcal{H}^{\lambda h, s+ \mez}}
\Big(\Vert  F\Vert_{F^{\lambda,s-2}}+ \Vert a(D_x)^\mez \psi \Vert_{\mathcal{H}^{\lambda h,s-\mez}}+\Vert \theta \Vert_{H^{s-\tdm}}\Big)\\
&\quad +C\lA \eta\rA_{\mathcal{H}^{\lambda h,s}}(I_1+I_2).
\end{aligned}
\ee
This will conclude the proof of desired estimate~\e{HF20}; by taking an appropriate linear combination of \e{I3-I2-I1}, \e{I2-I1} and \e{I1-I2} and then taking $\overline{\eps}\ge \Vert \eta \Vert_{  \mathcal{H}^{\lambda h, s}}$ small enough to absorb the contribution of $I_1+I_2$ in the right-hand side 
of \e{I1-I2}. 
Notice that
\be\label{HF0}
I_1=\blA   \nabla_{x,z}\widetilde{w}\brA_{F^{\lambda, s}}
\le \blA\nabla_{x,z}\widetilde{w}\brA_{F^{\lambda,s-1}}
+\blA\nabla_{x,z}\nabla_x \widetilde{w}\brA_{F^{\lambda, s-1}}.
\ee
We will prove that these two terms are bounded by the right-hand side of \e{I1-I2}.
The first term $\Vert\nabla_{x,z}\widetilde{w}\Vert_{F^{\lambda,s-1}}$ 
has   been estimated in Corollary~\ref{regul12}. By  \e{est-wideu12} we have
$$
\Vert  \nabla_{x,z}  \widetilde{w} \Vert _{F^{\lambda,s-1}}\leq
C\Big(\Vert   F \Vert_{F^{\lambda,s-2}} +\blA a(D_x)^\mez \psi\brA_{\mathcal{H}^{\lambda h,s-1}} + \Vert \theta \Vert_{H^{s- \tdm}(\xT^d)}\Big),
$$
which immediately implies that it is bounded by the right-hand side of \e{I1-I2}. 
Hence, it remains only to estimate $\blA\nabla_{x,z}\nabla_x \widetilde{w}\brA_{F^{\lambda, s-1}}$. 
We will estimate the $F^{\lambda, s-1}$-norm of $\nabla_{x,z}\partial_{j}\widetilde{w}$ for $1\le j\le d$. 
Notice that $\partial_{j}\widetilde{w}$ satisfies 
\[
\left\{
\begin{aligned}
& (\Delta_{x,z}+R)\partial_{j}\widetilde{w}=  -[\partial_{j},R]w+\partial_{j}F  \quad\text{ in  } \xT^d \times I_h,\\
&\widetilde{w}\arrowvert_{z=0} = \partial_j\psi, \quad    \partial_ z \widetilde{w} \arrowvert_{z=-h} = \partial_{j}\theta,
\end{aligned}
\right.
\]
where $[\partial_{j},R]\widetilde{w}=\partial_j(R \widetilde{w})-R\partial_j \widetilde{w}$. 
It follows from Corollary~\ref{regul12} that
\be\label{HF40}
\begin{aligned}
\lA \nabla_{x,z}\partial_{j}\widetilde{w}\rA_{F^{\lambda,s-1}}\le C \Big( &
\blA [\partial_{j},R]\widetilde{w}\brA_{F^{\lambda,s-2}}+\lA \partial_{j}F\rA_{F^{\lambda,s-2}} \\
&+\blA a(D_x)^\mez \partial_{j}\psi\brA_{\mathcal{H}^{\lambda h,s-1}}
+\lA \partial_j\theta\rA_{H^{s-\tdm}}\Big). 
\end{aligned}
\ee
The key point is to estimate the 
commutator $[\partial_{j},R]w$. We claim that
\be\label{HF30}
\blA [\partial_{j},R]\widetilde{w}\brA_{F^{\lambda,s-2}}\le 
C \Vert \eta \Vert_{  \mathcal{H}^{\lambda h, s}}\blA \nabla_{x,z}^2\widetilde{w}\brA_{F^{\lambda,s-1}}+
C \Vert \eta \Vert_{  \mathcal{H}^{\lambda h, s+ \mez}}\blA \nabla_{x,z}\widetilde{w}\brA_{F^{\lambda,s-1}}.
\ee
Let us assume this claim and conclude the proof. 
By combining \e{HF40}, \e{HF30} and \e{HFI-37}, we have
\begin{align*}
\lA \nabla_{x,z}\partial_{j}\widetilde{w}\rA_{F^{\lambda,s-1}}
&\leq  C
\Big(\Vert  F\Vert_{F^{\lambda,s- 1}} + \Vert a(D_x)^\mez \psi \Vert_{\mathcal{H}^{\lambda h,s}}+\Vert \theta \Vert_{H^{s-\mez}}\Big)
\\
&+C \Vert \eta \Vert_{  \mathcal{H}^{\lambda h, s+ \mez}}
\Big(\Vert  F\Vert_{F^{\lambda,s-2}} + \Vert a(D_x)^\mez \psi \Vert_{\mathcal{H}^{\lambda h,s-\mez}}+\Vert \theta \Vert_{H^{s-\tdm}}\Big)
\\
& +C\Vert \eta \Vert_{  \mathcal{H}^{\lambda h, s}}\blA \nabla_{x,z}^2 w \brA_{F^{\lambda,s-1}}.
\end{align*}
By summing all these estimates for $1\le j\le d$,  
this will give the required estimate for
$$\blA\nabla_{x,z}\nabla_x w\brA_{F^{\lambda, s-1}},$$
 which will conclude the proof of the theorem as explained after \e{I1-I2}.

We now have to prove the claim~\e{HF30}. Recall that
\begin{equation}\label{rappel-abcd}
R =a\partial_z^2 +b\Delta_x  
+c\cdot\nabla_x\partial_z  -d\partial_z , 
\end{equation}
where 
\[
\left\{
\begin{aligned}
&a=\frac{1+\vert\nabla_x\rho\vert^2}{\partial_z\rho}-1
,\quad 
b=\partial_z\rho-1,\quad c=-2\nabla_x\rho,\\
&d=\frac{1+\vert\nabla_x\rho\vert^2}{\partial_z\rho}\partial_z^2\rho+\partial_z\rho \Delta_x \rho -2\nabla_x\rho
\cdot\nabla_x\partial_z\rho.
\end{aligned}
\right.
\]
We have
$$
[\partial_{j},R] w=(\partial_ja)\partial_z^2 w+(\partial_jb)\Delta_x w
+(\partial_jc)\cdot\nabla_x\partial_z w-(\partial_jd)\partial_z w.
$$
Then we use statement (i) in Lemma~\ref{uvF} (applied with $s_1=s_3=s_1=s-2>d/2$) 
to write that
\begin{align*}
&\blA (\partial_ja)\partial_z^2 \widetilde{w}\brA_{F^{\lambda,s-2}}\les 
\lA \partial_ja\rA_{E^{\lambda,s-2}}\blA \partial_z^2 \widetilde{w}\brA_{F^{\lambda,s-2}},\\
&\blA (\partial_jb)\Delta_x \widetilde{w}\brA_{F^{\lambda,s-2}}
\les \lA \partial_jb\rA_{E^{\lambda,s-2}}\blA \Delta_x \widetilde{w}\brA_{F^{\lambda,s-2}},\\
&\blA (\partial_jc)\cdot\nabla_x\partial_z \widetilde{w}\brA_{F^{\lambda,s-2}}\les 
\lA \partial_jc\rA_{E^{\lambda,s-2}}\blA \nabla_x\partial_z \widetilde{w}\brA_{F^{\lambda,s-2}}.
\end{align*}
By lemma~\ref{algebre0} we have
$$
\lA a\rA_{E^{\lambda,s-1}}+\lA b\rA_{E^{\lambda,s-1}}
+\lA c\rA_{E^{\lambda,s-1}} \les \Vert \eta \Vert_{  \mathcal{H}^{\lambda h, s}}.
$$
Therefore,
\begin{align*}
& \blA (\partial_ja)\partial_z^2 \widetilde{w}\brA_{F^{\lambda,s-2}}+
\blA (\partial_jb)\Delta_x \widetilde{w}\brA_{F^{\lambda,s-2}}+
\blA (\partial_jd)\partial_z \widetilde{w}\brA_{F^{\lambda,s-2}}
\\
\les&\, \Vert \eta \Vert_{  \mathcal{H}^{\lambda h, s}}\blA \nabla_{x,z}^2 \widetilde{w}\brA_{F^{\lambda,s-2}}. 
\end{align*}
It remains to estimate the term $(\partial_jd)\partial_z \widetilde{w}$. 
To do so, again, we begin by applying the product rule given by 
(ii) in Lemma~\ref{uvF} (applied with $s$ replaced by $s-2>d/2$) to write that
$$
\blA (\partial_jd)\partial_z \widetilde{w}\brA_{F^{\lambda,s-2}}
\les \lA \partial_jd\rA_{F^{\lambda,s-2}}\blA \partial_z \widetilde{w}\brA_{E^{\lambda,s-2}}.
$$
Then,  by Lemma~\ref{algebre0} we have
$$
\lA \partial_jd\rA_{F^{\lambda,s-2}}\le \lA d\rA_{F^{\lambda,s-1}}\les \Vert \eta \Vert_{  \mathcal{H}^{\lambda h, s+\mez}}.
$$
By combining the previous estimates, we see that, to complete the proof of the claim \e{HF30}, it remains  
to estimate $\blA \partial_z \widetilde{w}\brA_{E^{\lambda,s-2}}$ and $\blA \nabla_{x,z}^2 \widetilde{w}\brA_{F^{\lambda,s-2}}$ in terms of $\blA \nabla_{x,z}\widetilde{w}\brA_{F^{\lambda,s-1}}$. 
Since we will need to prove a similar result later on, we pause here to 
prove a general result. 
\begin{lemm}\label{L316}
Consider two real numbers
$$
s>\frac{d}{2}+2,\quad \lambda \in [0,1).
$$ 
Then there exist two constants $\overline{\eps}>0$ and $C>0$ such that 
for all $\eta \in \mathcal{H}^{\lambda h, s}(\xT^d)$ 
satisfying $\Vert \eta \Vert_{\mathcal{H}^{\lambda h,s}} \leq \overline{\eps}$, 
if $v$ satisfies $\Delta_{x,z}v+Rv=f$, then
\be\label{ndz2}
\lA \nabla_{x,z} v\rA_{E^{\lambda,s-\tdm}}+\blA \partial_z^2 v\brA_{F^{\lambda,s-2}}
\le C\lA \nabla_{x,z}v\rA_{F^{\lambda,s-1}}+C\lA f\rA_{F^{\lambda,s-2}}.
\ee
\end{lemm}
\begin{proof}
By interpolation (see statement~(iv) in Lemma~\ref{uvF}), we have
$$
\blA \partial_z v\brA_{E^{\lambda,s-\tdm}}\les 
\blA \partial_z^2 v\brA_{F^{\lambda,s-2}}+\blA \partial_z v\brA_{F^{\lambda,s-1}}.
$$
Therefore, it is sufficient to prove that $\blA \partial_z^2 v\brA_{F^{\lambda,s-2}}$ 
is estimated by the right-hand side of \e{ndz2}. To do so, we repeat the arguments used in Step 2. Namely, we write $\partial_z^2v=-\Delta_x v-Rv+f$, to infer that
$$
\lA \partial_z^2v\rA_{F^{\lambda,s-2}}\le \lA \nabla_xv\rA_{F^{\lambda,s-1}}+\lA Rv\rA_{F^{\lambda,s-2}}+\lA f\rA_{F^{\lambda,s-2}}.
$$
Then we use the estimate~\e{estR:s-2} to estimate the contribution of $Rv$, which 
implies that
$$
\lA \partial_z^2v\rA_{F^{\lambda,s-2}}\le \lA \nabla_xv\rA_{F^{\lambda,s-1}}
+C \overline{\eps} \big( \lA \nabla_{x,z}v\rA_{F^{\lambda,s-1}}+\lA \partial_z^2v\rA_{F^{\lambda,s-2}}\big)+\lA f\rA_{F^{\lambda,s-2}}.
$$
Then we conclude the proof by taking $\overline{\eps}$ so small that $C\overline{\eps}\le 1/2$. 
\end{proof}

By applying the previous lemma to $(v,f)=(\widetilde{w},F)$, we complete the proof of the claim~\e{HF30}, which in turn concludes the proof of Theorem~\ref{regul-fin'}.
\end{proof} 

We consider eventually the  problem
\begin{equation}\label{eq-utilde}
(\Delta_{x,z}+R)\widetilde{u}=  0 \,  \text{ in } \xT^d \times I_h,\quad
\widetilde{u}\arrowvert_{z=0} = \psi, \quad    \partial_ z \widetilde{u} \arrowvert_{z=-h}=(\partial_ z \rho \arrowvert_{z=-h})b.
\end{equation}

\begin{coro}\label{regul-ouf}
Consider two real numbers $ \lambda_0, s$ such that
$$
0\le \lambda_0 <1, \quad s>\frac{d}{2}+2.
$$
Then there exist two constants $\overline{\eps}>0$ and $C>0$  such that for all $0\leq \lambda \leq \lambda_0$,    all $\eta \in \mathcal{H}^{\lambda  h,s+ \mez }(\xT^d)$  
satisfying $\Vert \eta \Vert_{\mathcal{H}^{\lambda h,s}}\leq \overline{\eps}$, 
all $\psi  \in   \mathcal{H}^{\lambda h, s+ \mez}(\xT^d), $ all $ b \in H^{s-\mez}(\xT^d)$ and all solutions $\widetilde{u}$ of  problem \eqref{eq-utilde}, 
if we set
$$I_1= \Vert   \nabla_{x,z}  \widetilde{u} \Vert _{F^{\lambda, s}}, \quad I_2= \Vert   \partial_z^2  \widetilde{u} \Vert_{F^{\lambda, s-1}}, \quad I_3=   \Vert   \nabla_{x,z}  \widetilde{u} \Vert _{E^{\lambda, s-\mez}},
$$
then
\begin{align*}
\sum_{k=1}^3 I_k  \leq  C
\Big(  \Vert a(D_x)^\mez \psi \Vert_{\mathcal{H}^{\lambda h,s}}&+\Vert b \Vert_{H^{s-\mez}}\\
 &+ \Vert \eta \Vert_{  \mathcal{H}^{\lambda h, s+ \mez}}
\Big( \Vert a(D_x)^\mez \psi \Vert_{\mathcal{H}^{\lambda h,s-\mez}}+\Vert b \Vert_{H^{s-\tdm}}\Big)\Big),
\end{align*}
where we recall that $a(D_x) =  G_0(0)$ is introduced in \eqref{a(D)=}.
\end{coro}
\begin{proof}
This follows from Theorem~\ref{regul-fin'} and the fact that, since 
$$\partial_z \rho\arrowvert_{z=-h} = 1+ \frac{1}{h}e^{-h\vert D_x \vert}\eta$$ 
with $\Vert \eta \Vert_{\mathcal{H}^{\lambda h, s}} \leq 1,$
we have, for  $\mu= s-\mez,$   $\mu = s-\frac{3}{2},$
$$
\Vert (\partial_z \rho\arrowvert_{z=-h})b\Vert_{H^\mu} \leq C \Vert  b\Vert_{H^\mu}.
$$
This completes the proof of Corollary \ref{regul-ouf}. 
\end{proof}

\section{The  Dirichlet Neumann operator}\label{S:DN}
Given  functions $\psi$,   $b$    we consider the  problem,
\begin{equation}\label{Diri1}
\Delta_{x,y} u = 0 \text{ in } \Omega, \quad u\arrowvert_{y=\eta}= \psi, \quad \partial_y u\arrowvert_{y=-h} = b.
\end{equation}
We set
\begin{equation*}
  G(\eta)(\psi, b) = \sqrt{1+ \vert \nabla_x \eta \vert^2} \, \partial_n u\arrowvert_{y=\eta} 
  = \big( \partial_y u - \nabla_x \eta  \cdot \nabla_x u\big) \arrowvert_{y=\eta}.
\end{equation*}
This is the  Dirichlet-Neumann operator  associated to   problem \eqref{Diri1}. Notice that 
using the notations in    \eqref{DNpsi} we have
$$
G_0(0)\psi = G(0)(\psi,0)= a(D_x)\psi=\vert D_x \vert \tanh  (h \vert D_x \vert)\psi.
$$

We have then the following result.
 \begin{theo}\label{est-DN}
Consider a real number $s>2+d/2$. For all $ 0\leq \lambda_0<1$, there exist $\overline{\eps} >0$ and $C>0$ 
such that for all $0\leq \lambda \leq \lambda_0$,  for all $\eta \in \mathcal{H}^{\lambda  h,s+\frac{1}{2}}(\xT^d)$  satisfying $\Vert \eta \Vert_{\mathcal{H}^{\lambda h,s }} \leq \overline{\eps}$, 
all $\psi$ such that $ a(D_x)^\mez \psi  \in   \mathcal{H}^{\lambda h, s}(\xT^d)$,  all $ b \in H^{s-\mez}(\xT^d)$  we have
\be\label{DNn1}
\begin{aligned}
\Vert G(\eta)(\psi,b) \Vert_{ \mathcal{H}^{\lambda h, s- \mez}}
\leq C \big(& \Vert a(D_x)^\mez\psi \Vert_{\mathcal{H}^{\lambda h, s}} + \Vert b \Vert_{H^{s-\mez}}\\
&\quad+ \Vert  \eta\Vert_{\mathcal{H}^{\lambda h, s+ \mez}} \big(  \Vert a(D_x)^\mez\psi \Vert_{\mathcal{H}^{\lambda h, s- \mez}} +\Vert b \Vert_{H^{s-\tdm}}\big)\big).
\end{aligned}
\ee
\end{theo}
\begin{rema}\label{G-s-1}
\begin{itemize}
\item [\rm{(i)}]  We insist on the fact that the constants $\overline{\eps}$ and $C$ in the above Theorem depend on $\lambda_0$ but not on $\lambda$ as soon as $0\leq \lambda\leq \lambda_0.$
\item [\rm{(ii)}]  Assume that $s> 2+\delta+ \frac{d}{2}$, $ b \in H^{s-\mez}(\xT^d)$ and for $j=1,2$,
$$
\eta_j \in \mathcal{H}^{\lambda  h,s+\frac{1}{2}}(\xT^d)
\text{ with }\Vert \eta_j \Vert_{\mathcal{H}^{\lambda h,s }} \leq \overline{\eps},\quad 
a(D_x)^\mez \psi_j  \in   \mathcal{H}^{\lambda h, s}(\xT^d) .
$$
Then we may apply the above Theorem with $s'= s-\delta$, 
and we obtain an estimate of $\Vert G(\eta)(\psi,b) \Vert_{\mathcal{H}^{\lambda  h, s-\mez-\delta}}$ by the right hand side of \eqref{DNn1} for $s$ replaced by $s-\delta.$
\end{itemize}
\end{rema}
\begin{proof}[Proof of Theorem \ref{est-DN}] 
We use the notations introduced in Section \ref{defiCOV}. In the 
variable $(x,z)\in \xT^d \times (-h,0)$, we have
\begin{equation*}
G(\eta)(\psi,b) =  \Big(\frac{1+ \vert \nabla_x \rho\vert^2}{\partial_z \rho} \partial_z \widetilde{u} - \nabla_x \rho \cdot \nabla_x \widetilde{u}\Big)\Big\arrowvert_{z=0},
\end{equation*}
where $\widetilde{u}=u(x,\rho(x,z))$ satisfies the following elliptic boundary value problem:
$$
(\Delta_{x,z}+R)\widetilde{u}=0, \quad \widetilde{u}\arrowvert_{z=0} = \psi, \quad   (\partial_ z \widetilde{u}) \arrowvert_{z=-h} = (\partial_z \rho\arrowvert_{z=-h}) b.
$$
Let us introduce the function
$$
U=\frac{1+ \vert \nabla_x \rho\vert^2}{\partial_z \rho} \partial_z \widetilde{u} - \nabla_x \rho \cdot \nabla_x \widetilde{u},
$$
and set $I_h=[-h,0]$. Since $G(\eta)(\psi,b) = U\arrowvert_{z=0}$, by definition of the spaces  $E^{\lambda,\mu}$, we have
$$
\Vert G(\eta)(\psi,b) \Vert_{ \mathcal{H}^{\lambda h, s- \mez}}\le 
\lA U\rA_{E^{\lambda,s-\mez}}.
$$
Now, we use an interpolation argument (see statement~(iv) in Lemma~\ref{uvF}) to infer that
\be\label{DNn2}
\Vert G(\eta)(\psi,b) \Vert_{ \mathcal{H}^{\lambda h, s- \mez}}
\le \lA U\rA_{E^{\lambda,s-\mez}}\le C  \lA \partial_z U\rA_{F^{\lambda,s-1}}
+C\lA  U\rA_{F^{\lambda h,s}}.
\ee
The rest of the proof consists in estimating $U$ and $\partial_z U$ in terms of 
$\nabla_{x,z}\widetilde{u}$, so that the required estimate~\e{DNn1} will be a consequence 
of Corollary~\ref{regul12} and Theorem~\ref{regul-fin'}. 

\begin{lemm}
For all $s>2+d/2$ and all $ 0\leq \lambda<1$, there exist $\overline{\eps} >0$ and $C>0$ 
such that for all $\eta \in \mathcal{H}^{\lambda  h,s+\frac{1}{2}}$  satisfying $\Vert \eta\Vert_{\mathcal{H}^{\lambda h,s }} \leq \overline{\eps}$, 
there holds
\be\label{estU1F}
\lA U\rA_{F^{\lambda ,s}}
\le C (\lA \nabla_{x,z}\widetilde{u}\rA_{F^{\lambda ,s}}
+\lA \eta\rA_{\mathcal{H}^{\lambda h,s+\mez}}
\lA \nabla_{x,z}\widetilde{u}\rA_{F^{\lambda ,s-1}}),
\ee
and
\be\label{estU2F}
\lA \partial_zU\rA_{F^{\lambda ,s-1}}
\le C (\lA \nabla_{x,z}\widetilde{u}\rA_{F^{\lambda ,s}}
+\lA \eta\rA_{\mathcal{H}^{\lambda h,s+\mez}}
\lA \nabla_{x,z}\widetilde{u}\rA_{F^{\lambda ,s-1}}).
\ee
\end{lemm}
\begin{proof}
Write $U$ under the form $$U=(1+a)\partial_z \widetilde{u}- \nabla_x \rho \cdot \nabla_x \widetilde{u},$$
where $a$ 
is as defined in~\e{rappel-abcd}. Then, the tame product estimate~\e{tame-EF} below
(applied with $s_0=s-2$) implies that 
\begin{align*}
\lA U\rA_{F^{\lambda ,s}}&\les (1+\lA a\rA_{E^{\lambda ,s-2}})
\lA \partial_z \widetilde{u}\rA_{F^{\lambda h,s}}+
\lA  \nabla_x \rho\rA_{E^{\lambda,s-2}}
\lA \nabla_x \widetilde{u}\rA_{F^{\lambda ,s}}\\
&\quad +\lA a\rA_{F^{\lambda h,s}}\lA \partial_z \widetilde{u}\rA_{E^{\lambda ,s-2}}+
\lA  \nabla_x \rho\rA_{F^{\lambda ,s}}
\lA \nabla_x \widetilde{u}\rA_{E^{\lambda ,s-2}}.
\end{align*}
The contribution of $\rho$ is estimated by means of 
Lemma~\ref{est-rho} and the one of $a$ is estimated by Lemma~\ref{algebre0} (which implies that $a$ belongs to the space $\mathcal{E}_1$). Consequently, 
provided $\lA \eta\rA_{\mathcal{H}^{\lambda h,s}}\le \overline{\eps}\le 1$, we have
\begin{align*}
&\lA a\rA_{E^{\lambda,s-2}}+\lA  \nabla_x \rho\rA_{E^{\lambda,s-2}}
\le C \lA \eta\rA_{\mathcal{H}^{\lambda h,s}}\le C,\\
&\lA a\rA_{F^{\lambda ,s}}+\lA  \nabla_x \rho\rA_{F^{\lambda ,s}}\le C 
\lA \eta\rA_{\mathcal{H}^{\lambda h,s+\mez}}.
\end{align*}
By combining the previous estimates, we obtain that
\be\label{estU1}
\lA U\rA_{F^{\lambda ,s}}
\le C \lA \nabla_{x,z}\widetilde{u}\rA_{F^{\lambda ,s}}
+\lA \eta\rA_{\mathcal{H}^{\lambda h,s+\mez}}
\lA \nabla_{x,z}\widetilde{u}\rA_{E^{\lambda ,s-2}},
\ee
Now, since $\Delta_{x,z}\widetilde{u}+R\widetilde{u}=0$, it follows from Lemma~\ref{L316} that 
\be\label{coroL316}
\lA \nabla_{x,z}\widetilde{u}\rA_{E^{\lambda ,s-2}}\le 
C\lA \nabla_{x,z}\widetilde{u}\rA_{F^{\lambda ,s-1}}.
\ee
By plugging this bound in~\e{estU1}, we conclude the proof of the estimate~\e{estU1F}.

We now estimate $\partial_zU$. 
To do so, we exploit the fact that the equation $$\Delta_{x,z}\widetilde{u}+R\widetilde{u}=0$$ 
can be written in divergence form, as we have seen in \e{formediv}. More precisely, we have
\begin{equation}\label{equat'}
\partial_z\Big(\frac{1+\la\nabla_x \rho\ra^2}{\partial_z \rho}
\partial_z \widetilde{u}-\nabla_x \rho\cdot\nabla_x\widetilde{u}\Big)
+\cnx_x\big(\partial_z \rho \nabla_x \widetilde{u}-\partial_z \widetilde{u}\nabla_x\rho\big)=0.
\end{equation}
This immediately implies that
$$
\lA \partial_z U\rA_{F^{\lambda,s-1}}\le 
\lA \partial_z \rho \nabla_x \widetilde{u}-\partial_z \widetilde{u}\nabla_x\rho\rA_{F^{\lambda ,s}}.
$$
Now, as above, we apply the same product estimate~\e{tame-EF} below to infer that
\begin{align*}
\lA \partial_z U\rA_{F^{\lambda ,s-1}}&\les 
\lA \nabla_{x,z} \rho\rA_{E^{\lambda ,s-2}}
\lA \nabla_{x,z} \widetilde{u}\rA_{F^{\lambda ,s}}+\lA \nabla_{x,z} \rho\rA_{F^{\lambda ,s}}
\lA \nabla_{x,z} \widetilde{u}\rA_{E^{\lambda,s-2}}.
\end{align*}
Then we use Lemma~\ref{est-rho} 
to obtain
$$
\lA \partial_zU\rA_{F^{\lambda  ,s-1}}
\le C \lA \nabla_{x,z}\widetilde{u}\rA_{F^{\lambda ,s}}
+\lA \eta\rA_{\mathcal{H}^{\lambda ,s+\mez}}
\lA \nabla_{x,z}\widetilde{u}\rA_{E^{\lambda ,s-2}},
$$
and hence, the desired estimate~\e{estU2F} follows from~\e{coroL316}.
\end{proof}

In view of \e{DNn2} and the previous lemma, the estimate~\e{DNn1} follows directly from Corollary~\ref{regul12} and Theorem~\ref{regul-fin'}.
 \end{proof}

In the following result we shall prove that, in a certain sense, 
the Dirichlet-Neuman operator is Lipschitz with respect to $(\psi, \eta)$. 
Let us introduce some notations. If $(\theta_1, \theta_2)$ is a given couple of functions and $t \in \xR $, we shall set
$$ \theta   = \theta_1- \theta_2, \quad \Vert (\theta_1, \theta_2)\Vert_{\mathcal{H}^{\lambda h,t}} =  \sum_{j=1}^2 \Vert \theta_j \Vert_{\mathcal{H}^{\lambda  h,t}}, $$
 and we shall use these notations if $\theta = \eta $ or $a(D_x)^\mez\psi.$ Moreover, we set
 \begin{equation}\label{lambdaj}
  \left\{
 \begin{aligned}
\mathcal{G}_j  &= G(\eta_j)(\psi_j,b),  j=1,2,\\
   \lambda_1 &=   \Vert (\eta_1, \eta_2) \Vert_{\mathcal{H}^{\lambda  h,s+ \mez }}   \big(\Vert (a(D_x)^\mez\psi_1,a(D_x)^\mez\psi_2) \Vert_{\mathcal{H}^{\lambda  h,s-\mez}}+\Vert b \Vert_{H^{s- \frac{3}{2}}}\big)\\
   &+    \Vert (a(D_x)^\mez\psi_1,a(D_x)^\mez \psi_2) \Vert_{\mathcal{H}^{\lambda  h,s }} 
  + \Vert b \Vert_{H^{s-\mez}},\\
 \lambda_2 &=   \Vert (a(D_x)^\mez\psi_1,a(D_x)^\mez\psi_2) \Vert_{\mathcal{H}^{\lambda  h,s-1}} + \Vert b \Vert_{H^{s-\frac{3}{2}}},\\
 \lambda_3 &=   \Vert (\eta_1, \eta_2) \Vert_{\mathcal{H}^{\lambda  h,s+ \mez }}.
   \end{aligned}
   \right.
   \end{equation}
 Then we have the following result.
 \begin{theo}\label{G-lip} 
For  all $s > d/2 + 2$ and all $0\leq \lambda_0<1$, there exist $C>0$ and  $\overline{\eps} >0$ such that for all $0\leq \lambda \leq \lambda_0,$ for    all $\eta_j \in \mathcal{H}^{\lambda  h,s+\frac{1}{2}}(\xT^d)$  satisfying $\Vert \eta_j \Vert_{\mathcal{H}^{\lambda h,s }} \leq \overline{\eps}, $ all $\psi_j$ such that  $ a(D_x)^\mez \psi_j  \in   \mathcal{H}^{\lambda h, s}(\xT^d),  j=1,2$ and  all $ b \in H^{s-\mez}(\xT^d)$  we have  
\begin{align*}
& \Vert \mathcal{G}_1 - \mathcal{G}_2\Vert_{\mathcal{H}^{\lambda  h, s-\mez}}\\
 \leq&\, C \Big( \lambda_1 \Vert \eta_1-\eta_2   \Vert_{\mathcal{H}^{\lambda  h,s}} + \lambda_2 \Vert \eta_1-\eta_2 \Vert_{\mathcal{H}^{\lambda  h,s + \mez}}\\
&+ \lambda_3 \Vert a(D_x)^\mez(\psi_1-\psi_2)  \Vert_{\mathcal{H}^{\lambda  h,s-\mez}} 
+ \Vert a(D_x)^\mez(\psi_1-\psi_2)    \Vert_{\mathcal{H}^{\lambda h,s}}\Big).
\end{align*}
\end{theo}
\begin{rema}\label{G1-G2-s-1}
\begin{itemize}
\item[\rm{(1)}] The constants $\overline{\eps}$ and $C$ in the above Theorem depend on $\lambda_0$ but not on $\lambda$ for   $0\leq \lambda\leq \lambda_0.$
\item[\rm{(2)}]  Assume that $s> 2+ \delta + \frac{d}{2}, \delta>0$, $b \in H^{s-\mez}(\xT^d)$ and for $j=1,2$,
$$
\eta_j \in \mathcal{H}^{\lambda  h,s+\frac{1}{2}}(\xT^d)\text{ with }\Vert \eta_j \Vert_{\mathcal{H}^{\lambda h,s }} \leq \overline{\eps},\quad 
a(D_x)^\mez \psi_j  \in   \mathcal{H}^{\lambda h, s}(\xT^d).
$$
Then we may apply the above Theorem to $s'= s-\delta$ and we get an estimate of the term, 
$\Vert \mathcal{G}_1 - \mathcal{G}_2\Vert_{\mathcal{H}^{\lambda  h, s-\mez -\delta}}$ by the right hand side for $s$ replaced by $s-\delta.$
\end{itemize}
\end{rema}
\begin{proof}[Proof of Theorem \ref{G-lip}] 
Introduce the functions 
$$U_j=\frac{1+ \vert \nabla_x \rho_j\vert^2}{\partial_z \rho_j} \partial_z \widetilde{u}_j - \nabla_x \rho_j \cdot \nabla_x \widetilde{u}_j $$
for $j=1,2,$
where 
\begin{equation}\label{eq-uj}
  (\Delta_{x,z}+R_j) \widetilde{u}_j= 0   , \quad  \widetilde{u}_j\arrowvert_{z=0} = \psi_j, \quad   (\partial_ z  \widetilde{u}_j) \arrowvert_{z=-h} = (\partial_z \rho_j\arrowvert_{z=-h}) b.
\end{equation}
 Then, by definition we have $$\mathcal{G}_1- \mathcal{G}_2 = (U_1-U_2)\arrowvert_{z=0}.$$ 
 As in \eqref{DNn2}, we have
\begin{align*}
\Vert \mathcal{G}_1- \mathcal{G}_2 \Vert_{ \mathcal{H}^{\lambda h, s- \mez}}
&\le \lA U_1 -U_2\rA_{E^{\lambda,s-\mez}}\\
&\le C( \lA  U_1 -U_2\rA_{F^{\lambda,s}} +  \lA \partial_z (U_1 -U_2)\rA_{F^{\lambda,s-1}}).
\end{align*} 
Now according to \eqref{equat'}, our equation  on $\widetilde{u}_j$ reads $$\partial_z U_j + \text{div}_x V_j=0$$ with
\begin{equation}\label{UjVj}
U_j=\frac{1+ \vert \nabla_x \rho_j\vert^2}{\partial_z \rho_j} \partial_z \widetilde{u}_j - \nabla_x \rho_j \cdot \nabla_x \widetilde{u}_j, \quad 
\quad V_j = \partial_z \rho_j\nabla_x \widetilde{u}_j -  \partial_z \widetilde{u}_j\nabla_x \rho_j, 
\end{equation}
 It follows that
 \begin{equation}\label{DNdiff}
 \Vert \mathcal{G}_1- \mathcal{G}_2 \Vert_{ \mathcal{H}^{\lambda h, s- \mez}}\leq C(\lA U_1 -U_2\rA_{F^{\lambda,s}}+ \lA  V_1 -V_2\rA_{F^{\lambda,s}}).
 \end{equation}
 Recall  that 
 \[
 \partial_z \rho_j = 1+ q_j \quad \text{and} \quad 
 \nabla_x \rho_j= \frac{1}{h}(z+h)e^{z \vert D_x \vert}\nabla_x \eta_j,
 \]
 where 
\begin{equation*}
\left\{
\begin{aligned}
& q_j = \frac{1}{h}e^{z \vert D_x \vert} \eta_j + \frac{1}{h}(z+h)e^{z \vert D_x \vert}\vert D_x \vert \eta_j,\\ 
    &\partial_z \rho_j\arrowvert_{z=-h}= 1 + \frac{1}{h} e^{-h \vert D_x \vert} \eta_j, \\
    &\frac{1+ \vert \nabla_x \rho_j\vert^2}{\partial_z \rho_j}    = 1+ A_j,\quad A_j = - f(q_j) - \vert\nabla_x \rho_j\vert^2 f(q_j) +  \vert\nabla_x \rho_j\vert^2,
\end{aligned}
\right.
\end{equation*}
with 
$$
f(q_j)= \frac{q_j}{1+ q_j}.
$$  
With these notations and using \eqref{UjVj} we can write
\begin{equation}\label{DiffUV}
\left\{
\begin{aligned}
\widetilde{u}&= \widetilde{u}_1- \widetilde{u}_2,\\
 U_1-U_2&= (A_1-A_2) \partial_z\widetilde{u}_2-(\nabla_x \rho_1- \nabla_x \rho_2)\cdot \nabla_x \widetilde{u}_2\\
 &\quad 
 + (1+ A_1)\partial_z\widetilde{u} - \nabla_x \rho_1\cdot\nabla_x\widetilde{u},\\ 
 V_1 -V_2&= (\partial_z \rho_1- \partial_z \rho_2) \nabla_x \widetilde{u}_2  \\
 &\quad -(\nabla_x \rho_1- \nabla_x \rho_2) \partial_z \widetilde{u}_2+ \partial_z \rho_1  \nabla_x \widetilde{u} -  \nabla_x \rho_1  \partial_z \widetilde{u}.
\end{aligned}
\right.
\end{equation}
The first two terms in  the right hand side of \eqref{DiffUV} are of the form $$(p_1-p_2)\nabla_{x,z} \widetilde{u}_2\quad   \text{ with }p\in \mathcal{F}_1$$
by Lemma~\ref{algebre1}, where $\mathcal{F}_1$ has been defined in Definition~\ref{Frond}. 
They are estimated as follows. Using Lemma~\ref{uvF} (iii) with $ s_0=s-\frac{3}{2}, $ we can write
\begin{align*}
 \Vert(p_1-p_2)\nabla_{x,z} \widetilde{u}_2\Vert_{F^{\lambda h, s}} 
& \les  \Vert p_1-p_2  \Vert_{E^{\lambda  , s-\frac{3}{2}}}\Vert \nabla_{x,z}\widetilde{u}_2\Vert_{F^{\lambda,s}} \\
&+\Vert p_1-p_2  \Vert_{F^{\lambda  , s}}\Vert \nabla_{x,z}\widetilde{u}_2\Vert_{E^{\lambda,s-\frac{3}{2}}},
\end{align*}
and hence,
 \begin{align*}
  \Vert(p_1-p_2)\nabla_{x,z} \widetilde{u}_2\Vert_{F^{\lambda h, s}}&\les \Vert  p_1-p_2 \Vert_{L^\infty(I_h, \mathcal{H}^{\lambda h, s-\frac{3}{2}})} \Vert \nabla_{x,z} \widetilde{u}_2\Vert_{F^{\lambda  , s}}\\
  &+ \Vert p_1-p_2 \Vert_{L^2(I_h, \mathcal{H}^{\lambda h, s})} \Vert \nabla_{x,z} \widetilde{u}_2\Vert_{E^{\lambda h, s-\frac{3}{2}}}.
  \end{align*}
   Using the definition of $\mathcal{F}_1$,   \eqref{eq-uj}, Corollary~\ref{regul-ouf} , Lemma~\ref{L316}, Corollary~\ref{regul12}     and the notation in \eqref{lambdaj}, we obtain
 \begin{equation}\label{G1-G2-1}
 \Vert(p_1-p_2)\nabla_{x,z} \widetilde{u}_2\Vert_{F^{\lambda h, s}}\leq \lambda_1 \Vert \eta_1- \eta_2\Vert_{\mathcal{H}^{\lambda h, s}} + \lambda_2 \Vert \eta_1- \eta_2\Vert_{\mathcal{H}^{\lambda h, s+\mez}}.
\end{equation}
The last two terms in  the right hand side of \eqref{DiffUV} are of the form 
$$q  \nabla_{x,z} (\widetilde{u}_1- \widetilde{u}_2)  = q \nabla_{x,z} \widetilde{u}, \quad \text{with } q =1 \text{ or } q\in \mathcal{E}_1,$$
where $\mathcal{E}_1$ has been defined in Definition~\ref{Erond}.
They are estimated as follows. By Lemma~\ref{uvF} with with $s_0 = s-\frac{3}{2}$ we can write
\begin{equation*} 
 \Vert q \nabla_{x,z} \widetilde{u}\Vert_{F^{\lambda,s}}  \les   \Vert q \Vert_{L^\infty(I_h,\mathcal{H}^{\lambda h, s-\frac{3}{2}})}\Vert \nabla_{x,z} \widetilde{u} \Vert_{F^{\lambda  , s} }
 + \Vert q \Vert_{L^2(I_h,\mathcal{H}^{\lambda h, s})}\Vert \nabla_{x,z} \widetilde{u} \Vert_{E^{\lambda  , s-\frac{3}{2}}}.
\end{equation*}
Since $q\in \mathcal{E}_1$ we deduce from Definition~\ref{Erond} that
\begin{equation}\label{G1-G2-2}
\Vert q \nabla_{x,z} \widetilde{u}\Vert_{F^{\lambda,s}}  \les  \Vert \nabla_{x,z} \widetilde{u} \Vert_{F^{\lambda h, s}}+ \Vert \eta_1 \Vert_{\mathcal{H}^{s+ \mez}}\Vert \nabla_{x,z} \widetilde{u} \Vert_{E^{\lambda h, s-\frac{3}{2}}}.
\end{equation}

We see therefore that we have to estimate $\nabla_{x,z}\widetilde{u}.$ For that, according to   \eqref{eq-uj}, we notice  that $\widetilde{u}=  \widetilde{u}_1-\widetilde{u}_2$ is solution of the problem
 $$(\Delta_{x,z} +R_1)\widetilde{u} = (R_2-R_1)\widetilde{u}_2, \quad \widetilde{u} \arrowvert_{z=0} = \psi_1- \psi_2, ~\partial_z \widetilde{u}\arrowvert_{z=-h}=\frac{1}{h} e^{-h \vert D_x \vert}(\eta_1- \eta_2) b.$$
 Notice that for every $\mu> \frac{d}{2}$ we have
 \begin{equation*}
 \bigg\Vert \frac{1}{h} e^{-h \vert D_x \vert}(\eta_1- \eta_2) b \bigg\Vert_{H^\mu} \leq C\Vert    \eta_1- \eta_2   \Vert_{\mathcal{H}^{\lambda h,\mu}}\Vert   b\Vert_{H^\mu}.
 \end{equation*}
 Therefore, using Theorem~\ref{regul-fin'}, we can write
 \begin{equation}\label{G1-G2-3}
 \begin{aligned}
&\Vert   \nabla_{x,z}  \widetilde{u} \Vert _{F^{\lambda, s}}
 \les   
 \Vert  (R_2-R_1)\widetilde{u}_2\Vert_{F^{\lambda,s- 1}} + \Vert a(D_x)^\mez (\psi_1-\psi_2) \Vert_{\mathcal{H}^{\lambda h,s}}\\
&+\Vert \eta_1-\eta_2 \Vert_{\mathcal{H}^{\lambda h,s-\mez}}\Vert   b\Vert_{H^{s-\mez}} 
  +C \Vert \eta_1 \Vert_{  \mathcal{H}^{\lambda h, s+ \mez}}
 \Vert(R_2-R_1)\widetilde{u}_2\Vert_{F^{\lambda,s-2}}\\
&+ \Vert a(D_x)^\mez (\psi_1-\psi_2) \Vert_{\mathcal{H}^{\lambda h,s-\mez}} 
 +\Vert \eta_1-\eta_2 \Vert_{\mathcal{H}^{\lambda h,s-\tdm}}\Vert   b\Vert_{H^{s-\tdm}}.
\end{aligned}
\end{equation}
Moreover using Lemma~\ref{L316} we can write
$$
\lA \nabla_{x,z} \widetilde{u}\rA_{E^{\lambda,s-\tdm}}
\les  \lA \nabla_{x,z}\widetilde{u}\rA_{F^{\lambda,s-1}}+ \lA (R_2-R_1)\widetilde{u}_2\rA_{F^{\lambda,s-2}}.
$$
Then we can use Corollary~\ref{regul12} and obtain
\begin{equation}\label{G1-G2-4}
\lA \nabla_{x,z} \widetilde{u}\rA_{E^{\lambda,s-\tdm}}\les  \lA (R_2-R_1)\widetilde{u}_2\rA_{F^{\lambda,s-2}}+ \Vert \eta_1-\eta_2 \Vert_{\mathcal{H}^{\lambda h,s-\tdm}}\Vert   b\Vert_{H^{s-\tdm}}.
\end{equation}
Using \eqref{G1-G2-2}, \eqref{G1-G2-3} and \eqref{G1-G2-4}, we obtain
\begin{equation*}
\begin{aligned}
 &\Vert q  \nabla_{x,z}  \widetilde{u} \Vert _{F^{\lambda, s}}\\
 & \les   
 \Vert  (R_2-R_1)\widetilde{u}_2\Vert_{F^{\lambda,s- 1}} + \Vert a(D_x)^\mez (\psi_1-\psi_2) \Vert_{\mathcal{H}^{\lambda h,s}}\\
 &+\Vert \eta_1-\eta_2 \Vert_{\mathcal{H}^{\lambda h,s-\mez}}\Vert   b\Vert_{H^{s-\mez}}\\ 
&  +  \Vert \eta_1 \Vert_{  \mathcal{H}^{\lambda h, s+ \mez}}
\Big(\Vert(R_2-R_1)\widetilde{u}_2\Vert_{F^{\lambda,s-2}}
 + \Vert a(D_x)^\mez (\psi_1-\psi_2) \Vert_{\mathcal{H}^{\lambda h,s-\mez}}\\ 
 &+\Vert \eta_1-\eta_2 \Vert_{\mathcal{H}^{\lambda h,s-\tdm}}\Vert   b\Vert_{H^{s-\tdm}}\Big).
 \end{aligned}
\end{equation*}
Using the definition of $\lambda_j, j=1,2,3$, we obtain
\begin{equation}\label{G1-G2-5}
\begin{split}
 & \Vert q  \nabla_{x,z}  \widetilde{u} \Vert _{F^{\lambda, s}}\\
  \les &\,   
 \Vert  (R_2-R_1)\widetilde{u}_2\Vert_{F^{\lambda,s- 1}} +\Vert \eta_1\Vert_{\mathcal{H}^{\lambda h, s + \mez}}\Vert(R_2-R_1)\widetilde{u}_2\Vert_{F^{\lambda,s-2}}
 + \lambda_1 \Vert \eta_1-\eta_2 \Vert_{\mathcal{H}^{\lambda h,s }}\\
& + \lambda_3 \Vert a(D_x)^\mez (\psi_1-\psi_2) \Vert_{\mathcal{H}^{\lambda h,s-\mez }}+\Vert a(D_x)^\mez (\psi_1-\psi_2) \Vert_{\mathcal{H}^{\lambda h,s}}.
  \end{split}
\end{equation}
To complete the proof, we are led to estimate $(R_1-R_2)\widetilde{u}_2.$ According to \eqref{rappel-abcd}, we have
\[
\left\{
\begin{aligned}
 & (R_2-R_1) =(a_1-a_2)\partial_z^2  +(b_1-b_2)\Delta_x  
+(c_1-c_2)\cdot\nabla_x\partial_z   -(d_1-d_2)\partial_z \\
&  a_j=\frac{1+\vert\nabla_x\rho_j\vert^2}{\partial_z\rho}-1,\quad b_j=\partial_z\rho_j-1,\quad c_j=-2\nabla_x\rho_j,\\
& d_j =\frac{1+\vert\nabla_x\rho_j\vert^2}{\partial_z\rho_j}\partial_z^2\rho_j+\partial_z\rho_j \Delta_x \rho_j -2\nabla_x\rho_j
\cdot\nabla_x\partial_z\rho_j.
\end{aligned}
\right.
\]
Estimate of $\Vert  (R_2-R_1)\widetilde{u}_2\Vert_{F^{\lambda,s- 1}}.$

The first three terms in $(R_1-R_2)\widetilde{u}_2$ are estimated in the same manner. Since $a,b,c \in \mathcal{F}_1$ (see Lemma~\ref{algebre1}) for $r\in \{a,b,c\}$, using Lemma~\ref{uvF} (ii), we can write
\begin{align*}
 \Vert (r_1-r_2) \nabla_{x,z}^2 \widetilde{u}_2\Vert_{F^{\lambda, s-1}}&\leq C\Vert  r_1-r_2  \Vert_{L^\infty(I_h, \mathcal{H}^{\lambda h, s-1})}\Vert   \nabla_{x,z}^2 \widetilde{u}_2\Vert_{F^{\lambda, s-1}},\\
 &\leq C\Vert \eta_1-\eta_2\Vert_{\mathcal{H}^{\lambda h, s}}\Vert   \nabla_{x,z}^2 \widetilde{u}_2\Vert_{F^{\lambda, s-1}}.
\end{align*}
Thanks to Corollary~\ref{regul-ouf}, we have
\begin{equation*}
\begin{aligned}
\Vert (r_1-r_2) \nabla_{x,z}^2 \widetilde{u}_2\Vert_{F^{\lambda, s-1}}&\leq\Vert \eta_1-\eta_2\Vert_{\mathcal{H}^{\lambda h, s}} 
\Big(  \Vert a(D_x)^\mez \psi_2 \Vert_{\mathcal{H}^{\lambda h,s}} +\Vert b \Vert_{H^{s-\mez}}\\
 &+ \Vert \eta_2 \Vert_{  \mathcal{H}^{\lambda h, s+ \mez}}
\big( \Vert a(D_x)^\mez \psi_2 \Vert_{\mathcal{H}^{\lambda h,s-\mez}}+\Vert b \Vert_{H^{s-\tdm}}\big)\Big).
\end{aligned}
\end{equation*}
 With the notations in \eqref{lambdaj} one can deduce eventually that
 \begin{equation}\label{G1-G2-6}
 \Vert (r_1-r_2) \nabla_{x,z}^2 \widetilde{u}_2\Vert_{F^{\lambda, s-1}}\leq \lambda_1  \Vert \eta_1-\eta_2\Vert_{\mathcal{H}^{\lambda h, s}},
 \end{equation}
 where $r =a$ or $b$ or $c$.
 
 We consider now   the term $\Vert (d_1-d_2)\partial_z \widetilde{u}_2\Vert_{F^{\lambda, s-1}}.$ Using Lemma~\ref{uvF} (iii) with $s_0 = s-2$, one can write
 \begin{align*}
  \Vert (d_1-d_2)\partial_z \widetilde{u}_2\Vert_{F^{\lambda, s-1}}&\les     \Vert  d_1-d_2 \Vert_{L^\infty(I_h, \mathcal{H}^{\lambda h, s-2})} \Vert  \partial_z \widetilde{u}_2\Vert_{F^{\lambda, s-1}}\\
  &+ \Vert  d_1-d_2 \Vert_{L^2(I_h, \mathcal{H}^{\lambda h, s-1})} \Vert  \partial_z \widetilde{u}_2\Vert_{E^{\lambda, s-2}}. 
\end{align*}
 By Lemma~\ref{algebre1} we have $d\in \mathcal{F}_2.$ Therefore,
\begin{align*}
 &\Vert  d_1-d_2  \Vert_{L^\infty(I_h, \mathcal{H}^{\lambda, s-2})} \les \Vert \eta_1-\eta_2\Vert_{\mathcal{H}^{\lambda h,s}},\\
 &\Vert  d_1-d_2  \Vert_{L^2(I_h, \mathcal{H}^{\lambda, s-1})} \les \Vert \eta_1-\eta_2\Vert_{\mathcal{H}^{\lambda h,s+\mez}}.
\end{align*}
By Corollary~\ref{regul12} we have
\begin{equation}\label{G1-G2-60}
 \Vert  \partial_z  \widetilde{u}_2 \Vert _{F^{\lambda,s-1}}\les
  \blA a(D_x)^\mez \psi_2\brA_{\mathcal{H}^{\lambda h,s-1}} + \Vert b \Vert_{H^{s- \tdm}(\xT^d)}.
  \end{equation}
  By   Lemma~\ref{L316} and Corollary~\ref{regul12} we have
    $$
 \Vert  \partial_z \widetilde{u}_2\Vert_{E^{\lambda, s-2}}\les \Vert  \nabla_{x,z}\widetilde{u}_2\Vert_{F^{\lambda, s-1}} \les \Vert a(D_x)\psi_2\Vert_{\mathcal{H}^{\lambda, s-1}}+ \Vert b \Vert_{H^{s-\frac{3}{2}}}.
$$
Using   the notation in \eqref{lambdaj}, we obtain
\begin{equation}\label{G1-G2-7}
 \Vert (d_1-d_2)\partial_z \widetilde{u}_2\Vert_{F^{\lambda, s-1}}\leq \lambda_1 \Vert \eta_1-\eta_2\Vert_{\mathcal{H}^{\lambda h, s}}+ \lambda_2 \Vert \eta_1-\eta_2\Vert_{\mathcal{H}^{\lambda h, s+\mez}}. 
 \end{equation}
 It follows from \eqref{G1-G2-6} and \eqref{G1-G2-7} that
 \begin{equation}\label{G1-G2-8}
 \Vert (R_1-R_2)\widetilde{u}_2\Vert_{F^{\lambda, s-1}} \les \lambda_1 \Vert \eta_1-\eta_2\Vert_{\mathcal{H}^{\lambda h, s}}+ \lambda_2 \Vert \eta_1-\eta_2\Vert_{\mathcal{H}^{\lambda h, s+\mez}}.
 \end{equation}
 
 Estimate of $\Vert (R_1-R_2)\widetilde{u}_2\Vert_{F^{\lambda, s-2}}.$
 
 We use the same notations as above. Since $s-2> \frac{d}{2}$, we can write
 $$\Vert (r_1-r_2)\nabla_{x,z}^2 \widetilde{u}_2\Vert_{F^{\lambda, s-2}}\les \Vert  r_1-r_2  \Vert_{L^\infty(I_h,\mathcal{H}^{\lambda h, s-2})}  \Vert  \nabla_{x,z}^2 \widetilde{u}_2\Vert_{F^{\lambda, s-2}}.$$
 Since $r \in \mathcal{F}_1$, we have $$\Vert  r_1-r_2  \Vert_{L^\infty(I_h,\mathcal{H}^{\lambda h, s-2})}\les \Vert \eta_1-\eta_2\Vert_{\mathcal{H}^{\lambda h, s-1}}.$$ 
 Moreover, from Lemma~\ref{L316} and Corollary~\ref{regul12} we have
 $$\Vert  \nabla_{x,z}^2 \widetilde{u}_2\Vert_{F^{\lambda, s-2}}\les \Vert a(D_x)^\mez \psi_2\Vert_{\mathcal{H}^{\lambda h, s-1}} + \Vert b \Vert_{H^{s-\frac{3}{2}}}.$$
Hence we find that
 \begin{equation*} 
 \Vert (r_1-r_2)\nabla_{x,z}^2 \widetilde{u}_2\Vert_{F^{\lambda, s-2}}\les \Vert \eta_1-\eta_2\Vert_{\mathcal{H}^{\lambda h, s-1}}\big(\Vert a(D_x)^\mez \psi_2\Vert_{\mathcal{H}^{\lambda h, s-1}} + \Vert b \Vert_{H^{s-\frac{3}{2}}}\big).
 \end{equation*}
 
 Now,
 $$\Vert (d_1-d_2)\partial_z \widetilde{u}_2\Vert_{F^{\lambda, s-2}}\les  \Vert  d_1-d_2   \Vert_{L^\infty(I_h, \mathcal{H}^{\lambda h, s-2})} \Vert  \partial_z \widetilde{u}_2\Vert_{F^{\lambda, s-2}}.$$
 Since $d \in \mathcal{F}_2$, we have $$\Vert  d_1-d_2   \Vert_{L^\infty(I_h, \mathcal{H}^{\lambda h, s-2})} \les \Vert \eta_1-\eta_2\Vert_{\mathcal{H}^{\lambda h, s}}.$$
 Using \eqref{G1-G2-60}, we obtain
 $$\Vert (d_1-d_2)\partial_z \widetilde{u}_2\Vert_{F^{\lambda, s-2}}\les \Vert \eta_1-\eta_2\Vert_{\mathcal{H}^{\lambda h, s}}\big(\Vert a(D_x)^\mez \psi_2\Vert_{\mathcal{H}^{\lambda h, s-1}} + \Vert b \Vert_{H^{s-\frac{3}{2}}}\big).$$
 Therefore, we get 
 \begin{equation}\label{G1-G2-9}
 \Vert (R_1-R_2)\widetilde{u}_2\Vert_{F^{\lambda, s-2}} \les \Vert \eta_1-\eta_2\Vert_{\mathcal{H}^{\lambda h, s}}\big(\Vert a(D_x)^\mez \psi_2\Vert_{\mathcal{H}^{\lambda h, s-1}} + \Vert b \Vert_{H^{s-\frac{3}{2}}}\big).
 \end{equation}
 It follows from \eqref{G1-G2-5}, \eqref{G1-G2-8} and \eqref{G1-G2-9} that
 \begin{equation}\label{G1-G2-10}
 \begin{aligned}
  \Vert q  \nabla_{x,z}  \widetilde{u} \Vert _{F^{\lambda, s}}
  &\les   \lambda_1\Vert \eta_1-\eta_2\Vert_{\mathcal{H}^{\lambda h, s}} + \lambda_2 \Vert \eta_1-\eta_2\Vert_{\mathcal{H}^{\lambda h, s+\mez}}\\
   &+ \lambda_3\Vert a(D_x)(\psi_1-\psi_2)\Vert_{\mathcal{H}^{\lambda h, s-\mez}}  + \Vert a(D_x)(\psi_1-\psi_2)\Vert_{\mathcal{H}^{\lambda h, s}}.
  \end{aligned}
 \end{equation}
 Then Theorem~\ref{G-lip} follows from \eqref{DNdiff}, \eqref{DiffUV},  \eqref{G1-G2-1} and \eqref{G1-G2-10}. 
 \end{proof}
 
The following  result is an  immediate consequence of Theorem~\ref{G-lip}.
\begin{coro}\label{estG0-lip}
Let $ 0\leq \lambda_0<1$  and   $   s > 2+ \frac{d}{2}$.  Then there exist two constant $\overline{\eps} >0$ and $C>0$ such that  for all $0\leq \lambda \leq \lambda_0$, all $\eta\in \mathcal{H}^{\lambda  h,s+\frac{1}{2}}, a(D_x)^\mez\psi \in \mathcal{H}^{\lambda  h,s }$  satisfying $\Vert \eta \Vert_{\mathcal{H}^{\lambda h,s }} \leq \overline{\eps}$, we have
\begin{align*}
& \Vert G(\eta)(\psi,0)- G(0)(\psi,0)\Vert_{\mathcal{H}^{\lambda h, s-\mez}} \\
\leq&\, C\big( \Vert \eta \Vert_{\mathcal{H}^{\lambda h,s+ \mez }}\Vert a(D_x)^\mez \psi \Vert_{\mathcal{H}^{\lambda h,s-\mez}} + \Vert \eta \Vert_{\mathcal{H}^{\lambda h,s }} \Vert a(D_x)^\mez\psi \Vert_{\mathcal{H}^{\lambda h,s}}\big) .
\end{align*}
\end{coro}
\begin{rema}
 Assume that $s> 2 + \delta+ \frac{d}{2}, \delta>0 $, $\eta  \in \mathcal{H}^{\lambda  h,s+\frac{1}{2}}(\xT^d)$  with $\Vert \eta  \Vert_{\mathcal{H}^{\lambda h,s }} \leq \overline{\eps}, $  and that $ a(D_x)^\mez \psi   \in   \mathcal{H}^{\lambda h, s}(\xT^d).$      Then we may apply Corollary \ref{estG0-lip} to $s'= s-\delta$, and obtain an estimate of $\Vert G(\eta)(\psi,0)- G(0)(\psi,0)\Vert_{\mathcal{H}^{\lambda h, s-\mez - \delta}}$ by the right hand side for $s$ replaced by $s-\delta.$
 \end{rema}

 We shall use later on the following result concerning the Dirichlet-Neumann operator defined in \eqref{def-DN}. Recall 
 from \eqref{def-VB} that
 $$ B=   \frac{\partial \phi}{\partial y} \big\arrowvert_{y= \eta(x)} =\frac{G(\eta)(\psi,b) + \nabla_x \eta \cdot \nabla_x \psi }{  1+ \vert \nabla_x \eta \vert^2}, 
 $$
 $$
V =    \nabla_x \phi\big \arrowvert_{y= \eta(x)} = \nabla_x \psi - B  \nabla_x \eta.$$

\begin{lemm}\label{der-DN}
 We have
$$\nabla_x\big[G(\eta)(\psi,b)\big] = G(\eta)(V, \nabla_x b)-(V \cdot \nabla_x)\nabla_x \eta - (\cnx V)\nabla_x \eta.$$
 \end{lemm}
 \begin{proof}
 According to \eqref{def-DN}, we have, with $\partial_i = \frac{\partial}{\partial x_i}, 1 \leq i \leq d,$
\[
\partial_i\big[G(\eta)(\psi,b)\big]= I + II + III,
\]
where 
 \begin{align*}
     I &= (\partial_y \partial_i \phi - \nabla_x \eta \cdot \nabla_x \partial_i \phi)(y, \eta(x)),\\
     II&= -\nabla_x \partial_i \eta(x) \cdot \nabla_x \phi(y, \eta(x)),\\
     III&= -\partial_i \eta(x) \big(\partial_y^2 \phi- \nabla_x \eta \cdot \nabla_x \partial_y \phi\big)(x, \eta(x)).
  \end{align*}
  Now, $\partial_i \phi$ is the solution of the problem:
\[
\left\{
\begin{aligned}
& \Delta_{x,y} \partial_i \phi = 0  \quad \text{ in} -h<y<\eta(x),\\
&\partial_i \phi\arrowvert_{y = \eta(x)} = V, \quad \partial_y \partial_i \phi\arrowvert_{y=-h} = \partial_i b.
\end{aligned}
\right.
\]
It follows that $I= G(\eta)(V, \partial_i b).$ By the definition of $V$ we have 
  $$II  = - (V\cdot \nabla_x)(\partial_i \eta).$$ Eventually, since $V = \nabla_x \phi(x, \eta(x))$, we have
  \begin{align*}
   \cnx V &= \sum_{i=1}^d\big[  \partial_i^2 \phi (x, \eta(x)) + \partial_i\partial_y \phi(x, \eta(x))\partial_i \eta(x)\big]\\
   &= -\partial^2_y \phi(y, \eta(x)) + \nabla_x \eta(x) \cdot \nabla_x \partial_y \phi (x, \eta(x)),
   \end{align*}
   since $\partial^2_y + \sum_{i=1}^d \partial_i^2 \phi = 0.$  It follows that $$III= -(\cnx V)\nabla_x \eta.$$
Thus, Lemma \ref{der-DN} is proved. 
 \end{proof}
 \begin{lemm}
We have
\begin{equation}\label{GB=V}
 G(\eta)(B, - \Delta_x( \phi\arrowvert_{y=-h})) = - \cnx V.
 \end{equation}
\end{lemm}
 \begin{proof}
 We have
\begin{align*}
 \text{div } V &= \sum_{j=1}^d (\partial_j^2 \phi + \partial_j \eta \partial_j \partial_y \phi)\arrowvert_{y=\eta} = (\Delta_x \phi + \nabla_x \eta \cdot \nabla_x \partial_y \phi)\arrowvert_{y=\eta} \\
 &= (- \partial_y^2 \phi + \nabla_x \eta \cdot \nabla_x \partial_y \phi)\arrowvert_{y=\eta} \\
 &= -((\partial_y - \nabla_x \eta \cdot \nabla_x) \partial_y \phi)\arrowvert_{y=\eta}.
\end{align*}
Since $\partial_y \phi$ is a  solution of the problem
$$\Delta_{x,y} (\partial_y \phi)= 0, \quad \partial_y \phi\arrowvert_{y=\eta}= B, \quad \partial_y(\partial_y \phi)\arrowvert_{y=-h} = - \Delta_x( \phi\arrowvert_{y=-h}),$$
we get
$$G(\eta)(B, - \Delta_x( \phi\arrowvert_{y=-h})) = ((\partial_y - \nabla_x \eta \cdot \nabla_x)) \partial_y \phi\arrowvert_{y=\eta}.$$
This gives the desired identity.
 \end{proof}

\section{Another  Dirichlet-Neumann operator}\label{S:DNsuite}
Let $h>0$ and  $\eta \in W^{1, \infty}(\xT^d)$ be such that 
$\Vert \eta \Vert_{L^\infty(\xT^d)} \leq \eps \ll h$. Setting
$$
\mathcal{O} = \{(x,y): x \in \xT^d, -\eta(x) -h < y<0\},
$$
we consider the Dirichlet problem
\begin{equation}\label{DN2}
\Delta_{x,y} v = 0 \text{ in } \mathcal{O},
\quad v \arrowvert_{y=0} = b, \quad v\arrowvert_{y = - \eta(x) -h} = B.
\end{equation}
\begin{prop}\label{est-DN2}
For all $ s \in \xR$ there exist    $C>0$  and $\mathcal{F}:\xR^+\to \xR^+$ non-decreasing   such that for all  solution of the  problem \eqref{DN2} we have
$$
\Big\Vert \frac{\partial v}{\partial y}\Big\arrowvert_{y=0} \Big \Vert_{H^s(\xT^d)} \leq C  \Vert b \Vert_{H^{s+1}(\xT^d)}  + \mathcal{F}\big(\Vert \eta \Vert_{W^{1, \infty}(\xT^d)}) \big( \Vert b \Vert_{H^{1}(\xT^d)} + \Vert B \Vert_{H^{1}(\xT^d)}\big).$$
\end{prop}
\begin{proof}
Let $\chi \in C^\infty(\xR)$ be such that $\chi(y) = 1$ if $ - \frac{h}{4} \leq y \leq 0$, $\chi(y) =0$ if $y \leq - \frac{h}{2}$. We set  $w = \chi(y) v$. Then $w$ is  solution of the  problem:
$$\Delta_{x,y} w = 2 \chi'(y) \partial_y v + \chi''(y) v: = F \quad \text{ for }   - \frac{h}{2}<y<0, \quad w \arrowvert_{y=0} = b, \quad w\arrowvert_{y \leq -\frac{h}{2}} = 0.$$
We can solve explicitly this  problem.  Indeed, taking the Fourier transform with respect to $x$, we are lead to solve the two following problems:  
\begin{align*}
& (\partial_y + \vert \xi \vert) w_1 = \widehat{F}, \quad w_1 \arrowvert_{y \leq - \frac{h}{2}} = 0,\\
&  (\partial_y - \vert \xi \vert) \widehat{w} = w_1, \quad  \widehat{w} \arrowvert_{y =0} = \widehat{b}.
\end{align*}
For $-\frac{h}{2} \leq y \leq 0$ we obtain
$$  \widehat{w}(y, \xi)= e^{y \vert \xi \vert}  \widehat{b} - \int_y^0 \int_{-\frac{h}{2}}^s e^{(y+ \sigma -2s) \vert \xi \vert}  \widehat{F}(\sigma, \xi)\, ds \, d \sigma.$$
It follows that
$$\frac{\partial  \widehat{v}}{\partial y}\arrowvert_{y=0} = \frac{\partial  \widehat{w}}{\partial y}\arrowvert_{y=0} = \vert \xi \vert  \widehat{b} + \int_{-\frac{h}{2}}^0 e^{\sigma \vert \xi \vert} (\chi'(\sigma) \partial_y \widehat{v}(\sigma, \xi) + \chi''(\sigma)   \widehat{v}(\sigma, \xi) ) \, d\sigma.$$
On the support of a  derivative   of $\chi$ we have $\sigma \leq - \frac{h}{4}$. Multiplying both members  by $\langle \xi \rangle^s$ and  using the fact that $\langle \xi \rangle^s e^{- \frac{h}{4} \vert \xi \vert} $ is  bounded on $\xT^d$, we obtain easily the estimate
$$ \bigg\Vert \frac{\partial v}{\partial y}\arrowvert_{y=0} \bigg\Vert_{H^s(\xT^d)} \leq C \big( \Vert b \Vert_{H^{s+1}(\xT^d)} + \Vert v \Vert_{H^1((- \frac{h}{2}, 0) \times \xT^d)}\big).$$
Since the problem \eqref{DN2} is  variational, we see easily that
$$
\Vert v \Vert_{H^1((- \frac{h}{2}, 0) \times \xT^d)} \leq \mathcal{F}\big(\Vert \eta \Vert_{W^{1, \infty}(\xT^d)}) \big( \Vert b \Vert_{H^{1}(\xT^d)} + \Vert B \Vert_{H^{1}(\xT^d)}\big).
$$
This completes the proof.
\end{proof}
\begin{coro}\label{est-phih}
Let $\phi$ be the   solution of the  problem \eqref{Diri1}, that is,
$$\Delta_{x,y} \phi = 0 \text{ in } -h<y<\eta(x), \quad \phi\arrowvert_{y= \eta(x)}= \psi, \quad \partial_y \phi\arrowvert_{y=-h} = b.$$
Let  $\phi_h = \phi(x,-h)$, $B=\partial_y\phi(x,\eta(x))$ and fix 
$s_0> \frac{d}{2}$. Then, for all $\mu \in \xR$ there exists $C>0$ such that
$$\Vert \Delta_x \phi_h \Vert_{H^\mu(\xT^d)} \leq C  \Vert b  \Vert_{H^{\mu+1}(\xT^d)} + \mathcal{F}\big(\Vert \eta \Vert_{H^{s_0 +1}(\xT^d)}\big) \big(\Vert  b\Vert_{H^1(\xT^d)} + \Vert B \Vert_{H^1(\xT^d)} \big)  .$$
\end{coro}
\begin{proof} 
Set  $u = \partial_y \phi$. It satisfies
$$\Delta_{x,y} u= 0, \quad u\arrowvert_{y= \eta(x)}= B, \quad u\arrowvert_{y=-h} = b.$$
On the other hand, we have
$$ \partial_y u \arrowvert_{y=-h} = \partial^2_y \phi \arrowvert_{y=-h}= - \Delta_x\phi \arrowvert_{y=-h} = - \Delta_x \phi_h.$$
Set  $v(x,y) = u(x, -y-h)$. Then $v$ is a solution of the  problem
$$\Delta_{x,y} v = 0 \text{ for  } -\eta(x) -h <y<0, \quad v \arrowvert_{y=0} = b, \quad v \arrowvert_{y=-\eta(x) -h} = B$$
and 
$\partial_y v \arrowvert_{y=0}= \partial_y u \arrowvert_{y=-h} =-\Delta_x \phi_h.$

Corollary~\ref{est-phih} follows from   Proposition~\ref{est-DN2} and from the fact that $H^{s_0+1}(\xT^d)$ is embedded in  $W^{1, \infty}(\xT^d)$.
\end{proof}

\section{Existence of a  solution on a time interval  of size 1}\label{S:size1}

In this section we prove Theorem~\ref{T=1}. 
We will obtain the solution as the limit of an iterative scheme. 
To avoid confusion of notations, 
we denote the initial data by $(u_0,v_0)$. Namely, we consider the Cauchy problem for the water-wave system~\e{system} 
with initial data
$$
(\eta,\psi)\arrowvert_{t=0}=(u_0,v_0).
$$
Let $u_0, v_0 \in \mathcal{H}^{\lambda h, s}$. We consider the sequence   $(\eta_\nu, \psi_\nu)_{ \nu \in \xN}$ defined by
\begin{equation}\label{ww-nu}
\begin{aligned}
& \eta_0 = u_0,\quad  \psi_0 = v_0, \text{ and for  } \nu \geq 0,\\
&\partial_t \eta_{\nu+1} = G(\eta_\nu)(\psi_\nu,b), \\
&\partial_t \psi_{\nu+1}= -g \eta_\nu - \mez \vert \nabla_x \psi_\nu\vert^2 
+ \frac{(G(\eta_\nu)(\psi_\nu,b)
+ \nabla_x \eta_\nu \cdot \nabla_x \psi_\nu)^2}{2(1+  \vert \nabla_x \eta_\nu\vert^2)},  \\
&\eta_{\nu +1}\arrowvert_{t=0} = u_0,\quad \psi_{\nu +1}\arrowvert_{t=0} = v_0.
\end{aligned}
\end{equation}
We set  for $s > \frac{d}{2}+2$, 
\begin{equation}\label{m-nu}
\begin{aligned}
m_\nu(t) &= \Vert \eta_\nu(t) \Vert^2_{\mathcal{H}^{\sigma(t), s}} 
+ \Vert \psi_\nu(t) \Vert^2_{\mathcal{H}^{\sigma(t), s}} \\
&\quad+ 2K \int_0^t \big(  \Vert  \eta_\nu(\tau) 
\Vert^2_{\mathcal{H}^{\sigma(\tau), s+\mez}} 
+ \Vert   \psi_\nu(\tau) 
\Vert^2_{\mathcal{H}^{\sigma(\tau), s+\mez }}\big)\, d \tau, \end{aligned}
\end{equation}
 where $\sigma(t) = \lambda h-Kt.$

\begin{prop}\label{eta-nu}
Let $T>0.$ Assume that 
$$
b\in L^\infty(\xR, H^{s-1}(\xT^d))\cap L^2(\xR, H^{s-\mez}(\xT^d)).
$$
Then there exist   positive constants $M,K_0,\eps_0$  such that for all $\eps \leq \eps_0$ and all $K \geq K_0,$  
\begin{multline*}
\big(\Vert b \Vert_{L^2(\xR, H^{s-\mez})\cap L^\infty(\xR, H^{s-1})} +  \Vert u_0 \Vert_{\mathcal{H}^{\lambda h, s}} + \Vert v_0 \Vert_{\mathcal{H}^{\lambda h, s}}\leq \eps\big)   \, \\
\Rightarrow \, \big(m_\nu(t) \leq M^2 \eps^2, \quad \forall t \leq T,    \forall \nu \geq 0\big).
\end{multline*}
 \end{prop}
\begin{proof}
First of all, we take  $\eps_0$ and  $M$ such that
$M \eps_0 \leq \overline{\eps}$, where  $\overline{\eps}$ is defined in Theorem~\ref{est-DN} and  $M\eps_0 \leq 1.$

We claim that $m_0(t) \leq 2 \eps^2$. In fact, we have  
 $ \Vert u_0 \Vert^2_{\mathcal{H}^{\sigma(t), s}} + \Vert v_0 \Vert^2_{\mathcal{H}^{\sigma(t), s}} \leq \eps^2$, and
 \begin{align*}
 2K \int_0^t \ \Vert    u_0 \Vert^2_{\mathcal{H}^{\sigma(\tau), s+\mez}}\, d\tau&= \sum_{\xZ^d}
 \langle \xi \rangle^{2s}e^{2 \lambda h\langle \xi \rangle} \vert \widehat{u}_0(\xi)\vert^2 \Big(\int_0^t 2K\langle \xi \rangle \, e^{-2K\tau\langle \xi \rangle}\, d\tau \Big)\\
& \leq \Vert u_0 \Vert^2_{\mathcal{H}^{\lambda h, s}},
\end{align*}
similarly for $v_0$.
So we take  $M \geq 2.$
Assume now that  $m_j(t) \leq M^2 \eps^2, 0 \leq j \leq \nu$. Set
$$\eta_\nu = e^{- \sigma(t) \langle \xi \rangle} \widetilde{\eta}_\nu, \quad \psi_\nu = e^{- \sigma(t) \langle \xi \rangle} \widetilde{\psi}_\nu.$$
Notice that  since $\sigma(t)= \lambda h -Kt \leq \lambda h$ and $\vert \xi \vert \leq \langle \xi \rangle \leq 1+ \vert \xi \vert$, we have
$$\sigma(t)\vert \xi\vert \leq \sigma(t)\langle \xi \rangle\leq \lambda h + \sigma(t) \vert \xi \vert.$$
It follows that
$$\Vert f(t)\Vert_{\mathcal{H}^{\sigma(t),\alpha}} \leq  \Vert \widetilde{f}(t)\Vert_{H^{\alpha}}\leq e^{\lambda h}\Vert f(t)\Vert_{\mathcal{H}^{\sigma(t),\alpha}}.$$
\begin{rema}
We can write $\sigma(t) = \lambda(t)h$ with $\lambda(t) = \lambda -\frac{Kt}{h}.$ Since $\lambda(t) \leq \lambda$, we may use the estimates  in Sections 2 and 3 with constants depending only on the fixed parameter $\lambda.$
\end{rema}

The system satisfied by $(\widetilde{\eta}_\nu,\widetilde{\psi}_\nu) $ is  then
\begin{align*}
  &\partial_t \widetilde{\eta}_{\nu+1} + K \langle D_x \rangle \widetilde{\eta}_{\nu+1}= e^{\sigma(t)\langle D_x \rangle}G(\eta_\nu)(\psi_\nu,b):= F_\nu\\  
    &\partial_t \widetilde{\psi}_{\nu+1} +  K \langle D_x \rangle \widetilde{\psi}_{\nu+1}=G_\nu, 
 \end{align*}
 where
  \[
   G_nu\defn e^{\sigma(t)\langle D_x \rangle}\Big[-g \eta_\nu - \mez \vert \nabla_x \psi_\nu\vert^2 + \frac{(G(\eta_\nu)(\psi_\nu,b)+ \nabla_x \eta_\nu \cdot \nabla_x \psi_\nu)^2}{2(1+  \vert \nabla_x \eta_\nu\vert^2)}\Big].
   \]
We have
\begin{align*}
\frac{d}{dt}\big[  \Vert \widetilde{\eta}_{\nu+1}(t) \Vert^2_{H^s} + \Vert \widetilde{\psi}_{\nu+1}(t) \Vert^2_{H^s} \big]& = 2  \Big( \partial_t  \widetilde{\eta}_{\nu+1}(t),  \widetilde{\eta}_{\nu+1}(t)\Big)_{H^s} \\
&\quad + 2  \Big( \partial_t  \widetilde{\psi}_{\nu+1}(t),  \widetilde{\psi}_{\nu+1}(t)\Big)_{H^s}.
\end{align*}
Using the above equations, we get
\begin{align*}
  &\Vert \widetilde{\eta}_{\nu+1}(t) \Vert^2_{H^s}   + \Vert \widetilde{\psi}_{\nu+1}(t) \Vert^2_{H^s}  + 2K \int_0^t \big( \Vert  \widetilde{\eta}_{\nu+1}(\tau) \Vert^2_{H^{s+\mez}}  + \Vert   \widetilde{\psi}_{\nu+1}(\tau) \Vert^2_{H^{s+\mez}}\big)\, d\tau\\& =\Vert u_0 \Vert^2_{\mathcal{H}^{\lambda h, s}} + \Vert v_0 \Vert^2_{\mathcal{H}^{\lambda h, s}} + 2  \int_0^t \Big( F_\nu(\tau),  \widetilde{\eta}_{\nu+1}(\tau)\Big)_{H^s}\, d\tau\\
  &+ 2  \int_0^t \Big( G_\nu(\tau),  \widetilde{\psi}_{\nu+1}(\tau)\Big)_{H^s}\, d\tau.
 \end{align*}
Set
$$A(\tau)= \la  \Big( F_\nu(\tau),  \widetilde{\eta}_{\nu+1}(\tau)\Big)_{H^s} \ra, \quad B(\tau)= \la\Big( G_\nu(\tau),  \widetilde{\eta}_{\nu+1}(\tau)\Big)_{H^s}\ra.$$
We deduce from the hypotheses that
 \begin{equation}\label{m-nu+1}
 m_{\nu+1}(t) \leq \eps^2 + \int_0^t A(\tau)\, d\tau + \int_0^t B(\tau)\, d\tau.
 \end{equation}
 We have
 $$ A(\tau) \leq \frac{C}{K}\Vert  G(\eta_\nu)(\psi_\nu, b)(\tau)  \Vert^2_{\mathcal{H}^{\sigma(\tau), s-\mez}} + \frac{2K}{20}\Vert  \widetilde{\eta}_{\nu+1}(\tau)\Vert^2_{H^{s+ \mez}}.$$
For all fixed $\tau$  (which is skipped), Theorem~\ref{est-DN} shows that
$$
\Vert  G(\eta_\nu)(\psi_\nu, b)   \Vert_{\mathcal{H}^{\sigma , s-\mez}}\leq   C \big( \Vert  \widetilde{\psi}_\nu \Vert_{H^{s+\mez}}  + \Vert b \Vert_{H^{s-\mez}} 
 + \Vert \widetilde{\eta}_\nu \Vert_{H^{s+\mez}}\big( \Vert  \widetilde{\psi}_\nu \Vert_{H^{s }} 
   +  \Vert b \Vert_{H^{s-\frac{3}{2}}}\big) \big).
$$
 Since $\Vert \widetilde{\psi}_\nu(\tau)\Vert_{H^s} \leq  M\eps_0\leq 1$,   there exists $C>0$   depending only on   $s$ such that
\begin{equation}\label{est-G}
\Vert  G(\eta_\nu)(\psi_\nu, b)   \Vert_{\mathcal{H}^{\sigma(\tau), s-\mez}}\leq C \big(\Vert  \widetilde{\psi}_\nu \Vert _{H^{s+\mez}}+ \Vert b \Vert _{H^{s-\mez}}  + \Vert \widetilde{\eta}_\nu \Vert_{H^{s+\mez}}(1+ \Vert b \Vert _{H^{s-\frac{3}{2}}})\big).
\end{equation}
Therefore, since $\Vert b \Vert_{L^\infty(\xR, H^{s-1})} \leq \eps\leq 1$, we can write
$$
  A (\tau) \leq \frac{1}{K}  \big( \Vert  \widetilde{\psi}_\nu(\tau)\Vert^2_{H^{s+\mez}}+ \Vert b(\tau)\Vert^2_{H^{s-\mez}}  + \Vert \widetilde{\eta}_\nu(\tau)\Vert^2_{H^{s+\mez}} 
      \big) 
    +\frac{2K}{20}\Vert   \widetilde{\eta}_{\nu+1}(\tau)\Vert^2_{H^{s+ \mez}}.
$$
So,
\begin{align*}
  \int_0^t A(\tau)\, d\tau &\leq  \frac{C}{K} \int_0^t \big( \Vert    \widetilde{\psi}_\nu(\tau)\Vert^2_{ H^{s+\mez}}   +\Vert    \widetilde{\eta}_\nu (\tau)\Vert^2_{ H^{s+\mez}}     &+\Vert b(\tau)\Vert^2_{  H^{s-\mez} } \big)\, d\tau\\
 \quad &+ \frac{2K}{20}\int_0^t \Vert\widetilde{\eta}_{\nu+1}(\tau)\Vert^2_{H^{s+ \mez}}  \,d\tau.
 \end{align*}
 It follows that
 \begin{equation}\label{est:A}
  \int_0^t A(\tau)\, d\tau \leq  \frac{C}{K^2}M^2 \eps^2 + \frac{C}{K}\Vert b\Vert^2_{L^2(\xR, H^{s-\mez})}  + \frac{1}{20} m_{\nu+1}(t).
  \end{equation}
Now,
\begin{equation}
\begin{aligned}
& B(\tau) \leq B_1(\tau)+B_2(\tau)+B_3(\tau),\\
& B_1(\tau) = g\la\Big( \widetilde{\eta}_{\nu}(\tau) ,  \widetilde{\eta}_{\nu+1}(\tau)\Big)_{H^s}\ra,\\
&B_2(\tau) =\mez \la\Big( e^{\sigma(\tau) \langle D_x \rangle}\vert \nabla_x \psi_\nu(\tau)\vert^2,  \widetilde{\eta}_{\nu+1}(\tau)\Big)_{H^s}\ra\\
&B_3(\tau)= \la\Big( e^{\sigma(\tau) \langle D_x \rangle}\frac{(G(\eta_\nu)(\psi_\nu,b)+ \nabla_x \eta_\nu \cdot \nabla_x \psi_\nu)^2}{2(1+  \vert \nabla_x \eta_\nu\vert^2)},  \widetilde{\eta}_{\nu+1}(\tau)\Big)_{H^s}\ra.
\end{aligned}
 \end{equation}
Using the Cauchy-Schwarz inequality and the induction hypothesis,  we obtain
\begin{equation}\label{est:B1}
\int_0^t B_1(\tau)\, d\tau \leq \frac{C}{K^2} M^2 \eps^2 + \frac{1}{20} m_{\nu+1}(t).
\end{equation}
On the other hand, using   Proposition~\ref{uv} (iii) with $s_0= s-1$,  we obtain
\begin{align*}
B_2(\tau) &\les\Vert \widetilde{\psi}_\nu(\tau)\Vert_{H^{s}}\Vert \widetilde{\psi}_\nu(\tau)\Vert_{H^{s +\mez}}\Vert \widetilde{\eta}_{\nu+1}(\tau)\Vert_{H^{s +\mez}}\\
&\les M\eps\Vert \widetilde{\psi}_\nu(\tau)\Vert_{H^{s +\mez}}\Vert \widetilde{\eta}_{\nu+1}(\tau)\Vert_{H^{s +\mez}}.
\end{align*}
It follows that
$$\int_0^t B_2(\tau)\, d\tau \les   M\eps \Vert  \Vert \widetilde{\psi}_\nu\Vert_{L^2((0,t), H^{s+ \mez})} 
\Vert \widetilde{\eta}_{\nu+1} \Vert_{L^2((0,t), H^{s+ \mez})}. $$
Therefore, by the Cauchy-Schwarz inequality and the induction, we get 
\begin{equation}\label{est-B2}
 \int_0^t B_2(\tau)\, d\tau \leq \frac{C}{K^2} (M^2 \eps^2)^2 + \frac{1}{20}m_{\nu+1}(t).
  \end{equation}
To  estimate the term  $B_3$, set
$$
N_\nu =  G(\eta_\nu)(\psi_\nu,b)+ \nabla_x \eta_\nu \cdot \nabla_x \psi_\nu
$$
and
$$
U_\nu = \frac{N^2_\nu}{1+ \vert \nabla_x \eta_\nu \vert^2}.
$$
Then we have
$$B_3(\tau) \les\Vert e^{\sigma(\tau) \vert D_x \vert}U_\nu \Vert_{H^{s-\mez}}\Vert\widetilde{\eta}_{\nu+1}(\tau)\Vert_{H^{s +\mez}},$$
from which we deduce that 
   \begin{equation}\label{est-B3} 
  \int_0^t B_3(\tau)\, d\tau \leq \frac{C}{K} \int_0^t  \Vert U_\nu(\tau) \Vert^2_{\mathcal{H}^{\sigma(\tau), s-\mez}}\, d\tau+ \frac{1}{20}m_{\nu+1}(t).
    \end{equation}
 Using   Propositions~\ref{uv} and~\ref{est-f(u)bis}, with $s_0 =s-1$, we can write, skipping $\tau$,   
 \begin{align*}
 \Vert  U_\nu  \Vert_{\mathcal{H}^{\sigma ,s-\mez}} &\les \ \Vert   N_\nu \Vert_{\mathcal{H}^{\sigma,  s-1}}\Vert  N_\nu   \Vert_{\mathcal{H}^{\sigma ,s-\mez}} + \Vert  \vert \nabla_x \eta_\nu   \vert^2 U_\nu  \Vert_{\mathcal{H}^{\sigma ,s-\mez}} \\
 & \les \Vert   N_\nu \Vert_{\mathcal{H}^{\sigma, s-1}}\Vert  N_\nu    \Vert_{\mathcal{H}^{\sigma ,s-\mez}}  + \Vert \widetilde{\eta}_\nu \Vert^2_{H^s}\Vert  U_\nu  \Vert_{\mathcal{H}^{\sigma ,s-\mez}} \\
 &\quad+  \Vert \widetilde{\eta}_\nu \Vert_{H^s} \Vert \widetilde{\eta}_\nu \Vert_{H^{s+ \mez}}\Vert  U_\nu  \Vert_{\mathcal{H}^{\sigma ,s-1}}. 
  \end{align*}
 Since $\Vert \widetilde{\eta}_\nu (\tau)\Vert_{H^{s}} \leq M \eps$, taking  $M, \eps_0$ such that $C (M \eps_0)^2 \leq \mez$, we can  absorb  the second   term of the right hand side by the left  hand side and deduce that
 $$ \Vert  U_\nu  \Vert_{\mathcal{H}^{\sigma ,s-\mez}}  \les   \Vert   N_\nu \Vert_{\mathcal{H}^{\sigma, s-1}}\Vert  N_\nu    \Vert_{\mathcal{H}^{\sigma ,s-\mez}}  +   M\eps\Vert \widetilde{\eta}_\nu \Vert_{H^{s+ \mez}}\Vert  U_\nu  \Vert_{\mathcal{H}^{\sigma ,s-1}}. $$
Similarly, since $s-1> \frac{d}{2}$ we have
 $$ \Vert  U_\nu  \Vert_{\mathcal{H}^{\sigma ,s-1}} \les   \Vert   N_\nu \Vert^2_{\mathcal{H}^{\sigma, s-1}} + \Vert \widetilde{\eta}_\nu \Vert^2_{H^{s}}\Vert  U_\nu  \Vert_{\mathcal{H}^{\sigma ,s-1}}.$$
Using again the fact that $\Vert \widetilde{\eta}_\nu (\tau)\Vert_{H^{s}} \leq M \eps$, we can  absorb the second term in the   right hand side by the left, one can deduce that 
$$\Vert  U_\nu  \Vert_{\mathcal{H}^{\sigma ,s-1}} \les \Vert   N_\nu \Vert^2_{\mathcal{H}^{\sigma, s-1}}. $$
So we obtain the inequality
\begin{equation*}
 \Vert  U_\nu  \Vert_{\mathcal{H}^{\sigma ,s-\mez}}  \les    \Vert  N_\nu    \Vert_{\mathcal{H}^{\sigma ,s-\mez}}\Vert   N_\nu \Vert_{\mathcal{H}^{\sigma, s-1}} +  M\eps \Vert \widetilde{\eta}_\nu \Vert_{H^{s+ \mez}} \Vert   N_\nu \Vert^2_{\mathcal{H}^{\sigma, s-1}}.
  \end{equation*}
By Theorem~\ref{est-DN},   Remark~\ref{G-s-1} and the induction we have
$$\Vert   N_\nu \Vert_{\mathcal{H}^{\sigma, s-1}} \les   \Vert \widetilde{\psi}_\nu \Vert_{H^s} + \Vert b \Vert_{H^{s-1}}+ \Vert \widetilde{\eta}_\nu \Vert_{H^s}\Vert \widetilde{\psi}_\nu \Vert_{H^s}.$$ 
We have $\Vert \widetilde{\eta}_\nu(\tau) \Vert_{H^s}\leq  M\eps\leq 1,  \Vert \widetilde{\psi}_\nu(\tau) \Vert_{H^s} \leq  M\eps\leq 1$ and $b\in L^\infty(\xR, H^{s-1})$,  so that
\begin{equation}\label{Nnu(s-1)}
 \Vert   N_\nu \Vert_{\mathcal{H}^{\sigma, s-1}}   \les M \eps + \Vert b\Vert_{H^{s-1}}= \mathcal{O}(1).  
  \end{equation}
  It follows that
  \begin{equation}\label{est-Unu}
   \Vert  U_\nu  \Vert_{\mathcal{H}^{\sigma ,s-\mez}} \les  \Vert N_\nu    \Vert_{\mathcal{H}^{\sigma ,s-\mez}} +   \Vert \widetilde{\eta}_\nu \Vert_{H^{s+ \mez}}.
   \end{equation}
It remains to  estimate the term $ \Vert  N_\nu    \Vert_{\mathcal{H}^{\sigma ,s-\mez}}.  $ We have
\begin{align*}
 \Vert \nabla_x \psi_\nu\cdot \nabla_x \eta_\nu  \Vert_{\mathcal{H}^{\sigma ,s-\mez}} &\les\Vert  \widetilde{\psi}_\nu   \Vert_{ H^{s}}  \Vert  \widetilde{\eta}_\nu   \Vert_{ H^{s+ \mez}} + \Vert  \widetilde{\psi}_\nu   \Vert_{ H^{s+ \mez}}\Vert  \widetilde{\eta}_\nu   \Vert_{ H^{s}} ,\\  
  &\les   M \eps \big(  \Vert  \widetilde{\eta}_\nu   \Vert_{ H^{s+ \mez}} + \Vert  \widetilde{\psi}_\nu   \Vert_{ H^{s+ \mez}}\big) .  
\end{align*}
 Using \eqref{est-G}, we get
\begin{align*}
   \Vert  N_\nu    \Vert_{\mathcal{H}^{\sigma ,s-\mez}}  &\les\Vert  \widetilde{\psi}_\nu \Vert _{H^{s+\mez}} + \Vert b \Vert _{H^{s-\mez}}   + \Vert \widetilde{\eta}_\nu \Vert_{H^{s+\mez}} 
    \Vert b \Vert _{H^{s-\frac{3}{2}}}  
     \\
     &\quad+   M \eps \big(  \Vert  \widetilde{\eta}_\nu   \Vert_{ H^{s+ \mez}} + \Vert  \widetilde{\psi}_\nu   \Vert_{ H^{s+ \mez}}\big),
\end{align*}
so,
 \begin{equation}\label{Nnu(s-mez)}
  \Vert  N_\nu    \Vert_{\mathcal{H}^{\sigma ,s-\mez}} \les   \Vert  \widetilde{\psi}_\nu   \Vert_{ H^{s+ \mez}} + \Vert \widetilde{\eta}_\nu \Vert_{H^{s+\mez}}+ \Vert b\Vert_{H^{-\mez}}.  
 \end{equation}
Using   \eqref{est-Unu}, \eqref{Nnu(s-1)} and \eqref{Nnu(s-mez)}, we get
$$ \Vert  U_\nu  \Vert^2_{\mathcal{H}^{\sigma ,s-\mez}} \les   \Vert  \widetilde{\psi}_\nu   \Vert^2_{ H^{s+ \mez}} +  \Vert  \widetilde{\eta}_\nu   \Vert^2_{ H^{s+ \mez}}     + \Vert b\Vert^2_{L^\infty(\xR, H^{s-\mez})}.$$
We deduce from \eqref{est-B3} and the induction that
\begin{equation}\label{est-B3suite}
 \int_0^t B_3(\tau)\, d\tau \leq \frac{C}{K^2}M^2 \eps^2      + \frac{1}{20}m_{\nu+1}(t).
\end{equation}
Using \eqref{m-nu+1}, \eqref{est:A}, \eqref{est:B1}, \eqref{est-B2}, \eqref{est-B3suite}  and  taking $K$ large enough and $\eps_0$ small enough, 
we obtain
$$
m_{\nu+1}(t) \leq M^2 \eps^2,
$$
which ends the proof of    Proposition~\ref{eta-nu}.
\end{proof}

\begin{nota}
We set $I= [0,T]$, and
 \begin{align*}
 U_{\nu+1} &= \widetilde{\eta}_{\nu+1} - \widetilde{\eta}_{\nu},  \quad V_{\nu+1}= \widetilde{\psi}_{\nu+1} - \widetilde{\psi}_{\nu},\\
  M_{\nu+1}(t) &= \Vert U_{\nu+1}(t) \Vert^2_{ H^s } + \Vert V_{\nu+1}(t) \Vert^2_{ H^s }  \\
  &\quad+ 2K \int_0^t  \Big(\Vert     U_{\nu+1}(\tau) \Vert^2_{ H^{s+\mez} }  
+    \Vert   V_{\nu+1}(\tau) \Vert^2_{ H^{s+\mez} } \Big)\, d \tau,
   \end{align*}
   and
 \begin{align*}
      \mathcal{M}_{\nu+1} & = \Vert U_{\nu+1}  \Vert^2_{L^\infty(I,H^s)} + \Vert V_{\nu+1}  \Vert^2_{L^\infty(I,H^s)}  \\
      &\quad+  2K \int_0^T \Big(\Vert    U_{\nu+1}(\tau) \Vert^2_{H^{s+\mez}}\, d\tau  
 + \Vert   V_{\nu+1}(\tau) \Vert^2_{H^{s+\mez}}\Big)  \, d \tau. 
\end{align*}
\end{nota}
\begin{prop}\label{Mnu-Mnu+1}
Under the hypotheses of Proposition~{\rm \ref{eta-nu}} there exists a constant 
$C>0$ independent of $K$ such that for all $\nu \geq 0,$
$$\mathcal{M}_{\nu+1} \leq \frac{C}{K}\big( \frac{1}{K} + \Vert b\Vert^2_{L^2(\xR,H^{s-\mez})}\big)\mathcal{M}_\nu.$$
\end{prop}
\begin{proof} 

Since $U_{\nu+1}\arrowvert_{t=0}= V_{\nu+1}\arrowvert_{t=0}=0$, we obtain as in the first part,
\begin{align*}
M_{\nu+1}(t)& =  2  \int_0^t \Big( (F_\nu- F_{\nu-1})(\tau),  U_{\nu+1}(\tau)\Big)_{H^s}\, d\tau\\
&\quad+ 2  \int_0^t \Big( (G_\nu - G_{\nu-1})(\tau),  V_{\nu+1}(\tau)\Big)_{H^s}\, d\tau,
\end{align*}
where
\begin{align*}
&F_\nu= e^{\sigma(t)\vert D_x \vert}G(\eta_\nu)(\psi_\nu,b),\\
&G_\nu= e^{\sigma(t)\vert D_x \vert}\Big[-g \eta_\nu - \mez \vert \nabla_x \psi_\nu\vert^2 + \frac{\big(G(\eta_\nu)(\psi_\nu,b)+ \nabla_x \eta_\nu \cdot \nabla_x \psi_\nu)^2}{2(1+  \vert \nabla_x \eta_\nu\vert^2\big)}\Big].
\end{align*}
Using the inequality $$\vert (f,g)_{H^s}\vert \leq \frac{1}{4K} \Vert f\Vert^2_{H^{s- \mez}} + K \Vert g\Vert^2_{H^{s+ \mez}}$$ and the  definition of $M_{\nu+1}$, we get
\begin{equation}\label{Mnu<}
M_{\nu+1}(t) \leq A(t)+ B(t), 
\end{equation}
where
\[
\begin{aligned}
&A(t)= \frac{C}{K} \int_0^t \Vert (F_\nu- F_{\nu-1})(\tau) \Vert^2_{  H ^ {s-\mez}}\, d\tau, \quad   \\
&B(t)= \frac{C}{K}\int_0^t \Vert (G_\nu- G_{\nu-1})(\tau) \Vert^2_{  H ^ {s-\mez}}\, d\tau.
\end{aligned}
\]
Estimate of the  term  $A(t)$.

We use Theorem~\ref{G-lip} and Proposition~\ref{eta-nu}. Keeping the notations in the Theorem, we can write
\begin{multline}\label{Fnu}
 \Vert F_\nu -F_{\nu-1}\Vert_{H^{s-\mez}} \les \lambda_1 \Vert \widetilde{\eta}_\nu-\widetilde{\eta}_{\nu-1} \Vert_{H^{s}} + \lambda_2 \Vert \widetilde{\eta}_\nu-\widetilde{\eta}_{\nu-1} \Vert_{H^{s+\mez}} \\
 + \lambda_3 \Vert \widetilde{\psi}_\nu-\widetilde{\psi}_{\nu-1} \Vert_{H^{s}} + \Vert \widetilde{\psi}_\nu-\widetilde{\psi}_{\nu-1} \Vert_{H^{s+\mez}}
  \end{multline}
  with
  \begin{align*}
  \lambda_1 &\les \Big(\sum_{\mu=\nu-1}^\nu
   \Vert \widetilde{\eta}_\mu \Vert_{H^{s+\mez}}\Big)      \Big(\sum_{\mu=\nu-1}^\nu \Vert \widetilde{\psi}_\mu \Vert_{H^{s}}  + \Vert b\Vert_{H^{s-\frac{3}{2}}}\Big) 
   + \sum_{\mu=\nu-1}^\nu \Vert \widetilde{\psi}_\mu \Vert_{H^{s+\mez}}   + \Vert b\Vert_{H^{s-\mez}},\\
   \lambda_2&  \les \sum_{\mu=\nu-1}^\nu \Vert  \widetilde{\psi}_\mu \Vert_{H^{s-\mez}} +\Vert b\Vert_{H^{s-\frac{3}{2}}},\\
   \lambda_3 & \les \sum_{\mu=\nu-1}^\nu
   \Vert \widetilde{\eta}_\mu \Vert_{H^{s+\mez}}.
  \end{align*}
  Since $b\in L^\infty(\xR, H^{s-1})$ and using \eqref{m-nu} together with  Proposition~\ref{eta-nu}, we deduce that 
  \begin{align*}
   \lambda_1 \les&\, \sum_{\mu=\nu-1}^\nu
   \big(\Vert \widetilde{\eta}_\mu \Vert_{H^{s+\mez}}+ \Vert \widetilde{\psi}_\mu \Vert_{H^{s+\mez}}\big) + \Vert b \Vert_{H^{s-\mez}},\\  
   \lambda_2 =&\, \mathcal{O}(1), \\
     \lambda_3 \les&\, \sum_{\mu=\nu-1}^\nu
   \Vert \widetilde{\eta}_\mu \Vert_{H^{s+\mez}}.
  \end{align*}
  It follows easily from \eqref{Fnu} that
  \begin{equation}\label{est-Fnu}
  \begin{split}
 &\Vert F_\nu -F_{\nu-1}\Vert_{H^{s-\mez}}\\
 \les &\,  \sum_{\mu=\nu-1}^\nu
   \Big(\Vert \widetilde{\eta}_\mu \Vert_{H^{s+\mez}}+ \Vert \widetilde{\psi}_\mu \Vert_{H^{s+\mez}} + \Vert b \Vert_{H^{s-\mez}}\Big)\big(\Vert\widetilde{\eta}_\nu-\widetilde{\eta}_{\nu-1} \Vert_{H^{s}}
+ \Vert\widetilde{\psi}_\nu-\widetilde{\psi}_{\nu-1} \Vert_{H^{s}}\big) \\
& + \Vert\widetilde{\eta}_\nu-\widetilde{\eta}_{\nu-1} \Vert_{H^{s+\mez}}+ \Vert \widetilde{\psi}_\nu-\widetilde{\psi}_{\nu-1} \Vert_{H^{s+\mez}},
   \end{split}
  \end{equation}
  from which we deduce that
  \begin{equation}\label{est-A'}
A(t) \les   \frac{1}{K} \Big(\frac{ 1}{K} +  \Vert b\Vert^2_{L^2(\xR,H^{s-\mez})}\Big)\mathcal{M}_\nu.
 \end{equation}

Let us look at the  term  $B(t)$. We can write
\begin{equation}\label{def-B'}
B(t) \leq C(B_1(t) + B_2(t) + B_3(t)),
 \end{equation}
 where 
 \[
\begin{aligned}
&B_1(t) = \frac{1}{K}\int_0^t \Vert \widetilde{\eta}_\nu(\tau)-\widetilde{\eta}_{\nu-1}\Vert^2_{H^{s-\mez} }(\tau)\, d\tau,\\
&B_2(t) =\frac{1}{K}\int_0^t \Vert \vert \nabla_x \widetilde{\psi}_\nu(t)\vert^2 - \vert \nabla_x \widetilde{\psi}_{\nu-1}(t)\vert^2\Vert^2_{H^{s-\mez} }\, d\tau,\\
&B_3(t) = \frac{1}{K}\int_0^t \lA \frac{N_\nu^2(\tau)}{(1+ \vert \nabla_x \eta_\nu(\tau)\vert^2) }-  \frac{N_{\nu-1}^2(\tau)}{(1+ \vert \nabla_x \eta_{\nu-1}(\tau)\vert^2)} \rA^2_{\mathcal{H}^{\sigma(\tau),s-\mez}}\, d\tau
\end{aligned}
\]
with 
\[
N_\nu =  G(\eta_\nu)(\psi_\nu,b)+ \nabla_x \eta_\nu \cdot \nabla_x \psi_\nu. 
\]
First of all, we have
\begin{equation}\label{est-B1'}
B_1(t) \les \frac{1}{K^2} \mathcal{M}_\nu.
\end{equation}
Now set 
$$I_j= \Vert   (\partial_j \widetilde{\psi}_\nu(t))^2 - (\partial_j \widetilde{\psi}_{\nu-1}(t))^2 \Vert_{H^{s-\mez} }.
$$ 
We can write
\begin{align*}
  I_j &\les    \Vert   \partial_j \widetilde{\psi}_\nu(t) - \partial_j \widetilde{\psi}_{\nu-1}(t)\Vert_{H^{s-\mez}} \Vert   \partial_j \widetilde{\psi}_\nu(t) + \partial_j \widetilde{\psi}_{\nu-1}(t)\Vert_{H^{s-1}}\\
  &\quad +  \Vert   \partial_j \widetilde{\psi}_\nu(t) - \partial_j \widetilde{\psi}_{\nu-1}(t)\Vert_{H^{s-1}}  \Vert   \partial_j \widetilde{\psi}_\nu(t) + \partial_j \widetilde{\psi}_{\nu-1}(t)\Vert_{H^{s-\mez}},
 \end{align*}
which implies that 
 \begin{align*}
  I_j &\les    \Vert   \widetilde{\psi}_\nu(t) -   \widetilde{\psi}_{\nu-1}(t)\Vert_{H^{s+\mez}} \big( \Vert     \widetilde{\psi}_\nu(t)\Vert_{H^s} +   \Vert     \widetilde{\psi}_{\nu-1}(t)\Vert_{H^s}\big)\\
  &\quad+  \Vert    \widetilde{\psi}_\nu(t) -   \widetilde{\psi}_{\nu-1}(t)\Vert_{H^{s}} \big( \Vert     \widetilde{\psi}_\nu(t)\Vert_{H^{s+\mez}} +\Vert     \widetilde{\psi}_{\nu-1}(t)\Vert_{H^{s+\mez}}\big).
 \end{align*}
It follows from Proposition~\ref{eta-nu}  that
\begin{equation}\label{est-B2'}
B_2(t) \les \frac{(M\eps)^2 }{K^2} \mathcal{M}_\nu.
\end{equation}
To estimate $B_3$, set
\begin{equation*}
H_\nu = \frac{N_\nu^2}{(1+ \vert \nabla_x \eta_\nu\vert^2) }-  \frac{N_{\nu-1}^2}{(1+ \vert \nabla_x \eta_{\nu-1}\vert^2)}, \quad f(t)= \frac{t}{1+t}.
\end{equation*}
 Then we can write
\begin{equation}\label{est-Hnu}
H_\nu = (1) -(2)-(3)
\end{equation}
with
\[
\begin{aligned}
(1)&=  N_\nu^2-N_{\nu-1}^2,\\
(2)&=  N_\nu^2\big(f(\vert\nabla_x\eta_\nu\vert^2)- f(\vert\nabla_x\eta_{\nu-1}\vert^2)\big), \\
(3)&= \big( N_\nu^2-N_{\nu-1}^2\big)f(\vert\nabla_x\eta_{\nu-1}\vert^2). 
\end{aligned}
\]
 We have
 \begin{align*}
  \Vert (1)\Vert_{\mathcal{H}^{\sigma, s-\mez }} &\les \Vert N_\nu-N_{\nu-1}\Vert_{\mathcal{H}^{\sigma, s-\mez }}  \big(\Vert N_\nu \Vert_{\mathcal{H}^{\sigma, s-1 }}+ \Vert N_{\nu-1}\Vert_{\mathcal{H}^{\sigma, s-1 }}\big) \\
  &\quad+ \Vert N_\nu-N_{\nu-1}\Vert_{\mathcal{H}^{\sigma, s-1 }} \big(\Vert N_\nu\Vert_{\mathcal{H}^{\sigma, s-\mez }}+ \Vert N_{\nu-1}\Vert_{\mathcal{H}^{\sigma, s-\mez }}\big)
\end{align*}
According to \eqref{Nnu(s-1)} and \eqref{Nnu(s-mez)}, we have for $\mu= \nu-1, \nu,$
\begin{equation}\label{est-Nnu1}
\begin{aligned}
&\Vert   N_\mu \Vert_{\mathcal{H}^{\sigma, s-1}}   \les M \eps + \Vert b\Vert_{H^{s-1}}= \mathcal{O}(1),\\
&\Vert  N_\mu    \Vert_{\mathcal{H}^{\sigma ,s-\mez}} \les   \Vert  \widetilde{\psi}_\mu   \Vert_{ H^{s+ \mez}} + \Vert \widetilde{\eta}_\mu \Vert_{H^{s+\mez}}+ \Vert b\Vert_{H^{-\mez}}.
\end{aligned}
\end{equation}
Moreover, according to \eqref{est-Fnu}, we have
\begin{equation*}
\begin{aligned}
&\Vert N_\nu-N_{\nu-1} \Vert_{\mathcal{H}^{\sigma ,s-\mez}}\\
\les &\,  \sum_{\mu=\nu-1}^\nu
   \Big(\Vert \widetilde{\eta}_\mu \Vert_{H^{s+\mez}}+ \Vert \widetilde{\psi}_\mu \Vert_{H^{s+\mez}} + \Vert b \Vert_{H^{s-\mez}}\Big)\big(\Vert\widetilde{\eta}_\nu-\widetilde{\eta}_{\nu-1} \Vert_{H^{s}}\\
   &\quad\quad+ \Vert\widetilde{\psi}_\nu-\widetilde{\psi}_{\nu-1} \Vert_{H^{s}}\big) + \Vert\widetilde{\eta}_\nu-\widetilde{\eta}_{\nu-1} \Vert_{H^{s+\mez}}+ \Vert \widetilde{\psi}_\nu-\widetilde{\psi}_{\nu-1} \Vert_{H^{s+\mez}},
   \end{aligned}
\end{equation*}
and 
\begin{equation*}
\Vert N_\nu - N_{\nu-1}\Vert_{ \mathcal{H}^{\lambda h, s- 1}} \les   \Vert \widetilde{\eta}_\nu- \widetilde{\eta}_{\nu-1} \Vert_{\mathcal{H}^{\sigma, s }} +  \Vert \widetilde{\psi}_\nu- \widetilde{\psi}_{\nu-1} \Vert_{H^s}.
\end{equation*}
It follows that
\begin{equation}\label{est-(1)}
\begin{aligned}
\Vert(1)\Vert_{\mathcal{H}^{\sigma, s-\mez}} \les &\sum_{\mu=\nu-1}^\nu
   \Big(\Vert \widetilde{\eta}_\mu \Vert_{H^{s+\mez}}+ \Vert \widetilde{\psi}_\mu \Vert_{H^{s+\mez}} + \Vert b \Vert_{H^{s-\mez}}\Big)\big(\Vert\widetilde{\eta}_\nu-\widetilde{\eta}_{\nu-1} \Vert_{H^{s}}\\
   &+ \Vert\widetilde{\psi}_\nu-\widetilde{\psi}_{\nu-1} \Vert_{H^{s}}\big) + \Vert\widetilde{\eta}_\nu-\widetilde{\eta}_{\nu-1} \Vert_{H^{s+\mez}}+ \Vert \widetilde{\psi}_\nu-\widetilde{\psi}_{\nu-1} \Vert_{H^{s+\mez}}.
\end{aligned}
\end{equation}
Now we have
\begin{align*}
  \Vert(2)\Vert_{\mathcal{H}^{\sigma, s-\mez}} &\les \Vert N_\nu\Vert^2_{\mathcal{H}^{\sigma, s-1}} \Vert f(\vert\nabla_x\eta_\nu\vert^2)- f(\vert\nabla_x\eta_{\nu-1}\vert^2)\Vert_{\mathcal{H}^{\sigma, s-\mez}}\\
 &+ \Vert N_\nu\Vert_{\mathcal{H}^{\sigma, s-1}}\Vert N_\nu\Vert_{\mathcal{H}^{\sigma, s-\mez}}\Vert f(\vert\nabla_x\eta_\nu\vert^2)- f(\vert\nabla_x\eta_{\nu-1}\vert^2)\Vert_{\mathcal{H}^{\sigma, s-1}}.
 \end{align*}
 Using Proposition~\ref{est-f(u)ter}, \eqref{est-Nnu1} and the estimates
 \begin{align*}
      & \Vert  \vert \nabla_x \eta_\nu\vert^2 - \vert \nabla_x \eta_{\nu-1}\vert^2 \Vert_{\mathcal{H}^{\sigma, s-\mez}} \les    \Vert  \eta_{\nu}- \eta_{\nu-1}\Vert_{\mathcal{H}^{\sigma, s+\mez}},\\
      & \Vert  \vert \nabla_x \eta_\nu\vert^2 - \vert \nabla_x \eta_{\nu-1}\vert^2 \Vert_{\mathcal{H}^{\sigma, s-1}} \les\Vert  \eta_{\nu}- \eta_{\nu-1} \Vert_{\mathcal{H}^{\sigma, s}},
      \end{align*}
      we obtain
      \begin{equation}\label{est-(2)}
      \begin{aligned}
       \Vert(2)\Vert_{\mathcal{H}^{\sigma, s-\mez}}\les &\sum_{\mu=\nu-1}^\nu
   \Big(\Vert \widetilde{\eta}_\mu \Vert_{H^{s+\mez}}+ \Vert \widetilde{\psi}_\mu \Vert_{H^{s+\mez}}    \Big)\big(\Vert\widetilde{\eta}_\nu-\widetilde{\eta}_{\nu-1} \Vert_{H^{s}}\\
   &+ \Vert\widetilde{\psi}_\nu-\widetilde{\psi}_{\nu-1} \Vert_{H^{s}}\big) + \Vert\widetilde{\eta}_\nu-\widetilde{\eta}_{\nu-1} \Vert_{H^{s+\mez}}+ \Vert \widetilde{\psi}_\nu-\widetilde{\psi}_{\nu-1} \Vert_{H^{s+\mez}}.
      \end{aligned}
      \end{equation}
      The same method gives the estimate
       \begin{equation}\label{est-(3)}
      \begin{aligned}
       \Vert(3)\Vert_{\mathcal{H}^{\sigma, s-\mez}}\les &\sum_{\mu=\nu-1}^\nu
   \Big(\Vert \widetilde{\eta}_\mu \Vert_{H^{s+\mez}}+ \Vert \widetilde{\psi}_\mu \Vert_{H^{s+\mez}}    \Big)\big(\Vert\widetilde{\eta}_\nu-\widetilde{\eta}_{\nu-1} \Vert_{H^{s}}\\
   &+ \Vert\widetilde{\psi}_\nu-\widetilde{\psi}_{\nu-1} \Vert_{H^{s}}\big) + \Vert\widetilde{\eta}_\nu-\widetilde{\eta}_{\nu-1} \Vert_{H^{s+\mez}}+ \Vert \widetilde{\psi}_\nu-\widetilde{\psi}_{\nu-1} \Vert_{H^{s+\mez}}.
      \end{aligned}
      \end{equation}
      Using \eqref{est-(1)}, \eqref{est-(2)}, \eqref{est-(3)} and \eqref{est-Hnu}, we deduce that
      \begin{equation*}
\begin{aligned}
\Vert H_\nu\Vert_{\mathcal{H}^{\sigma ,s-\mez}}\les  &\sum_{\mu=\nu-1}^\nu
   \Big(\Vert \widetilde{\eta}_\mu \Vert_{H^{s+\mez}}+ \Vert \widetilde{\psi}_\mu \Vert_{H^{s+\mez}} + \Vert b \Vert_{H^{s-\mez})}\Big)\big(\Vert\widetilde{\eta}_\nu-\widetilde{\eta}_{\nu-1} \Vert_{H^{s}}\\
   &+ \Vert\widetilde{\psi}_\nu-\widetilde{\psi}_{\nu-1} \Vert_{H^{s}}\big) + \Vert\widetilde{\eta}_\nu-\widetilde{\eta}_{\nu-1} \Vert_{H^{s+\mez}}+ \Vert \widetilde{\psi}_\nu-\widetilde{\psi}_{\nu-1} \Vert_{H^{s+\mez}}.
   \end{aligned}
\end{equation*}
Going back to \eqref{def-B'}, we obtain eventually that 
\begin{equation}\label{est-B3'}
\vert B_3(t)\vert \les \frac{1}{K}\Big(\frac{1}{K}+ \Vert b \Vert^2_{L^2(\xR, H^{s-\mez}}\Big)\mathcal{M}_\nu.
\end{equation}
  Now we use \eqref{Mnu<}, \eqref{est-A'}, \eqref{def-B'}, \eqref{est-B1'}, \eqref{est-B2'} and \eqref{est-B3'}.  We obtain
  \begin{equation*}
\mathcal{M}_{\nu+1} \les \frac{1}{K} \Big( \frac{1}{K} + \Vert b \Vert^2_{L^2(\xR, H^{s-\mez}}\Big) \mathcal{M}_{\nu},
\end{equation*}
which completes the proof of Proposition~\ref{Mnu-Mnu+1}.
\end{proof}
\begin{proof}[Proof of Theorem~\ref{T=1}.]
 It follows from   Proposition~\ref{Mnu-Mnu+1} that there exist  $\eps_0 $ and $K $ such that  $\mathcal{M}^\mez_\nu \leq \delta^\nu \mathcal{M}^\mez_0$, where  $\delta<1$. Take $T< \frac{\lambda h}{K}$ and  set  $$\mathcal{X} = C^0([0,T], \mathcal{H}^{\sigma(t),s})\cap L^2((0,T),  \mathcal{H}^{\sigma(t),s+ \mez}).$$ It follows that the sequence $(\eta_\nu, \psi_\nu)$ converges in $\mathcal{X}\times \mathcal{X}$ to $(\eta, \psi)$. It remains to prove that $(\eta, \psi)$ is a solution of system \eqref{system}. 
 According to    \eqref{ww-nu}, we can write
 \begin{equation*}
 \begin{aligned}
 \eta_{\nu+1}(t) &= u_0 + \int_0^t F_\nu(\tau)\, d\tau,\quad F_\nu= G(\eta_\nu)(\psi_\nu,b),\\
   \psi_{\nu+1}(t)&= v_0 + \int_0^t G_\nu(\tau)\, d\tau,\\
   G_\nu&=   -g \eta_\nu - \mez \vert \nabla_x \psi_\nu\vert^2 + \frac{(G(\eta_\nu)(\psi_\nu,b)+ \nabla_x \eta_\nu \cdot \nabla_x \psi_\nu)^2}{2(1+  \vert \nabla_x \eta_\nu\vert^2)}.
 \end{aligned}
 \end{equation*}
 It is enough to prove that $\int_0^t F_\nu(\tau)\, d\tau$ converges in $L^\infty([0,T], \mathcal{H}^{\sigma(t) , s-\mez})$ to $\int_0^t F(\tau)\, d\tau$, where  $F= G(\eta)(\psi,b)$, together with a similar result for $\int_0^t G_\nu(\tau)\, d\tau$. We have with fixed  $t$,
 \begin{align*}
 A&:= \lA \int_0^t (F_\nu(\tau)-F(\tau))\, d\tau \rA^2_{\mathcal{H}^{\sigma(t) , s-\mez}} \\
 &= \sum_{\xZ^d}
 e^{2\sigma(t)\vert \xi \vert} \langle \xi \rangle^{2s-1}\la \int_0^t(\widehat{F_\nu-F})(\tau, \xi)\, d\tau\ra^2 .
 \end{align*}
Since for  $0 \leq \tau \leq t$ we have $\sigma(t) \leq \sigma(\tau)$, we deduce from the Cauchy-Schwarz  inequality that
$$A \leq T \int_0^t \Vert (F_\nu-F)(\tau)\Vert^2_{\mathcal{H}^{\sigma(\tau) , s-\mez}}\, d\tau.$$
Using Theorem~\ref{G-lip} applied with $\eta_1= \eta_\nu, \eta_2 = \eta$ and the fact that   the sequences $\Vert \eta_\nu\Vert_{\mathcal{X}}$ and  $\Vert \psi_\nu\Vert_{\mathcal{X}}$ are uniformly   bounded,  we deduce that   there exists     $C= C( \Vert \eta \Vert_{\mathcal{X}},\Vert \psi \Vert_{\mathcal{X}})>0$  such that  
$$
A \leq C (\Vert \eta_\nu- \eta \Vert_{\mathcal{X}} + \Vert \psi_\nu- \psi \Vert_{\mathcal{X}})\to 0, \quad \text{ if } \nu \to + \infty.
$$
Now we estimate
\begin{equation*}
B:= \lA \int_0^t (G_\nu -G)(\tau) \, d\tau \rA^2_{\mathcal{H}^{\sigma(t) , s-\mez}} \leq  C(B_1+B_2+B_3),
\end{equation*}
where 
\[
\left\{
\begin{aligned}
B_1 &= \lA \int_0^t   (\eta_\nu -\eta)(\tau)\, d\tau \rA^2_{\mathcal{H}^{\sigma(t) , s-\mez}},\\
B_2 & = \lA \int_0^t   (\vert \nabla_x \psi_\nu\vert^2  -\vert \nabla_x \psi \vert^2 )(\tau)\, d\tau \rA^2_{\mathcal{H}^{\sigma(t) , s-\mez}},\\
B_3&=  \lA \int_0^t  \frac{N_\nu^2(\tau)}{(1+ \vert \nabla_x \eta_\nu(\tau)\vert^2) }-  \frac{N ^2(\tau)}{(1+ \vert \nabla_x \eta (\tau)\vert^2)} \rA^2_{\mathcal{H}^{\sigma(t) , s-\mez}} 
\end{aligned}
\right.
\]
with 
$$N_\nu =  G(\eta_\nu)(\psi_\nu,b)+ \nabla_x \eta_\nu \cdot \nabla_x \psi_\nu, \quad N =  G(\eta )(\psi ,b)+ \nabla_x \eta  \cdot \nabla_x \psi. $$
As for the  term $B_1$, we have
$$B_1 \leq C \int_0^t  \Vert (\eta_\nu -\eta)(\tau)\Vert^2_{\mathcal{H}^{\sigma(t) , s-\mez}}\, d\tau \leq C \Vert (\eta_\nu -\eta)(\tau)\Vert^2_{\mathcal{X}} \to 0.$$
The  terms $B_2$ and $B_3$ can be estimated similarly by using \eqref{def-B'} with $\eta_{\nu-1} $ and $ \psi_{\nu-1}$ replaced by  $\eta, \psi$, together with the  estimates \eqref{est-B2'} and \eqref{est-B3'}. Hence, we conclude that $B$ tends to zero in $L^\infty([0,T], \mathcal{H}^{\sigma(t),s-\mez})$, which  implies that $(\eta, \psi)$ is a solution of system \eqref{system}.

The uniqueness of the solution follows from the computation made in Proposition~\ref{Mnu-Mnu+1}, where we replace   $(\widetilde{\eta}_{\nu}, \widetilde{\eta}_{\nu+1})$ by $(\widetilde{\eta}_1, \widetilde{\eta}_2)$ and  $(\widetilde{\psi}_{\nu}, \widetilde{\psi}_{\nu+1})$ by $(\widetilde{\psi}_1, \widetilde{\psi}_2)$, where $(\eta_1, \psi_1), (\eta_2, \psi_2)$ are the two supposed solutions.
 \end{proof}

\section{Existence of a  solution on a time interval  of size  $\eps^{-1}$}\label{S:sizeeps}
In this section we prove Theorem~\ref{T=2} about the well-posedness of the Cauchy problem on large time intervals. We shall construct solutions as limits of solutions to a sequence of approximate nonlinear systems. The analysis is in three different steps:
\begin{itemize}
\item [\rm{(1)}]  Firstly, we define approximate 
systems and prove that the Cauchy problem for the latter are 
well-posed locally in time by means of an ODE argument.
\item[\rm{(2)}]  Secondly, we prove that 
the solutions of the approximate systems are 
bounded on a uniform time interval.
\item [\rm{(3)}] Third, we prove that 
these approximate solutions converge to a solution of the water-waves system.
\end{itemize}

\subsection{Approximate systems.}
Let us rewrite the water-wave system under the form
\begin{equation*}
\partial_t f=\mathcal{T}(f;b),
\end{equation*}
where $f=\begin{pmatrix}\eta\\ \psi\end{pmatrix}$,
$$
\mathcal{T}(f;b)=\begin{pmatrix}G(\eta )(\psi ,b) \\[1ex]
-g \eta - \mez \vert \nabla_x \psi \vert^2 + \frac{1}{2(1+  \vert \nabla_x \eta \vert^2)}(G(\eta )(\psi,b)+ \nabla_x \eta  \cdot \nabla_x \psi )^2
\end{pmatrix},
$$
and $b=b(t,x)$ is a given function. We denote by $f_0=(\eta_0,\psi_0)$ the initial data.
 
To define the approximate systems, we use a 
version of Galerkin's method based on Friedrichs mollifiers. 
To do so, we shall  use smoothing operators. 
  We consider, for $n\in \xN\setminus\{0\}$, the operators $J_n$ defined by 
\begin{equation}\label{defi:Jn}
 \widehat{J_n u}(\xi) = \chi\Big(\frac{\xi  }{n}\Big)\widehat{u}(\xi) , 
\end{equation}
where $\chi\in C^\infty(\xR^d)$ is such that
$$
\chi(y)=1 \text{ if }\vert y \vert \leq 1 \quad\text{and}\quad \chi(y)= 0 \text{ if }\vert y \vert \geq 2.
$$ 
 
Now we consider the following approximate Cauchy problems:
\begin{equation}\label{A3}
\left\{
\begin{aligned}
&\partial_t f=J_n\big(\mathcal{T}(f;b)\big),\\
& f\arrowvert_{t=0}=J_n f_{0}.
\end{aligned}
\right.
\end{equation} 

The following lemma states that, for each $n\in\xN\setminus\{0\}$, the Cauchy problem~\e{A3} is well-posed locally in time.

\begin{lemm}\label{lemm-Tn}
Let $s> 2+\frac{d}{2}$. For all initial data 
$f_0=(\eta_0,\psi_0)\in L^{2}(\xT^d)^2$, 
for all source term $b\in C^0([0 , +\infty); H^{s- \mez}(\xT^d))$ 
and for all $n\in \xN\setminus \{0\}$, there exists $T_n>0$ such that 
the Cauchy problem~\e{A3} has a unique maximal solution
$$
f_n=(\eta_n,\psi_n)\in C^{1}\big([0,T_n);L^2(\xT^d)^2\big).
$$
Moreover, $f_n$ is a smooth function which belongs to $C^{0}([0,T_n);\mathcal{H}^{\lambda,\mu}(\xT^d))$ 
for any $\lambda,\mu\ge 0$. 

Finally, either 
\be\label{A4}
T_n=+\infty\qquad\text{or}\qquad \limsup_{t\rightarrow T_n} \lA f_n(t)\rA_{L^2}=+\infty.
\ee
\end{lemm}
\begin{proof}
Fix $f_0=(\eta_0,\psi_0)\in L^{2}(\xT^d)^2$ and $b\in L^2(\xT^d)$.  
Let $\widetilde {\chi}\in C^\infty_0(\xR^d)$ equal to $1$ on the support of $\chi$ and set 
$\widehat{ \widetilde{J}_n u} (\xi) = \widetilde{\chi} ( \frac {\xi} n ) \widehat{u} (\xi)$. 
With this choice we have $\widetilde{J}_nJ_n =J_n$.  
We begin by studying an auxiliary Cauchy problem which reads
\begin{equation}\label{A3-twisted}
\left\{
\begin{aligned}
&\partial_t f=F_n(f)\quad\text{where}\quad F_n(f)=J_n\big(\mathcal{T}(\widetilde{J}_n f;b)\big),\\
& f\arrowvert_{t=0}=J_n f_{0}.
\end{aligned}
\right.
\end{equation}
We will prove that the Cauchy problem is well-posed by using the classical fixed point argument. 
Then we will prove that if $f_n$ solves \eqref{A3-twisted}, then it also solves the original problem \eqref{A3}. 
Notice that the operators $J_n$ and $\widetilde{J}_n$ are 
smoothing operators: they are bounded from $L^2(\xT^d)$ 
into $H^\mu(\xT^d)$ for any $\mu\ge 0$. Consequently, it would be sufficient 
to exploit some rough estimates for the Dirichlet-Neumann 
operator which could be proved by elementary variational estimates (see~\cite[Theorems 3.8 and 3.9]{ABZ3}). For completeness (due to the additional source term $b$), we shall use here the more elaborates estimates from Section \ref{S:DN} and Appendix \ref{AppendixA} restricted to the case of Sobolev spaces ($\lambda =0$). From~\eqref{DNn1}, we have (for $s\geq \frac 3 2 + \frac d 2$)
\begin{equation*}
\Vert G(\eta)(\psi,b) \Vert_{ {H}^{ s}}
\leq C \big( \Vert \psi \Vert_{{H}^{ s+1 }} + \Vert b \Vert_{H^{s}}
+ \Vert  \eta\Vert_{{H}^{s+ 1}} \big(  \Vert \psi \Vert_{{H}^{s+1}} +\Vert b \Vert_{H^{s-1}}\big)\big).
\end{equation*}
and from~Theorem~\ref{G-lip}, we have 
\begin{equation}\label{DNn1ter}
\begin{split}
& \Vert G(\eta_1)(\psi_1,b) - G(\eta_2)(\psi_2,b) \Vert_{ {H}^{ s}}\\
\leq &\, C \big( 1+  \| (\eta_1 , \eta_2)\|_{H^{s+1}})   ( \| (\psi_1 , \psi_2)\|_{H^{s+1}} +\| b\|_{H^{s}}  )\big) \\
& \times\bigl( \| \eta_1 - \eta_2\|_{H^{s+1}}+  \| \psi_1 - \psi_2\|_{H^{s+1}}\bigr).
\end{split}
\end{equation}
It follows from  standard nonlinear estimates that  for some non decreasing functions $\mathcal{F}_j: \xR_+ \rightarrow \xR$, 
\begin{equation*}
\|  F_n (f) \|_{H^{s}}\\
  \leq C\mathcal{F}_1 ( \| \widetilde{J}_n f \|_{H^{s+ 1}}) ( \| b\|_{H^s}+  \| \widetilde{J}_nf\| _{H^{s+1}})
\end{equation*}
and (using~\eqref{DNn1ter})
\begin{equation*}
\Big\|   F_n (f_1) -  F_n (f_2)\Big\|_{H^{s}}\\
  \leq C \mathcal{F} _2   (\| \widetilde{J}_n f \|_{H^{s+ \frac 3 2}}+ \| b\|_{H^s})
  \| \widetilde{J}_n (f_1 - f_2)  \|_{H^s} .
\end{equation*}

Since the operators $\widetilde{J}_n$ are bounded from $L^2$ to $H^k$ by $C n^k$, It follows that 
the operator $f\mapsto F_n(f)$ is locally Lipschitz from $L^{2}(\xT^d)$ to itself. 
Consequently, the Cauchy-Lipschitz theorem implies that the Cauchy problem~\e{A3-twisted} has 
a unique maximal solution~$f_n$ in~$C^{1}([0,T_n);L^{2}(\xT^d))$. 
Since    $(I- \widetilde{J}_n)J_n =0$, we check that the function $(I-\widetilde{J}_n)f_n$ solves
$$
\partial_t (I-\widetilde{J}_n)f_n=0,\quad (I-\widetilde{J}_n)f_n\arrowvert_{t=0}=0.
$$
This shows that $(I-\widetilde{J}_n)f_n=0$, so $\widetilde{J}_nf_n=f_n$. Consequently, 
the fact that $f_n$ solves \e{A3-twisted} implies 
that $f_n$ is also a solution to~\e{A3}. In addition, since the Fourier transform of $f_n$ is compactly supported, 
the function $f_n$ belongs to $C^{0}([0,T_n);\mathcal{H}^{\lambda,\mu}(\xT^d))$  
for any $\lambda,\mu\ge 0$. 
 The alternative \eqref{A4} is a consequence of the usual 
continuation principle for ordinary differential equations. 
\end{proof}

\subsection{Reformulation of the equations.}

Let $n\in\xN\setminus\{0\}$ and denote by $(\eta_n,\psi_n)$ the 
approximate solution as constructed in the previous paragraph. 
Recall that
\begin{equation}\label{ww1'}
\left\{
\begin{aligned}
&\partial_t \eta_n  = J_n G(\eta_n )(\psi_n ,b),  \\
  &\partial_t \psi_n 
  =J_n\Big( -g \eta_n - \mez \vert \nabla_x \psi_n \vert^2 
  + \frac{(G(\eta_n )(\psi_n,b)+ \nabla_x \eta_n  \cdot \nabla_x \psi_n )^2}{2(1+  \vert \nabla_x \eta_n \vert^2)}\Big),\\   
  &\eta_n \arrowvert_{t=0} = J_n\eta_0, \quad \psi_n \arrowvert_{t=0} = J_n\psi_0.
 \end{aligned}
\right.
\end{equation}
In this paragraph we shall 
reformulate the above equations into a new set of equations where  the unknowns are $(\zeta_n,V_n,B_n)$ defined by 
\be\label{defiBnVn}
\zeta_n = \nabla_x \eta_n, \quad B_n=\frac{G(\eta_n)(\psi_n,b) + \zeta_n\cdot\nabla_x\psi_n}{1+| \zeta_n|^2},\quad 
V_n=\nabla_x\psi_n-B_n \zeta_n.
\ee
Notice that with the above notations the equation on $\psi_n$ can be reformulated as follows:
\begin{equation}\label{new-psin}
\partial_t \psi_n = J_n\Big(-g \eta_n -\mez \vert V_n + B_n   \zeta_n\vert^2 + \mez(1+ \vert   \zeta_n\vert^2) B_n^2\Big).
\end{equation}
Moreover, we have
\begin{equation}\label{DN-VB}
G(\eta_n)(\psi_n,b) = B_n -V_n \cdot   \zeta_n.
\end{equation}
Recall that by definition, we have
 \begin{equation}\label{DNn}
 \begin{aligned}
  &G(\eta_n)(\psi_n,b)= (\partial_y \phi_n -   \zeta_n\cdot \nabla_x  \phi_n)\arrowvert_{y= \eta_n(x)}, \quad \text{where}\\
  &\Delta_{x,y}\phi_n = 0 \text{ in} \, \{-h<y< \eta_n(x)\}, \quad \phi_n\arrowvert_{y= \eta_n(x)} = \psi_n, 
  \quad \partial_y\phi_n\arrowvert_{y=-h}= b.
  \end{aligned}
  \end{equation}
  Notice that by definition, we have
  \begin{equation*}
  \vert \nabla_{x,y} \phi_n\vert^2\arrowvert_{y = \eta_n} = \vert V_n \vert^2 + B_n^2.
  \end{equation*}
  Before reformulating the equations, we need several identities that will be used later.
  \begin{lemm}\label{der2phi}
  Let $\phi_n$ be as defined by \eqref{DNn}. Then
  \begin{equation*}
  \left\{
  \begin{aligned}
  \partial_y\nabla_x \phi_n \arrowvert_{y = \eta_n} &=   \nabla_x B_n  -\frac{(\nabla_x B_n \cdot \zeta_n) \zeta_n - (\cnx V_n)  \zeta_n }{1+ \vert  \zeta_n \vert^2} ,\\
 \partial_y^2 \phi_n \arrowvert_{y = \eta_n} &=     \frac{\nabla_x B_n\cdot  \zeta_n - \cnx V_n}{1+ \vert   \zeta_n \vert^2}.
   \end{aligned}
   \right.
  \end{equation*}
  \end{lemm}
  \begin{proof}
Notice that  $  B_n = \partial_y\phi_n(x, \eta_n(x))$. 
Differentiating this identity with respect to $x$ and using the equation satisfied by $\phi_n$, 
we obtain
$$ \nabla_x B_n = \partial_y\nabla_x \phi_n \arrowvert_{y = \eta_n} 
+ \partial_y^2\phi_n \arrowvert_{y = \eta_n}  \zeta_n 
=  \partial_y\nabla_x \phi_n \arrowvert_{y = \eta_n} - \Delta_x\phi_n \arrowvert_{y = \eta_n}  \zeta_n.
$$
We have also $V_n = \nabla_x \phi_n\arrowvert_{y = \eta_n}.$ Taking the divergence of both members, we obtain
\begin{equation} \label{xyphin1}
\cnx V_n = \Delta_x\phi_n \arrowvert_{y = \eta_n}+  \partial_y\nabla_x \phi_n \arrowvert_{y = \eta_n}\cdot  \zeta_n.
\end{equation}
It follows that
\begin{equation} \label{xyphin2}
\partial_y\nabla_x \phi_n \arrowvert_{y = \eta_n}= \nabla_x B_n 
+ \big[\cnx V_n -\partial_y\nabla_x \phi_n \arrowvert_{y = \eta_n}\cdot  \zeta_n\big]  \zeta_n.
\end{equation}
Taking the scalar product of both members with $\zeta_n=\nabla_x \eta_n$, we obtain
$$
(1+ \vert   \zeta_n \vert^2)(\partial_y\nabla_x \phi_n \arrowvert_{y = \eta_n}\cdot   \zeta_n)
= \nabla_x B_n \cdot\zeta_n+ \vert   \zeta_n \vert^2 \cnx V_n,
$$
therefore,
\begin{equation} \label{xyphin3}
\partial_y\nabla_x \phi_n \arrowvert_{y = \eta_n}\cdot   \zeta_n
= \frac{\nabla_x B_n\cdot   \zeta_n + \vert   \zeta_n \vert^2 \cnx V_n}{1+ \vert   \zeta_n\vert^2}. 
\end{equation}
Then we use \eqref{xyphin2} to obtain the first claim of 
the lemma. Since $$\partial_y^2\phi_n\arrowvert_{y= \eta_n} = - \Delta_x\phi_n\arrowvert_{y= \eta_n},$$ 
the second claim follows from \eqref{xyphin1} and \eqref{xyphin3}.    
\end{proof}

\subsubsection{The pressure.}
We introduce  now the pressure associated to the new system \eqref{ww1'}. 
We shall define $P_n$ by
\begin{equation}\label{press-n}
P_n + gy = -(\partial_t \phi_n + \mez \vert \nabla_{x,y}\phi_n \vert^2),
\end{equation}
where $\phi_n$ is defined in \e{DNn}. 
Then setting $v_n = \nabla_{x,y}\phi_n$ and differentiating \eqref{press-n} with respect to $x$ an $y$ we find that $v_n$ satisfies the system
\begin{equation}\label{euler-n}
\partial_t v_n + (v_n\cdot\nabla_{x,y}) v_n = -\nabla_{x,y}(P_n + gy).
\end{equation} 
We define now $Q_n$ by
\begin{equation}\label{def-Qn}
P_n = - \mez \vert \nabla_{x,y}\phi_n \vert^2 + Q_n - gy.
\end{equation}
Then according to \eqref{DNn}, we find that $Q_n$ is the solution of the problem
\begin{equation}\label{eq-Qn}
\left\{
\begin{aligned}
& \Delta_{x,y} Q_n = 0, \quad\text{ in} -h<y<\eta_n(x),\\   
& Q_n\arrowvert_{y = \eta_n} = P_n\arrowvert_{y = \eta_n} + g \eta_n + \mez(\vert V_n\vert ^2 + B_n^2),\\
& \partial_y Q_n\arrowvert_{y= -h} = - \partial_t b.
\end{aligned}
\right.
\end{equation}
This shows that $Q_n$ is solution of an elliptic problem. 
We will have good estimates on $Q_n$  as soon as we have 
described $P_n\arrowvert_{y = \eta_n},$ which we 
do now. 
\begin{lemm}
We have
\begin{equation}\label{press3}
P_n\arrowvert_{y= \eta_n}= g(J_n-1)\eta_n + \mez(J_n-1)(\vert V_n\vert^2 + B_n^2)   + [B_n,J_n](B_n -(V_n \cdot \zeta_n)). \\
\end{equation}
\end{lemm}
\begin{proof}
It follows from  \eqref{press-n} that
\begin{equation}\label{press2}
P_n\arrowvert_{y = \eta_n} = -g \eta_n -(\partial_t \phi_n)\arrowvert_{y = \eta_n} - \mez(\vert V_n\vert^2 + B_n^2).
\end{equation}
Now since $\psi_n(t,x) = \phi_n(t,x, \eta_n(t,x))$, we have
\begin{align*}
  \partial_t \psi_n(t,x) &= \partial_t\phi_n(t,x, \eta_n(t,x)) + \partial_y \phi_n(y,x, \eta_n(t,x))\partial_t \eta_n(t,x)\\
  &=  \partial_t\phi_n\arrowvert_{y=\eta_n} + B_n J_n\big[G(\eta_n)(\psi_n,b)\big].
  \end{align*}
The identity \eqref{new-psin} for $\partial_t\psi_n$ implies that 
 $$J_n\big[-g \eta_n -\mez \vert V_n + B_n   \zeta_n\vert^2 + \mez(1+ \vert   \zeta_n\vert^2) B_n^2\big]=  \partial_t\phi_n\arrowvert_{y=\eta_n} + B_n J_n\big[G(\eta_n)(\psi_n,b)\big].$$
 The lemma follows from \eqref{DN-VB} and \eqref{press2}.
 \end{proof}
 \begin{lemm}\label{der-x-y-P}
 We have
 \begin{equation}
 \begin{aligned}
 &(\nabla_x P_n)\arrowvert_{y = \eta_n} = gJ_n \zeta_n -\big(\partial_y P_n\arrowvert_{y = \eta_n} +g)\zeta_n + S_n ,\\
 & S_n =  (J_n-1)(V_n\cdot \nabla_x) V_n -(J_n+1)B_n (\nabla_x B_n)
    +(\nabla_x B_n) J_n B_n  
  \\
     & \qquad  + B_n J_n (\nabla_x B_n)  +J_n\nabla_x(B_nV_n\cdot\zeta_n)  
      -\nabla_x(B_n J_nV_n\cdot\zeta_n).
 \end{aligned}
 \end{equation}
 \end{lemm}
 \begin{proof}
 This follows from  \eqref{press3} which we differentiate with respect to $x.$
 \end{proof}
 
  \subsubsection{The new equations.}
  We are now in position to reformulate our equations.
\begin{prop}\label{new-eq}
Let $(\eta_n, \psi_n)$ be a solution of \eqref{ww1'}. Then, with the notations in \eqref{defiBnVn} we have
\begin{align*}
&{\rm (i)} \quad \partial_t\zeta_n    + J_n (V_n\cdot \nabla_x \zeta_n ) = J_n\big[G(\eta_n)(V_n,\nabla_x b) -(\cnx V_n) \zeta_n\big],\\
& {\rm (ii)} \quad (\partial_t   + V_n\cdot\nabla_x) B_n  \\
& \ \qquad \hfill = -(\partial_y P_n\arrowvert_{y= \eta_n} +g) -\partial_y^2\phi_n\arrowvert_{y= \eta_n}\big((J_n-1)(V_n \cdot  \zeta_n-B_n)\big),\\
& {\rm (iii)} \quad (\partial_t   + V_n\cdot\nabla_x) V_n \\
&\qquad = -\nabla_x P_n\arrowvert_{y= \eta_n}- \partial_y\nabla_x \phi_n\arrowvert_{y= \eta_n}\big((J_n-1)(V_n \cdot  \zeta_n-B_n)\big).
 \end{align*}
\end{prop}
\begin{proof}
(i)\quad We have
$$\partial_t \zeta_n = \nabla_x (\partial_t \eta_n) = J_n\nabla_x\big[ G(\eta)(\psi_n,b)\big],$$
so using Lemma~\ref{der-DN} we obtain
$$\partial_t \zeta_n  = J_n\big[G(\eta_n)(V_n, \nabla_x b)-(V_n \cdot \nabla_x) \zeta_n - (\cnx V_n)  \zeta_n\big],$$
which proves (i).

(ii) We have $B_n = \partial_y \phi_n(t,x, \eta_n(t,x)).$ Therefore,
\begin{equation}\label{eq-Bn1}
 \partial_t B_n = \partial_t \partial_y \phi_n\arrowvert_{y= \eta_n} + \partial_y^2 \phi_n\arrowvert_{y= \eta_n} \partial_t \eta_n .
 \end{equation}
Now, since $v_n = \nabla_{x,y}\phi_n$, the last equation of the system \eqref{euler-n} restricted to $y = \eta_n$ reads
\begin{equation}\label{eq-Bn2}
\partial_t \partial_y \phi_n\arrowvert_{y= \eta_n} + V_n \cdot \nabla_x \partial_y \phi_n\arrowvert_{y= \eta_n} + B_n \partial_y^2 \phi_n\arrowvert_{y= \eta_n} = -(\partial_y P_n\arrowvert_{y= \eta_n} + g).
\end{equation}
On the other hand, we have
\begin{equation}\label{eq-Bn3}
\begin{aligned}
V_n \cdot \nabla_x B_n &= (V_n \cdot \nabla_x)[\partial_y \phi_n(t,x,\eta(t,x))],\\
&= V_n \cdot \nabla_x \partial_y \phi_n\arrowvert_{y = \eta_n} + (V_n\cdot   \zeta_n)\partial_y^2\phi_n\arrowvert_{y = \eta_n}.
\end{aligned}
\end{equation}
 Using \eqref{eq-Bn1}, \eqref{eq-Bn2} and \eqref{eq-Bn3}, we obtain
 $$\partial_t B_n + V_n \cdot \nabla_x B_n = -(\partial_y P_n \arrowvert_{y = \eta_n} +g) + \partial_y^2 \phi_n\arrowvert_{y = \eta_n}\big( V_n \cdot   \zeta_n -B_n + \partial_t \eta_n\big).$$
 Then (ii) follows from the fact that
 \begin{equation}\label{eq-Bn4}
  \partial_t \eta_n = J_n G(\eta_n)(\psi_n,b) = J_n(B_n - (V_n \cdot \nabla_x) \eta_n),
  \end{equation}
 by \eqref{DN-VB}.
 
 (iii) We have $V_n = \nabla_x \phi_n(t,x,\eta(t,x)).$ Therefore,
 \begin{equation}\label{eq-Vn1}
 \partial_t V_n = \partial_t\nabla_x \phi_n\arrowvert_{y = \eta_n} + \partial_y \nabla_x \phi_n\arrowvert_{y = \eta_n}\partial_t \eta_n.
 \end{equation}
 On the other hand,
 \begin{equation}\label{eq-Vn2}
 (V_n \cdot\nabla_x) V_n = (V_n \cdot\nabla_x)\nabla_x \phi_n\arrowvert_{y = \eta_n}+( (V_n \cdot\nabla_x)\eta_n )\nabla_x \partial_y \phi_n\arrowvert_{y = \eta_n}.
 \end{equation}
Now, since $v_n = \nabla_{x,y}\phi_n$, the $d$-first equations of the system \eqref{euler-n} restricted to $y = \eta_n$ 
are written as 
\begin{equation}\label{eq-Vn3}
\partial_t \nabla_x \phi_n \arrowvert_{y = \eta_n}
+ (V_n\cdot\nabla_x) \nabla_x \phi_n \arrowvert_{y = \eta_n}
+ B_n \partial_y \nabla_x \phi_n \arrowvert_{y = \eta_n}= -\nabla_x P_n\arrowvert_{y= \eta_n}.
\end{equation}
Using \eqref{eq-Vn1}, \eqref{eq-Vn2}, \eqref{eq-Vn3}, we obtain
$$
\partial_t V_n + (V_n \cdot\nabla_x) V_n
= -\nabla_x P_n\arrowvert_{y= \eta_n}
+ \nabla_x \partial_y \phi_n\arrowvert_{y = \eta_n}(\partial_t \eta_n + (V_n \cdot\nabla_x)\eta_n -B_n).
$$
Then (iii) follows from \eqref{eq-Bn4}.
\end{proof}
\begin{coro}\label{new-eq2}
We have
\begin{equation*}
\left\{
\begin{aligned}
&\partial_t V_n +gJ_n\zeta_n = (\partial_y P_n \arrowvert_{y = \eta_n} + g\big)\zeta_n   + R_n^{(0)},\\
&    \partial_t \zeta_n - J_n G(0)(V_n, 0) = R_n^{(1)},\\
\end{aligned}
\right.
\end{equation*}
where
 \begin{align*}
R_n^{(0)} =& \, -J_n( V_n\cdot\nabla_x) V_n -(\partial_y \nabla_x\phi_n\arrowvert_{y= \eta_n}) (J_n-1)(V_n\cdot \zeta_n - B_n)    \\
&+ (J_n+1)B_n \nabla_x B_n -\nabla_x B_n J_n B_n   
 - B_n J_n \nabla_x B_n\\  
     &-J_n\nabla_x(B_nV_n\cdot \zeta_n)  
      +\nabla_x(B_n J_nV_n\cdot \zeta_n), 
\end{align*}
\[
R_n^{(1)} = J_n \big[G(\eta_n)(V_n, 0) - G(0)(V_n,0)  
+  G(\eta_n)(0, \nabla_x b) -(\cnx V_n)\zeta_n -(V_n\cdot \nabla_x) \zeta_n\big].
 \]
 \end{coro}
\begin{proof}
The first claim follows from Proposition~\ref{new-eq} (iii) and Lemma~\ref{der-x-y-P}.
The second claim follows from Proposition~\ref{new-eq} (i)  writing
$$G(\eta_n)(V_n, \nabla_x b)= G(0)(V_n,0) + \big[G(\eta_n)(V_n, 0) - G(0)(V_n,0  )\big] +  G(\eta_n)(0, \nabla_x b).$$
\end{proof}
 Recall  that we have set
 $$G(0)(V_n,0) = a(D_x)V_n = \vert D_x \vert \tanh  (h \vert D_x \vert)V_n.$$

\subsubsection{The final equation.}
We shall use Corollary~\ref{new-eq2}   to deduce a new equation on a new unknown with which we shall work. 
  
 Let us set
\begin{equation}\label{def-un}
u_n = \sqrt{g}\zeta_n + i\big(a(D_x)\big)^\mez  V_n.
\end{equation}
Then, with the notations in Corollary~\ref{new-eq2} we have:
\begin{lemm}
The function $u_n$ satisfies the equation
 \begin{equation}\label{eq-fin}
\partial_t u_n + iJ_n \big(ga(D_x)\big)^\mez u_n = \sqrt{g }R_n^{(1)} + i a(D_x) ^\mez\big[(\partial_y P_n \arrowvert_{y = \eta_n} + g)\zeta_n + R_n^{(0)}\big].
\end{equation}
\end{lemm}
\begin{proof}
 This follows immediately from Corollary~\ref{new-eq2}.
\end{proof}
Set
$$U_{s,n}(t) = e^{\sigma(t) \langle D_x \rangle}\langle D_x \rangle^{s-\mez} u_n(t), \quad \sigma(t) = \lambda h - K \eps t,$$
where  $K$ is a  large positive constant to be chosen.

Since
$$ e^{\sigma(t)\vert \xi \vert} \leq e^{\sigma(t)\langle \xi \rangle} \leq e^{\sigma(t)}   e^{\sigma(t)\vert \xi \vert} \leq e^{\lambda h} e^{\sigma(t)\vert \xi \vert},$$
and since $\zeta_n$ and $a(D_x)^\mez V_n$ are real valued functions, there exist  two absolute positive constants $C_1, C_2$ 
such that
for all $t\in [0,T]$ and $\mu = 0$ or $\mu = \mez$ 
\begin{equation}\label{equiv-Us-u}
\begin{aligned}
& \Vert \zeta_n(t)\Vert_{\mathcal{H}^{\sigma(t), s-\mez+\mu}} + \Vert a(D_x)^\mez V_n(t)\Vert_{\mathcal{H}^{\sigma(t), s-\mez+\mu}}  \leq C_1\Vert U_{s,n}(t)\Vert_{H^\mu},\\
&\Vert U_{s,n}(t)\Vert_{H^\mu} \leq C_2\big(\Vert \zeta_n(t)\Vert_{\mathcal{H}^{\sigma(t), s-\mez+\mu}} + \Vert a(D_x)^\mez V_n(t)\Vert_{\mathcal{H}^{\sigma(t), s-\mez+\mu}}\big).
\end{aligned}
\end{equation}

Fix two real numbers $h>0$ and $\lambda \in [0,1)$. Consider initial data $(\eta_0,\psi_0)$ such that 
$$
 \Vert   \eta_0  \Vert_{\mathcal{H}^{\lambda h,s+ \mez}}  +  \Vert   a(D_x)^\mez \psi_0  \Vert_{\mathcal{H}^{\lambda h,s}} 
+     \Vert   V_0  \Vert_{\mathcal{H}^{\lambda h,s}}  +  \Vert  B_0 \Vert_{\mathcal{H}^{\lambda h,s}}<+\infty,
$$
where, as above,
$$
B_0=\frac{G(\eta_0)\psi_0+\nabla_x\eta_0\cdot\nabla_x\psi_0}{1+|\nabla_x\eta_0|^2},\quad 
V_0=\nabla_x\psi_0-B_0\nabla_x\eta_0.
$$
Notice that, when $\psi_0=0$, the functions $V_0$ and $B_0$ vanish and hence the previous assumption is satisfied whenever $\eta_0\in \mathcal{H}^{\lambda h,s+\mez}$. 
Recall that we denote by $T_n$ the 
lifespan of the approximate solution $(\eta_n,\psi_n)$ and that 
we denote by $\overline{\eps}$ the constant determined by means of Theorem~\ref{est-DN}. 
We are now in position to state our main Sobolev estimates. 
For this we introduce some notations. We set
\begin{equation}\label{Nsb}
\left\{
\begin{aligned}
&N_s(b) =   \Vert  b \Vert_{L^\infty(\xR,H^{s+ \mez})} + \Vert  \partial_t  b \Vert_{L^\infty(\xR,H^{s- \mez})} +\Vert b \Vert_{L^1(\xR, H^{s+ \mez})},\\
& M_{s,n} (T) =   \Vert   \eta_n  \Vert_{X_T^{\infty,s+ \mez}}  +  \Vert   a(D_x)^\mez \psi_n \Vert_{X_T^{\infty,s}} +     \Vert  V_n  \Vert_{X_T^{\infty,s}}  
+  \Vert  B_n \Vert_{X_T^{\infty,s}}
\end{aligned}
\right.
\end{equation}
where
\[
X_T^{\infty,s} = L^\infty([0,T], \mathcal{H}^{\sigma(\cdot), s})\quad\text{with}\quad \sigma(t) = \lambda h - K \eps t.
\]

\begin{prop}\label{Prop:C0}
Fix $g>0, h>0, 0<\lambda <1$ and  $s>5/2+d/2$. Set
\begin{equation}\label{defi:epsnorm}
\eps=\Vert   \eta_0  \Vert_{\mathcal{H}^{\lambda h,s+ \mez}}  +  \Vert   a(D_x)^\mez \psi_0  \Vert_{\mathcal{H}^{\lambda h,s}} 
+     \Vert   V_0  \Vert_{\mathcal{H}^{\lambda h,s}}  +  \Vert  B_0 \Vert_{\mathcal{H}^{\lambda h,s}} + N_s(b).
\end{equation}
There exists a constant $C_0\ge 1$ such that for all $n\in \xN\setminus\{0\}$, 
for all $T\le T_n$ and for all $K>0$  
the norm $M_{s,n} (T)$ 
satisfies the following inequality: if $M_{s,n}(T)\le \overline{\eps}$, then 
\begin{equation}\label{a-priori30}
\begin{split}
& M_{s,n}(T)^2 + 2K\eps\int_0^T\Vert U_{s,n}(t)\Vert^2_{H^\mez}\, dt    \\
 \leq &\,C_0 \Big(  M_{s,n}(T)\int_0^T \Vert U_{s,n}(t)\Vert^2_{H^\mez}\, dt  + \frac{1}{K\eps}h_\eps(T)
 +(1+ T^2)M_{s,n}(T)^4 + \eps^2\Big),
 \end{split}
 \end{equation}
where
\begin{equation}\label{heps}
 h_\eps(T)= TM_{s,n}(T)^4+   T^5M_{s,n}(T)^{8} + \eps^2 T  M_{s,n}(T)^2.
 \end{equation}
\end{prop}

The proof of this proposition is postponed to Section \ref{S:proofPC0}.

Let us assume Proposition~\ref{Prop:C0} for the moment. 
Our goal in the end of this paragraph is to explain how to deduce some uniform estimates for the approximate solutions on a large time interval.

\begin{coro}\label{coro-c*}
Fix $g>0, h>0, 0<\lambda <1$ and  $s>5/2+d/2$. There exist four positive real numbers $\eps_*$, $c_*$, $C_*$ and $K_*$ such that for all $n\in \xN\setminus\{0\}$, 
the following properties hold: if the initial norm $\eps$ (as defined 
by~\e{defi:epsnorm}) satisfies $\eps\le \eps_*$, then for all $n\in \xN\setminus\{0\}$, the lifespan is bounded from below by
$$
T_n\ge \frac{c_*}{\eps},
$$
and moreover,
$$
M_{s,n}\Big(\frac{c_*}{\eps}\Big)\le C_*\eps,
$$
where the norm $M_{s,n}$ has been  defined in~\eqref{Nsb}, with $K$ replaced by $K_*$.
\end{coro}
\begin{proof}
Fix $\mu$ such that $\mu \geq 16$ and $  \mu C_0\geq 1.$ Then set
\begin{equation}\label{param}
 C_*= \sqrt{\mu C_0}\geq 1,\quad \eps_*
 = \frac{\overline{\eps}}{4 C^2_*},\quad  c_*=\frac{1}{C_*^4}\leq 1,\quad  K_* \geq \text{max}\Big(3, \frac{1}{\mu} C_*^3\Big), 
 \end{equation}
where $C_0$ is the constant whose existence 
is the main assertion of Proposition~\ref{Prop:C0}.

Hereafter we assume that $\eps\le\eps_*$. 
Given $n\in \xN\setminus\{0\}$, let us introduce the interval
$$
I_n=\Big[0,\min\Big\{T_n,\frac{c_*}{\eps}\Big\}\Big].
$$
We prove that, for all $n\in \xN\setminus\{0\}$,
\be\label{claim:uni}
\forall T\in I_n,\quad M_{s,n}(T)\le C_*\eps.
\ee
Notice that if \e{claim:uni} holds, then it follows   that $T_n\ge c_*/\eps$ in light of the alternative~\e{A4}. 
In particular, 
we can apply \e{claim:uni} for $T=c_*/\eps$, which will give the wanted result.

It remains to prove~\e{claim:uni}. 
To do so, introduce the set $$J_n=\{T\in I_n\, :\, M_{s,n}(T)\le C_*\eps\}.$$ 
With this notation, we prove that $J_n=I_n$. 
Notice that $0\in J_n$ since $M_{s,n}(0)\le \eps$ and $C_*\ge 1$. Since $T\mapsto M_{s,n}(T)$ is continuous, the set $J_n$ is closed. 
Hence, to conclude the proof, it is sufficient to prove that $J_n$ is open. This is turn will be a straightforward corollary of the following claim:
\begin{equation}\label{claimcontinuity}
\forall T\in J_n,\quad M_{s,n}(T)\le \frac{1}{2}C_*\eps.
\end{equation}
Let us prove this claim. To do so, we will exploit \eqref{a-priori30}. Since $\eps\le \eps_*=\overline{\eps}/(4C_*)$, notice that if $M_{s,n}(T)\le C_*\eps$, then we automatically obtain that 
$M_{s,n}(T)\le \overline{\eps}$. Then we are in position to apply the estimate~\eqref{a-priori30}. We use \eqref{claim:uni} and the fact that $T_n \leq \frac{c_*}{\eps}.$ Recall from \eqref{a-priori30} that
\begin{align*}
 & M_{s,n}(T)^2 + 2K\eps\int_0^T\Vert U_{s,n}(t)\Vert^2_{H^\mez}\, dt  \\
   \leq &\, \frac{1}{\mu}C_*^2 \left( C_* \eps\int_0^T \Vert U_{s,n}(t)\Vert^2_{H^\mez}\, dt  + \frac{1}{K\eps}h_\eps(T)
 +(1+ T^2)M_{s,n}(T)^4 + \eps^2\right).
 \end{align*}
Since $\frac{1}{\mu}C_*^3 \leq K_*$, we can absorb the integral term in the right hand side into the left hand side.
Now according to \eqref{heps}, we have
$$\frac{1}{K_*\eps}h_\eps(T) \leq \frac{1}{K_*}\big(c_* C_*^4 + c_*^5 C_*^8 + c_* C_*\big) \eps^2.$$
Since $c_* = \frac{1}{C^4_*},$  $C_* \geq 1$ and $K_* \geq 3$, we obtain
$$\frac{1}{K_*\eps}h_\eps(T) \leq \frac{3}{K_*}  \eps^2 \leq \eps^2.$$
On the other hand, by \eqref{param} we have for $\eps \leq \eps_*$,
 $$(1 + T_n^2)M_{s,n}(T)^4 + \eps^2\leq (\eps^2 + c_*^2)C^4_*\eps^2 + \eps^2 \leq 3 \eps^2.$$
 It follows that
\begin{align*}
 M_{s,n}(T)^2  \leq \frac{4}{\mu} C_*^2\eps^2 \leq \frac{1}{4} C_*^2\eps^2,
\end{align*}
 since $\mu \geq 16.$
This completes the proof of the claim~\eqref{claimcontinuity}.
\end{proof}

\subsection{Proof of Proposition~\ref{Prop:C0}.}\label{S:proofPC0}
We fix $d\ge 1$, $g>0$, $\lambda<1$, $h>0$, $s>5/2+d/2$.

To simplify notations, the indexes $n$ will be skipped: we fix an integer 
$n$ in $\xN\setminus\{0\}$ and denote simply by $(\eta,\psi)$ the 
solutions to the approximate system~\e{ww1'}. 
We shall denote by   $C$ many different constants, 
whose values may  change  from a  line to another, and which 
depend only on the parameters which are considered fixed (that is 
$d,\lambda,h$ and $s$). 
In particular, these constants are independent of $T,K,\overline{\eps}$ and $n$.

We set 
\begin{align*}
& I = [0,T],    \\  
 &\sigma(t) = \lambda h - K \eps t,  \quad t \in [0,T], \quad \eps>0, \quad   K>0 \quad  \text{(to be chosen)}, \\  
 &a(D_x)  = \vert D_x \vert \tanh (h \vert D_x \vert)\,  (= G_0(0)),\\
 &X^{\infty,s} = L^\infty(I, \mathcal{H}^{\sigma(\cdot), s}),\\
&M_s (T) =   \Vert   \eta  \Vert_{X^{\infty,s+ \mez}}  +  \Vert   a(D_x)^\mez \psi  \Vert_{X^{\infty,s}} +     \Vert   V  \Vert_{X^{\infty,s}}  +  \Vert  \B \Vert_{X^{\infty,s}},\\
& N_s(b) 
=   \Vert  b \Vert_{L^\infty(\xR,H^{s+ \mez})} + \Vert  \partial_t  b \Vert_{L^\infty(\xR,H^{s- \mez})} 
+\Vert b \Vert_{L^1(\xR, H^{s+ \mez})},\\
&U_{s,n}(t) = e^{\sigma(t)\langle D_x \rangle} \langle D_x\rangle^{s-\mez}u_n.
\end{align*}
Recall (see \eqref{E-F}) that
\begin{align*}
E^{\lambda,\mu} &= \{u: e^{\lambda z\vert D_x \vert} u 
\in C_z^0([-h,0], \mathcal{H}^{\lambda h,\mu}(\xT^d)\},\\
 F^{\lambda,\mu} &=  \{u: e^{\lambda  z\vert D_x \vert} u \in L_z^2((-h,0), 
 \mathcal{H}^{\lambda h,\mu}(\xT^d)\}.
\end{align*}
In what follows,  we fix  $t \in [0,T]$  and write 
$\sigma(t) = \lambda(t)h, $ where  
$$\lambda(t)= \lambda - \frac{K \eps t}{h}  \leq \lambda <1.$$  
Recall the basic hypotheses  made in Proposition~\ref{Prop:C0}:  
\begin{equation}\label{hypot}
\begin{aligned}
  &\text{(i)}\quad \Vert   \eta_0  \Vert_{\mathcal{H}^{\lambda h,s+ \mez}}  +  \Vert   a(D_x)^\mez \psi_0  \Vert_{\mathcal{H}^{\lambda h,s}} 
+     \Vert   V_0  \Vert_{\mathcal{H}^{\lambda h,s}}  +  \Vert  B_0 \Vert_{\mathcal{H}^{\lambda h,s}} + N_s(b) = \eps \leq 1, \\
 &\text{(ii)}\quad M_s(T)\leq \overline{\eps}\leq 1.
\end{aligned}
\end{equation}

We begin by preliminaries. The goal of the two following sections  is to show that it is sufficient to have estimates on $V$ to control the terms   $B$ and $\psi$.

  \subsubsection{ Estimates of $\psi$ and $B$.}
  \begin{lemm}
     There exists  $C>0$ such that if $M_s(T)\leq \overline{\eps}\leq 1$,  then for all $t \in [0,T]$,
     \begin{align*}
    &\text{\rm{(i)}}\quad  \Vert B(t) \Vert_{\mathcal{H}^{\sigma,s+ \mez} }\leq  C \big(\Vert V(t) \Vert_{\mathcal{H}^{\sigma,s+ \mez} } + \Vert b(t) \Vert_{ H^{s + \mez}} + \Vert B(t) \Vert_{ H^1 } \big),\\
      &\text{\rm{(ii)}}\quad     \Vert B(t) \Vert_{\mathcal{H}^{\sigma,s} }\leq  C \big(\Vert V(t) \Vert_{\mathcal{H}^{\sigma,s} } + \Vert b(t) \Vert_{ H^{s}} + \Vert B(t) \Vert_{ H^1 } \big).
          \end{align*}
        \end{lemm}
\begin{proof} 
 {(i)}\quad   We fix  $t \in [0,T]$  and take  $\sigma \leq \lambda h$. We start from the identity:
        $$ - \text{div }V= G(\eta)(B, - \Delta_x \phi_h) =G(\eta)(B, 0)+G(\eta)(0, - \Delta_x \phi_h),$$
   where $\phi_h = \phi\arrowvert_{y=-h},$     proved in \eqref{GB=V}.
        We have  seen   that
$$
G(0)(B,0) = a(D_x)B = \vert D_x \vert \tanh(h \vert D_x\vert)B.
$$
We write
\begin{equation}\label{est-B0}
a(D_x)B = -\text{div }V +G(\eta)(0,  \Delta_x \phi_h) +(G(0) - G_0(\eta))(B,0) : = F.
\end{equation}
Let $\chi \in C_0^\infty(\xR^d)$ be such that $\chi(\xi) = 1 $ if $\vert \xi \vert \leq 1$ and  $\chi(\xi) = 0 $ if $\vert \xi \vert \geq 2$. 
We have
$$
\Vert \chi(D) B \Vert_ {\mathcal{H}^{\sigma, s+\mez}} \leq C \Vert B \Vert_{L^2}.
$$
On the other hand, 
it follows from \eqref{est-B0} that
$$
(1- \chi(\xi)) \widehat{B}(\xi) = \frac{1- \chi(\xi)}{  a(\xi)}\widehat{F}(\xi),
$$ 
so that
$$
\Vert (1-\chi(D))B \Vert_{\mathcal{H}^{\sigma, s+\mez}} \leq C \Vert F\Vert_{\mathcal{H}^{\sigma, s-\mez}}.
$$ 
First of all, we have $\Vert \text{div } V \Vert_{\mathcal{H}^{\sigma, s- \mez}}
\les \Vert V \Vert_{\mathcal{H}^{\sigma, s+\mez}}$. 
By using Theorem~\ref{G-lip} we have
$$
\Vert (G(0) - G_0(\eta))(B,0) \Vert_{\mathcal{H}^{\sigma, s- \mez}} \leq C M_s(T) \Vert B \Vert_{\mathcal{H}^{\sigma, s+\mez}}  \leq C  \overline{\eps}  \Vert B \Vert_{\mathcal{H}^{\sigma, s+\mez}}.
$$
Eventually from  Theorem~\ref{est-DN} and Corollary~\ref{est-phih} 
applied 
with $\mu = s-\frac{1}{2}$ we deduce that 
$$
\Vert G(\eta)(0,  \Delta_x \phi_h)\Vert_{\mathcal{H}^{\sigma, s- \mez}} \les    \Vert \Delta \phi_h \Vert_{H^{s -\frac{1}{2}}} \les \Vert b\Vert_{H^{s+\mez}}+ \Vert B \Vert_{H^1}.
$$
It follows that
$$
\Vert (1-\chi(D))B \Vert_{\mathcal{H}^{\sigma, s+\mez}} \leq C \big(\Vert V \Vert_{\mathcal{H}^{\sigma, s + \mez}} +  \Vert b \Vert_{H^{s+\mez}} + \Vert B \Vert_{H^1} +\overline{\eps} \Vert B \Vert_{\mathcal{H}^{\sigma, s +\mez}} \big).
$$
We see from the above estimate of $\chi(D)B$ that
$$
\Vert B \Vert_{\mathcal{H}^{\sigma, s+\mez}}    \leq C \big(\Vert V \Vert_{\mathcal{H}^{\sigma, s + \mez}} +  \Vert b \Vert_{H^{s+\mez}} + \Vert B \Vert_{H^1} +\overline{\eps} \Vert B \Vert_{\mathcal{H}^{\sigma, s+ \mez}} \big).
$$ 
Taking $\overline{\eps}$ sufficiently small we obtain the desired result. 

The proof of (ii) is identical. We have just to use Remark~\ref{G1-G2-s-1}, Remark~\ref{G-s-1} and Corollary~\ref{est-phih} 
with $\mu = s-1.$
\end{proof}  
 
\begin{lemm} 
   There exists $C>0$ such that for all $t\in (0,T)$ we have
$$\Vert a(D_x)^\mez \psi(t) \Vert_{\mathcal{H}^{\sigma,s} } \leq C  \big( \Vert V(t) \Vert_{\mathcal{H}^{\sigma,s-\mez} } + \Vert B(t)\Vert_{\mathcal{H}^{\sigma,s-\mez} } \Vert \eta(t) \Vert_{\mathcal{H}^{\sigma,s+ \mez} }\big).$$
\end{lemm}
\begin{proof} 
We write, for fixed  $t \in [0,T]$ and $\sigma \leq \lambda h,$
$$
\Vert a(D_x)^\mez \psi \Vert^2_{\mathcal{H}^{\sigma, s}} = \sum  \langle \xi \rangle^{2s}e^{2 \sigma \vert \xi \vert} \vert \xi \vert  \tanh(h \vert \xi \vert)\vert \widehat{\psi}(\xi)\vert^2.  
$$
Since $\vert \xi \vert \tanh(h \vert\xi \vert)\les \frac{\vert \xi \vert^2}{\langle \xi \rangle}$, we have
$$
\Vert a(D_x)^\mez \psi \Vert_{\mathcal{H}^{\sigma, s}}
\leq C \Vert \nabla_x \psi \Vert_{\mathcal{H}^{\sigma, s-\mez}}.
$$
Eventually we notice that    $\nabla_x \psi = V + B \nabla_x \eta$, so that, since $s -\mez >  \frac{d}{2}$, 
we obtain
$$
\Vert \nabla_x \psi(t) \Vert_{\mathcal{H}^{\sigma,s-\mez} } \les \Vert  V(t)\Vert_{\mathcal{H}^{\sigma,s-\mez} } +  \Vert B \Vert_{\mathcal{H}^{\sigma,s-\mez} }\Vert \eta \Vert_{\mathcal{H}^{\sigma,s+\mez} },
$$
which proves the  lemma.
\end{proof}
 
 \subsubsection{Low frequency estimates of  $\eta,V, \psi, B$.}
 \begin{lemm}\label{basseF} 
 Assume that $s>3+\frac{d}{2}$ and $M_s(T)\leq \overline{\eps}\leq 1.$ There exists a constant $C= C(d,g,\lambda,h,s)>0$ such that for every $t\in (0,T)$, 
 \begin{equation} \label{basseF1}
 \begin{split}
& \Vert \eta(t) \Vert_{H^{s-1}} +\Vert a(D_x)^\mez\psi(t) \Vert_{H^{s-1}}\\
\leq&\, C \Big(\Vert \eta_0\Vert_{H^{s-1}} +\Vert a(D_x)^\mez\psi_0 \Vert_{H^{s-1}}
 + T M_s(T)^2 + \Vert b \Vert^2_{L^2((0,T), H^{s-\mez})}\\
 & \qquad \qquad \qquad \qquad \qquad \qquad \qquad \qquad \qquad \qquad \qquad  
 + \Vert b  \Vert_{L^1((0,T),H^{s-\mez})}\Big),
 \end{split}
 \end{equation}
 \begin{equation}\label{basseF2}
 \begin{aligned}
& \Vert V(t) \Vert_{H^{s-\frac{3}{2}}} + \Vert B(t)\Vert_{H^{s-\frac{3}{2}} } \\
\leq&  C \Big( \Vert \eta_0\Vert_{H^{s-1}} +\Vert a(D_x)^\mez\psi_0 \Vert_{H^{s-1}}
 + T M_s(T)^2 + \Vert b \Vert^2_{L^2((0,T), H^{s-\mez})}\\ 
 & \qquad \qquad \qquad \qquad \qquad \qquad 
 + \Vert b \Vert_{L^\infty((0,T), H^{s-\frac{3}{2}})} + \Vert b  \Vert_{L^1((0,T),H^{s-\mez})}\Big).
 \end{aligned}
 \end{equation}
 \end{lemm}
 \begin{coro}
 We have
 \begin{equation}\label{B-H1}
  \Vert B(t)\Vert_{H^1} \les  TM_s(T)^2 + \eps.
  \end{equation}
 \end{coro}
 \begin{proof}
 Since $s-\frac{3}{2} \geq 1$, this follows from   \eqref{basseF2} and the hypotheses \eqref{hypot} made on the data and $b$.
 \end{proof}
\begin{proof}[Proof of Lemma~\ref{basseF}] 
We start from  system~\eqref{ww1'} which we write as
\begin{align*}
  &\partial_t \eta -Ja(D_x)\psi = f_1:= J(G(\eta)-G(0))(\psi,0) + G(\eta)(0,b),\\
  &\partial_t \psi + J g \eta= f_2: =J \Bigl( \mez \vert \nabla_x \psi \vert^2 -  \frac{(G(\eta) (\psi,b)+ \nabla_x \eta \cdot \nabla_x \psi)^2}{2(1+ \vert \nabla_x \eta \vert^2)}\Bigr),
  \end{align*}
where $ a(D_x)\psi = \vert D_x\vert\tanh(h\vert D_x\vert)\psi.$
   
We set
$$
u= g \eta +i\big(g \,a(D_x)\big)^\mez \psi.
$$
Then $u$ is a  solution   of the equation
\begin{equation*}
\partial_t u +iJ(ga(D_x))^\mez u = g f_1+ i (ga(D_x))^\mez f_2.
\end{equation*}
Computing  $\frac{d}{dt} \Vert u(t) \Vert^2_{H^{s-1}}$ on the interval  $I= [0,T]$, we obtain the inequality
\begin{equation}\label{est-u1}
\Vert u(t) \Vert_{  H^{s-1} } \leq C \Big( \Vert u_0 \Vert_{H^{s-1}} + \int_0^t \Vert f_1(t') \Vert_{H^{s-1}}\, dt' +  \int_0^t \Vert f_2(t') \Vert_{H^{s-\mez}}\, dt'\Big).
\end{equation}
Let us estimate $f_1$. First of all, Corollary~\ref{estG0-lip} gives
\begin{align*}
 \Vert (G(\eta(t)) - G(0))(\psi(t),0) \Vert _{H^{s-1}}  
  \leq &\, \Vert (G(\eta(t)) - G(0))(\psi(t),0) \Vert _{\mathcal{H}^{\sigma(t), s-\mez}} \\
   \leq &\,  CM_s(T)^2.
\end{align*}
 It follows  that
 \begin{equation}\label{bf1}
 \int_0^t \Vert  (G(\eta(t))-G(0)  (\psi(t'),0) \Vert _{H^{s-1}}\, dt' \leq CTM_s(T)^2.
 \end{equation}
Now, from Theorem ~\ref{est-DN}   we get
\begin{equation}\label{bf2}
\begin{aligned}
 \int_0^t \Vert G(\eta(t'))(0,b(t')) \Vert_{ H^{s-1}}\, dt'  &\leq \int_0^t \Vert  G(\eta(t')) (0, b(t')) \Vert_{  \mathcal{H}^{\sigma(t), s-\mez}} \, dt'\\
 &\leq  C \Vert b  \Vert_{L^1(I,H^{s-\mez})}.
 \end{aligned}
 \end{equation}
It follows from \eqref{bf1}, \eqref{bf2} that
 \begin{equation}\label{est-psi1}
 \int_0^t \Vert f_1(t')\Vert_{H^{s-1}}\, dt' \leq  C\big( TM_s(T)^2 + \Vert b  \Vert_{L^1(I,H^{s-\mez})}\big).
 \end{equation}
Let us estimate $f_2$. Using  Remark~\ref{nabla-G}, we can write
\begin{align*}
  \Vert \vert \nabla_x \psi(t) \vert^2 \Vert_{H^{s - \mez}} \leq &\, C \Vert \nabla_x \psi(t) \Vert^2_{H^{s- \mez}} \leq C \Vert \nabla_x \psi(t) \Vert^2_{\mathcal{H}^{\sigma(t),s- \mez}} \\
  \leq &\, C' \Vert a(D_x)^\mez \psi(t) \Vert^2_{\mathcal{H}^{\sigma(t), s}}.
  \end{align*}
 Therefore, we get 
  $$  \Vert \vert \nabla_x \psi \vert^2 \Vert_{L^\infty(I,H^{s -\mez})} \leq C M_s(T)^2.$$
 It follows  that
  \begin{equation}\label{bf3}
  \int_0^T  \Vert \vert \nabla_x \psi(t) \vert^2 \Vert_{H^{s-\mez}}\, dt \leq CTM_s(T)^2.  
  \end{equation}
  Now we can write 
  \begin{align*}
   \widetilde{f}_2 &= \frac{(G(\eta) (\psi,b)+ \nabla_x \eta \cdot \nabla_x \psi)^2}{2(1+ \vert \nabla_x \eta \vert^2)}:= U  \big(1-g(\vert \nabla_x \eta \vert^2)\big),\\
  U & =   \mez (G(\eta) (\psi,b)+ \nabla_x \eta \cdot \nabla_x \psi)^2, \quad g(t)= \frac{t}{1+t}.
  \end{align*}
 Then 
 $$ \Vert \widetilde{f}_2   \Vert_{H^{s-\mez}}\leq C\Vert U\Vert_{H^{s-\mez}}\big(1+  \Vert g(\vert \nabla \eta\vert^2)\Vert_{H^{s-\mez}}\big).$$
By the product laws in the usual Sobolev spaces we can write
\begin{align*}
   \Vert U  \Vert_{H^{s-\mez}} &\les \Vert  G(\eta) (\psi,b)\Vert^2_{H^{s-\mez}} +  \Vert  \eta \Vert^2_{H^{s+\frac{1}{2}}} \Vert a(D_x)^\mez \psi \Vert^2_{H^{s}} \\
   & \les \Vert  G(\eta) (\psi,b)\Vert^2_{H^{s-\mez}}+ M_s(T)^4,
   \end{align*}
\[    \Vert g(\vert\nabla\eta\vert^2)\Vert_{H^{s-\mez}} \les \Vert \eta \Vert^2_{H^{s+\mez}}\les M_s(T)^2.
  \]
 Using Theorem~\ref{est-DN}, we have
 \begin{align*}
  & \Vert  G(\eta) (\psi,b)\Vert^2_{H^{s-\mez}} \\
  \les &\, \big(\Vert a(D_x)^\mez\psi \Vert^2_{\mathcal{H}^{\sigma, s }} + \Vert b\Vert^2_{H^{s-\mez}} + \Vert \eta\Vert^2_{H^{s+\mez}}(\Vert a(D_x)^\mez\psi \Vert^2_{\mathcal{H}^{\sigma, s-\mez }} + \Vert b\Vert^2_{H^{s-\frac{3}{2}}})\big)\\
   \les&\, M_s(T)^2 + \Vert b\Vert^2_{  H^{s-\mez} }.
  \end{align*}
   Therefore, 
 $$ \Vert \widetilde{f}_2(t)   \Vert_{H^{s-\mez}} \leq C \big( M_s(T)^2 + \Vert b(t) \Vert^2_{H^{s-\mez}}\big).$$
It follows that
 \begin{equation}\label{est-psi3}
 \int_0^t \Vert \widetilde{f}_2(t)\Vert_{H^{s-\mez}} \, dt \les    TM_s(T)^2 +     \Vert b\Vert^2_{L^2(I, H^{s-\mez})}.
 \end{equation}
Using \eqref{est-u1}, \eqref{est-psi1}, \eqref{bf3}, and \eqref{est-psi3}, we obtain that 
\begin{equation*}
\Vert u(t) \Vert_{  H^{s-1} } \les \Vert u_0 \Vert_{H^{s-1}} +   TM_s(T)^2 +   \Vert b\Vert^2_{L^2(\xR, H^{s-\mez})} + \Vert b  \Vert_{L^1(I,H^{s-\mez})}.
\end{equation*}
Using \eqref{est-u1} and the  definition of $u$, we get
\begin{equation} \label{est-etapsi}
\begin{split}
& \Vert \eta(t)  \Vert_{ H^{s-1} } + \Vert a(D_x)^\mez \psi(t) \Vert_{H^{s-1}} \\
\les&\,  \Vert \eta_0  \Vert_{  H^{s-1} }  + \Vert a(D_x)^\mez \psi_0 \Vert_{  H^{s-1} }
+   TM_s(T)^2 +  \Vert b\Vert_{L^1(\xR, H^{s-\mez})}+  \Vert b\Vert^2_{L^2(\xR, H^{s-\mez})}.
\end{split}
\end{equation}
This completes the proof of~\eqref{basseF1}.

To prove~\eqref{basseF2}, we  estimate  $B$  with fixed $t$.  By definition we can write
\begin{align*}
&B = \frac{G(\eta)(\psi,b) + \nabla_x \psi \cdot \nabla_x \eta}{1+ \vert \nabla_x \eta \vert^2}= W\big(1-g(\vert \nabla_x \eta \vert^2\big),\\
&W= G(\eta)(\psi,b) + \nabla_x \psi \cdot \nabla_x \eta, \quad g(t)= \frac{t}{1+t}.
\end{align*}
Now, we use as before the product laws in the usual Sobolev spaces. Since $s>3+d/2$, it follows 
from Remark~\ref{G-s-1} applied with $\delta=1 $ and $\lambda=0$, that
$$
\Vert G(\eta)(\psi,b)(t)\Vert_{H^{s-\frac{3}{2}}} \les \Vert a(D_x)^\mez \psi(t)\Vert_{H^{s-1}} + \Vert b \Vert_{H^{s-\frac{3}{2}}}.
$$
Therefore,
$$
\Vert W(t)\Vert_{H^{s-\frac{3}{2}}} \les \Vert a(D_x)^\mez \psi(t)\Vert_{H^{s-1}} 
+ \Vert b (t)\Vert_{H^{s-\frac{3}{2}}}+ \Vert a(D_x)^\mez \psi(t)\Vert_{H^{s-1}}\Vert \eta(t)\Vert_{H^{s-\mez}}.
$$
Moreover,
$$
\Vert g(\vert \nabla \eta\vert^2)(t)\Vert_{H^{s-\frac{3}{2}}}
\les \Vert \eta(t)\Vert_{H^{s-\mez}}\les \overline{\eps}\leq 1.
$$
It follows that
$$
\Vert B(t) \Vert_{H^{s-\frac{3}{2}}}
\les \Vert a(D_x)^\mez \psi(t)\Vert_{H^{s-1}} + \Vert b(t) \Vert_{H^{s-\frac{3}{2}}}.
$$
Using \eqref{est-etapsi}, we deduce that 
\begin{equation}\label{estbf-B}
\begin{aligned}
\Vert B\Vert_{L^\infty(I, H^{s-\frac{3}{2}}) }  \les &\big( \Vert \eta_0  \Vert_{  H^{s-1} }   + \Vert a(D_x)^\mez \psi_0 \Vert_{  H^{s-1} } 
  +   TM_s(T)^2\\
 &+   \Vert b\Vert^2_{L^2(\xR, H^{s-\mez})}+ \Vert b\Vert_{L^1(\xR, H^{s-\mez})}+ \Vert b\Vert_{L^\infty(\xR, H^{s-\frac{3}{2}})}\big)
 \end{aligned}
\end{equation}
Now we have $V = \nabla_x \psi - B \nabla_x \eta$. It follows that
$$\Vert V \Vert_{L^\infty(I, H^{s-\frac{3}{2}})}\les  \Vert a(D_x)^\mez\psi \Vert_{L^\infty(I, H^{s-1})} + \Vert B \Vert_{L^\infty(I, H^{s-\frac{3}{2}})}\Vert \eta \Vert_{L^\infty(I, H^{s-\mez})}.$$
Since $\Vert \eta \Vert_{L^\infty(I, H^{s-\mez})}\leq \overline{\eps}\leq 1$, using \eqref{est-etapsi} and \eqref{estbf-B} we obtain
 \begin{equation*}
\begin{aligned}
\Vert V\Vert_{ L^\infty(I, H^{s-\frac{3}{2}}) } \les  &\big( \Vert \eta_0  \Vert_{  H^{s-1} }   + \Vert a(D_x)^\mez \psi_0 \Vert_{  H^{s-1} } 
  +   TM_s(T)^2\\
 & +   \Vert b\Vert^2_{L^2(\xR, H^{s-\mez})}+\Vert b\Vert_{L^1(\xR, H^{s-\mez})}+ \Vert b\Vert_{L^\infty(\xR, H^{s-\frac{3}{2}})}\big).
 \end{aligned}
\end{equation*}
This completes the proof of Lemma \ref{basseF} .
\end{proof}
  \begin{coro}\label{est-B-V2}
  Under the assumptions \eqref{hypot} we have for every $t \in [0,T]$,
  \begin{align*}
  &\Vert B(t)\Vert_{\mathcal{H}^{\sigma, s + \mu}} \les \Vert V(t)\Vert_{\mathcal{H}^{\sigma, s + \mu}} + T M_s (T)^2+ \eps,\quad   \mu = 0, \mu = \mez,  \\
  &\Vert a(D_x)^\mez \psi(t) \Vert_{\mathcal{H}^{\sigma,s} } \les \Vert V(t) \Vert_{\mathcal{H}^{\sigma,s-\mez} } + \Vert B(t)\Vert_{\mathcal{H}^{\sigma,s-\mez} } M_s(T).
  \end{align*}
  \end{coro}
\subsubsection{Proof of Proposition~\ref{Prop:C0} (continued).} 

Recall that  the function  $ u = \sqrt{g}\zeta +i \big(a(D_x)\big)^\mez V$ 
that was introduced in \eqref{def-un}  satisfies (see \eqref{eq-fin}) the equation
\begin{equation}\label{eq-fin2}
 \partial_t u + iJ \big(g a(D_x)\big)^\mez u = \sqrt{g }R^{(1)} + ia(D_x) ^\mez\big[(\partial_y P \arrowvert_{y = \eta} + g)\zeta + R^{(0)}\big],
 \end{equation}
where
\begin{equation}\label{R1=}
\left\{ 
\begin{aligned}
R^{(1)} &= R_1^{(1)} + R_2^{(1)},\\
R_1^{(1)}&= J \big[G(\eta)(V, 0) - G(0)(V,0)  +  G(\eta)(0, \nabla_x b)\big],\\
R_2^{(1)}&= J \big[-(\cnx V)\zeta -(V\cdot \nabla_x) \zeta\big]\\
 R^{(0)} &= R^{(0)}_1 + R^{(0)}_2 + R^{(0)}_3 + R^{(0)}_4,\\
  R^{(0)}_1 &= -J(V \cdot\nabla_x )V,\\
   R^{(0)}_2 &= -(\partial_y \nabla_x \phi\arrowvert_{y = \eta})(J -1)( V \cdot \zeta-B), \\
   R^{(0)}_3& =  (J+1)B \nabla_x B -\nabla_x B J B   
 - B J \nabla_x B,\\   
   R^{(0)}_4 &=  -\nabla_x J (BV\cdot \zeta)  
      +\nabla_x(B JV\cdot \zeta).
 \end{aligned}\right. \end{equation}
We shall now   estimate   the remainder terms. Recall that $$\sigma(t) = \lambda h - K \eps t = \lambda(t) h$$ and
$$ U_s(t) = e^{\sigma(t) \langle D_x\rangle } \langle D_x\rangle ^{s - \mez} u(t) .$$
\begin{prop}\label{est-restes}
We have
\begin{align*}
&\text{\rm{(i)}}\quad \Vert \sqrt{g}R_1^{(1)}   \Vert_{\mathcal{H}^{\sigma, s-\mez}}   \les  M_s(T) \Vert U_s \Vert_{H^\mez} + \Vert b \Vert_{H^{s+\mez}},\\
&\text{\rm{(ii)}}\quad\Vert \sqrt{g}R_2^{(1)}   \Vert_{\mathcal{H}^{\sigma, s-1}}\les   M_s(T) \Vert U_s \Vert_{H^\mez},\\
& \text{\rm{(iii)}}\quad \Vert ia(D_x) ^\mez\big[(\partial_y P \arrowvert_{y = \eta} + g)\zeta   \Vert_{\mathcal{H}^{\sigma , s-1}} \\
&\qquad  \les  M_s(T)^2 + T^2 M_s(T)^5 + \eps M_s(T)+ M^2_s(T) \| U_s\|_{H^{\mez}},\\
&\text{\rm{(iv)}}\quad\Vert ia(D_x)^\mez R^{(0)}\Vert_{\mathcal{H}^{\sigma, s-1}}\les
  M_s(T) \Vert U_s \Vert_{H^\mez} + M_s(T)^2+ T M_s(T)^3 + \eps M_s(T).
\end{align*}
\end{prop}
\begin{proof}
Proof of (i) and (ii): Estimate of $\sqrt{g}R^{(1)}$. 
  
We shall estimate each term separately. 
All the estimates will be with fixed $t$, 
with constants independent of $t$. 
Therefore $t$ will be omitted in what follows.

Moreover we notice that the family of operators  $J_n$ is uniformly bounded with respect to $n$ on any space $\mathcal{H}^{\sigma,s}.$ Therefore we can skip it in what follows.
 
First of all, we have
\begin{equation}\label{est-F_11}
\Vert  G(\eta)(V,0)-G(0)(V,0)\Vert_{\mathcal{H}^{\sigma, s-\mez}}\les M_s(T)   \Vert   U_s\Vert_{H^\mez}.
\end{equation}
To see this, 
we use Theorem~\ref{G-lip} with $\eta_1 = \eta, \eta_2 =0, \psi_1 = \psi_2 = V, 
b=0$. 
Then, with the notations in \eqref{lambdaj}, we have
$$
\Vert  G(\eta)(V,0)-G(0)(V,0)\Vert_{\mathcal{H}^{\sigma, s-\mez}} \les \lambda_1\Vert \eta\Vert_{\mathcal{H}^{\sigma, s}} + \lambda_2\Vert \eta\Vert_{\mathcal{H}^{\sigma, s+\mez}}\leq (\lambda_1 + \lambda_2)M_s(T).
$$
Now, by definition of $U_s$, we have
\begin{align*}
 \lambda_1=&\, \Vert \eta\Vert_{\mathcal{H}^{\sigma, s+\mez}}\Vert a(D_x)^\mez V\Vert_{\mathcal{H}^{\sigma(t), s-\mez}} + \Vert a(D_x)^\mez V\Vert_{\mathcal{H}^{\sigma, s}} \\
 \les&\, M_s(T)\Vert U_s\Vert_{L^2} + \Vert U_s\Vert_{H^\mez} ,\\
 \lambda_2=&\, \Vert a(D_x)^\mez V\Vert_{\mathcal{H}^{\sigma, s-\mez}} \les \Vert U_s\Vert_{L^2},
\end{align*}
which implies \eqref{est-F_11}. 
Moreover, according to Theorem~\ref{est-DN}, we have
  \begin{equation*}
  \Vert G(\eta)(0, \nabla_x b)\Vert_{\mathcal{H}^{\sigma, s-\mez}} \les \Vert b\Vert_{H^{s+\mez}} + M_s(T)\Vert b\Vert_{H^{s-\mez}}.
  \end{equation*}
Since $M_s(T)\leq \overline{\eps}\leq 1$, we obtain
\begin{equation*}
\Vert \sqrt{g}R_1^{(1)}   \Vert_{\mathcal{H}^{\sigma, s-\mez}} \les  M_s(T) \Vert U_s \Vert_{H^\mez} + \Vert b \Vert_{H^{s+\mez}}.
\end{equation*}
 This proves (i).

Now, since $\zeta=\nabla_x\eta$, it follows from the product rule in Proposition~\ref{uv} that
\begin{equation*}
\Vert (V\cdot\nabla_x)\zeta\Vert_{\mathcal{H}^{\sigma, s-1}}\les \Vert  V  \Vert_{\mathcal{H}^{\sigma, s-1}}\Vert  \zeta\Vert_{\mathcal{H}^{\sigma, s}}\les M_s(T) \Vert U_s\Vert_{H^\mez},
\end{equation*}
and
\begin{equation*}
\Vert (\text{div}V) \zeta\Vert_{\mathcal{H}^{\sigma, s-1}}\les \Vert  V  \Vert_{\mathcal{H}^{\sigma, s}}\Vert  \zeta\Vert_{\mathcal{H}^{\sigma, s-1}}\les M_s(T) \Vert U_s\Vert_{L^2}.
\end{equation*}
It follows that
\begin{equation*}
\Vert \sqrt{g}R_2^{(1)}   \Vert_{\mathcal{H}^{\sigma, s-1}}\les  M_s(T) \Vert U_s \Vert_{H^\mez},
\end{equation*}
which proves (ii)

Proof of (iii):  Estimate of\,: 
\begin{equation*}
(1)= i  a(D_x)^{\frac{1}{2}}  \big[(\partial_y P \arrowvert_{y = \eta} + g)\zeta)\big]. 
\end{equation*}

For simplicity we shall set
\begin{equation*}
\mathfrak{a}= -\partial_y P\arrowvert_{y = \eta}.
\end{equation*}
We have
 \begin{equation*}
\Vert (1)\Vert_{\mathcal{H}^{\sigma, s-1}} \les  \Vert  \mathfrak{a} -g   \Vert_{\mathcal{H}^{\sigma, s- \mez}}\Vert \eta\Vert_{\mathcal{H}^{\sigma, s+ \mez}} 
  \les  M_s(T) \Vert  \mathfrak{a} -g   \Vert_{\mathcal{H}^{\sigma, s- \mez}}.
\end{equation*}
   We claim that, for fixed $t,$  
   \begin{equation}\label{est-F22}
 \Vert  \mathfrak{a} -g   \Vert_{\mathcal{H}^{\sigma, s- \mez}}   \les    M_s(T)+ T^2M_s(T)^4+ \eps + M_s(T) \| U_s\|_{H^\mez}
 \end{equation}
 It will follow that
 \begin{multline*}
   \Vert  i a(D_x)^{\frac{1}{2}}  \big[(\partial_y P \arrowvert_{y = \eta} + g)\zeta_n)\big]   \Vert_{\mathcal{H}^{\sigma , s-1}}  \\
   \les  M_s(T)^2 + T^2 M_s(T)^5 + \eps M_s(T) +  M^2_s(T) \| U_s\|_{H^\mez}. 
     \end{multline*}
 
  We prove \eqref{est-F22}. 
 Since the  function $\phi$ is the   solution  of the  problem
\begin{equation*}
 \Delta_{x,y} \phi= 0, \quad \phi \arrowvert_\Sigma = \psi, \quad \partial_y \phi \arrowvert_{y=-h} = b,
 \end{equation*}
we have
 \begin{equation}\label{PsurG}
 \partial_y(P- gy)\arrowvert_{y=-h} = - \partial_t b +  \nabla_x b \cdot (\nabla_x \phi_h) - b (\Delta_x \phi_h) := \mathcal{P}_1,
 \end{equation}
where  $\phi_h = \phi\arrowvert_{y=-h}.$
Let us recall that  according to   \eqref{def-Qn}, \eqref{press3}, \eqref{eq-Qn} and \eqref{PsurG}, the  pressure $P$ satisfies
\begin{equation}\label{pression}
\begin{aligned}
  \Delta_{x,y} (P -gy)&= - \la \nabla^2_{x,y} \phi \ra^2,\\  (P- gy) \arrowvert_{y= \eta(x)} &= g(J-2)\eta + \mez (J-1)(\vert V \vert^2 + B^2)+ [B,J](B - V \cdot \zeta):= \mathcal{P}_0, \\ 
  \partial_yP- g \arrowvert_{y=-h} &= - \partial_t b +  \nabla_x b \cdot (\nabla_x \phi_h) - b (\Delta_x \phi_h) := \mathcal{P}_1
 \end{aligned}
 \end{equation}
 We are going to work in the $(x,z)$ variables  defined previously. Let us recall some  notations (see section~\ref{defiCOV}): 
\begin{align*}
 \rho(x,z)&= \frac{1}{h}(z+h)e^{z \vert D_x \vert}(\eta(x))+z,\quad
  \Lambda_1 = \frac{1}{\partial_z \rho} \partial_z, \quad \Lambda_2 = \nabla_x - \frac{\nabla_x \rho}{\partial_z \rho}\partial_z,\\ 
    \widetilde{\phi}(x,z) &= \phi(x, \rho(x,z)).
\end{align*}
Set
 $$\phi_h(x):=\phi(x,-h) = \widetilde{\phi}(x, -h):= \widetilde{\phi}_h(x),$$ 
and
$$\mathcal{P}(x,z) = P(x, \rho(x,z)) - g \rho(x,z).$$
Since $\Lambda_1 \rho=1$, we have 
\begin{equation*}
  \mathfrak{a} -g = - (\Lambda_1 \mathcal{P})\arrowvert_{z=0}.
  \end{equation*}
 It follows from~\eqref{pression} that $\mathcal{P}$ is a solution of the problem
 \begin{equation}\label{eq-mathcalP}
  (\Lambda^2_1 + \Lambda^2_2) \mathcal{P} = G, \quad  \mathcal{P}\arrowvert_{z=0}  =   \mathcal{P}_0,\quad 
  \partial_z \mathcal{P}\arrowvert_{z=-h}  = -(\partial_z \rho\arrowvert_{z=-h})  \mathcal{P}_1,
  \end{equation}
 where   $G$ is a  linear combination   of $(\Lambda_j \Lambda_k \widetilde{\phi})^2, 1 \leq j,k \leq 2.$
 
From Theorem~\ref{regul-fin'}    we have
 \begin{equation}\label{est-P1}
 \begin{split}
 & \Vert e^{\lambda(z+h) \vert D_x \vert} \nabla_{x,z} \mathcal{P}\Vert_{C^0([-h,0], H^{s-\mez})}\\
  \les&\, 
  \big(1+  \Vert \eta\Vert_{\mathcal{H}^{\sigma,s+\mez}}\big)\big(\Vert G \Vert_{F^{\lambda, s-1}} + \Vert a(D_x)^\mez \mathcal{P}_0 \Vert_{\mathcal{H}^{\sigma, s}} +\Vert  (\partial_z \rho)\arrowvert_{z=-h}\mathcal{P}_1\Vert_{H^{s-\mez}}\big).
\end{split} 
  \end{equation}
 Estimate of   $G$ in $F^{\lambda,s-1}$. According to Lemma~\ref{uvF}, we have, since $s > 2+ \frac{d}{2}$, 
$$
  \Vert (\Lambda_j \Lambda_k \widetilde{\phi})^2 \Vert_{F^{\lambda h, s-1}} \les \Vert \Lambda_j \Lambda_k \widetilde{\phi}\Vert_{E^{\lambda h, s-\frac{3}{2}}}\Vert \Lambda_j \Lambda_k \widetilde{\phi}\Vert_{F^{\lambda h, s-1}}.
 $$
 We shall prove the following estimate.
\begin{equation}\label{est-lambda2}
\sum_{j,k=1}^2\Vert \Lambda_j \Lambda_k \widetilde{\phi} \Vert_{E^{\lambda h, s-\frac{3}{2}  }}+ \sum_{j,k=1}^2\Vert \Lambda_j \Lambda_k \widetilde{\phi} \Vert_{F^{\lambda h, s-1}} \les   \Vert a(D_x)^\mez V \Vert_{\mathcal{H}^{\lambda h, s-1}} + \Vert b \Vert_{H^{s- \mez}}.\\
  \end{equation}
In fact, since $\Lambda_1$ and  $\Lambda_2$ commute , we have $(\Lambda^2_1 + \Lambda^2_2) \Lambda_2 \widetilde{\phi} = 0$. On the other hand, by definition $\Lambda_2 \widetilde{\phi}\arrowvert_{z=0} = V$. Now,
$$\Lambda_1(\Lambda_2 \widetilde{\phi})\arrowvert_{z=-h} = \Lambda_2(\Lambda_1 \widetilde{\phi})\arrowvert_{z=-h} = \nabla_x b.$$
Indeed, the right hand side is the image by our diffeomorphism of the quantity $\nabla_x(\partial_y \phi)\arrowvert_{y=-h}= \nabla_x[(\partial_y \phi)\arrowvert_{y=-h}] = \nabla_x b$. Summing up $U = \Lambda_2 \widetilde{\phi}$ is a  solution of the  problem
\begin{equation}
\left\{
\begin{gathered}
 (\Lambda^2_1 + \Lambda^2_2) U = 0, \quad U\arrowvert_{z=0} = V, \\
  \partial_z U \arrowvert_{z=-h} = (\partial_z \rho)\arrowvert_{z=-h} \nabla_x b = (1+ \frac{1}{h}e^{-h\vert D_x\vert}\eta) \nabla_xb.
  \end{gathered}
  \right.
  \end{equation}
    Using Corollary~\ref{regul12}  we obtain when $M_s(T)\leq \overline{\eps},$
  \begin{equation}\label{est-nablaU-F}
 \Vert \nabla_{x,z} U\Vert_{F^{\lambda, s-1}} \les  \Vert a(D_x)^\mez V \Vert_{\mathcal{H}^{\lambda h, s-1}} + \Vert b \Vert_{H^{s-\mez}}.
 \end{equation}
 Recall that $\Lambda_1 = \frac{1}{\partial_z \rho}\partial_z$ and $ \Lambda_2 = \nabla_x -  \frac{\nabla_x \rho}{\partial_z \rho}\partial_z.$ Since $\nabla_x \rho$ and $\partial_z \rho -1$ belong to $ \mathcal{E}_1$ and since  $\Lambda_1^2 \widetilde{\phi} = - \Lambda_2^2 \widetilde{\phi}$, we obtain
 $$ \sum_{j,k=1}^2\Vert \Lambda_j \Lambda_k \widetilde{\phi} \Vert_{F^{\lambda h, s-1}} \les   \Vert a(D_x)^\mez V \Vert_{\mathcal{H}^{\lambda h, s-1}} + \Vert b \Vert_{H^{s- \mez}}.$$
Now, using Lemma~\ref{L316}, we can write
 $$  \Vert \nabla_{x,z} U\Vert_{E^{\lambda, s-\frac{3}{2}}} \leq C  \Vert \nabla_{x,z} U\Vert_{F^{\lambda, s-1}}.$$ 
 Using same argument as above and \eqref{est-nablaU-F}, we deduce that
   $$ \sum_{j,k=1}^2 \Vert \Lambda_j \Lambda_k \widetilde{\phi} \Vert_{E^{\lambda h, s-\frac{3}{2}}} \les \Vert a(D_x)^\mez V \Vert_{\mathcal{H}^{\lambda h, s-1}} + \Vert b \Vert_{H^{s- \mez}},$$
   which proves  \eqref{est-lambda2}. Therefore, with the notation in \eqref{eq-mathcalP} we have
   \begin{equation*}
   \Vert G \Vert_{F^{\lambda , s-1}} \les\Vert a(D_x)^\mez V \Vert^2_{\mathcal{H}^{\lambda h, s-1}} + \Vert b \Vert^2_{H^{s- \mez}}. 
\end{equation*}
 Estimate of $I=\Vert a(D_x)^\mez \mathcal{P}_0\Vert_{\mathcal{H}^{\sigma, s}}.$ 
 Recall that
  $$\mathcal{P}_0= g(J-2)\eta + \mez (J-1)(\vert V \vert^2 + B^2)+ [B,J](B - V \cdot \zeta).$$ 
 According to Proposition~\ref{uv} (iii) (with $t = s+\mez, s_0=s$) and to Proposition~\ref{commutat} applied to $B(D) = J_n(D)= \chi(\frac {D}{n})$ and $a= B$, with $\nu = 0$ and $s+\mez$ instead of $s$, we obtain (using that $\xi \mapsto \chi( \frac \xi n) \in S^0_{1,0}$ with semi-norms bounded with respect to the $n$ parameter)
  \begin{align*}
   I&\les\Vert \eta\Vert_{\mathcal{H}^{\sigma, s+\mez}} +  \Vert V\Vert_{\mathcal{H}^{\sigma, s}}\Vert V\Vert_{\mathcal{H}^{\sigma, s+\mez}}\\ 
   &\qquad +  \Vert B\Vert_{\mathcal{H}^{\sigma, s}}\Vert B\Vert_{\mathcal{H}^{\sigma, s+\mez}} + \Vert B\Vert_{\mathcal{H}^{\sigma,s+\mez}} \Vert  B- V\cdot\zeta \Vert_{\mathcal{H}^{\sigma,s-\mez}} ,\\
    &\les \Vert \eta\Vert_{\mathcal{H}^{\sigma, s+\mez}}+  \Vert V\Vert_{\mathcal{H}^{\sigma, s}}\Vert V\Vert_{\mathcal{H}^{\sigma, s+\mez}}\\
    &\qquad +  \Vert B\Vert_{\mathcal{H}^{\sigma, s}}\Vert B\Vert_{\mathcal{H}^{\sigma, s+\mez}} 
     +  \Vert B\Vert_{\mathcal{H}^{\sigma, s+\mez}}\Vert \eta\Vert_{\mathcal{H}^{\sigma, s+\mez}} \Vert V\Vert_{\mathcal{H}^{\sigma, s-\mez}}.
 \end{align*}
 Here $S^0_{1,0}$ is the  class of symbols $p$   for which the seminorms
  $$\sup_{(x,\xi)\in \xR^d \times \xR^d} \langle \xi \rangle^{\vert \alpha \vert}\vert\partial_x^\beta \partial_{\xi}^\alpha p(\xi)\vert  $$
 are finite for every $(\alpha, \beta) \in \xN^d\times\xN^d.$
 
 Using Corollary~\ref{est-B-V2} and the fact that $$\Vert \eta\Vert_{\mathcal{H}^{\sigma, s+\mez}}\leq M_s(T)\leq \overline{\eps}\leq 1,$$
 we obtain eventually
 \begin{equation*}
\Vert a(D_x)^\mez \mathcal{P}_0\Vert_{\mathcal{H}^{\sigma, s}}  \les      M_s(T) \Vert U_s\Vert_{H^\mez}
 +  M_s(T) +TM_s(T)^3 .
\end{equation*}
  Estimate of $II= \Vert(\partial_z\rho)\arrowvert_{z=-h}\mathcal{P}_1\Vert_{H^{s-\mez}}.$
  
  We have $$ \mathcal{P}_1 =  \partial_t b -  \nabla_x b \cdot (\nabla_x \phi_h) + b (\Delta_x \phi_h)$$ in $H^{s-\mez}$. Therefore,
  $$\Vert \mathcal{P}_1 \Vert_{H^{s-\mez}} \les \Vert \partial_t b \Vert_{H^{s-\mez}} +   \Vert b \Vert_{H^{s+ \mez}}(\Vert \nabla_x \phi_h\Vert_{H^{s- \mez}} +\Vert \Delta_x \phi_h \Vert_{H^{s- \mez}}).$$
Now, we have 
$$ \Vert \nabla_x  {\phi}_h \Vert_{H^{s-\mez}} \les C \Vert \nabla_x {\phi}_h \Vert_{H^{s-1}} + \Vert \Delta_x  \phi_h \Vert_{H^{s - \frac{3}{2}}},$$
so that
  $$\Vert \mathcal{P}_1 \Vert_{H^{s-\mez}} \les \Vert \partial_t b \Vert_{H^{s-\mez}} +   \Vert b \Vert_{H^{s+ \mez}}(\Vert \nabla_x \phi_h\Vert_{H^{s- 1}} +\Vert \Delta_x \phi_h \Vert_{H^{s- \mez}}).$$
 By Corollary~\ref{est-phih}   we have
$$\Vert \Delta_x \phi_h \Vert_{H^{s-\mez} } \leq C \big( \Vert b  \Vert_{H^{s+ \mez} }  + \Vert B \Vert_{H^1} \big).$$
 Theorem~\ref{regul-fin'} and the fact that
 $$\mathcal{X}^{\lambda h, s-1} \subset\{ u: e^{\lambda (z+h) \vert D_x \vert} u \in C^0([-h,0], H^{s-1})\}$$ imply that
$$
 \Vert \nabla_x {\phi}_h \Vert_{H^{s-1}} \les  \Vert a(D_x)^\mez \psi \Vert_{  \mathcal{H}^{\lambda h, s-\mez}} + \Vert b \Vert_{H^{s-1}}. $$
Therefore,
$$\Vert \mathcal{P}_1\Vert_{H^{s-\mez}}\les \Vert \partial_t b \Vert_{H^{s-\mez}} +   \Vert b \Vert_{H^{s+ \mez}} \big( \Vert a(D_x)^\mez \psi \Vert_{  \mathcal{H}^{\lambda h, s-\mez}} +   \Vert b \Vert_{H^{s+ \mez}} + \Vert B \Vert_{H^1}\big)$$
and  eventually,
\begin{equation}\label{est-P4}
   \Vert \mathcal{P}_1 \Vert_{H^{s-\mez}} \les
   \Vert \partial_t b \Vert_{H^{s-\mez}} +   \Vert b \Vert_{H^{s+ \mez}} \big( M_s(T) +   \Vert b \Vert_{H^{s+ \mez}} + TM_s(T)^2 + \eps\big).   
   \end{equation}
  Then \eqref{est-F22}  follows from   \eqref{est-P1} to \eqref{est-P4}, and (iii) is proved.

(iv)\quad Estimate of $(2) = i a(D_x)^\mez R^{(0)}$ 
 Recall for convenience that
 \begin{equation}\label{R0=}
 \left\{
 \begin{aligned}
 &R^{(0)} = R^{(0)}_1 + R^{(0)}_2 + R^{(0)}_3 + R^{(0)}_4,\\
 &R^{(0)}_1= -J(V \cdot\nabla_x )V,\\
 & R^{(0)}_2= -(\partial_y \nabla_x \phi\arrowvert_{y = \eta})(J -1)( V \cdot \zeta-B), \\
 & R^{(0)}_3 =  (J+1)B \nabla_x B -(\nabla_x B) J B   
 - B J \nabla_x B,\\   
 & R^{(0)}_4=  -\nabla_x J (BV\cdot \zeta)  
      +\nabla_x(B JV\cdot \zeta),
 \end{aligned}
 \right.
 \end{equation}
 First of all, we have
 \begin{align*}
 \Vert a(D_x)^\mez J (V\cdot\nabla_x)V\Vert_{\mathcal{H}^{\sigma, s-1}}&\les \Vert   (V\cdot\nabla_x)V\Vert_{\mathcal{H}^{\sigma, s-\mez}}\les\Vert    V  \Vert_{\mathcal{H}^{\sigma, s-\mez}}\Vert V\Vert_{\mathcal{H}^{\sigma, s+\mez}}\\ 
 &\les \Vert    V  \Vert_{\mathcal{H}^{\sigma, s-\mez}}(\Vert    V  \Vert_{\mathcal{H}^{\sigma, s }} +\Vert   a(D_x)^\mez V  \Vert_{\mathcal{H}^{\sigma, s }}).
 \end{align*}
 Therefore,
    \begin{equation}\label{est-F1}
     \Vert a(D_x)^\mez R^{(0)}_1\Vert_{\mathcal{H}^{\sigma, s-1}}\les M_s(T)^2+ M_s(T)\Vert U_s\Vert_{H^\mez}.
  \end{equation}
Let us consider the term $R^{(0)}_2.$ Since $s > 1+ \frac{d}{2}$, 
we have
$$ \Vert a(D_x)^\mez R^{(0)}_2\Vert_{\mathcal{H}^{\sigma, s-1}} \les  \Vert\partial_y \nabla_x \phi\arrowvert_{y = \eta}\Vert_{\mathcal{H}^{\sigma, s-\mez}} \Vert V\cdot\zeta-B\Vert_{\mathcal{H}^{\sigma, s-\mez}}.$$
Here we have
$$ \Vert V\cdot\zeta-B\Vert_{\mathcal{H}^{\sigma, s-\mez}}\les \Vert V \Vert_{\mathcal{H}^{\sigma, s-\mez}} \Vert  \eta \Vert_{\mathcal{H}^{\sigma, s+\mez}} + \Vert  B\Vert_{\mathcal{H}^{\sigma, s-\mez}}\leq  C M_s(T).$$
Recall that by Lemma~\ref{der2phi}, we have
$$ \partial_y\nabla_x \phi \arrowvert_{y = \eta} =   \nabla_x B  -\frac{(\nabla_x B \cdot \zeta) \zeta + (\cnx V)  \zeta }{1+ \vert  \zeta \vert^2} .$$
It follows from the product rules and Proposition~\ref{est-f(u)bis} that
\begin{align*}
 &\Vert\partial_y \nabla_x \phi\arrowvert_{y = \eta}\Vert_{\mathcal{H}^{\sigma, s-\mez}}\\
 \les &\,   \Vert B\Vert_{\mathcal{H}^{\sigma, s+ \mez}}(1+\Vert \eta\Vert^2_{\mathcal{H}^{\sigma, s+ \mez}})^2 + \Vert V\Vert_{\mathcal{H}^{\sigma, s+ \mez}}\Vert \eta\Vert_{\mathcal{H}^{\sigma, s+ \mez}}(1+ \Vert \eta\Vert^2_{\mathcal{H}^{\sigma, s+ \mez}})\\
 \les &\,  \Vert B\Vert_{\mathcal{H}^{\sigma, s+ \mez}}+ \Vert V\Vert_{\mathcal{H}^{\sigma, s+ \mez}} \\  
   \les&\, \Vert V\Vert_{\mathcal{H}^{\sigma, s+ \mez}}+ TM_s(T)^2 + \eps,
\end{align*}
since $\Vert \eta\Vert_{\mathcal{H}^{\sigma, s+ \mez}}\leq M_s(T)\leq \overline{\eps}\leq 1$  and using   Corollary~\ref{est-B-V2}. Now,
$$
  \Vert V\Vert_{\mathcal{H}^{\sigma, s+ \mez}}  \les  \Vert  a(D_x)^\mez V\Vert_{\mathcal{H}^{\sigma, s}}+ \Vert V\Vert_{\mathcal{H}^{\sigma, s}}  
   \les\Vert U_s\Vert_{H^\mez} + M_s(T).
$$
 It follows that
 $$\Vert\partial_y \nabla_x \phi\arrowvert_{y = \eta}\Vert_{\mathcal{H}^{\sigma, s-\mez}}\les  \Vert U_s\Vert_{H^\mez} + M_s(T) + TM_s(T)^2 + \eps.$$
  Eventually we obtain
  \begin{equation*}
  \Vert a(D_x)^\mez R^{(0)}_2\Vert_{\mathcal{H}^{\sigma, s-1}}\les  \Vert U_s\Vert_{H^\mez}M_s(T) + M^2_s(T) + TM_s(T)^3 + \eps M_s(T).
  \end{equation*}
We consider now the term $R^{(0)}_3.$ It is easy to see that
$$ \Vert a(D_x)^\mez R^{(0)}_3\Vert_{\mathcal{H}^{\sigma, s-1}}\les \Vert B \Vert_{\mathcal{H}^{\sigma, s-\mez}}\Vert B \Vert_{\mathcal{H}^{\sigma, s+ \mez}}.$$
So using Corollary~\ref{est-B-V2} we obtain eventually,
\begin{equation*}
  \Vert a(D_x)^\mez R^{(0)}_3\Vert_{\mathcal{H}^{\sigma, s-1}}\les    \Vert U_s\Vert_{H^\mez}M_s(T) + M^2_s(T) + TM_s(T)^3 + \eps M_s(T).
    \end{equation*}
    Consider now the term $R^{(0)}_4.$ We notice that $R^{(0)}_4 = \nabla_x \big( [J,B] V \zeta\big).$ 
    It follows that
    $$\Vert a(D_x)^\mez R^{(0)}_4\Vert_{\mathcal{H}^{\sigma, s-1}}\leq C\Vert [J,B] V \zeta \Vert_{\mathcal{H}^{\sigma, s+ \mez}}.$$
    Since the operator $J = \varphi(2^{-n}D_x)$ belongs to $\text{Op}S^0_{1,0}$ with semi-norms  bounded with respect to  $n$,  Proposition~\ref{commutat} gives
    $$\Vert a(D_x)^\mez R^{(0)}_4\Vert_{\mathcal{H}^{\sigma, s-1}}\les  \Vert B   \Vert_{\mathcal{H}^{\sigma, s+ \mez}}\Vert   V   \Vert_{\mathcal{H}^{\sigma, s- \mez}}\Vert   \eta \Vert_{\mathcal{H}^{\sigma, s+ \mez}}.$$ 
   Then by  Corollary~\ref{est-B-V2}, we obtain
    \begin{equation}\label{est-F4}
    \Vert a(D_x)^\mez R^{(0)}_4\Vert_{\mathcal{H}^{\sigma, s-1}}\les  M_s(T)^2\big( \Vert U_s\Vert_{H^\mez} + T M_s(T)^2 + \eps\big).
    \end{equation}
    Summing up, using \eqref{est-F1} to \eqref{est-F4}, then \eqref{R0=} and the fact that $M_s(T)\leq \overline{\eps}\leq 1$, we obtain
    \begin{equation*}
    \Vert  a(D_x)^\mez R^{(0)}\Vert_{\mathcal{H}^{\sigma, s-1}}\les
  M_s(T) \Vert U_s \Vert_{H^\mez} + M_s(T)^2+ T M_s(T)^3 + \eps M_s(T).
    \end{equation*}
      The proof of Proposition~\ref{est-restes} is complete. 
       \end{proof}
    \begin{coro}\label{est-fin-R}
    We have
\begin{align*}
&\text{\rm{(i)}}\quad \Vert \sqrt{g}R_1^{(1)}   \Vert_{\mathcal{H}^{\sigma, s-\mez}}   \les  M_s(T) \Vert U_s \Vert_{H^\mez} + \Vert b \Vert_{H^{s+\mez}},\\
&\text{\rm{(ii)}} \quad \Vert \sqrt{g}R_2^{(1)}+ ia(D_x) ^\mez\big[(\partial_y P \arrowvert_{y = \eta} + g)\zeta+  ia(D_x)^\mez R^{(0)}\Vert_{\mathcal{H}^{\sigma, s-1}}\\
&\qquad\qquad\qquad\qquad\qquad\qquad \les M_s(T) \Vert U_s \Vert_{H^\mez} + g_\eps(T), 
 \end{align*}
    where
    \begin{equation}\label{geps}
     g_\eps(T)=  M_s(T)^2+ T M_s(T)^3 +  T^2 M_s(T)^5+ \eps M_s(T).
     \end{equation}
  \end{coro}
\subsubsection{A priori estimates.}

In this paragraph we first bound the terms in $M_s(T)$ containing $\eta$ and $V$.
 For the reader's convenience we recall some  notations. We have set
$$\zeta = \nabla_x \eta, \quad  u = \sqrt{g}\zeta +i a(D_x)^\mez V.$$
Then $u$ is solution of the equation (see~\eqref{eq-fin2})
\begin{equation*}
  \partial_t u + iJ \big(g a(D_x)\big)^\mez u = \sqrt{g }R^{(1)} + ia(D_x) ^\mez\big[(\partial_y P \arrowvert_{y = \eta} + g)\zeta + R^{(0)}\big]:= f
\end{equation*}
We have also set
$$U_s(t) = e^{\sigma(t) \langle D_x \rangle}\langle D_x \rangle^{s-\mez} u(t), \quad \sigma(t) = \lambda h - K \eps t,$$
where  $K$ is a  large positive constant to be chosen.
Then $U_s$ satisfies the equation
$$\partial_t U_s+iJ(g \,a(D_x))^\mez U_s+ K \eps \langle D_x \rangle  U_s = e^{\sigma(t)\langle D_x \rangle}  \langle D_x \rangle^{s-\mez}f.$$
So we have
\begin{align*}
  \frac{d}{dt} \Vert U_s(t)\Vert^2_{L^2}  &= 2  \big(U_s(t), \partial_t U_s(t)\big)_{L^2},\\
 &= - 2K \eps  \Vert \, \langle D_x \rangle^\mez U_s(t)\Vert^2_{L^2}  + 2  \big(U_s(t), e^{\sigma(t) \langle D_x \rangle} \langle D_x \rangle^{s-\mez}f(t) \big)_{L^2},
 \end{align*}
 because  the term  $ 2\text{Re } i \big(U_s(t), (g \,Ja(D_x))^\mez U_s(t)\big)_{L^2} $ vanishes,  since the symbol of $Ja$ is real. We deduce the estimate
 \begin{equation}\label{a-priori1}  
 \begin{split}
&  \Vert U_s(t)\Vert^2_{L^2} +  2K \eps  \int_0^t \Vert \,   U_s(t')   \Vert^2_{H^\mez} \,dt' \\
  =&\, \Vert U_s(0)\Vert^2_{L^2} + 2 \int_0^t \big(U_s(t'),\mathcal{K}_s(t')f(t') \big)_{L^2}\, dt',
 \end{split}
 \end{equation}
 where (see \eqref{R1=})
\begin{align*}
 &\mathcal{K}_s(t') = e^{\sigma(t') \langle D_x \rangle} \langle D_x \rangle^{s-\mez}, \quad f= f_1 + f_2,\\
 &\quad f_1= \sqrt{g }R_1^{(1)},\quad 
  f_2= \sqrt{g }R_2^{(1)} + ia(D_x) ^\mez\big[(\partial_y P \arrowvert_{y = \eta} + g)\zeta + R^{(0)}\big].
\end{align*}
  Set
 \begin{equation}\label{defA}
A_j(t) = \int_0^t F_j(t')\, dt' 
\end{equation}
where 
\[
\begin{aligned} 
F_1(t')&=\vert \big(U_s(t'), K_s(t')f_1(t')  \big)_{L^2} \vert,\quad  
   F_2(t')  = \vert \big(U_s(t'), K_s(t')f_2 \big)_{L^2} \vert 
  \end{aligned}
\]

 Estimate of $A_1(t)$. 
  By Corollary~\ref{est-fin-R} (i) we can write
$$
F_1(t') \leq \Vert U_s(t')\Vert_{L^2}  \Vert f_1(t')\Vert_{\mathcal{H}^{\sigma, s-\mez}}
 \les M_s(T) \Vert U_s(t')\Vert^2_{H^\mez} +  \Vert U_s(t')\Vert_{L^2} \Vert b(t')\Vert_{H^{s+\mez}}. $$
 It follows that
 \begin{equation}\label{est-A1}
 A_1(t)\les M_s(T)\int_0^t   \Vert U_s(t')\Vert^2_{H^\mez}\, dt' + \int_0^t  \Vert U_s(t')\Vert_{L^2}\Vert b(t')\Vert_{H^{s+\mez}}\, dt'.
  \end{equation}
Estimate of $A_2(t)$. By Corollary~\ref{est-fin-R} (ii) we can write
$$
 F_2(t')  \leq \Vert U_s(t')\Vert_{H^{\mez}} \Vert f_2(t')\Vert_{\mathcal{H}^{\sigma, s-1}}\\
   \les M_s(T) \Vert U_s(t')\Vert^2_{H^\mez} + g_\eps(T).  
$$
 It follows that
 \begin{equation}\label{est-A2}
  A_2(t)    \les M_s(T)\int_0^t \Vert U_s(t')\Vert^2_{H^\mez}\, dt' + g_\eps(T)\int_0^t \Vert U_s(t')\Vert_{H^\mez}\, dt'
\end{equation}
 Let us set
 \begin{equation}\label{def-Es}
\mathcal{E}_s(t)=  \Vert U_s(t)\Vert^2_{ L^2 }  +  2K \eps  \int_0^t \Vert  U_s(t')   \Vert^2_{H^\mez} \,dt'.
\end{equation}
 Using \eqref{a-priori1}, \eqref{defA}, \eqref{est-A1} and \eqref{est-A2}, we obtain
  \begin{equation}\label{a-priori2}
   \mathcal{E}_s(t) \les  \Vert U_s(0)\Vert^2_{L^2}    + M_s(T)\int_0^t \Vert U_s(t')\Vert^2_{H^\mez}\, dt' +   I_1(t) + I_2(t),
\end{equation}
where 
\begin{align*}
    I_1(t)&= \int_0^t \Vert b(t')\Vert_{H^{s+\mez}} \Vert U_s(t')\Vert_{L^2}\, dt', \\ 
    I_2(t) &= g_\eps(T)\int_0^t \Vert U_s(t') \Vert_{H^\mez}\, dt'.
     \end{align*}
Here $g_\eps$ is defined by \eqref{geps}.
We estimate separately each term $I_k.$ We have
\begin{equation}\label{I1}
 I_1(t)\leq \Vert U_s\Vert_{L^\infty(I,L^2)} \Vert b\Vert_{L^1(\xR, H^{s+\mez})} \leq \delta \Vert U_s\Vert^2_{L^\infty(I,L^2)}+ C_\delta\Vert b\Vert^2_{L^1(\xR, H^{s+\mez})},\\
 \end{equation}
  Using the Cauchy-Schwarz inequality, we write
 $$I_2(t) \leq g_\eps(T)T^\mez\Big(\int_0^t \Vert U_s(t') \Vert^2_{H^\mez}\, dt'\Big)^\mez,$$
 so, with absolute constants $C_1$ and $C_2$,
 \begin{equation}\label{I2}
 I_2(t) \les \frac{1}{K \eps}T\, g_\eps(T)^2 +  K\eps \int_0^t \Vert U_s(t') \Vert^2_{H^\mez}\, dt'.
 \end{equation}
 We will absorb the term $\frac{K\eps}{C_2}\int_0^t \Vert U_s(t') \Vert^2_{H^\mez}\, dt'$ by the corresponding term in $\mathcal{E}_s(t).$
  It follows from \eqref{a-priori2}, \eqref{I1} and \eqref{I2} that
 \begin{equation}\label{a-priori3}
 \begin{aligned}
 \mathcal{E}_s(t)  \les \Vert U_s(0)\Vert^2_{L^2}  + \delta \Vert U_s\Vert^2_{L^\infty(I,L^2)}&+  M_s(T)\int_0^t \Vert U_s(t')\Vert^2_{H^\mez}\, dt' \\
 &+ \frac{C_1}{K \eps}T g_\eps(T)^2 +\Vert b\Vert^2_{L^1(\xR, H^{s+\mez})}.
  \end{aligned}
 \end{equation}
 Now according to \eqref{equiv-Us-u}, we have
 \begin{equation}\label{a-priori4}
 \begin{aligned}
 &\Vert U_s(0)\Vert_{L^2} \les \Vert \eta_0\Vert^2_{\mathcal{H}^{\lambda h, s+\mez}}+ \Vert V_0\Vert^2_{\mathcal{H}^{\lambda h, s}},\\
 &\Vert \nabla \eta(t)\Vert^2_{\mathcal{H}^{\sigma(t), s-\mez}}+ \Vert a(D_x)^\mez V(t)\Vert^2_{\mathcal{H}^{\sigma(t), s-\mez}}\les \Vert U_s(t)\Vert^2_{L^2}.
 \end{aligned}
 \end{equation}
 Moreover, we have
 \begin{equation}\label{a-priori5}
 \begin{aligned}
 &\Vert   \eta(t)\Vert^2_{\mathcal{H}^{\sigma(t), s+\mez}}\les \Vert \nabla \eta(t)\Vert^2_{\mathcal{H}^{\sigma(t), s-\mez}}+ \Vert   \eta(t)\Vert^2_{L^2},\\
 &\Vert  V(t)\Vert^2_{\mathcal{H}^{\sigma(t), s }}\les \Vert a(D_x)^\mez V(t)\Vert^2_{\mathcal{H}^{\sigma(t), s-\mez}}+ \Vert   V(t)\Vert^2_{L^2}.
\end{aligned}
\end{equation}
Using \eqref{def-Es}, \eqref{a-priori3}, \eqref{a-priori4} and \eqref{a-priori5}, 
we obtain
\begin{align*}
&\mathcal{E}_s(t)+ \Vert   \eta(t)\Vert^2_{\mathcal{H}^{\sigma(t), s+\mez}}
+\Vert  V(t)\Vert^2_{\mathcal{H}^{\sigma(t), s }} \\
\les&\, \Vert \eta_0\Vert^2_{\mathcal{H}^{\lambda h, s+\mez}}
+\Vert V_0\Vert^2_{\mathcal{H}^{\lambda h, s}}
+ \delta \Vert U_s(t)\Vert^2_{L^\infty(I,L^2)}\\
&\, +M_s(T)\int_0^t \Vert U_s(t')\Vert^2_{H^\mez}\, dt' + \frac{1}{K \eps}T g_\eps(T)^2 +\Vert b\Vert^2_{L^1(\xR, H^{s+\mez})}.
\end{align*}

To complete the proof of Proposition~\ref{Prop:C0} 
we are left with the estimate of  the part in $M_s(T)$ containing  $B$ and $\psi.$ For that we use Corollary~\ref{est-B-V2}. Using also the hypotheses \eqref{hypot} we obtain
 \begin{coro}
There exists $C = C(d,g,\lambda,h,s)>0$ such that for every $t\in (0,T)$ we have
\begin{equation}
\begin{split}
& \mathcal{E}_s(t)+ \Vert   \eta(t)\Vert^2_{\mathcal{H}^{\sigma(t), s+\mez}}+ \Vert a(D_x)^\mez 
\psi\Vert^2_{\mathcal{H}^{\sigma(t), s}} 
  + \Vert  V(t)\Vert^2_{\mathcal{H}^{\sigma(t), s }}+ \Vert  B(t)\Vert^2_{\mathcal{H}^{\sigma(t), s }}\\ 
  \les  &\, 
    \delta \Vert U_s\Vert^2_{L^\infty(I,L^2)} 
  + M_s(T)\int_0^t \Vert U_s(t')\Vert^2_{H^\mez}, dt' 
+\frac{1}{K \eps}T g_\eps(T)^2  \\
&\qquad \qquad + (1+T^2) M_s(T)^4 + \eps^2,
\end{split}
  \end{equation}
   where
   $$g_\eps(T)=   M_s(T)^2+ T M_s(T)^3 +  T^2 M_s(T)^5+ \eps M_s(T)$$
\end{coro}
 
 \begin{proof}[End of the proof of Proposition~\ref{Prop:C0}]
 According to the definition of $\mathcal{E}_s(T)$ (see \eqref{def-Es}) taking the supremum of both members with respect to 
 $t$ in $(0,T)$ and  $\delta$ small enough, we can absorb the term $\delta \Vert U_s\Vert^2_{L^\infty(I,L^2)}$ in the right hand side  by the left hand side. 
This completes the proof.
 \end{proof}

\subsection{End of the proof.}

We are now in position to complete the proof of Theorem~\ref{T=2}. 

\textbf{Uniqueness.} Without source term (that is when $b=0$), 
the uniqueness of smooth solutions is a well-known result. 
When $b$ is non-trivial, we notice that the uniqueness result asserted by 
Theorem~\ref{T=1} implies the uniqueness of the solutions 
satisfying the regularity assumptions in Theorem~\ref{T=2}. Indeed, 
if we consider an initial data $(\eta_0,\psi_0)$ satisfying assumption ~\e{assu:T=2}, 
and two possible solutions $(\eta_1,\psi_1)$ 
and $(\eta_2,\psi_2)$, satisfying the Cauchy problem~\eqref{system}, with the same initial data $(\eta_0,\psi_0)$, and satisfying the regularity result~\e{result:T=2}, then they are both solutions satisfying trivially \e{result:T=1}. 

\textbf{Passing to the limit.} It remains to prove the existence part of the result. 
For the convenience of the readers, let us recall that we have proved the existence 
of approximate solutions $(\eta_n\psi_n)$ to the Cauchy problem
\begin{equation}\label{ww:approxf}
\left\{
\begin{aligned}
&\partial_t \eta_n  = J_n G(\eta_n )(\psi_n ,b),  \\
  &\partial_t \psi_n 
  =J_n\Big( -g \eta_n - \mez \vert \nabla_x \psi_n \vert^2 
  + \frac{(G(\eta_n )(\psi_n,b)+ \nabla_x \eta_n  \cdot \nabla_x \psi_n )^2}{2(1+  \vert \nabla_x \eta_n \vert^2)}\Big),\\   
  &\eta_n \arrowvert_{t=0} = J_n\eta_0, \quad \psi_n \arrowvert_{t=0} = J_n\psi_0,
 \end{aligned}
\right.
\end{equation}
where $J_n$ is a truncation in frequency space defined in \e{defi:Jn}. 
In this paragraph, 
we shall prove that we can extract a sub-sequence of $((\eta_{n'},\psi_{n'}))$ 
that converges weakly 
to a solution of the water-wave system (thanks to the uniqueness of the solution to the water-wave system, this will imply that the whole sequence converges, without extraction of a sub-sequence). 
This part relies on classical arguments from functional analysis, but since we work in analytic spaces and since the problem is nonlinear and nonlocal, some verifications are needed.

Recall from Lemma \ref{lemm-Tn} and Corollary \ref{coro-c*} that there exist 
four positive real 
numbers $\eps_*$, $c_*$, $C_*$ and $K_*$ such that for all $n\in \xN\setminus\{0\}$, 
the following properties hold: if the initial norm $\eps$ (as defined 
by~\e{defi:epsnorm}) satisfies $\eps\le \eps_*$, then for all $n\in \xN\setminus\{0\}$, the lifespan is bounded from below by
$T_n\ge \frac{c_*}{\eps}$, 
and moreover,
$$
M_{s,n}\Big(\frac{c_*}{\eps}\Big)\le C_*\eps,
$$
where the norm $M_{s,n}(T)$ is defined by
\[
M_{s,n} (T) =   \Vert   \eta_n  \Vert_{X_T^{\infty,s+ \mez}}  +  \Vert  a(D_x)^\mez \psi_n \Vert_{X_T^{\infty,s}} 
+    \Vert  V_n  \Vert_{X_T^{\infty,s}}  
+  \Vert  B_n \Vert_{X_T^{\infty,s}}.
\]
Here we put 
\[
 X^{\infty,s}_T = L^\infty([0,T], \mathcal{H}^{\sigma, s}) \quad \text{with} \quad \sigma(t) = \lambda h - K_* \eps t.
\]
Let us notice that Theorem~\ref{est-DN} implies that
$$
\lA G(\eta_n)(\psi_n,b)\rA_{X^{\infty,s-\mez}_{c_*/\eps}}\le C_*'\eps,
$$
for some constant $C_*'$ independent of $\eps$ and $n$. Then, by using the product rule given by point~$\ref{HorNLii)})$ in Proposition~\ref{uv}, 
as we already did 
repeatedly in the previous paragraph, we infer from the equations~\e{ww:approxf} that
$$
\lA \partial_t\eta_n\rA_{X^{\infty,s-\mez}_{c_*/\eps}}
+\lA \partial_t\psi_n\rA_{X^{\infty,s-\mez}_{c_*/\eps}}\le C_*''\eps.
$$
Thanks to the Arzela-Ascoli theorem and the compact embedding of 
$\mathcal{H}^{\lambda h,s}(\xT^d)$ in $\mathcal{H}^{\lambda h,s'}(\xT^d)$ for $s'<s$, 
there exist a sub-sequence $((\eta_{n'},\psi_{n'}))$ and the limit $(\eta,\psi)$ such that $(\eta_{n'},\psi_{n'})$ converges to $(\eta,\psi)$ in $X^{\infty,s'}_{c_*/\eps}$. 
Now, the contraction result for the Dirichlet-to-Neumann operator given by Theorem~\ref{G-lip} implies that the sequence $(G(\eta_{n'})(\psi_{n'},b))$ converges to $G(\eta)(\psi,b)$. Thus we conclude that the limit $(\eta,\psi)\in X^{\infty,s'}_{c_*/\eps}$ 
solves the water-waves equations.

\appendix

\section{Some properties of the $\mathcal{H}^{\sigma,s}$ spaces}
\label{AppendixA}

\subsection{Characterization.}

In this paragraph we prove Theorem~\ref{Gsigmas} whose statement is recalled below, 
together with the fact that functions 
in $\mathcal{H}^{\sigma,s}(\xT^d)$ are the traces on $\xT^d$ of holomorphic functions in 
$$
S_\sigma= \{(x,y)\in \xT^d\times\xR^d: \vert y \vert <\sigma\},
$$
where
\[
|y|= \Bigg(\sum_{j=1}^d y_j^2\Bigg)^\mez.
\]
Recall that, given $U \colon S_\sigma\to \xC$, 
we denote by $U_y$ the function from $\xT^d$ to $\xC$ defined by $x \mapsto  U(x+iy)$. 


\begin{theo}\label{Gsigmas2}
Let $\sigma>0$ and $s\in \xR$. 
\begin{itemize}
\item[\rm{(1)}]  Let $u \in \mathcal{H}^{\sigma,s}(\xT^d)$. 
There exists $U\in \mathcal{H}ol(S_\sigma)$  such that $U_0 = u$ and
$$\sup _{\vert y \vert<\sigma} \Vert U_y\Vert_{H_x^s(\xT^d)} \leq\Vert u \Vert_{ \mathcal{H}^{\sigma,s}}.$$
\item[\rm{(2)}]  Let $U\in \mathcal{H}ol(S_\sigma)$ 
such that  $M_0:= \sup_{\vert y \vert<\sigma} \Vert U_y\Vert_{H_x^s(\xT^d)}<+ \infty$. Set $u = U_0$. Then, 
\begin{itemize}
\item[\rm{(i)}] If $d=1$, then $u$ 
belongs to $\mathcal{H}^{\sigma,s}(\xT^d)$ and  $\Vert u \Vert_{\mathcal{H}^{\sigma,s}} \leq 2M_0.$

\item[\rm{(ii)}] If $d\geq 2$, 
then $u$ belongs to $\mathcal{H}^{\delta,s}(\xT^d)$ for any $\delta<\sigma$ and there exists a constant 
$C_\delta>0 $ such 
that $\Vert u \Vert_{\mathcal{H}^{\delta,s}} \leq C_\delta M_0$.
\end{itemize}
\item[\rm{(3)}] Let $U\in \mathcal{H}ol(S_\sigma)$ be such that
$$
M_1:= \sup_{\vert y \vert<\sigma} \Vert U_y\Vert_{H_x^{s'}(\xT^d)}<+ \infty \quad\text{with}\quad 
s'>s + \frac{d-1}{4}.
$$
Then the function $u = U_0$ belongs to $\mathcal{H}^{\sigma,s}(\xT^d)$ 
and there exists a constant $C>0$ such that 
$\Vert u \Vert_{\mathcal{H}^{\sigma,s}} \leq C   M_1$.
\end{itemize}
\end{theo}
\begin{proof}
We divide the proof of Theorem \ref{Gsigmas2} into six steps. \\

{\textbf{Step 1: Existence of an holomorphic extension}.} Let us prove statement~(1). 
Fix $\sigma>0$ and $s\in \xR$ and consider a function 
$u \in \mathcal{H}^{\sigma,s}(\xT^d)$. We want to prove that 
there exists $U\in \mathcal{H}ol(S_\sigma)$  such that $U_0 = u$ and
$$
\sup _{\vert y \vert<\sigma}
\Vert U_y\Vert_{H_x^s(\xT^d)} \leq 
\Vert u \Vert_{ \mathcal{H}^{\sigma,s}}.
$$

We begin by observing that, for any $z\in S_\sigma$, the function $\xZ^d \to \xC$, 
$\xi \mapsto e^{iz\cdot \xi} \widehat{u}(\xi)$ belongs to $\ell^1(\xZ^d)$. To see this, we write
\begin{equation*}
\vert e^{iz\cdot \xi} \widehat{u}(\xi) \vert = e^{-y \cdot \xi}\vert \widehat{u}(\xi) \vert = \big[ \langle \xi \rangle^{-s} e^{-\sigma\vert \xi \vert} e^{-y \cdot \xi}\big] \times \big[\langle \xi \rangle^{s}e^{\sigma \vert \xi \vert} \vert \widehat{u}(\xi) \vert\big] : = f_1(\xi) \times f_2(\xi),
\end{equation*}
and then conclude since 
$\vert f_1(\xi) \vert \leq  \langle \xi \rangle^{-s} e^{-(\sigma - \vert y \vert)\vert \xi \vert} \in \ell^2(\xZ^d)$  and $f_2 \in \ell^2(\xZ^d)$, by assumption on $u$. So we can define the function  
 \begin{equation*}
  U(z) = (2\pi)^{-d}\sum_{\xi\in\xZ^d} e^{iz\cdot \xi} \widehat{u}(\xi).
 \end{equation*}
The previous inequality implies that, for any $\eps>0$,
$$
\sup_{z\in S_{\sigma-\eps}}\sum_{\xi\in\xZ^d} \vert e^{iz\cdot \xi} \widehat{u}(\xi)\vert<+\infty.
$$
Since the function $z\mapsto e^{iz\cdot \xi} \widehat{u}(\xi)$ is holomorphic, we deduce that 
$U\in \mathcal{H}ol(S_\sigma)$.

We next observe that the Fourier inversion formula implies that $U_0 = \overline{\mathcal{F}} \widehat {u} = u$. 
In addition, we have 
$U(x+iy) =  \overline{\mathcal{F}}(e^{-y \cdot \xi} \widehat{u})$, hence 
$\widehat{U_y} = e^{-y \cdot \xi} \widehat{u}$ which in turn implies that
\begin{equation*}
 \Vert U_y \Vert^2_{H^s(\xT^d)}= \sum_{\xi\in\xZ^d} \langle \xi \rangle^{2s} e^{-2 y \cdot \xi} \vert \widehat{u}(\xi) \vert^2  \leq \sum_{\xi\in\xZ^d} \langle \xi \rangle^{2s} e^{ 2 \sigma \vert \xi \vert } \vert \widehat{u}(\xi) \vert^2  = \Vert u \Vert^2_{\mathcal{H}^{\sigma,s}}.
 \end{equation*}
This completes the proof of statement~(1). \\

{\textbf{Step2: A Sobolev estimate}.} For later purpose, let us prove an additional estimate. 
Consider a real number $s_0>d/2$ and write
$$
e^{-y\cdot\xi}\vert  \widehat{u}(\xi)\vert = \langle \xi \rangle^{-s_0}e^{-y\cdot\xi}e^{ -  \sigma \vert \xi \vert }e^{   \sigma \vert \xi \vert } \langle \xi \rangle^{s_0}\vert  \widehat{u}(\xi)\vert.
$$
Now, compared to the previous proof, we see that the factor 
$\langle \xi \rangle^{-s_0}e^{-y\cdot\xi}e^{ -  \sigma \vert \xi \vert }$ is summable in $\xi$, uniformly for $\la y\ra\le \sigma$, thanks to the assumption $s_0>d/2$. Then, the Cauchy-Schwarz inequality implies that 
there exists a positive constant $C=C(d,s_0)$ such that  
\begin{equation*}
\Vert U\Vert_{L^\infty{(S_\sigma)}}=\sup_{\la y\ra< \sigma}\lA U_y\rA_{L^\infty(\xT^d)}  \leq C\Vert u \Vert_{\mathcal{H}^{\sigma,s_0}}.
\end{equation*}

\vspace{5mm}

{\textbf{Step 3: Trace of an holomorphic function}.} 
In this step, we initiate the proof of the various points in statement~(2). 
Consider a function $U\in \mathcal{H}ol(S_\sigma)$ 
such that
\begin{equation}\label{T1assuM0}
M_0:= \sup_{\vert y \vert<\sigma} \Vert U_y\Vert_{H_x^s(\xT^d)}<+ \infty.
\ee
We want to study the regularity of the 
trace $u = U_0$. 

First, observe that the assumption~\e{T1assuM0} implies that $u = U_0 \in H^s(\xT^d)$.

Now, let $\psi= 1_{\vert k \vert \leq 1}.$   Given $\lambda >0$, we introduce the functions 
$\psi_\lambda(\xi) = \psi(\frac{\xi}{\lambda})$ and 
$\varphi_\lambda = \overline{\mathcal{F}} \psi_\lambda$. Set
$$
F_\lambda(z) = (2\pi)^{-d} \sum_{\xi \in \xT^d} e^{iz\cdot \xi} \psi_{\lambda}(\xi) \widehat{u}(\xi).
$$
This function is holomorphic in $S_\sigma$. Indeed, the summand is holomorphic and, for $z \in S_\sigma$,
\begin{align*}
\vert e^{iz\cdot \xi} \psi_{\lambda}(\xi) \widehat{u}(\xi) \vert &= e^{-(\IM z) \cdot\xi} \psi_{\lambda}(\xi)\vert \widehat{u}(\xi) \vert \leq  e^{\sigma\vert \xi\vert}\psi_{\lambda}(\xi) \vert\widehat{u}(\xi) \vert \leq e^{2\sigma \lambda}   \psi_{\lambda}(\xi)\vert \widehat{u}(\xi) \vert,\\
  & \leq \big[ e^{2\sigma \lambda} \langle \xi \rangle^{-s}  \psi_{\lambda}(\xi)\big] \times \big[ \langle \xi \rangle^{s}\vert\widehat{u}(\xi) \vert] \in \ell^1(\xZ^d).
  \end{align*}
Notice that
$$
F_\lambda(x+iy) = \overline{\mathcal{F}}(e^{-y\cdot \xi}\psi_{\lambda}(\xi) \widehat{u}(\xi))(x).
$$
For $z = x+iy \in S_\sigma$  we set
$$
V_\lambda(z) = U_y\star \varphi_\lambda(x) = \int_{\xT^d} U(x+iy-t) \varphi_\lambda(t)\, dt.
$$
This function is holomorphic in the strip $S_\sigma$. In addition,  
$$
V_\lambda\arrowvert_ {y=0} = u\star \varphi_\lambda = \overline{\mathcal{F}} (\psi_\lambda \widehat{u}) = (2\pi)^{-d} \sum_{\xi \in \xZ^d} e^{ix\cdot \xi}\psi_\lambda \widehat{u} = F_\lambda\arrowvert_{y=0}.
$$
By the uniqueness for analytic functions, this implies that $V_\lambda=F_\lambda$ in $S_\sigma$. By taking the Fourier transform 
of the previous identity, we obtain
$$
\mathcal{F}  U_y (\xi) \psi_\lambda(\xi) = e^{-y\cdot \xi}\psi_{\lambda}(\xi) \widehat{u}(\xi).
$$
By letting $\lambda$ goes to $+ \infty$, we infer that $\mathcal{F}  U_y (\xi)  = e^{-y\cdot \xi}  \widehat{u}(\xi)$. 
Consequently,
$$
\Vert U_y \Vert^2_{H^s(\xT^d)}=\sum_{\xi \in \xZ^d} e^{-2y\cdot \xi} \langle \xi \rangle^{2s} \vert \widehat{u}(\xi) \vert^2.
$$
Then  the assumption on $U$ implies that
\begin{equation}\label{equation2}
\sup_{\vert y \vert <\sigma} \sum_{\xi \in \xZ^d} e^{-2y\cdot \xi} \langle \xi \rangle^{2s} \vert \widehat{u}(\xi )\vert^2  = M_0^2<+ \infty.
\end{equation}
This is the key ingredient to prove the two statements in point~(2). \\

{\textbf{Step 4: Trace of an holomorphic function in dimension one}.} 

Assume that $d=1$. Set $v(\xi) = \langle \xi \rangle^{2s}\vert \widehat{u}(\xi) \vert^2$. 
For any real number $0<b<\sigma$, 
the inequality \eqref{equation2} applied with $y=-b$ (resp.\ $y=b$) implies that 
$$
\sum_{\xi=0}^{+\infty}  e^{2b \xi}v(\xi)  \leq M_0,\quad \text{resp.}\quad 
\sum_{\xi=-\infty}^{0}e^{-2b \xi}v(\xi)  \leq M_0.
$$
It follows that $\sum_{\xi \in \xZ} e^{2b \vert \xi \vert}v(\xi)\leq 2M_0 $.   Fatou's lemma 
then implies that, when $b$ goes to $ \sigma$, we have
$$
\sum_{\xi \in \xZ} e^{2\sigma \vert \xi \vert}v(\xi)
\leq 2M_0,
$$
which proves statement~(i). \\

{\textbf{Step 5: Arbitrary dimension}.} 

Let us prove statement~(ii).
We now assume that $d\geq 2$ and consider a real number $\delta<\sigma$. 
We can write $\delta = (1- \frac{\eps^2}{2}) \sigma$ for some $\eps>0$. 
Then there exists $N= N(\eps)$ and $\omega_1, \ldots, \omega_N \in \xS^{d-1}$  such that 
$$
\xZ^d\setminus \{0\} = \bigcup_{j=1}^N \Gamma_j\quad\text{where}\quad
\Gamma_j = \Big\{\xi\in\xZ^d: \la \frac{\xi}{\vert \xi \vert} - \omega_j\ra < \eps \Big\}.
$$
Notice that
$$
\xi \in \Gamma_j\quad\Rightarrow \quad\la \frac{\xi}{\vert \xi \vert}\ra^2  +\vert\omega_j\vert^2  - 2 \frac{\xi}{\vert \xi \vert}\cdot \omega_j< \eps^2\quad\Rightarrow\quad \left(1- \frac{\eps^2}{2}\right)\vert \xi \vert \leq \xi\cdot \omega_j.
$$
Consider $0<b<\delta$ and set $\widetilde{b} = \frac{b}{1- \eps^2/2} <\sigma$. 
We have
$$
\sum_{\xZ^d} e^{2 b\vert \xi \vert} v(\xi)
\leq \sum_{j=1}^N \sum_{\Gamma_j} e^{2 b\vert \xi \vert} v(\xi)
\leq \sum_{j=1}^N \sum_{\Gamma_j} e^{2 \widetilde{b } (\xi \cdot \omega_j)}v(\xi).
$$
Since the vector $y_j=\widetilde{b}  \omega_j$ satisfies $\vert y_j \vert := \widetilde{b} \vert \omega_j \vert < \sigma$, 
the key estimate \eqref{equation2} implies that 
$$
\sum_{\xZ^d} e^{2 b\vert \xi \vert} v(\xi) \leq N M_0^2.
$$ 
As above, we conclude by using Fatou's lemma, which implies that
$$
\sum_{\xZ^d} e^{2 \delta \vert \xi \vert} v(\xi) \leq N M_0^2,
$$
which concludes the proof of statement~(ii).\\
 
{\textbf{Final step: Arbitrary dimension, sharp estimate}.}  
 
We now prove statement~(3), which gives a smaller loss in analyticity. Namely, we assume that, 
$U\in \mathcal{H}ol(S_\sigma)$ is such that
$$
M_1:= \sup_{\vert y \vert<\sigma} \Vert U_y\Vert_{H_x^{s'}}<+ \infty
\quad\text{with}\quad 
s'>s + \frac{d-1}{4}.
$$
Our goal is to prove that, 
the function $u = U_0$ belongs to $\mathcal{H}^{\sigma,s}$ 
and there exists a constant $C>0$ such that 
$\Vert u \Vert_{\mathcal{H}^{\sigma,s}} \leq C   M_1$.

Let $b<\sigma$ and $s'>s + \frac{d-1}{4}$.  By replacing $s$ by $s'$ in \eqref{equation2}  we have
\begin{equation}\label{equation3}
\sup_{\vert y \vert <\sigma} \sum_{\xZ^d}
e^{-2y\cdot \xi} \langle \xi \rangle^{2s'} \vert \widehat{u}(\xi) \vert^2 
= \sup_{\vert y \vert <\sigma} \Vert U_y \Vert^2_{H^{s'}} = M_1^2<+ \infty.
\end{equation}
Set $v_s(\xi) = \langle \xi \rangle^{2s} \vert \widehat{u}(\xi)\vert^2$ and consider a real number $R_0$ 
such that $\sigma R_0 \gg 1$. We first notice that
\begin{equation}\label{equation3.5} 
\sum_{\vert \xi \vert \leq R_0}e^{2b \vert \xi \vert} v_s(\xi)
\leq e^{2b R_0}\sum   v_s(\xi) \leq e^{2bR_0 }\sum
\langle \xi \rangle^{\frac{d-1}{2}} v_s(\xi)
\leq e^{2bR_0 } M_1^2.
\end{equation} 
Define $\ell_0$ as the largest integer such that 
$\mez 2^{\ell_0} \leq R_0$ and then, for $\ell \geq \ell_0$, introduce the dyadic rings:
$$
\mathcal{C}_\ell = \{\xi \in \xZ^d\,:\, \mez 2^\ell \leq \vert \xi \vert \leq   2^{\ell +1}\}.
$$
Write
\begin{equation}\label{equation4}
\sum_{\vert \xi \vert > R_0}e^{2b \vert \xi \vert} v_s(\xi)
\leq  \sum_{\ell =\ell_0}^{+ \infty}
\sum_{\mathcal{C}_\ell}e^{2b \vert \xi \vert} v_s(\xi) 
:=   \sum_{\ell =\ell_0}^{+ \infty} I_\ell.
\end{equation}
Fix $\ell \geq \ell_0$ and set $\delta_\ell = \frac{  1}{b 2^\ell} \ll 1$. 
Let $\omega_0$ be an arbitrary point on the sphere $\xS^{d-1} $ and introduce 
$$
\Omega_{\omega_0} = \{ \omega \in \xS^{d-1}: \omega \cdot \omega_0 > 1- \delta_\ell\} =  \{ \omega \in \xS^{d-1}: \vert \omega- \omega_0 \vert < \sqrt{2 \delta_\ell}\}.
$$
The sets $\Omega_{\omega_0}$ have a $(d-1)$-dimensional measure independent of $\omega_0$. Moreover, 
there exists two positive constants $c_1, c_2$ independent of the dimension such that
$$
c_1 \delta_\ell^{\frac{d-1}{2}} \leq \mu(\Omega_{\omega_0}) \leq c_2 \delta_\ell^{\frac{d-1}{2}}.
$$ 
There exists $\omega_1, \ldots, \omega_{N_\ell} \in \xS^{d-1}$ where  ${N_\ell} \sim C_d\delta_\ell^{- \frac{d-1}{2}}$ such that $\xS^{d-1} = \cup_{j=1}^{N_\ell} \Omega_{\omega_j}$. 
Set
$$
\mathcal{C}_{\ell,j} = \left\{\xi \in \mathcal{C}_\ell : \frac{\xi}{\vert \xi \vert} \in \Omega_{\omega_j}\right\} .
$$
Then one can split the dyadic ring, $\mathcal{C}_\ell$ as 
$\mathcal{C}_\ell = \cup_{j=1}^{N_\ell} \mathcal{C}_{\ell,j}$, to obtain 
$$
I_\ell  \leq \sum_{j=1}^{N_\ell}
\sum_{\mathcal{C}_{\ell,j} } e^{2b \vert \xi \vert} v_s(\xi)
:= \sum_{j=1}^{N_\ell} I_{\ell,j}.
$$
If $\xi$ belongs to $\mathcal{C}_{\ell,j}$ one has, 
$ \frac{\xi}{\vert \xi \vert}\cdot \omega_j >1- \delta_\ell$, so $\vert \xi \vert < \xi\cdot \omega_j + \vert \xi \vert \delta_\ell$. Then  we write
$$
I_\ell  \leq  \frac{1}{N_\ell}\sum_{j=1}^{N_\ell}
\sum_{\mathcal{C}_{\ell,j} } e^{2(b \omega_j) \cdot\xi + 2b\vert \xi \vert \delta_\ell} {N_\ell} \, v_s(\xi).
$$
Recall that $\delta_\ell=1/(b2^\ell)$. Consequently, 
if $\xi$ in $\mathcal{C}_{\ell,j}$, we have
$2b \vert \xi \vert \delta_\ell \leq 4$. We deduce that 
$e^{2b\vert \xi \vert \delta_\ell} \leq e^4$. Moreover, 
$$N_\ell  \leq C_d \delta_\ell^{-\frac{d-1}{2}} \leq C_d(b 2^\ell)^{\frac{d-1}{2}} \leq C_d(2b)^{\frac{d-1}{2}} \vert \xi \vert^{\frac{d-1}{2}},$$ and 
$\ell \leq c \log \vert \xi \vert$.
 
Remembering that $v_s(\xi) = \langle \xi \rangle^{2s} \vert \widehat{u}(\xi)\vert^2$, 
we deduce that, for any $\eps>0$, 
\begin{equation*}
I_\ell  \leq \frac{C'_d b^{\frac{d-1}{2}}}{ \ell^{1+\eps }}   \frac{1}{N_\ell}\sum_{j=1}^{N_\ell}
\sum_{\mathcal{C}_{\ell,j}} e^{2(b \omega_j) \cdot \xi} \vert \xi \vert^{\frac{d-1}{2}}  (\log \vert \xi \vert)^{1+\eps}\langle \xi \rangle^{2s} \vert \widehat{u}(\xi)\vert^2.
\end{equation*}
Now since $s'> s + \frac{d-1}{4}$, there exists $\eps>0$ so small enough that $s' \geq s+\eps.$ We use \eqref{equation3} to infer that 
\begin{equation}\label{equation5} 
I_\ell \leq  \frac{C'_d b^{\frac{d-1}{2}}}{ \ell^{1+\eps }}  M_1^2.
\end{equation}
Since $\sum_{\ell\ge 1}\ell^{-1-\eps}<+\infty$, 
it follows from \eqref{equation4} that
$$
\sum_{\vert \xi \vert > R_0}e^{2b \vert \xi \vert} \langle \xi \rangle^{2s} \vert \widehat{u}(\xi)\vert^2 \leq C M_1^2.
$$
Using \eqref{equation3.5} we obtain eventually,
$$
\sum_{\xZ^d}
e^{2b \vert \xi \vert} \langle \xi \rangle^{2s} \vert \widehat{u}(\xi)\vert^2
\leq C M_1^2.
$$
Again, we conclude the proof thanks to Fatou's lemma. 
This completes the proof of Theorem~\ref{Gsigmas2}. 
\end{proof}

Notice that, if $u$ is radial, then one can remove the factor 
$(\log \vert \xi \vert)^{1+\eps}$ in \eqref{equation5} and hence it is sufficient to 
assume that 
$$M_1  = \sup_{\vert y \vert< \sigma} \Vert U_y \Vert_{H_x^{s+ \frac{d-1}{4}}}<+ \infty.$$ 
Notice also that, in this case, the above assumption 
on $U_y$ is optimal to insure that $u \in \mathcal{H}^{\sigma,s}$.

\subsection{Technical lemmas.}

We recall the following interpolation lemma. 
 \begin{lemm}\label{lions}
Let $s\in \xR$ and $h>0$.  
There exists $C>0 $ such that for all $\sigma\ge 0$, if  
$f$ is in $L^2((-h,0),\mathcal{H}^{\sigma,  s+\mez})$ and   
$\partial_zf$ is in $L^2((-h,0),\mathcal{H}^{\sigma,  s-\mez})$, then  
$f$ belongs to $C^0([-h,0], \mathcal{H}^{\sigma,  s})$ together with the estimate
$$
\sup_{z\in [-h,0]}\Vert f(z, \cdot)\Vert_{\mathcal{H}^{\sigma,  s}} \leq C \big( \Vert f \Vert_{L^2((-h,0),\mathcal{H}^{\sigma,  s+\mez})}
+\Vert \partial_zf \Vert_{L^2((-h,0),\mathcal{H}^{\sigma,  s-\mez})}\big).
$$
\end{lemm} 

We shall use the following lemma from H\"ormander~(see \cite[Theorem~8.3.1]{HormanderL}).
\begin{lemm}\label{Theo:Hormander}
Consider three real numbers $s_1,s_2,s_3$ such that
$$
s_1+s_2\ge 0,\quad s_3\le \min\{s_1,s_2\},\quad s_3\le s_1+s_2-\frac{d}{2},
$$
with the last inequality strict if $s_1$ or $s_2$ or $-s_3$ is equal to $d/2$. 
For $\xi,\zeta$ in $\xR^d$, define
$$
F(\xi,\zeta)=\frac{\langle \xi\rangle^{s_3}}{\langle \xi-\zeta\rangle^{s_1}\langle \zeta\rangle^{s_2}},
$$
and then set
$$
T_F(f,g)(\xi)=\sum_{\zeta\in  \xZ^d}F(\xi,\zeta)f(\xi-\zeta)g(\zeta),
$$
when $f$ and $g$ are continuous functions with compact support. Then there exists a positive constant $C$ such that
$$
\lA T_F(f,g)\rA_{\ell^2}\le C \lA f\rA_{\ell^2}\lA g\rA_{\ell^2}.
$$
\end{lemm}

  \subsection{Paradifferential operators on analytic spaces and non linear estimates.}
  
In this section we investigate the continuity of the paradifferential 
operators on the spaces $\mathcal{H}^{\sigma,s}$ and we apply these 
results to prove nonlinear estimates in these spaces.

We will use the following lemma.
\begin{lemm}\label{produit}
Let $s_0> \frac{d}{2}.$ There exists $C>0$ such that for all 
$\sigma>0,$  and for all 
$f,g \in \mathcal{H}^{\sigma,s_0},$ we have 
\begin{equation}\label{produit1}
\Vert f g \Vert_{\mathcal{H}^{\sigma,0}} \leq C \Vert f \Vert_{\mathcal{H}^{\sigma,s_0}} \Vert g\Vert_{\mathcal{H}^{\sigma,0}}.  
  \end{equation}
In addition, if $\supp \widehat{f} \subset \{\xi: \vert \xi \vert \leq R\}$, then
\begin{equation}\label{produit2}
 \Vert f g \Vert_{\mathcal{H}^{\sigma,0}} \leq C R^{\frac{d}{2}}\Vert f \Vert_{\mathcal{H}^{\sigma,0}} \Vert g\Vert_{\mathcal{H}^{\sigma,0}}.  
 \end{equation}
\end{lemm}
\begin{proof}
Set $U= e^{\sigma \vert D \vert}(fg).$ Then 
$$
\widehat{U}(\xi) = e^{\sigma \vert \xi \vert}\sum_{\eta \in \xZ^d} \widehat{f}(\xi - \eta) \widehat{g}(\eta),
$$
which implies that
$$
\vert \widehat{U}(\xi) \vert \leq \sum_{\eta\in\xZ^d}e^{\sigma \vert \xi- \eta \vert} \vert\widehat{f}(\xi - \eta)\vert  e^{\sigma \vert \eta\vert}  \vert\widehat{g}(\eta)\vert
=  \vert\widehat{e^{\sigma \vert D \vert}f}\vert  \star \vert\widehat{e^{\sigma \vert D \vert}g} \vert.
$$
Therefore,  
\begin{equation}\label{produit3}
 \Vert U \Vert_{L^2}    \leq C\Vert  \widehat{e^{\sigma \vert D \vert}f} \Vert_{\ell^1} \Vert  \widehat{e^{\sigma \vert D \vert}g} \Vert_{\ell^2} \leq C' \Vert  \widehat{e^{\sigma \vert D \vert}f}\Vert_{\ell^1} \Vert g \Vert_{\mathcal{H}^{\sigma,0}}.
 \end{equation}
If $s_0>\frac{d}{2}$, we have by using the H\" older inequality,
$$
\Vert u \Vert_{\ell^1} \leq C_{s_0} \Vert \langle \cdot \rangle^{s_0} u \Vert_{\ell^2},
$$
which proves the required inequality~\eqref{produit1}.  

To prove~\eqref{produit2} we introduce a function $\psi \in C_0^\infty(\xR^d)$ such that 
$ \supp \psi \subset \{ \xi : \vert \xi \vert \leq 2\}$ and $\psi = 1$ for 
$\vert \xi \vert \leq 1.$ Then
\begin{align*}
\sum_{\xi \in \xZ^d} e^{ \sigma \vert \xi \vert} \vert \widehat{f}(\xi) \vert &= \sum_{\xi \in \xZ^d}e^{ \sigma \vert \xi \vert} \psi\big(\frac{\xi}{R} \big)\vert  \widehat{f}(\xi) \vert
\leq \Big( \sum_{\xi \in \xZ^d} \psi\big(\frac{\xi}{R}\big)^2 \Big)^\mez \Vert f \Vert_{\mathcal{H}^{\sigma, 0}}\\
& \leq C R^{\frac 
{d}{2}}\Vert f \Vert_{\mathcal{H}^{\sigma, 0}},
\end{align*}
and we conclude the proof by using~\eqref{produit3}.
\end{proof}

Let us introduce a symbol class. 
Given $m \in \xR, \sigma>0,s\in\xR $, $\Gamma^{m,\sigma,s}$ denotes the space 
of those functions $p$ that are $C^\infty$ on 
$\xT^d\times\xR^d$, of the form 
$$p(x, \xi) \sim \sum_{j=0}^{+\infty} p_{m-j}(x,\xi),$$  where 
$p_{m-j}$ is homogeneous of order $m-j$ in $\xi$ for $\vert   \xi \vert \geq 1$,  and satisfies
\begin{equation*}
\mathcal{N}_s(p):= \Vert \langle D \rangle^s e^{\sigma \vert D \vert}(1- \Delta_\omega)^k p(x, \omega)\Vert_{L^2(\xT^d \times \xS^{d-1})} < + \infty, 
 \end{equation*}
with $k> \frac{d}{2}.$
 
We denote by $T_p$ the paradifferential operator associated with this symbol. By definition:
\begin{equation}\label{para}
\widehat{T_p u}(\xi) = \sum_{\eta\in \xZ^d} \chi(\xi- \eta, \eta)\widehat{p}(\xi-\eta, \xi) \psi(\eta)\widehat{u}(\eta),
\end{equation}
where 
$\chi \in C^\infty(\xR^d \times \xR^d), \chi(\theta, \eta) = 1$ if 
$ \vert \theta\vert \leq \eps_1 \vert \eta \vert, \, \chi(\theta, \eta) = 0$ if 
$ \vert \theta\vert \geq \eps_2 \vert \eta \vert,$ $0<\eps_1<\eps_2<1$  and where $\widehat{p}$ is the Fourier transform of 
$p$ with respect to $x$, whereas $\psi \in C^\infty(\xR^d)$ is a cut-off function such that 
$\psi(\xi) = 0$ for $\vert \xi \vert \leq1$ and $\psi(\xi) = 1$ for $\vert \xi \vert \geq2.$

\begin{theo}\label{continuite}
Let $s_0> \frac{d}{2}.$ For all $s \in \xR$ there exists a constant $C>0$ such that for all 
$p \in \Gamma^{m,\sigma,s_0}$ and all $u \in \mathcal{H}^{\sigma,s+m}$ we have
$$
\Vert T_p u\Vert_{\mathcal{H}^{\sigma,s}} \leq C \mathcal{N}_{s_0}(p) \Vert  u\Vert_{\mathcal{H}^{\sigma,s+m}}.
$$
\end{theo}
\begin{proof}
(i) Suppose $p(x,\xi) = a(x) h(\xi)$ where $h$ 
is homogeneous of degree $m$, so that 
$h(\xi) = \vert \xi \vert^m \widetilde{h}\big( \frac{\xi}{ \vert \xi \vert}\big)$.

In the formula \eqref{para} on the support of $\chi$ we have 
$\vert \xi \vert \sim \vert \eta \vert $ so that we can write
$$
e^{\sigma \vert \xi \vert} \langle \xi \rangle^s \vert \widehat{T_pu}(\xi)\vert \leq C\sum_{\eta \in \xZ^d}
e^{\sigma \vert \xi -\eta\vert} \vert \widehat{a}(\xi-\eta)\vert  \psi(\eta)  \langle \eta \rangle^{s+m}   \vert \widetilde{h}\big( \frac{\eta}{ \vert \eta \vert}\big) \vert e^{\sigma \vert \eta\vert} \vert \widehat{u}(\eta)\vert .
$$
Since the right-hand side is a convolution, we have
$$
\Vert e^{\sigma \vert \xi \vert} \langle \xi \rangle^s   \widehat{T_pu} \Vert_{\ell^2} \leq \Vert e^{\sigma \vert \cdot \vert}     \widehat{a} \Vert_{\ell^1}
\Vert \widetilde{h} \Vert_{L^\infty(\xS^{d-1})} \Vert u \Vert_{\mathcal{H}^{\sigma,s+m}}.
$$
For $s_0 > \frac{d}{2}$, it follows from the Cauchy-Schwarz inequality that
$$
\Vert e^{\sigma \vert \cdot \vert}     \widehat{a} \Vert_{\ell^1} \leq C\Vert e^{\sigma \vert \cdot \vert}   \langle \cdot\rangle^{s_0}    \widehat{a} \Vert_{\ell^2},
$$
so
\begin{equation}\label{cont2}
\Vert e^{\sigma \vert \xi \vert} \langle \xi \rangle^s   \widehat{T_pu}(\xi)\vert \Vert_{\ell^2} \leq C\Vert e^{\sigma \vert \xi \vert}   \langle \xi \rangle^{s_0}    \widehat{a} \Vert_{\ell^2} \Vert \widetilde{h} \Vert_{L^\infty(\xS^{d-1})} \Vert u \Vert_{\mathcal{H}^{\sigma,s+m}}. 
\end{equation}

(ii) Let $(\widetilde{h}_\nu)_{\nu \geq 1}$ be an orthonormal basis of 
$L^2(\xS^{d-1})$ with 
$-\Delta_\omega \widetilde{h}_\nu = \lambda_\nu \widetilde{h}_\nu.$   
Let $p (x,\omega) = \sum_\nu a_\nu(x) \widetilde{h}_\nu(\omega)$. Then 
$$a_\nu(x) = \int_{\xS^{d-1}} p(x, \omega) \overline{\widetilde{h}_\nu(\omega)}\, d\omega.$$ 
As a result, for $k \in \xN,$ we have
\begin{align*}
 \lambda_\nu^{2k} \langle D\rangle^{s_0} e^{\sigma\vert D\vert}a_\nu(x) &= \int_{\xS^{d-1}} \langle D\rangle^{s_0} e^{\sigma\vert D\vert}p(x, \omega)(- \Delta_\omega)^k \overline{\widetilde{h}_\nu(\omega)}\, d\omega\\
 &=\int_{\xS^{d-1}}  (- \Delta_\omega)^k\langle D\rangle^{s_0} e^{\sigma\vert D\vert}p(x, \omega) \overline{ \widetilde{h}_\nu(\omega)}\, d\omega,
 \end{align*}
whence
\begin{equation}\label{a-nu}
 \lambda_\nu^{2k}  \Vert \langle D\rangle^{s_0} e^{\sigma\vert D\vert} a_\nu\Vert_{L^2}
 \leq \Vert \Delta_\omega^k \langle D\rangle^{s_0} e^{\sigma\vert D\vert}p(x, \cdot)\Vert_{L^2(\xT^d \times\xS^{d-1})}.
 \end{equation}
On the other hand,  
 \begin{equation}\label{h-nu}
 \Vert \widetilde{h}_\nu\Vert_{L^\infty(\xS^{d-1})} \leq C\Vert \widetilde{h}_\nu\Vert_{H^{s_1}(\xS^{d-1})} \leq C'\lambda_\nu^{s_1}, \quad s_1= \frac{d-1}{2} + \eps.
 \end{equation}
By using \eqref{cont2}, \eqref{a-nu} and \eqref{h-nu}, we conclude that
\begin{align*}
 \Vert T_pu \Vert_{\mathcal{H}^{\sigma,s}} &\leq \sum_{\nu=1}^{+\infty} \Vert T_{a_\nu {h}_\nu}u \Vert_{\mathcal{H}^{\sigma,s}} \leq C \sum_{\nu=1}^{+\infty} \Vert a_\nu \Vert_{\mathcal{H}^{\sigma, s_0}} \Vert \widetilde{h}_\nu \Vert_{L^\infty(\xS^{d-1})} \Vert u \Vert_{\mathcal{H}^{\sigma,s+m}},\\
 & \leq C \sum_{\nu=1}^{+\infty} \lambda_\nu^{-2k} \Vert \Delta_\omega^k \langle D\rangle^{s_0} e^{\sigma\vert D\vert}p(x, \cdot)\Vert_{L^2(\xT^d \times\xS^{d-1})}\lambda_\nu^{s_1}  \Vert u \Vert_{\mathcal{H}^{\sigma,s+m}}.
 \end{align*}
As $\lambda_\nu \sim C_d \nu^{\frac{2}{d}} $, we have 
$\lambda_\nu^{s_1-2k} \sim C'_d \nu^{\frac{2}{d}(s_1-2k)}.$ If 
$k> \frac{d}{2}- \frac 1 4$ (and $\varepsilon>0$ in~\eqref{h-nu} is small enough), we have $2(s_1-2k)<-d $ and the series in $\nu$ converges. 
Eventually, we obtain that
$$
\Vert T_pu \Vert_{\mathcal{H}^{\sigma,s}} \leq C_d  \Vert \Delta_\omega^k \langle D\rangle^{s_0} e^{\sigma\vert D\vert}p(x, \cdot)\Vert_{L^2(\xT^d \times\xS^{d-1})}\Vert u \Vert_{\mathcal{H}^{\sigma,s+m}}.
$$
\end{proof}

In the sequel, we will also use the Littlewood-Paley decomposition (see \cite{MePise}) and consider the 
paraproduct. Notice that when $p=a$ is a function (thus independent of $\xi$), the paradifferential operator $T_a$ with symbol $a$ is called a paraproduct. Modulo a regularizing operator it can be defined as follows.

Consider a partition of unity, 
$$ 1 = \sum_{j\geq -1} \chi_j(\xi), \qquad \forall j \geq 1, \chi_j (x) = \chi_0(2^{-j} \xi ),$$
with 
$ \chi_0 \in C^\infty_0 ( \{ \frac 1 2\leq \vert \xi\vert \leq 2\} )$.  Recall that if $a, u$ are two functions, 
then 
$$ \Delta_j u = \chi_j (D_x) u, \qquad S_j (u) = \sum_{-1\leq k \leq j-1} \Delta _k u, \text{ and } T_a u = \sum_{j\geq 2} S_{j-2}(a) \Delta_j u.$$  
\begin{prop}\label{prod1}
Consider three real numbers $\alpha, \beta, \gamma$ such that
$$
\alpha \leq \gamma, \quad \alpha < \beta + \gamma - \frac{d}{2}.
$$
There exists $C>0$ such that for all $a\in \mathcal{H}^{\sigma,\beta}, u \in \mathcal{H}^{\sigma, \gamma}$, we have
$$
\Vert T_a u \Vert_{\mathcal{H}^{\sigma, \alpha}} \leq C\Vert  a   \Vert_{\mathcal{H}^{\sigma, \beta}}\Vert   u \Vert_{\mathcal{H}^{\sigma, \gamma}}.
$$
\end{prop}
\begin{proof} 
Case 1: $\beta >\frac{d}{2}.$ 

We apply Theorem~\ref{continuite} with $s_0 = \beta, p = a, m=0.$ Since 
$\alpha \leq \gamma$, it follows that
$$
\Vert T_a u \Vert_{\mathcal{H}^{\sigma, \alpha}} \leq C\Vert  a   \Vert_{\mathcal{H}^{\sigma, \beta}}\Vert   u \Vert_{\mathcal{H}^{\sigma, \alpha}} \leq C\Vert  a   \Vert_{\mathcal{H}^{\sigma, \beta}}\Vert   u \Vert_{\mathcal{H}^{\sigma, \gamma}} $$
 
Case  2: $ \beta \leq \frac{d}{2}.$ 

Let $\eps>0$ be such that $ \alpha + \eps <  \beta + \gamma - \frac{d}{2}.$ We write 
as above 
$$T_a u = \sum_q S_{q-2}(a) \Delta_q u:= \sum_q v_q.$$ 
It follows from Lemma~\ref{produit} that
$$
2^{q \alpha}\Vert v_q \Vert_{\mathcal{H}^{\sigma,0}} \leq 2^{q \alpha}\Vert   S_{q-2}(a) \Vert_{\mathcal{H}^{\sigma, \frac{d}{2} + \eps}} \Vert \Delta_qu  \Vert_{\mathcal{H}^{\sigma, 0}} \leq 2^{q( \alpha - \gamma)} c_q \Vert  u  \Vert_{\mathcal{H}^{\sigma, \gamma}}.
$$
We next notice that
\begin{align*}
 \Vert S_{q-2}(a) \Vert_{\mathcal{H}^{\sigma, \frac{d}{2} + \eps}}  &\leq  \sum_{p=-1}^{q-2} \Vert \Delta_p a\Vert_{\mathcal{H}^{\sigma, \frac{d}{2} + \eps}} \leq  \sum_{p=-1}^{q-2} 2^{p(\frac{d}{2} + \eps)} \Vert \Delta_p a\Vert_{\mathcal{H}^{\sigma, 0}},\\
  & \leq \Big(\sum_{p=-1}^{q-2} 2^{2p(\frac{d}{2} + \eps - \beta)}\Big)^\mez\Big(\sum_{p=-1}^{q-2} 2^{2p\beta} \Vert \Delta_p a\Vert^2_{\mathcal{H}^{\sigma, 0}}\Big)^\mez  \\
  & \leq C 2^{q(\frac{d}{2} + \eps - \beta)}\Vert a \Vert_{\mathcal{H}^{\sigma, \beta}}.
\end{align*}
This implies that 
$$
2^{q \alpha}\Vert v_q \Vert_{\mathcal{H}^{\sigma,0}} \leq C2^{q( \alpha + \eps -(\beta + \gamma - \frac{d}{2}))} \Vert a \Vert_{\mathcal{H}^{\sigma, \beta}}  \Vert  u  \Vert_{\mathcal{H}^{\sigma, \gamma}},
$$
and the wanted result follows using that $\alpha + \eps -(\beta + \gamma - \frac{d}{2}) < 0.$
\end{proof}

We are in position to state the bilinear estimates.


\begin{prop}\label{uv}
\begin{itemize}
\item[\rm{(i)}] \label{HorNLi)}
Consider three real numbers $s_1,s_2,s_3$ such that
$$
s_1+s_2\ge 0,\quad s_3\le \min\{s_1,s_2\}, \quad s_3 < s_1+s_2-\frac{d}{2}.
$$
Then there exists $C>0$ such that for all $\sigma\ge 0$,
\be\label{piH}
\Vert u_1u_2 \Vert_{\mathcal{H}^{\sigma,s_3}} \leq C \Vert u_1 \Vert_{\mathcal{H}^{\sigma,s_1}}\Vert  u_2 \Vert_{\mathcal{H}^{\sigma,s_2}}.
\ee
\item[\rm{(ii)}] \label{HorNLii)} For all $s>d/2$, there exists $C>0$ such that for all $\sigma\ge 0$, 
$$
\Vert u_1u_2 \Vert_{\mathcal{H}^{\sigma,s}} \leq C \Vert u_1 \Vert_{\mathcal{H}^{\sigma,s}}
\Vert  u_2 \Vert_{\mathcal{H}^{\sigma,s}}.
$$
\item[\rm{(iii)}] \label{HorNLiii)} For all $s_0>d/2$ and all $t\ge 0$, there exists $C>0$ such that for all $\sigma\ge 0$, 
$$
\Vert u_1u_2 \Vert_{\mathcal{H}^{\sigma,t}} \leq C \Vert u_1 \Vert_{\mathcal{H}^{\sigma,s_0}}
\Vert  u_2 \Vert_{\mathcal{H}^{\sigma,t}}
+C \Vert u_2 \Vert_{\mathcal{H}^{\sigma,s_0}}
\Vert  u_1 \Vert_{\mathcal{H}^{\sigma,t}}.
$$
\end{itemize}
\end{prop}
 \begin{proof}
(i) One writes 
$ u_1 u_2 = T_{u_1} u_2 + T_{u_2} u_1 + R(u_1, u_2)$. 
The first two terms of the right-hand side are estimated by the Proposition \ref{prod1} 
by taking $\alpha = s_1$ (resp. $\alpha = s_2$), $\beta = s_2$ (resp. $\alpha = s_1$). 
It remains to estimate $R(u_1, u_2).$ We claim that 
\begin{equation*}
\Vert R(u_1, u_2) \Vert_{\mathcal{H}^{\sigma, s_3}}  \leq \Vert R(u_1, u_2) \Vert_{\mathcal{H}^{\sigma, s_1 + s_2 - \frac{d}{2}}} \leq C\Vert u_1 \Vert_{\mathcal{H}^{\sigma, s_1}}\Vert u_2 \Vert_{\mathcal{H}^{\sigma, s_2}}.
\end{equation*}
To see this, recall that
$$
R(u_1,u_2) = \sum_{\vert r-q \vert\leq 2} \Delta_r u_1 \Delta_q u_2 = \sum_{q} R_q, \quad R_q = \sum_{\vert r-q\vert \leq 2} \Delta_r u_1 \Delta_q u_2 .
$$
As the spectrum of $R_q$ is contained in a ball of radius $C2^q$ and not in an annulus, 
it is not enough to directly estimate $\Vert R_q \Vert_{L^2}$ 
because that would require that $s_1+s_2-\frac{d}{2}$ be positive (which is not necessarily the case). So we further decompose 
by writing
$$
R = \sum_{p\geq -1} \Delta_p R = \sum_{p\geq -1} \sum_{q} \Delta_p R_q =\sum_{p\geq -1} \sum_{q\geq p-n_0} \Delta_p R_q,
$$
since for $p> q+ n_0$ we have $\Delta_p R_q =0$. 
All the terms of $R_q$ are bounded in the same way, 
so that it is enough to consider the one where $r=q$. 
We will prove that
\begin{equation}\label{Rq}
\Vert \Delta_p (\Delta_q u_1 \Delta_q u_2)\Vert_{\mathcal{H}^{\sigma,0}} \leq C 2^{p\frac{d}{2}}\Vert \Delta_q u_1\Vert _{\mathcal{H}^{\sigma,0}} \Vert  \Delta_q u_2  \Vert_{\mathcal{H}^{\sigma,0}}.
\end{equation}
Indeed, set $U = \Delta_pe^{\sigma \vert D \vert}(fg).$ We have 
$$\vert \widehat{U}(\xi)\vert \leq \varphi(2^{-p} \xi) \vert (\widehat{e^{\sigma \vert D \vert}f})\vert \star \vert (\widehat{e^{\sigma \vert D \vert}g})(\xi)\vert.$$ Therefore, 
\begin{align*}
\Vert U \Vert_{L^2} &\leq C \Vert \varphi(2^{-p} \cdot)\Vert_{\ell^2} \Vert \vert (\widehat{e^{\sigma \vert D \vert}f})\vert \star \vert (\widehat{e^{\sigma \vert D \vert}g})(\xi)\vert \Vert_{\ell^\infty},\\
& \leq C 2^{p\frac{d}{2}}  \Vert  \widehat{e^{\sigma \vert D \vert}f} \Vert_{\ell^2}  \Vert  \widehat{e^{\sigma \vert D \vert}g}  \Vert_{\ell^2} \leq C'  2^{p\frac{d}{2}}\Vert  f \Vert_{\mathcal{H}^{\sigma,0}}\Vert   g\Vert_{\mathcal{H}^{\sigma,0}}.
\end{align*}
We deduce from \eqref{Rq} that
$$
\Vert \Delta_p R_q\Vert_{\mathcal{H}^{\sigma,0}} \leq C 2^{p\frac{d}{2}} 2^{-q(s_1 + s_2)}   \big(2^{ q s_1 } \Vert \Delta_q u_1\Vert _{\mathcal{H}^{\sigma,0}}\big)  \big(2^{ q s_2 } \Vert  \Delta_q u_2  \Vert_{\mathcal{H}^{\sigma,0}}\big).
$$

Now, set $r_p = 2^{-p\frac{d}{2}} 2^{p(s_1 +s_2)} \Vert \Delta_p R \Vert _{L^2} $ 
so that, using the previous inequality,
$$
r_p \leq C_1\sum_{q \geq p-n_0} 2^{(p-q)(s_1 +s_2)}g_q, \quad g_q = 2^{qs_1}\Vert \Delta_q u_1 \Vert_{\mathcal{H}^{\sigma,0}} 2^{q s_2}\Vert \Delta_q u_2 \Vert_{\mathcal{H}^{\sigma,0}} .
$$
Notice that
$$
g_q \leq C c_q d_q \Vert u_1 \Vert_{\mathcal{H}^{\sigma,s_1}}\Vert u_2 \Vert_{\mathcal{H}^{\sigma,s_2}}, \quad (c_q), (d_q) \in l^2.
$$
We infer that
$$
r_p^2 \leq C_1 \sum_{q \geq p-n_0} 2^{(p-q)(s_1 +s_2)}\sum_{q \geq p-n_0} 2^{(p-q)(s_1 + s_2)}g_q^2 \leq C_2 \sum_{q \geq p-n_0} 2^{(p-q)(s_1 + s_2)}g_q^2$$
since $s_1 +s_2 >0.$ 
Then, 
\begin{align*}
  \sum_{p\geq -1} r_p^2 &\leq C_2 \sum_{q \geq -1} \Big( \sum_{p \leq q+n_0}2^{-(q-p)(s_1 + s_2)}\Big) g_q^2 \leq C_3 \sum_{q\geq -1} g_q^2,\\
  & \leq C_4 \Big(\sum_{q\geq -1}  c^2_q d^2_q \Big)\Vert \Delta_q u_1 \Vert^2_{\mathcal{H}^{\sigma,0}}  \Vert \Delta_q u_2 \Vert^2_{\mathcal{H}^{\sigma,0}}.
  \end{align*}
We conclude that
$$ \Vert \Delta_p R \Vert _{L^2} =   2^{-p(s_1 + s_2 - \frac{d}{2})} r_p,  \quad \sum_{p\geq -1} r_p^2 \leq C_5\Vert \Delta_q u_1 \Vert^2_{\mathcal{H}^{\sigma,0}} 2^{q s_2}\Vert \Delta_q u_2 \Vert^2_{\mathcal{H}^{\sigma,0}}. 
$$
Since the spectrum of $\Delta_p R$ is contained in an annulus, this proves the inequality and 
completes  the proof of the claim (i).

(ii) follows from (i) setting $ s_1 = s_2 = s_3=s.$

To prove (iii)   we use the following facts: 
$$ \langle \xi \rangle^t \les  \langle \xi-\eta \rangle^t + \langle \eta \rangle^t, e^{\sigma\vert \xi \vert} \leq  e^{\sigma\vert \xi-\eta \vert} e^{\sigma\vert \eta \vert}, \Vert \varphi\star \theta\Vert_{L^2} \leq \Vert \varphi   \Vert_{L^1}\Vert  \theta\Vert_{L^2}$$ and eventually $\Vert \varphi\Vert_{L^1} \les  \Vert \langle \xi \rangle^{s_0}\varphi\Vert_{L^2}$ if $s_0> \frac{d}{2}.$
\end{proof}

As a byproduct, we obtain the following result.
\begin{prop}\label{a-Ta}
Let $ \alpha_0, \alpha_1, \alpha_2$ be real numbers such that 
$$
\alpha_1 + \alpha_2 >0, \quad \alpha_0 \leq \alpha_1, \quad \alpha_0 < \alpha_1 + \alpha_2 - \frac{d}{2}.
$$
There exists $C>0$ such that for all $a \in {\mathcal{H}^{\sigma, \alpha_1}}$ and for 
all 
$u \in {\mathcal{H}^{\sigma, \alpha_2}}$, 
$$
\Vert au-T_a u \Vert_{\mathcal{H}^{\sigma, \alpha_0}} \leq C \Vert a \Vert_{\mathcal{H}^{\sigma, \alpha_1}} \Vert u \Vert_{\mathcal{H}^{\sigma, \alpha_2}}.
$$
\end{prop}
The proof is identical to that of the above proposition, 
for it is sufficient to notice that $au -T_au = T_ua + R(a,u).$

\begin{coro}
Consider two real numbers $s_0>d/2$ and $s\geq 0$. Let  $P$ be a polynomial of 
degree $m \geq 2$ such that $P(0)=0 $ and let 
$u \in \mathcal{H}^{\sigma,s}  \cap \mathcal{H}^{\sigma,s_0}$. 
Then $P(u) \in  \mathcal{H}^{\sigma,s}$ and
$$\Vert P(u)\Vert_{\mathcal{H}^{\sigma,s}}
\leq Q( \Vert u \Vert_{\mathcal{H}^{\sigma,s_0}}) \Vert u \Vert_{\mathcal{H}^{\sigma,s}}$$
where $Q$ is a polynomial of    degree $m-1.$
\end{coro}
\begin{proof}
This corollary is an immediate consequence of statement (i) in Proposition~\ref{uv}.
\end{proof}

\begin{prop}\label{est-f(u)bis}
Consider three real numbers $s_0>d/2$, $s \geq 0$, $M_0>0$ and let $f$ be a holomorphic function in the ball 
$\{z \in \xC: \vert z \vert< M_0\}$, such that $f(0)=0$. There exists $\eps_0>0$ such that for all  $\sigma>0$, if 
$u \in \mathcal{H}^{\sigma,s} \cap \mathcal{H}^{\sigma,s_0}$ satisfies  
$\Vert u \Vert_{\mathcal{H}^{\sigma, s_0}} \leq \eps_0$, then 
$f(u)$ belongs to $\mathcal{H}^{\sigma, s}$. Moreover, there exists   $C>0$ depending only on $f, s,s_0,  \eps_0$ such that
$$
\Vert f(u) \Vert_{\mathcal{H}^{\sigma,s}} \leq C \Vert u \Vert_{\mathcal{H}^{\sigma, s}} .
$$
\end{prop}
\begin{proof} 
Set $C_s =  2^s C(d,s_0)$. It follows from  Proposition~\ref{uv} an from an induction that    for all $n \geq 1,$
\begin{equation}\label{u^n}
  \Vert u^n \Vert_{ \mathcal{H}^{\sigma,s}} \leq ( 2 C_{\text{max}(s,s_0)})^{n-1}\Vert u  \Vert^{n-1}_{\mathcal{H}^{\sigma,s_0}}\Vert  u \Vert_{\mathcal{H}^{\sigma,s}}.
  \end{equation}
For $\vert z \vert <M_0$ we can write $f(z) = \sum_{n=1}^{+ \infty} a_n z^n$ where $a_n$ is such that $\vert a_n \vert \leq K^n, K>0$. We shall show that the   series   $\sum a_n u^n$ is normaly convergent  in $\mathcal{H}^{\sigma,s}$. Indeed according to   \eqref{u^n} and the hypothesis we have
$$
\Vert a_n u^n \Vert_{\mathcal{H}^{\sigma,s}} \leq K(2 C_{\text{max}(s,s_0)}K\eps_0)^{n-1}\Vert u \Vert_{\mathcal{H}^{\sigma, s}}.
$$
We have just to take $\eps_0$ small enough, so that 
$2 C_{\text{max}(s,s_0)}K \eps_0<1$. 
Since, on the other hand,  $\Vert u \Vert_{L^\infty(\xR^d)} \leq C(d,s_0)\Vert u \Vert_{\mathcal{H}^{\sigma, s_0}}$, taking moreover  $\eps_0$ such that $C(d,s_0) \eps_0<M_0$ we will have,
$$
\Vert f(u) \Vert_{ \mathcal{H}^{\sigma,s}} \leq K\Big(\sum_{n=1}^{+ \infty}(2K C_{\text{max}(s,s_0)})^{n-1}\Vert u  \Vert^{n-1}_{\mathcal{H}^{\sigma,s_0}} \Big)\Vert  u \Vert_{\mathcal{H}^{\sigma,s}}.
$$
This completes the proof.
\end{proof}
\begin{coro}\label{est-f(u)ter}
Consider three real numbers $s_0>d/2$, $t \geq 0$, $M_0>0$ and let $f$ be a holomorphic function in the ball $\{z \in \xC: \vert z \vert< M_0\}.$ There exists $\eps_0>0, C>0$  such that for all  $\sigma>0$, if $u_1, u_2 \in \mathcal{H}^{\sigma,t} \cap \mathcal{H}^{\sigma,s_0}$ satisfy  
$\Vert u_j \Vert_{\mathcal{H}^{\sigma, s_0}} \leq \eps_0$, for $j=1,2$  then, 
$$
\Vert f(u_1)-f(u_2)\Vert_{\mathcal{H}^{\sigma, t}} \leq     C\Big(\Vert  u_1 - u_2 \Vert_{\mathcal{H}^{\sigma, t}} + (\Vert u_1\Vert_{\mathcal{H}^{\sigma, t}}+\Vert u_2\Vert_{\mathcal{H}^{\sigma, t}})\Vert  u_1 - u_2 \Vert_{\mathcal{H}^{\sigma, s_0}}\Big).
$$ 
\end{coro}
\begin{proof}
Set $g(z)= f'(z)- f'(0).$ Then $g$ is holomorphic in the set $\{z\in \xC: \vert z \vert < M_0\}$ and $g(0)=0.$
Then one can write
$$f(u_1) - f(u_2) = f'(0)(u_1 - u_2) + (u_1 - u_2)\int_0^1 g(\lambda u_1 + (1- \lambda) u_2)\, d \lambda.$$
Since 
$$\Vert  \lambda u_1 +  (1-\lambda)u_2\Vert_{\mathcal{H}^{\sigma, s_0}}\leq \lambda \Vert  u_1\Vert_{\mathcal{H}^{\sigma, s_0}}+ (1-\lambda)\Vert  u_2\Vert_{\mathcal{H}^{\sigma, s_0}}\leq \eps_0,$$
we can apply proposition~\ref{est-f(u)bis} with $s=t$ and $s=s_0$ to $g$ and write
\begin{align*}
&\Vert  g(\lambda u_1 + (1- \lambda) u_2)\Vert_{\mathcal{H}^{\sigma, t}}\leq C(\Vert  u_1\Vert_{\mathcal{H}^{\sigma, t}}+ \Vert  u_2\Vert_{\mathcal{H}^{\sigma, t}}),\\
&\Vert  g(\lambda u_1 + (1- \lambda) u_2)\Vert_{\mathcal{H}^{\sigma, s_0}}\leq C.
\end{align*}
Using Proposition~\ref{uv} and the above inequalities, we get
\begin{multline*}
 \Vert f(u_1) - f(u_2)\Vert_{\mathcal{H}^{\sigma,t}} \\
 \leq \vert f'(0)\vert \Vert u_1-u_2\Vert_{\mathcal{H}^{\sigma,t}} + C\big(\Vert u_1-u_2\Vert_{\mathcal{H}^{\sigma,t}} + (\Vert u_1\Vert_{\mathcal{H}^{\sigma,t}}\Vert+ u_2\Vert_{\mathcal{H}^{\sigma,t}}) \Vert u_1-u_2\Vert_{\mathcal{H}^{\sigma,s_0}}\big).
\end{multline*}
This completes the proof.
\end{proof}

Recall (see Definition~\ref{def-EF}) that  for  $\lambda\in [0,1]$ and $\mu \in \xR$, we have introduced the spaces
\begin{equation*}
\begin{aligned}
E^{\lambda,\mu} &= \{u: e^{\lambda z\vert D_x \vert} 
u \in C^0([-h,0], \mathcal{H}^{\lambda h,\mu})\},\\
F^{\lambda,\mu} &=  \{u: e^{\lambda  z\vert D_x \vert} u \in 
L^2(I_h, \mathcal{H}^{\lambda h,\mu})\},\quad I_h =(-h,0),\\
\mathcal{X}^{\lambda ,\mu} &= E^{\lambda,\mu} \cap F^{\lambda,\mu+ \mez},
\end{aligned}
\end{equation*}
equipped with their natural  norms. \\
 
We have several bilinear estimates by using these norms.

 \begin{lemm}\label{uvF}
\begin{itemize}
\item[\rm{(i)}] \label{uvFi} Consider three real numbers $s_1,s_2,s_3$ such that
$$
s_1+s_2\ge 0,\quad s_3\le \min \{s_1,s_2\},\quad   s_3 < s_1+s_2-\frac{d}{2}.
$$
 Then there exists $C>0$ such that for any $\lambda\in [0,1]$, 
$$
\Vert u_1u_2 \Vert_{F^{\lambda,s_3}} \leq C \Vert u_1 \Vert_{E^{\lambda,s_1}}\Vert  u_2 \Vert_{F^{\lambda,s_2}}.
$$
\item[\rm{(ii)}] \label{uvFii} For all $s>d/2$, there exists $C>0$ such that for any $\lambda\in [0,1]$, 
\begin{align*}
&\Vert u_1u_2 \Vert_{F^{\lambda,s}} \leq C \Vert u_1 \Vert_{E^{\lambda,s}}\Vert  u_2 \Vert_{F^{\lambda,s}},\\
&\Vert u_1u_2 \Vert_{E^{\lambda,s}} \leq C \Vert u_1 \Vert_{E^{\lambda,s}}\Vert  u_2 \Vert_{E^{\lambda,s}}.
\end{align*}
\item [\rm{(iii)}] \label{uvFiii}
Let $s_0>d/2$ and $t \geq 0$. There exists  $C>0$ such that
\be\label{tame-EF}
\Vert u_1u_2 \Vert_{F^{\lambda,t}} \leq C \Vert u_1 \Vert_{E^{\lambda,s_0}}\Vert  u_2 \Vert_{F^{\lambda,t}} + C\Vert u_1  \Vert_{F^{\lambda,t}}\Vert  u_2\Vert_{E^{\lambda,s_0}}.
\ee
\item [\rm{(iv)}] \label{uvFiv} For any $\mu\in \xR$ and any $\lambda\in [0,1]$, there exists a constant $C>0$ such that
$$
\lA u\rA_{E^{\lambda,\mu}}\le C \lA \partial_z u\rA_{F^{\lambda,\mu-\mez}}+\lA u\rA_{F^{\lambda,\mu+\mez}}.
$$
\end{itemize}
 \end{lemm}
\begin{proof} 
(i) We  first use  Proposition~\ref{uv} with fixed  $z$ and $\sigma = \lambda(z+h)$, then we we bound   the  $L^2$ norm  in $z$ of the  products by the   $L^\infty$ and $L^2$ norms. The proof of statement (ii) and (iii) are similar and (iv) follows directly from Lemma~\ref{lions}.
\end{proof}

\subsection{Commutators.}
In this section we study the commutators between  a Fourier multiplier and the multiplication by a function $a\in \mathcal{H}^{\sigma, s}$. In our applications, the Fourier multiplier will be $$b(D_x) = J_n = \chi \left(\frac{|D_x|} n\right).$$ 
Recall that   $S^m_{1,0}$ is the  class of symbols $p$   for which the seminorms
  $$\sup_{(x,\xi)\in \xR^d \times \xR^d} \langle \xi \rangle^{\vert \alpha \vert-m}\vert\partial_x^\beta \partial_{\xi}^\alpha p(\xi)\vert  $$ 
 are finite for every $(\alpha, \beta) \in \xN^d\times \xN^d.$
\begin{prop}\label{commutat}
    Let $ \nu ,s \in \xR$ be such that $\nu \geq 0$, $2s+ \nu> 1 $ and $\nu +s > 1+ \frac {d}{2}.$  Let $b(\xi) \in S^\nu _{1,0}.$ There exists $C>0$ {\rm(}depending only on   $m$ and on a finite number on seminorms of   $b${\rm )}  such that for all $\sigma>0$    and all  $a\in \mathcal{H}^{\sigma, \nu + s }$   we have
\begin{equation}\label{est-commut3}
 \Vert \big[b(D), a\big] u \Vert_{\mathcal{H}^{\sigma,s}} \leq C \Vert  a \Vert_{\mathcal{H}^{\sigma, \nu +s}} \Vert u \Vert_{\mathcal{H}^{\sigma, \nu +s-1}}, \quad \forall u \in \mathcal{H}^{\sigma,\nu +s-1}. 
 \end{equation}
\end{prop}
The proof is based on the following lemma where as before $T_a$ denotes the paraproduct.

\begin{lemm}\label{commut1}
Let $ s_0> \frac{d}{2}, \, m \in \xR, $ such that $-s_0< m \leq s_0+1$ and $p(\xi) \in S^m_{1,0}.$ There exists $C>0$ 
{\rm (}depending only on   $s_0,m$ and on a finite number on seminorms of $p${\rm )} such that for all $\sigma>0$ and  all  $a$ satisfying $\nabla_x a \in \mathcal{H}^{\sigma, s_0}$ we have
\begin{align*}
&\text{\rm {(i)}}\quad \Vert \big[p(D), T_a\big] u \Vert_{\mathcal{H}^{\sigma,0}} \leq C \Vert \nabla_x a \Vert_{\mathcal{H}^{\sigma, s_0}} \Vert u \Vert_{\mathcal{H}^{\sigma, m-1}}, \quad \forall u \in \mathcal{H}^{\sigma,m-1},\\
 &\text{\rm {(ii)}}\quad \Vert \big[p(D), a\big] u \Vert_{\mathcal{H}^{\sigma,0}} \leq C \Vert  a \Vert_{\mathcal{H}^{\sigma, s_0 +1}} \Vert u \Vert_{\mathcal{H}^{\sigma, m-1}}, \quad \forall u \in \mathcal{H}^{\sigma,m-1}.
\end{align*}
 \end{lemm}
\begin{proof} 
(i)\, We have
\begin{align*}
\big[p(D), T_a\big] u =&\, \sum_{j \geq 1} \big[ p(D), S_{j-3}(a)\big] \Delta_j u \\
=&\,  \sum_{j \geq 1} \big[ p(D)\varphi_1(2^{-j}D), S_{j-3}(a)\big] \Delta_j u\\
 : =&\,  \sum_{j \geq1} w_j,
\end{align*}
where  $\text{supp} \varphi_1 \subset \{\xi: \frac{1}{4} \leq \vert \xi \vert \leq 4\}.$ Set $p_j(\xi) = p(\xi) \varphi_1( 2^{-j}\xi).$ 
Opening the bracket we have
$$
e^{\sigma \vert \xi \vert} \widehat{w_j}(\xi) = e^{\sigma \vert \xi \vert}
\sum_{\zeta \in \xZ^d}\widehat{S_{j-3}(a)}( \xi- \zeta) (p_j(\xi) - p_j(\zeta)) \widehat{\Delta_j u }(\zeta).
$$
Since on the support of $\varphi_1(2^{-j}\xi)$ we have $\vert \xi \vert \sim 2^j$   there exists   $C>0$ such that $\vert \nabla_\xi p_j(\xi) \vert \leq C 2^{j(m-1)}$ for all $\xi \in \xR^d$ and all  $j \geq -1.$ It follows that
$$
\vert p_j(\xi) - p_j(\zeta) \vert \leq C 2^{j(m-1)}\vert \xi- \zeta \vert
$$
for all  $\xi, \zeta \in \xR^d$. Then
$$
e^{\sigma \vert \xi \vert} \vert \widehat{w_j}(\xi)\vert \leq C \sum_{\zeta \in \xZ^d}
e^{\sigma \vert \xi - \zeta \vert} \vert \widehat{S_{j-3}(a)}( \xi- \zeta) \vert \,  \vert \xi - \zeta \vert\, 2^{j(m-1)}
\vert  e^{\sigma \vert \zeta \vert}\Delta_j u(\zeta) \vert .
$$
The right hand side  being a  convolution we can write,
\begin{align*}
  \Vert e^{\sigma \vert D_x\vert}   w_j \Vert_{L^2(\xR^d)} &\leq 
  C \Big(\sum_{\xi \in \xZ^d}e^{\sigma \vert \xi \vert} \vert \widehat{S_{j-3}(\vert D_x \vert a)}\vert \Big) 2^{j(m-1)}\Vert e^{\sigma \vert D_x \vert} \Delta_j u \Vert_{L^2},\\
  & \leq C' c_j \,  \Vert \nabla_x a \Vert_{\mathcal{H}^{\sigma, s_0}} \Vert u \Vert_{\mathcal{H}^{\sigma, m-1}}
  \end{align*}
  where   $ \sum_j c_j^2 < + \infty.$ This proves (i).  To prove (ii) we write
   \begin{align*}
 & \Vert \big[p(D), a\big] u \Vert_{\mathcal{H}^{\sigma,0}} \\
 \leq &\, C \Vert \big[p(D), T_a\big] u \Vert_{\mathcal{H}^{\sigma,0}} + \Vert p(D)(T_a-a)u\Vert_{\mathcal{H}^{\sigma,0}} + \Vert  (T_a-a) p(D)u\Vert_{\mathcal{H}^{\sigma,0}},\\
\leq   &\,  A_1 + A_2 + A_3.
  \end{align*}
  The   term $A_1$ is bounded by the right hand side of (ii) according to (i).  To estimate the term $A_2$ we use Lemma~\ref{a-Ta} with $\alpha_0= m, \alpha_1= s_0+1, \alpha_2= m-1.$ The same lemma with $\alpha_0 = 0, \alpha_1= s_0+1, \alpha_2= -1$ gives the estimate of $A_3.$
\end{proof}
\begin{proof}[Proof of Proposition~\ref{commutat}]
We write
$$\Vert [b(D),a]u\Vert_{\mathcal{H}^{\sigma, s}}\leq \Vert [\langle D\rangle^s b(D),a]u\Vert_{\mathcal{H}^{\sigma, 0}}+
\Vert [\langle D\rangle^s,a]b(D)u\Vert_{\mathcal{H}^{\sigma, 0}}.$$ 
Then we apply (ii) in Lemma~\ref{commut1} to the operators $p(D)= \langle D\rangle^s b(D)$ with $m= s + \nu, s_0= s+\nu-1$ and to $p(D)= \langle D\rangle^s$ with $m=s,s_0= s+\nu-1. $ We obtain \eqref{est-commut3}.
\end{proof}
 
\subsection{Estimates on the coefficients.}
In this paragraph, we prove some elementary estimates for the 
the derivatives of the functions $\rho$ and $\underline{\psi}$ introduced in Section~\ref{defiCOV}.  
 \begin{lemm}\label{est-rho}
 For all $t\in \xR$ there exists  $C>0$ such that  
 \begin{align*}
 &\Vert  \partial_z  \rho-1  \Vert_{L^\infty(I_h, \mathcal{H}^{\lambda h, t})} +  \Vert \nabla_x \rho \Vert_{L^\infty(I_h, \mathcal{H}^{\lambda h, t})} 
 +\lA \nabla_{x,z}^2 \rho \rA_{L^\infty(I_h, \mathcal{H}^{\lambda h, t-1})} \\
 &\qquad \leq C \Vert \eta \Vert_{ \mathcal{H}^{\lambda h, t +1}},\\
 &   \Vert  \partial_z  \rho-1  \Vert_{L^2(I_h, \mathcal{H}^{\lambda h, t+\mez})} +  \Vert \nabla_x \rho \Vert_{L^2(I_h, \mathcal{H}^{\lambda h, t+\mez})}
 +\lA \nabla_{x,z}^2 \rho \rA_{L^2(I_h, \mathcal{H}^{\lambda h, t-\mez})}\\
 &\qquad \leq C \Vert \eta \Vert_{ \mathcal{H}^{\lambda h, t +1}}.
 \end{align*}
\end{lemm}
\begin{proof}
We have
\begin{equation}\label{drho}
\begin{aligned}
\partial_z \rho(x,z)  &= 1 + \frac{1}{h} e^{z \vert D_{x} \vert} \eta(x) + \frac{1}{h}(z+h) e^{z \vert D_{x} \vert} \vert D_{x} \vert \eta(x),\\
\nabla_{x} \rho(x,z)  &= \frac{1}{h}(z+h) e^{z \vert D_{x}\vert} \nabla_{x} \eta(x). 
\end{aligned}
\end{equation}
Then the first set of estimates follow from the fact that, since $z\le 0$, we have 
$e^{z\la \xi\ra}\le 1$ so 
the Fourier multiplier $e^{z\la D_x\ra}$ 
is bounded from $\mathcal{H}^{\lambda h,s}$ to itself for any   $s$. To prove the second set of estimates, we use the special choice for $\rho$ involving the operator $e^{z\la D_x\ra}$. 
Notice that, by using Fubini's lemma,
$$
\int_{I_h}\sum_{\xi \in \xZ^d} \la \xi\ra e^{2z\la \xi\ra}\big\vert \widehat{f}(\xi)\big\vert^2 dz
=\sum_{\xi \in \xZ^d}\left(\int_{I_h} \la \xi\ra e^{2z\la \xi\ra}\, dz\right)\big\vert \widehat{f}(\xi)\big\vert^2
\le \sum_{\xi \in \xZ^d}\mez \big\vert \widehat{f}(\xi)\big\vert^2.
$$
Consequently, by applying this result with $f=\langle D_x\rangle^{s}e^{\lambda h\la D_x\ra}u$ and 
using the Plancherel identity, we get
$$
\lA \la D_x\ra e^{z\la D_x\ra} u\rA_{L^2(I_h, \mathcal{H}^{\lambda h,s})}\les \lA u\rA_{ \mathcal{H}^{\lambda h,s+ \mez}}.
$$
Using again \eqref{drho}, we complete the proof of the lemma.
\end{proof}

\begin{defi}\label{Erond}
For $k=1,2$  let $\mathcal{E}_k$ be the set of functions $\varphi(z, \eta)$ defined on,
 $$ \mathcal{A}=  I_h\times \{\eta \in \mathcal{H}^{\lambda h, s+ \mez}:\Vert \eta \Vert_{\mathcal{H}^{\lambda h, s}} \leq \overline{\eps} \leq 1\},$$
 such that for every $\frac{d}{2}< t \leq s-k$ there exists $C>0$ satisfying
$$
\Vert \varphi(\cdot, \eta)\Vert_{L^\infty(I_h, \mathcal{H}^{\lambda h,t})}
\leq C \Vert\eta\Vert_{\mathcal{H}^{\lambda h, t+k}}
$$
and 
$$
\Vert \varphi(\cdot, \eta)
\Vert_{L^2(I_h, \mathcal{H}^{\lambda h,t+1})} \leq C \Vert \eta\Vert_{\mathcal{H}^{\lambda h, t +k+ \frac{1}{2}}}.$$
\end{defi}

\begin{lemm}\label{algebre0}
\begin{itemize}
\item[\rm{(1)}] \label{item:(1)} $\mathcal{E}_k$ is an algebra and $\mathcal{E}_1 \subset \mathcal{E}_2$.
\item[\rm{(2)}] \label{(2)} If $f$ is a holomorphic function in a ball $\{Z\in \xC: \vert Z \vert < M_0\}$ such that $f(0)=0$ and if $\varphi\in \mathcal{E}_k$, where $\overline{\eps}$ is small enough then $f(\varphi) \in \mathcal{E}_k.$
\item[\rm{(3)}] \label{(3)} If $\overline{\eps}$ is small enough we have
$$\nabla_x \rho \in \mathcal{E}_1, \quad \partial_z \rho -1 \in \mathcal{E}_1, \quad \nabla_{x,z}^2 \rho \in \mathcal{E}_2, \quad \frac{1+ \vert \nabla_x \rho\vert^2- \partial_z \rho}{\partial_z \rho} \in \mathcal{E}_1.$$
\end{itemize}
\end{lemm}
\begin{proof}
Statement (1) follows  from point~(ii) in Proposition~\ref{uv} 
applied with $ s= t$  and from point~(iii) applied 
$s_0= t$ and $t$ replaced by $t +1$. 
Statement (2) is a consequence of  Proposition~\ref{est-f(u)bis}. 
The first three claims in point (3) follow directly from Lemma~\ref{est-rho}. 
For the last one   we notice that,
$$
q\defn\partial_z \rho -1  = \frac{1}{h}e^{z\vert D_x \vert}\eta + \frac{1}{h}(z+h)e^{z\vert D_x \vert}\vert D_x \vert\eta \in \mathcal{E}_1.
$$
Then we can write
$$
\frac{1+|\nabla_x\rho|^2 }{\partial_z\rho}= 1 - \frac{q}{1+q}+ \vert\nabla_x \rho\vert^2 - \vert\nabla_x \rho\vert^2  \frac{q}{1+q},
$$
and hence, the desired result follows from the previous statements.
\end{proof}

We shall need the following  extension of Lemma~\ref{algebre0}. We begin by a definition.
  
\begin{defi}\label{Frond}
For $k=1,2$   let $\mathcal{F}_k $ be the set of functions $\varphi(z, \eta)$ defined on,
$$
\mathcal{A}=  I_h\times \{\eta \in \mathcal{H}^{\lambda h, s+ \mez}:\Vert \eta \Vert_{\mathcal{H}^{\lambda h, s}} \leq \overline{\eps} \leq 1\},
$$
such that  $\varphi(z, 0)=0$ and such that 
for all\,  $\frac{d}{2}<t\leq s-k$ there exists $C>0$ such that 
the function $\Phi(z, \eta_1, \eta_2) = \varphi(z, \eta_1) -\varphi(z, \eta_2)$ satisfies the 
two following estimates
\begin{align*}
& \Vert \Phi(\cdot, \eta_1, \eta_2)\Vert_{L^2(I_h, \mathcal{H}^{\lambda h,t+1})}\\
\leq &\, C\Big(\Vert \eta_1- \eta_2\Vert_{\mathcal{H}^{\lambda h, t+k+\frac{1}{2}}}  
+ \sum_{j=1}^2\Vert \eta_j \Vert_{\mathcal{H}^{\lambda h, t+k+\frac{1}{2}}}
\Vert \eta_1- \eta_2\Vert_{\mathcal{H}^{\lambda h, t+k}}\Big),\\
& \Vert\Phi(\cdot, \eta_1, \eta_2)\Vert_{L^\infty(I_h, \mathcal{H}^{\lambda h,t})}
\leq C \Vert \eta_1- \eta_2\Vert_{\mathcal{H}^{\lambda h, t +k}}.
\end{align*}
\end{defi}
\begin{lemm}\label{algebre1}
\begin{itemize}
\item[\rm{(1)}] \label{(1i)}  For $k=1,2,$ $\mathcal{F}_k$ is an algebra 
and $\mathcal{F}_1 \subset \mathcal{F}_2.$   
\item[\rm{(2)}] \label{(2i)} If $\overline{\eps}$ is small enough, as functions of $(z, \eta)$, the functions
$$
\nabla_x \rho ,  \quad \vert \nabla_x \rho\vert^2, \quad q= \partial_z \rho -1, \quad f(q) = \frac{q}{1+q}, \quad \frac{1+ \vert \nabla_x \rho\vert^2}{\partial_z \rho}-1,
$$
belong to $\mathcal{F}_1,$ and the functions
$$
\nabla^2_x \rho, \quad \nabla_x\partial_z \rho, \quad \partial^2_z \rho
$$
belong to $\mathcal{F}_2.$ 
\end{itemize}
\end{lemm}
\begin{proof}
\begin{enumerate}[(1)]
\item It is obvious that $\mathcal{F}_1\subset \mathcal{F}_2$ so 
it is sufficient to prove that 
if $\varphi, \theta \in \mathcal{F}_k$ then, $\varphi \,\theta \in \mathcal{F}_k.$ 
Taking $\eta_1=\eta$ and $\eta_2=0$, 
we first notice that the hypotheses imply that, 
for every $ \eta \in \mathcal{H}^{\lambda h, s+ \mez}$ such that $\Vert \eta \Vert_{\mathcal{H}^{\lambda h, s}}\leq\overline{\eps}\leq 1, $  every $\frac{d}{2}<t\leq s-k$, we have
\begin{equation}\label{alg1}
 \Vert \theta(\cdot, \eta )  \Vert_{L^2(I_h, \mathcal{H}^{\lambda h,t +1})}\les  \Vert \eta   \Vert_{\mathcal{H}^{\lambda h, t+k+\frac{1}{2}}}, \quad \Vert \theta(\cdot, \eta )  \Vert_{L^\infty(I_h, \mathcal{H}^{\lambda h,t})}\les 1,
 \end{equation}
together with similar estimates for $\varphi.$ Then we write
\begin{align*}
 &(\varphi \,\theta)(z, \eta_1)-(\varphi \,\theta)(z, \eta_2)\\
 =&\, (\varphi  (z, \eta_1)- \varphi (z, \eta_2))\theta(z, \eta_1) + (\theta (z, \eta_1)- \theta (z, \eta_2))\varphi(z, \eta_2)\\
 &\qquad = A+B.
 \end{align*}
The   terms $A$ and $B$  are handled in the same way. Using point~(iii) 
in Proposition~\ref{uv} with $t=t +1, s_0= t>\frac{d}{2},$ we can write
\begin{align*}
&\Vert A\Vert_{L^2(I_h, \mathcal{H}^{\lambda h, t +1})} \\
\les &\,  \Vert\varphi  (\cdot, \eta_1)- \varphi (\cdot, \eta_2) \Vert_{L^2(I_h, \mathcal{H}^{\lambda h, t +1})} \Vert \theta(\cdot, \eta_1)\Vert_{L^\infty(I_h, \mathcal{H}^{\lambda h, t})}\\
&\qquad + \Vert\varphi(\cdot, \eta_1)- \varphi (\cdot, \eta_2)\Vert_{L^\infty(I_h, \mathcal{H}^{\lambda h, t})}
\Vert \theta(\cdot, \eta_2)\Vert_{L^2(I_h, \mathcal{H}^{\lambda h, t +1})}.
\end{align*}
Then using the hypotheses and \eqref{alg1} we can write
$$\Vert A\Vert_{L^2(I_h, \mathcal{H}^{\lambda h, t +1})}\les     \Vert \eta_1- \eta_2\Vert_{\mathcal{H}^{\lambda h, t+k+\frac{1}{2}}} + \sum_{j=1}^2\Vert \eta_j   \Vert_{\mathcal{H}^{\lambda h, t+k+\frac{1}{2}}}\Vert \eta_1- \eta_2\Vert_{\mathcal{H}^{\lambda h, t+k}} .$$
On the other hand, using point~(ii) in Proposition~\ref{uv} with $s$ replaced by $t$, we obtain
\begin{align*}
 \Vert A\Vert_{L^\infty(I_h, \mathcal{H}^{\lambda h, t})}
 &\les  \Vert\varphi  (z, \eta_1)- \varphi (z, \eta_2)\Vert_{L^\infty(I_h, \mathcal{H}^{\lambda h, t})} \Vert \theta(z, \eta_1)\Vert_{L^\infty(I_h, \mathcal{H}^{\lambda h, t})}\\
 &\les \Vert \eta_1- \eta_2\Vert_{\mathcal{H}^{\lambda h, t+k}}.
 \end{align*}
This completes the proof of the first statement.
 \item For the three first functions this follows from Lemma~\ref{est-rho} and statement~(1). For $f(q)$  this follows from   Corollary~\ref{est-f(u)ter} applied for fixed $z \in I_h$. Now one can write
 \begin{equation*}
  \frac{1+ \vert \nabla_x \rho\vert^2}{\partial_z \rho}    = 1+ A,\quad A = - f(q) - \vert\nabla_x \rho\vert^2 f(q) +  \vert\nabla_x \rho\vert^2 \in \mathcal{F}_1. 
  \end{equation*}
  The last claim follows from Lemma~\ref{est-rho}.
 \end{enumerate}
\end{proof} 
 \begin{lemm}\label{est-relevappendix}
 For all $\mu \in \xR$, 
 there exists $C>0$ such that for all $\sigma \geq 0$ and all  $\psi$ such that $a(D_x) ^\mez \psi \in \mathcal{H}^{\sigma,\mu}$,  there holds
\begin{align*}
  &\Vert \nabla_{x,z} \underline{\psi}\Vert_{L^2(I_h, \mathcal{H}^{\sigma,\mu})} \leq C \Vert  a(D_x)^\mez \psi\Vert_{ \mathcal{H}^{\sigma,\mu}},\\
  & \Vert \partial_z^2 \underline{\psi}\Vert_{L^2(I_h, \mathcal{H}^{\sigma,\mu-1})} \leq C \Vert  a(D_x)^\mez \psi\Vert_{ \mathcal{H}^{\sigma,\mu}},\\
  &\Vert \nabla_{x,z} \underline{\psi}\Vert_{L^\infty(I_h, \mathcal{H}^{\sigma,\mu-\mez})} \leq C \Vert  a(D_x)^\mez \psi\Vert_{ \mathcal{H}^{\sigma,\mu}}.
  \end{align*}
 \end{lemm}
\begin{proof} 
To obtain the first  estimates, we use the Fourier transform formula to express 
$\underline{\psi}$ in terms of $\psi$ (see \e{Fourierpsi}), and then the required estimates follow from arguments similar to the ones used 
in the proof of Lemma~\ref{est-rho}. The second one follows, since
$$\Vert \partial_z^2 \underline{\psi}\Vert_{F^{\lambda,\mu-1} } = \Vert \Delta_x \underline{\psi}\Vert_{F^{\lambda,\mu-1} }\leq \Vert \nabla_x \underline{\psi}\Vert_{F^{\lambda,\mu}} \leq C \Vert  a(D_x)^\mez\psi\Vert_{\mathcal{H}^{\sigma,\mu}}.$$
The third estimate is then 
a consequence of the interpolation argument given by Lemma~\ref{lions}.
  \end{proof}
 \subsection{Some estimates on the remainder.}
Consider a function $\eta=\eta(x)$ and defined $\rho=\rho(x,z)$ 
as in Section~\ref{defiCOV}. In this paragraph, we prove several estimates 
for the action of the operator $R$ defined by
\begin{equation}
\left\{
\begin{aligned}
&R =a\partial_z^2 +b\Delta_x  
+c\cdot\nabla_x\partial_z  -d\partial_z,  
\quad\text{where}\\
&a=\frac{1+\vert\nabla_x\rho\vert^2}{\partial_z\rho}-1,\quad 
b=\partial_z\rho-1,\quad c=-2\nabla_x\rho,\\
&d=\frac{1+\vert\nabla_x\rho\vert^2}{\partial_z\rho}\partial_z^2\rho+\partial_z\rho \Delta_x \rho -2\nabla_x\rho
\cdot\nabla_x\partial_z\rho.
  \end{aligned}
\right.
  \end{equation}
\begin{prop}\label{lemm:HFI0}
Consider two real numbers $s>d/2+2$ and $\lambda \in [0,1)$. 
There exist $\overline{\eps}>0$ and a constant $C>0$ such that 
for all $\eta \in \mathcal{H}^{\lambda h, s+\mez}$  satisfying $\Vert \eta \Vert_{\mathcal{H}^{\lambda h,s}} \leq \overline{\eps}$, the two following properties hold{\rm :} 
\begin{align}
\Vert R u\Vert_{F^{\lambda,s-2}}&\le 
C \overline{\eps} \big( \lA \nabla_{x,z}u\rA_{F^{\lambda,s-1}}+\lA \partial_z^2u\rA_{F^{\lambda,s-2}}\big),\label{estR:s-2}\\
\Vert Ru\Vert_{F^{\lambda,s-1}}&\le 
C \overline{\eps} \big( \lA \nabla_{x,z}u\rA_{F^{\lambda,s}}+\lA \partial_z^2u\rA_{F^{\lambda,s-1}}\big)+
C \Vert \eta \Vert_{  \mathcal{H}^{\lambda h, s+ \mez}}\lA \partial_z u\rA_{E^{\lambda,s-1}},\label{HF2-u2}\\
\Vert Ru\Vert_{F^{\lambda,s-1}}&\le 
C \overline{\eps} \big( \lA \nabla_{x,z}u\rA_{F^{\lambda,s}}+\lA \partial_z^2u\rA_{F^{\lambda,s-1}}\big)\label{HF2-u3}\\
&\quad+
C  \Vert \eta \Vert_{\mathcal{H}^{\lambda h, s+\mez}}\big(\Vert \partial_z^2 u\Vert_{F^{\lambda, s-\frac{5}{2}}} + \lA   \partial_z u\rA_{F^{\lambda,s-\frac{3}{2}}}\big).\notag
\end{align}
\end{prop}
 
\begin{proof}

We deduce from Lemma~\ref{algebre0} that
$$a, b, c \in \mathcal{E}_1,\quad  d \in \mathcal{E}_2.$$
 
 {\textbf{Step 1.}} We begin by studying $a\partial_z^2u+b\Delta_x u +c\cdot\nabla_x\partial_z u$. We will prove an estimate which holds for 
both cases. Namely, we claim that  for all $s-1\le \nu\le s$  we have
$$
\lA a\partial_z^2u+b\Delta_x u
+c\cdot\nabla_x\partial_z u\rA_{F^{\lambda,\nu-1}}\le C \overline{\eps}\lA \nabla_{x,z}^2u\rA_{F^{s,\nu-1}}.
$$
 Consider the first term $a\partial_z^2u$. Since $\frac{d}{2}<\nu-1\leq s-1$    the product rule given by statement~ (ii)  in Lemma~\ref{uvF}   implies that
\be\label{HF7-u}
\begin{aligned}
\lA a\partial_z^2u\rA_{F^{\lambda,\nu-1}}
&\leq C\lA a\rA_{E^{\lambda,s-1}}\lA\partial_z^2u\rA_{F^{\lambda,\nu-1}}\\
&\leq C\lA a\rA_{L^\infty(I_h,\mathcal{H}^{\lambda h,s-1})}\lA\partial_z^2u\rA_{F^{\lambda,\nu-1}},\\
&\leq C\overline{\eps}\lA\partial_z^2u\rA_{F^{\lambda,\nu-1}},
\end{aligned}
\ee 
since $a \in \mathcal{E}_1.$
By the same way,
\be\label{est-betaDelta} 
\begin{aligned}
\lA b \Delta_x u\rA_{F^{\lambda,\nu-1}}
&\leq C\lA b\rA_{E^{\lambda,s-1}}\lA\Delta_x u\rA_{F^{\lambda,\nu-1}}\\
&\leq C\lA b\rA_{L^\infty(I_h,\mathcal{H}^{\lambda h,s-1})}\lA \Delta_x u\rA_{F^{\lambda,\nu-1}},\\
&\leq C\overline{\eps}\lA \nabla_x u \rA_{F^{\lambda,\nu}},
\end{aligned}
\ee
since $b\in \mathcal{E}_1$
 and similarly, 
$$
\lA c\cdot\nabla_x\partial_z u\rA_{F^{\lambda,s-1}}\leq C  \overline{\eps}  \lA \partial_z u\rA_{F^{\lambda,\nu}}.
$$

 {\textbf{Step 2.}} We now estimate the $F^{\lambda,s-2}$ norm of $d \partial_z u$, 
using parallel arguments to those used above. 
Firstly, since $s-1>d/2$ and $2s-3\ge 0$, the product rule given by Lemma~\ref{uvF},  with $s_3= s-2, s_1= s-2, s_2 = s-1$ implies that
$$
\lA d \partial_z u\rA_{F^{\lambda,s-2}}\leq C 
\lA d \rA_{E^{\lambda,s-2}}\lA \partial_z u\rA_{F^{\lambda,s-1}}\leq C \overline{\eps}\lA \partial_z u\rA_{F^{\lambda,s-1}},
$$
since $d\in \mathcal{E}_2.$
Taking $\nu= s-1$ we obtain \eqref{estR:s-2}. 

{\textbf{Step 3.}} We now estimate the $F^{\lambda,s-1}$ norm of $d \partial_z u$. 
We proceed as in the previous step, but 
now the term $d $ is estimated in $F^{\lambda,s-1}$. This is here the only place where 
we are taking advantage of the special definition of $\rho$ (the so-called smoothing diffeomorphism). 
Namely, we write
\be\label{HF12-u-bis}
\begin{aligned}
\lA d \partial_z u\rA_{F^{\lambda,s-1}}&\leq C 
\lA d \rA_{F^{\lambda,s-1}}
\lA \partial_z u\rA_{E^{\lambda,s-1}}\leq C \lA d\rA_{L^2(I_h, \mathcal{H}^{\lambda,s-1})}
\lA \partial_z u\rA_{E^{\lambda,s-1}}\\
&\leq C\lA \eta \rA_{\mathcal{H}^{\lambda h,s+\mez}}\lA \partial_z u\rA_{E^{\lambda,s-1}}.
\end{aligned}
\ee
Using~\e{HF7-u}, \eqref{est-betaDelta}, \eqref{HF12-u-bis} with $\nu = s$, we complete the proof of   \e{HF2-u2}.

To obtain \eqref{HF2-u3} we use Lemma~\ref{uvF} (iii) with $s_0= s-2.$ We obtain
$$\lA d \partial_z u\rA_{F^{\lambda,s-1}}\leq C(\lA d  \rA_{E^{\lambda,s-2}} \lA   \partial_z u\rA_{F^{\lambda,s-1}} + \lA d \rA_{F^{\lambda,s-1}}\lA   \partial_z u\rA_{E^{\lambda,s-2}}).$$
 Since $d\in \mathcal{E}_2$ we have
 \begin{align*}
 \lA d  \rA_{E^{\lambda,s-2}}&\leq C \Vert \eta \Vert_{\mathcal{H}^{\lambda h, s}} \leq C \overline{\eps},\\
 \lA d  \rA_{F^{\lambda,s-1}}&\leq C \Vert \eta \Vert_{\mathcal{H}^{\lambda h, s+\mez}}.
 \end{align*}
On the other hand, by the interpolation lemma (see Lemma~\ref{uvF} (iv)) we have
$$
\lA   \partial_z u\rA_{E^{\lambda,s-2}} \leq C (\Vert \partial_z^2 u\Vert_{F^{\lambda, s-\frac{5}{2}}} + \lA   \partial_z u\rA_{F^{\lambda,s-\frac{3}{2}}}).
$$
Therefore,
$$
\lA d \partial_z u\rA_{F^{\lambda,s-1}}\leq C\overline{\eps}\lA   \partial_z u\rA_{F^{\lambda,s-1}}+ C\Vert \eta \Vert_{\mathcal{H}^{\lambda h, s+\mez}}(\Vert \partial_z^2 u\Vert_{F^{\lambda, s-\frac{5}{2}}} + \lA   \partial_z u\rA_{F^{\lambda,s-\frac{3}{2}}}).
$$
This completes the proof of \eqref{HF2-u3}.  
\end{proof}

\begin{coro}
Let $\underline{\psi}$ be  the lifting of the function $\psi$ defined in \eqref{relev}.
Under the assumptions of Lemma~{\rm \ref{lemm:HFI0}}, there exists a constant $C>0$ such that
\be\label{HF2(s-2)}
\Vert R\underline{\psi}\Vert_{F^{\lambda,s-2}}\le C\overline{\eps}\blA a(D_x)^\mez \psi\brA_{\mathcal{H}^{\lambda h,s-1}},
\ee
and
\be\label{HF2}
\Vert R\underline{\psi}\Vert_{F^{\lambda,s-1}}\le 
C \overline{\eps}\Vert a(D_x)^\mez \psi \Vert_{\mathcal{H}^{\lambda h,s}}+
C \Vert \eta \Vert_{  \mathcal{H}^{\lambda h, s+ \mez}}\Vert a(D_x)^\mez \psi \Vert_{\mathcal{H}^{\lambda h,s-\mez}}.
\ee
\end{coro}
\begin{proof}
Notice that $$\blA \nabla_{x,z}\underline{\psi}\brA_{F^{\lambda,\mu}}\le \blA \nabla_{x,z}\underline{\psi}\brA_{L^2_z(I_h,\mathcal{H}^{\lambda h,\mu})}$$ for any $\mu$ in $\xR$. By Lemma~\ref{est-relevappendix} we have, for any real number $\mu$,  
$$
\Vert \nabla_{x,z} \underline{\psi}\Vert_{E^{\lambda,\mu-\mez}}+ \Vert \partial_z^2\underline{\psi}\Vert_{F^{\lambda,\mu-1}}+ \blA \nabla_{x,z}\underline{\psi}\brA_{F^{\lambda,\mu}}\les \blA a(D_x)^\mez\psi\brA_{\mathcal{H}^{\lambda h,\mu}}.
$$
  Then \e{HF2(s-2)} and \e{HF2} follow from Proposition~\ref{lemm:HFI0}.
\end{proof}

\textbf{Acknowledgements} 
We would like to warmly thank the anonymous referee for his comments and extraordinarily 
attentive reading, which allowed us to significantly improve the presentation of the article.

\vspace{3mm}

\begin{flushleft}
\textbf{Thomas Alazard}\\
Universit{\'e} Paris-Saclay, ENS Paris-Saclay, CNRS,\\
Centre Borelli UMR9010, Avenue des Sciences,
F-91190 Gif-sur-Yvette, France,\\
and \\
Department of Mathematics, University of California, Berkeley\\
CA 94720, USA

\vspace{5mm}

\textbf{Nicolas Burq}\\
{Laboratoire de Math\'ematiques d'Orsay, CNRS,\\
Universit\'e Paris-Saclay,\\
B\^atiment 307, 
F-91405 Orsay\\
France,\\
and Institut Universitaire de France.}

\vspace{5mm}

\textbf{Claude Zuily}\\
{Laboratoire de Math\'ematiques d'Orsay, CNRS,\\
Universit\'e Paris-Saclay,\\
B\^atiment 307, 
F-91405 Orsay\\
France}

\end{flushleft}

 \end{document}